\documentclass[11pt,a4paper]{article}

\usepackage[T1]{fontenc}
\usepackage[utf8]{inputenc}
\usepackage[english]{babel}
\usepackage{lmodern}
\usepackage{amsmath,amssymb,amsthm,mathtools}
\usepackage{microtype}
\usepackage[a4paper,margin=3.1cm]{geometry}
\usepackage{enumitem}
\usepackage{booktabs}
\usepackage{xcolor}

\usepackage{tikz}
\usetikzlibrary{decorations.markings,arrows.meta, calc}
\tikzset{
  orient/.style={
    postaction={decorate},
    decoration={markings,
      mark=at position 0.2 with {\arrow[scale=1.1]{Stealth}},
      mark=at position 0.7 with {\arrow[scale=1.1]{Stealth}}}},
  orientb/.style={
    postaction={decorate},
    decoration={markings,
      mark=at position 0.15 with {\arrow[scale=1.1]{Stealth}},
      mark=at position 0.60 with {\arrow[scale=1.1]{Stealth}}}},
}

\usepackage[font=small,labelfont=bf]{caption}

\usepackage{hyperref}

\allowdisplaybreaks

\hypersetup{
  colorlinks=true,
  linkcolor=blue!55!black,
  citecolor=blue!55!black,
  urlcolor=blue!55!black
}

\newif\ifworknotes
\worknotesfalse

\newcommand{\R}{\mathbb{R}}
\newcommand{\C}{\mathbb{C}}
\newcommand{\Hh}{\mathbb{H}}
\newcommand{\Sphere}{\mathbb{S}}
\newcommand{\imq}{\operatorname{Im}}
\newcommand{\req}{\operatorname{Re}}
\newcommand{\dd}{\,\mathrm{d}}
\newcommand{\dsigma}{\,\mathrm{d}\sigma}
\newcommand{\darea}{\,\mathrm{d}A}
\newcommand{\ci}{\mathrm{i}}
\newcommand{\qi}{\mathbf{i}}
\newcommand{\qj}{\mathbf{j}}
\newcommand{\qk}{\mathbf{k}}

\DeclareMathOperator{\Res}{Res}
\DeclareMathOperator{\PRes}{PRes}
\DeclareMathOperator{\SRes}{SRes}
\DeclareMathOperator{\Area}{Area}

\newcommand{\Flux}[2]{\mathcal F_{#1}(#2)}
\newcommand{\DeltaSphere}[1]{\Delta_{#1}}

\newtheorem{theorem}{Theorem}[section]
\newtheorem{proposition}[theorem]{Proposition}
\newtheorem{lemma}[theorem]{Lemma}
\newtheorem{corollary}[theorem]{Corollary}

\theoremstyle{definition}
\newtheorem{definition}[theorem]{Definition}
\newtheorem{example}[theorem]{Example}

\theoremstyle{remark}
\newtheorem{remark}[theorem]{Remark}

\theoremstyle{plain}
\newtheorem{maintheorem}{Main Theorem}
\mathtoolsset{showonlyrefs}

\title{A four-dimensional residue calculus
for quaternionic slice regular functions}
\author{%
  Amedeo Altavilla\thanks{%
    Dipartimento di Matematica,
    Università degli Studi di Bari Aldo Moro,
    Via Edoardo Orabona 4, 70125 Bari, Italy.
    E-mail:
    \href{mailto:amedeo.altavilla@uniba.it}
      {\texttt{amedeo.altavilla@uniba.it}}.
    ORCID:
    \href{https://orcid.org/0000-0002-3290-371X}
      {\texttt{0000-0002-3290-371X}}.
}}
\date{}

\begin{document}

\maketitle

\begin{abstract}
We develop a four-dimensional residue calculus for quaternionic slice semiregular functions. Its local singularities are isolated real poles and two-dimensional spheres. For the latter we introduce a spherical residue and characterize it as a spherical Laurent coefficient, as the sum of two classical residues on any full complex slice, and by a finite differential procedure involving alternating spherical and slice derivatives. The residue theorem expresses a weighted boundary flux on a circular domain as the sum of the real and spherical residues it contains. Beyond this degree-zero formula, polynomially weighted flux moments recover every coefficient of every finite principal part, including the companion coefficient at a spherical pole that is invisible to ordinary residue summation. Once the polar components and order bounds are fixed, the quotient of the real polynomial ring by the global pole polynomial is the natural finite-dimensional algebra of weights. We prove that the associated residue pairing is well defined and nondegenerate, and identifies formal principal parts with the full quaternion-valued real dual of this algebra. Hermite interpolation and the Chinese remainder theorem provide explicit global weights selecting each local coefficient. We also separate the regularity domain of the function from the geometry of integration: for a smooth noncircular integration domain compactly contained in a circular ambient domain, away from the real axis and with boundary disjoint from the polar set, we establish an exact residue formula with an explicit volume correction governed by the Cauchy--Fueter operator. The theory includes higher-order poles, exceptional points, product domains, and zero divisors. Applications give formulas for rational and oscillatory integrals on the imaginary hyperplane and characterize the rational data defined on the three--dimensional imaginary subspace of the quaternions admitting a canonical slice regular extension.
\end{abstract}

\tableofcontents

\section{Introduction}\label{sec:introduction}

The residue theorem is one of the central local-to-global principles of
complex analysis. If \(D\subset\C\) is a bounded domain with positively
oriented boundary and \(h\) is meromorphic near \(\overline D\), with no
pole on \(\partial D\), then
\[
  \frac{1}{2\pi\ci}\int_{\partial D}h(z)\,\dd z
  =\sum_{a\in D}\Res_a h ,
\]
so that a global contour integral is determined by finitely many local
Laurent coefficients. Besides evaluating real and complex integrals, this
identity governs the behaviour of contour integrals under deformation and
underlies the argument principle, the counting of zeros and poles, and
several forms of holomorphic functional calculus; see for instance
\cite[Chapter~4]{Ahlfors} and \cite[Chapter~10]{Rudin}. Beyond the plane,
the notion of residue has taken several forms. On Riemann surfaces it is
attached to meromorphic differentials and obeys the global residue theorem
\cite[Theorem~10.21]{Forster}; in several complex variables it leads to
local and global multidimensional residues and, in algebraic geometry, to
residue duality \cite{Tsikh,Hartshorne}; and in quaternionic and Clifford
analysis the Cauchy--Fueter theory has its own Cauchy kernels and residue
formulas \cite[Section~9]{Sudbery}, 
\cite{BrackxDelangheSommen,DelangheSommenSoucek, GuerlebeckHabethaSprossig}. In
each of these settings the cycles, the kernels and the local singularity
data are dictated by the function theory at hand.

The function theory considered here is that of slice regular functions,
introduced in \cite{GentiliStruppa2007} and developed both in the original
quaternionic setting and, through stem functions, over real alternative
\(*\)-algebras \cite{GentiliStoppatoStruppa,GhiloniPerotti2011}. One of its
characteristic features is that it contains quaternionic polynomials and,
more generally, ordinary power series in the quaternionic variable,
\begin{equation}\label{eq:intro-power-series}
  f(q)=\sum_{n=0}^{\infty}q^na_n ,
  \qquad a_n\in\Hh ,
\end{equation}
with coefficients written on the right. This feature has proved useful well
beyond the function theory itself: slice regular functions have been related
to twistor geometry and orthogonal complex structures
\cite{GentiliSalamonStoppato, AltavillaTwistor,AltavillaSarfatti,LiuMorenoShi}
, to quaternionic manifolds and other questions
in differential geometry \cite{GentiliGoriSarfatti,GentiliGoriSarfattiToric,AngellaBisiHopf,GentiliGoriSarfattiAffine}, and to the continuous slice functional calculus and the
$S$-functional calculus for quaternionic operators
\cite{ColomboSabadiniStruppa, ColomboSabadiniStruppaFunctional,
GhiloniMorettiPerotti,ColomboGantnerKimsey}.

The family of slice regular functions carries a natural product. On
polynomials and formal power series with coefficients on the right it is the
one obtained by letting the variable commute with the coefficients,
\[
  \Bigl(\sum_nq^na_n\Bigr)*\Bigl(\sum_mq^mb_m\Bigr)
  =\sum_k q^k\sum_{n+m=k}a_nb_m ,
\]
and it extends to all slice functions as the pointwise product of their stem
functions \cite{GhiloniPerotti2011,GPS-algebra}. The resulting algebra is
associative and non-commutative; the two factors do commute
whenever one of them is slice preserving, in which case \(*\) agrees with
the pointwise product. It has zero divisors, a feature that will matter
later.
Power series centred at a non-real quaternion display a genuinely
non-commutative phenomenon. The natural expansion
\(\sum_{n\geq0}(q-q_0)^{*n}a_n\) converges on balls for the
\(\sigma\)-distance (see \cite{GentiliStoppatoStruppa}), and such a ball need not be a Euclidean neighbourhood
of its centre: the topology naturally attached to these expansions is finer
than the Euclidean one \cite[Sections~2.1--2.3]{GentiliStoppatoStruppa}. To
recover expansions on Euclidean open sets, Stoppato introduced spherical
series and the corresponding notion of spherical analyticity
\cite{StoppatoSeries}. If \(\Sigma=\alpha+\Sphere\beta\) is the conjugacy
sphere through a real centre (i.e., the set of all \(\alpha+I\beta\in\Hh\) such that \(\alpha,\beta\in\mathbb{R}\) and \(I\) varies in the set of quaternionic imaginary units), and
\[
  \DeltaSphere{\Sigma}(q)=q^2-2\alpha q+\alpha^2+\beta^2 ,
\]
then a regular function is locally represented on an open set (namely, a Cassini neighbourhood), by
a series of the form
\begin{equation}\label{eq:intro-spherical-series}
  f(q)=\sum_{k\geq0}\DeltaSphere{\Sigma}(q)^k\bigl(qu_k+v_k\bigr) ,
  \qquad u_k,v_k\in\Hh .
\end{equation}
The construction was extended to singularities and to real alternative
\(*\)-algebras in \cite{GhiloniPerottiPowerSpherical,GPSSingularities,GhiloniPerottiStoppatoDivision}. A further step, crucial here, is
that the coefficients of \eqref{eq:intro-spherical-series} can be recovered
constructively from a finite jet, by alternating the so--called spherical derivative
\(\partial_s\) and the slice  derivative \(\partial_c\)
\cite[Theorem~4.1]{AltavillaSphericalCoefficients}.
The two operators probe complementary directions. At a non-real point
\(q=\alpha+I\beta\), the tangent space of \(\Hh\) splits orthogonally into
the tangent space of the slice \(\C_I\), spanned by \(1\) and \(I\), and
that of the conjugacy sphere \(\alpha+\Sphere\beta\). The slice derivative \(\partial_c\)
governs the behaviour of \(f\) along the first, being the holomorphic
derivative of the restriction \(f_I\), while the spherical derivative
\(\partial_s\) governs it along the second, measuring the variation of \(f\)
on the sphere through \(q\), on which it is constant. Alternating them is
thus a way of reading the two-variable jet of \(f\) at the sphere one
direction at a time.

This differential access to spherical coefficients raises the question of
whether the complex residue theorem has a four-dimensional counterpart for
slice regular functions, and the question has to be posed in terms of the
actual geometry of their singularities. A semiregular function can have
isolated real poles and non-real polar two-spheres
\(\Sigma=\alpha+\Sphere\beta\); along such a sphere the generic pole order is
constant, and at most one exceptional point may have lower order, possibly
zero \cite[Theorems~9.4--9.5 and Remark~11.2]{GPSSingularities}. Allowing
negative indices in \eqref{eq:intro-spherical-series}, we single out the
coefficient \(u_{-1}\) of the resulting paired Laurent expansion and call it
the \emph{spherical residue} \(\SRes_\Sigma f\). Restricting \(f\) to a slice \(\C_I\) makes the sphere \(\Sigma\) visible as
the conjugate pair \(\alpha\pm I\beta\), and one of the results proved below
is that the two ordinary point residues there always add up to the spherical
residue:
\begin{equation}\label{eq:intro-paired-residue}
  \SRes_\Sigma f
  =\Res_{\alpha+I\beta}^{I}f_I+\Res_{\alpha-I\beta}^{I}f_I
  \qquad\text{for every }I\in\Sphere .
\end{equation}
Each summand depends on the chosen slice, but with our normalization their
sum does not. After
multiplication by a sufficient power of \(\DeltaSphere{\Sigma}\), the same
quantity is a confluent divided difference at alternating conjugate nodes,
and hence an alternating \(\partial_s\)--\(\partial_c\) derivative of the
regularized function.

Each of the ingredients used here is, on its own, well established: the
theory of divided differences and Hermite interpolation
\cite{deBoor,LopezTemme}, two-point expansions, and the differential
recovery of spherical coefficients
\cite[Theorem~4.1]{AltavillaSphericalCoefficients}. What the present
construction contributes is the identification of a single local invariant
in three languages at once. The spherical residue is
\begin{enumerate}[label=\textup{(\roman*)}]
  \item the sum of the two point residues on any full slice, which connects
        it to the classical residue theorem;
  \item a balanced confluent divided difference of a common intrinsic
        regularization, which makes it depend on a finite jet only;
  \item an alternating differential expression in that regularization,
        which makes it computable.
\end{enumerate}
The identification survives when the two slice poles have different orders,
and when the common regularizing exponent overshoots at one of them, and it
is this threefold correspondence that supplies the local mechanism of the
four-dimensional theorem.

The paper belongs to a four-dimensional line of development in slice
analysis, which includes volume Cauchy formulas on circular domains
\cite{GhiloniPerottiVolume}, a global Borel--Pompeiu-type formula associated
with a differential characterization of slice regularity
\cite{GonzalezCervantesGonzalezCampos}, and a local Cauchy formula whose
integration domain need not be axially symmetric \cite{PerottiLocalCauchy}.
Four-dimensional Jensen formulas, in which the integration takes place on a
three-sphere and zeros and poles enter through their quaternionic geometry,
were obtained first for slice-preserving functions
\cite{AltavillaBisiJensen} and then for general slice regular and
semiregular functions \cite{PerottiJensen}. In the same spirit, the left-hand
side of our formula is an integral over a three-dimensional hypersurface of
\(\Hh\), and its right-hand side is indexed by the real points and
two-spheres that form the polar set.

In compact form, the main result reads as follows; the complete statement,
with the differential evaluation of every residue and the precise regularity
assumptions, is Theorem~\ref{thm:4d-residue}.

\begin{maintheorem}\label{thm:intro-main}
Let \(U\) be an admissible circular open set and let \(f\) be semiregular on
a circular neighbourhood of \(\overline U\), with no polar component meeting
\(\partial U\). If \(a_1,\ldots,a_\ell\) are the real poles and
\(\Sigma_1,\ldots,\Sigma_s\) the non-real polar spheres contained in \(U\),
then
\begin{equation}\label{eq:intro-main-residue-formula}
  \boxed{
  \frac{1}{4\pi^2}\int_{\partial U\setminus\R}
    \frac{n(q)f(q)}{|\imq q|^2}\,\dsigma(q)
  =\sum_{i=1}^{\ell}\Res_{a_i}f+\sum_{j=1}^{s}\SRes_{\Sigma_j}f .
  }
\end{equation}
\end{maintheorem}

Here \(n\) is the outward Euclidean unit normal, regarded as a quaternion,
and \(\dsigma\) is three-dimensional surface measure. The order of the
factors in \(n(q)f(q)\) is part of the statement. On each slice \(\C_I\) the
quaternions form a left \(\C_I\)-vector space, and it is with respect to this
structure that the boundary integral reduces to the contour integral
\(\int\dd z\,f_I(z)\); the opposite order corresponds to the right module
structure and in general gives a different functional, the two agreeing for
slice-preserving functions. The weighted integral converges absolutely, and
both kinds of residue on the right can be evaluated from finite jets after
multiplication by intrinsic polynomial factors, whose exponents need not be
minimal. The theorem covers product domains, domains disjoint from the real
axis, exceptional points on polar spheres, and non-zero slice functions whose
symmetrization vanishes identically: since its proof never divides by a
value of \(f\) or of its symmetrization, the zero divisors of the algebra of
slice functions cost no cases.

Formula \eqref{eq:intro-main-residue-formula} is the first term of a larger
residue calculus. Suppose, for example, that near
\(\Sigma=\alpha+\Sphere\beta\) the principal part is finite,
\[
  \sum_{m=1}^{M}\DeltaSphere{\Sigma}(q)^{-m}(qu_{-m}+v_{-m}) .
\]
If \(T_\Sigma\) is a sufficiently small circular tube around the sphere,
Theorem~\ref{thm:4d-principal-part-recovery} gives, for \(0\leq r<M\),
\begin{equation}\label{eq:intro-principal-part-moments}
  \Flux{T_\Sigma}{\DeltaSphere{\Sigma}^{\,r}f}=u_{-r-1} ,
  \qquad
  \Flux{T_\Sigma}{(q-2\alpha)\DeltaSphere{\Sigma}^{\,r}f}=v_{-r-1} ,
\end{equation}
and there is an analogous formula at a real pole. Finitely many boundary
moments thus recover the \emph{complete} principal part, and each moment has an
explicit finite-jet expression in alternating spherical and slice
derivatives.

So far the weights have been local: each moment in
\eqref{eq:intro-principal-part-moments} is computed on a tube surrounding a
single sphere. On a domain containing several polar components one would
rather integrate once, over \(\partial U\), and still read off a prescribed
coefficient of a prescribed component. This requires a weight that isolates
that coefficient and annihilates every other, and such weights turn out to
form a finite-dimensional algebra determined by the polar data alone.

Indeed, let the pole locations and order bounds determine the real
polynomial
\[
  Q(X)=\prod_i(X-a_i)^{L_i}\prod_j\DeltaSphere{\Sigma_j}(X)^{M_j} ,
\]
and consider the quotient \(\mathcal A_Q:=\R[X]/(Q)\). A \emph{formal
principal part} \(\mathfrak p\) subordinate to these data is a choice of
finitely many quaternionic coefficients, one family at each real pole and
two at each polar sphere, within the orders allowed by \(Q\); we write
\(R_{\mathfrak p}\) for the rational slice function having exactly
\(\mathfrak p\) as principal part and no other singularity. Two polynomial
weights differing by a multiple of \(Q\) produce the same flux, the
difference cancelling the whole principal part, so that \(\mathcal A_Q\)
records exactly the polynomial jets that can interact with such data.
Hermite interpolation and the Chinese remainder theorem split this algebra
into its real and spherical local factors, and thereby produce global
weights isolating any prescribed principal coefficient. The residue pairing
\[
  \bigl\langle[P],\mathfrak p\bigr\rangle_{\mathrm{res}}
  :=\Flux{U}{P(q)R_{\mathfrak p}(q)}
\]
is therefore well defined on \(\mathcal A_Q\), and
Theorem~\ref{thm:4d-residue-duality} states that the induced map from formal
principal parts to quaternion-valued real-linear functionals on
\(\mathcal A_Q\) is an isomorphism: no principal-part information is lost by
passing to fluxes, and none is redundant. Equivalently, writing
\(R_{\mathfrak p}=Q^{-1}N_{\mathfrak p}\) with
\(\deg N_{\mathfrak p}<\deg Q\), the pairing takes the algebraic form
\[
  \bigl\langle[P],\mathfrak p\bigr\rangle_{\mathrm{res}}
  =
  \lambda_Q\bigl([PN_{\mathfrak p}]\bigr),
\]
where \(\lambda_Q\) extracts the coefficient of
\(X^{\deg Q-1}\) from the remainder modulo \(Q\).
It is therefore the perfect Frobenius residue pairing of the Artinian
Gorenstein algebra \(\R[X]/(Q)\), extended to quaternionic coefficients.
This provides a concrete one-variable realization of the algebraic residue
formalism underlying residue duality \cite{Hartshorne}, while the
Hermite--Chinese-remainder selectors give explicit weights recovering each
individual principal coefficient.

The main theorem integrates over the boundary of a circular set, and it is
natural to ask how much of it survives on an arbitrary hypersurface. The
answer is that circular symmetry is used once only, and that its failure can
be measured exactly. Note that only the integration domain is at issue here:
\(f\) remains semiregular on a circular ambient open set \(\Omega\), as it
must be for its polar components to be spheres in the first place.

Let then \(W\) be a bounded smooth domain with
\(\overline W\subset\Omega\setminus\R\), whose boundary avoids every polar
sphere; since a polar sphere is connected, each of them is then either
contained in \(W\) or disjoint from it. Excising small circular tubes around
those contained in \(W\), and calling \(W_{\mathrm{exc}}\) the remaining
domain, Theorem~\ref{thm:4d-noncircular-residue} gives the exact identity
\begin{equation}\label{eq:intro-noncircular-formula}
  \frac{1}{4\pi^2}\int_{\partial W}
    \frac{n(q)f(q)}{|\imq q|^2}\,\dsigma(q)
  +\frac{1}{2\pi^2}\int_{W_{\mathrm{exc}}}
    \frac{I(q)f_s^\circ(q)}{|\imq q|^3}\,\dd V_4(q)
  =\sum_{\Sigma_j\subset W}\SRes_{\Sigma_j}f .
\end{equation}
The second term is a genuine volume integral, governed by the
Cauchy--Fueter operator, and it vanishes when \(W\) is circular because the
integrand is then odd in the angular variable. This is where the hypothesis
of circularity was doing its work, and \eqref{eq:intro-noncircular-formula}
isolates it: everything else in the theorem is geometry that any smooth
hypersurface enjoys. One case remains outside the statement, namely a polar
sphere crossing \(\partial W\), where the boundary integral itself is
singular and some principal-value or indentation prescription would be
needed.

The proof combines three mechanisms. First, the boundary of a circular set
splits into its planar profile and the sphere of imaginary units; the factor
\(|\imq q|^2\) appearing in the boundary Jacobian is cancelled by the weight
of \eqref{eq:intro-main-residue-formula}, and the flux becomes a
vector-valued contour integral on any full slice. Second, the ordinary
complex residue theorem evaluates that contour integral. Third, the residues
at the two conjugate intersection points of each polar sphere are grouped by
\eqref{eq:intro-paired-residue} and identified with a confluent divided
difference, which produces the differential formula. One-variable complex analysis is an essential ingredient: once the weighted
boundary functional has been reduced to a full slice, the planar residue
theorem is applied there without modification. This step alone, however, does
not yield the four-dimensional statement. The construction of the boundary
functional, the organization of the singularities into intrinsic polar
components, and the slice-independent identification of the two point
residues with a single spherical residue all depend on the quaternionic slice
structure and on the spherical Laurent theory developed in this paper. The
genuinely four-dimensional content therefore lies both in the boundary
functional, and in its localization on the intrinsic singularity data of \(f\),
namely real poles and polar two-spheres.

For Cauchy--Fueter regular functions there is a classical theory of Cauchy
and residue formulas \cite{Sudbery,GuerlebeckHabethaSprossig,BrackxDelangheSommen,DelangheSommenSoucek}, whose
singularities are organized differently: the classification by real poles,
conjugacy spheres, exceptional points and spherical Laurent coefficients is
specific to the slice setting. Closer to the present circle of ideas, in what were his last works Dolbeault
proposed residue currents associated with inverses of quaternionic
functions, together with a related local theory
\cite{DolbeaultQuaternionic,DolbeaultLocal}, in the setting of a
Cauchy--Fueter-type hyperholomorphy. The comparison shows how
closely a residue theorem is tied to the singularity geometry of its
regularity class.

As a first application of our theory, tangent four-balls turn
\eqref{eq:intro-main-residue-formula} into formulas for absolutely convergent
integrals on the imaginary hyperplane \(\imq\Hh\simeq\R^3\). Within the class
of rational data considered, the restrictions admitting a slice regular
extension are exactly the spherically compatible functions
\[
  \varphi(v)=\mathsf a(|v|^2)+v\,\mathsf b(|v|^2) ,
  \qquad v\in\imq\Hh ,
\]
with quaternion-valued rational \(\mathsf a,\mathsf b\) satisfying suitable
pole and decay conditions; Remark~\ref{rem:rat-obstruction} shows that some
natural rational data fall outside this class. The family was chosen because
it makes the residue mechanism, and its oscillatory version, transparent. A
natural open problem is to characterize the polynomial or rational functions
on an arbitrary affine hyperplane of \(\Hh\) that arise as restrictions of
slice regular or semiregular functions.

Several further directions would require substantial additional
preliminaries and are left for future work: an operator-valued version
related to the $S$-functional calculus
\cite{GantnerDirect,ColomboGantnerKimsey} and to its spectral projectors
and spectral measures
\cite{AlpayColomboGantnerSabadini,
AlpayColomboKimseySabadiniUnitary,
GhiloniMorettiPerottiPVM}; a comparison with the quaternionic argument
principle \cite{VlacciArgument,GentiliStoppatoStruppa} and with the
four-dimensional Jensen formulas mentioned above; the extension of the
local residue construction and, where a suitable boundary geometry can be
developed, of the weighted flux theorem to other hypercomplex slice
theories, notably slice monogenic functions on Clifford algebras
\cite{ColomboSabadiniStruppaSliceMonogenic} and slice regular functions
over real alternative $*$-algebras, including the octonions
\cite{GhiloniPerotti2011,GhiloniPerottiPowerSpherical,
GhiloniPerottiStoppatoDivision,GPSSingularities}; and a transfer of the
construction from circular stem domains to local or non-axially symmetric
frameworks, such as the local Cauchy theory of
\cite{PerottiLocalCauchy}, the slice-topological approach of
\cite{DouJinRenSabadini}, and the theory beyond slice domains of
\cite{GhiloniStoppatoBeyond}. In the more general algebraic settings, the
available spherical expansions and singularity theory provide the local
ingredients, whereas the geometry of the conjugacy sets, the boundary
weight and, in the nonassociative case, the order and parenthesization of
the products would have to be reconsidered. The boundary regularity and
decay hypotheses used in our proofs are chosen to keep the geometric
arguments, the absolute convergence and the limiting procedures
transparent, and we have not tried to make them minimal.

The paper is organized as follows. The three sections following this
introduction are preparatory and closely linked, each setting up what the
next one needs. Most of their ingredients are classical, and we indicate the
sources as they appear; they are recalled in some detail because the main
proof uses them in a precise form not usually presented as a single package.
For the same reason some standard proofs are retained: they display the
normalizations, orientations, orders of multiplication and interpolation
steps that the quaternionic argument reuses, and they make it clear where the
four-dimensional construction begins.

Section~\ref{sec:complex-model} builds the one-variable blueprint. It
rewrites the complex residue theorem as a normal-flux identity, fixes the
left-sided convention for vector-valued integration, and isolates the data
carried by a pair of conjugate poles, allowing them to have different orders
and to be regularized by one common, possibly nonminimal, exponent.
Reflection, tangent disks and model integrals are included because the same
geometric operations reappear in the four-dimensional reduction.

Section~\ref{sec:divided-differences} provides the algebraic bridge between
the complex and the quaternionic parts. After recalling confluent divided
differences and Hermite interpolation, it specializes them to alternating
conjugate nodes, in which form paired residues become balanced confluent
divided differences and then differential evaluations.

Section~\ref{sec:slice-preliminaries} transports this mechanism to slice
analysis. Besides fixing the conventions on stem functions and derivatives,
it records the features that prevent a formal transcription of complex
analysis: a full slice may be disconnected on a product domain, the
evaluation of a complexified quaternionic stem is complex linear but not
multiplicative, slice functions may be zero divisors, and a polar sphere may
carry an exceptional point of lower pole order. Keeping track of these
features is what allows the later statements to hold simultaneously on
slice and product domains, without inverses and without assuming equal pole
orders at the two slice points.

Section~\ref{sec:spherical-residues} develops the local theory on which the
main theorem rests. Starting from Cassini neighbourhoods, it proves a paired
spherical Laurent expansion that remains meaningful when the two point
singularities visible on a slice have different orders. The real residue is
defined as usual, and the spherical residue is selected from the first of the
two families of paired coefficients. We prove that it equals the sum of the
two point residues on every full slice and that, after multiplication by any
sufficiently high power of the characteristic polynomial, it is both a
balanced confluent divided difference and an alternating expression in the
spherical and slice derivatives, with a value independent of the exponent,
of the representative and of the slice. Removable exceptional points,
regularizations vanishing on the sphere and zero divisors are all covered,
since no division by values of the function or of its symmetrization takes
place. The method is then extended to all paired Laurent coefficients and
read componentwise in terms of the intrinsic meromorphic coordinates of the
stem.

Section~\ref{sec:weighted-flux} constructs the global functional that
detects these local coefficients. It introduces admissible planar profiles
and their circularizations, allowing disconnected sets, product domains and
domains with holes, and studies the regularity and orientation of their
boundaries. The boundary measure factors into a profile part and a spherical
part, and this explains both the constant \(4\pi^2\) and the weight
\(|\imq q|^{-2}\), which cancels the quadratic area factor of the conjugacy
spheres and makes the apparent singularity along the real axis integrable.
The main reduction theorem then identifies the weighted four-dimensional flux
with the vector-valued contour integral on any full slice, a geometric
statement requiring only a continuous slice boundary datum. Direct
computations on balls and on tubes around a non-real sphere check the
normalization independently and show that the flux extracts \(u_{-1}\) (and
not \(v_{-1}\)). A weighted Cauchy--Fueter divergence identity, due in its
unweighted form to Perotti, closes the section and shows where circularity
of the domain enters.

Section~\ref{sec:residue-theorem} assembles the local and the global
constructions. The weighted flux is reduced to a contour integral on a full
slice, the classical residue theorem evaluates it through point residues,
and conjugate point residues are grouped into intrinsic spherical residues;
this proves the four-dimensional residue formula, together with finite
differential evaluations of all its local terms. Local extraction formulas
follow, a sufficiently small admissible boundary around one polar component
recovering exactly its real or spherical residue, and lead to a
local-to-global decomposition and to domain independence. The section then
separates circularity of the function domain from circularity of the
integration domain. On a smooth noncircular integration set away from the
real axis, whose boundary avoids the polar spheres, the failure of angular
cancellation is measured by an explicit Cauchy--Fueter volume correction;
this gives an exact corrected residue formula and also explains why a sphere
cut by the boundary is outside its scope.

The rest of Section~\ref{sec:residue-theorem} goes beyond the degree-zero
sum. Weighting the flux with polynomials recovers every coefficient of a
finite principal part, one family at a real pole and two at a spherical one,
and each of them again has a finite differential formula, valid at
exceptional points and for nonminimal regularizing exponents. These local
selectors are then assembled into a global one. We take as algebra of
weights the quotient of the real polynomial ring by the pole polynomial, use
Hermite interpolation and the Chinese remainder theorem to produce a basis
of selector classes in it, and show that the residue pairing identifies
formal principal parts with the whole quaternion-valued real dual of that
algebra. Two examples close the section: one recovers the companion
coefficient of a simple spherical pole, the other separates a real pole from
a spherical one by means of global weights.

Section~\ref{sec:examples} checks the theorem against computations carried
out independently of it.
Whenever possible, the boundary integrals are computed directly and only then
compared with the theorem's predictions. The first example checks
additivity for two polar spheres carrying non-commuting coefficients; the
second combines spherical order four with an exceptional point, showing why
the alternating differential evaluation is needed; and the third concerns a
non-zero slice function on a product domain whose symmetrization vanishes
identically, on which arguments based on slice reciprocals would be
unavailable. Shorter checks cover a real pole of higher order, a vanishing
spherical residue at a genuine polar sphere, and a function combining
contributions of all three kinds.

Finally, Section~\ref{sec:hyperplane} applies the theorem to integrals over
the imaginary hyperplane \(\imq\Hh\simeq\R^3\). It first characterizes the
rational data on this hyperplane admitting a rational slice extension, the
spherically compatible functions formed by a radial even part and a radial
coefficient of the imaginary variable, and relates the absolute
integrability of the datum to the decay of its extension. A family of
tangent four-balls then exhausts the right half-space, and the residue
theorem, applied to the auxiliary function \(q^2f(q)\), evaluates the
hyperplane integral through the real and spherical poles lying there.
Multiplication by the intrinsic factor \(e^{-tq}\) yields an oscillatory
version whose values are finite quaternionic exponential--polynomial
expressions in the parameter. Explicit families illustrate higher spherical
order, exceptional points, degeneration to a real pole, and the interaction
of several spheres with non-commuting coefficients. Since the imaginary unit
in the phase varies with the point, the oscillatory integral differs from
the three-dimensional Fourier transform with a fixed frequency.
\section{The complex one-variable model}\label{sec:complex-model}

The aim of this section is to work out, in the simpler setting of one complex variable, the mechanism that will later produce the four-dimensional residue theorem.

Throughout the section, $V$ denotes a finite-dimensional complex vector space,
with complex scalars acting on the left. After a basis of $V$ has been chosen,
a function $h\colon\Omega\to V$ can be written as
\[
  h(z)=\sum_{\ell=1}^{N}h_\ell(z)e_\ell .
\]
We say that $h$ is holomorphic, or meromorphic, if all its scalar
components $h_\ell$ are so; derivatives, Laurent coefficients and
integrals are likewise defined componentwise. None of these notions
depends on the choice of basis. Moreover, since $V$ is
finite-dimensional, componentwise convergence coincides with
convergence in any norm on $V$.

If $\gamma\colon[a,b]\to\Omega$ is a piecewise $C^1$ curve, we set
\begin{equation}\label{eq:complex-vector-line-integral}
  \int_\gamma \dd z\,h(z)
  :=\int_a^b \gamma'(t)h(\gamma(t))\,\dd t.
\end{equation}
\begin{remark}
\label{rem:complex-left-differential}
 At a point \(z\), the differential \(\dd z\) is a
complex-valued covector. Since \(V\) is regarded as a left complex
vector space, it determines the \(V\)-valued one-form
\[
  \omega_z(\xi):=(\dd z)_z(\xi)\,h(z),
  \qquad \xi\in T_z\Omega .
\]
Its pullback along \(\gamma\) satisfies
\[
  (\gamma^*\omega)_t(\partial_t)
  =\gamma'(t)h(\gamma(t)),
\]
which is precisely the integrand in
\eqref{eq:complex-vector-line-integral}. Thus the position of
\(\dd z\) records which scalar action is being used.

This becomes essential when \(V=\Hh\), which is the case the whole paper
leads to and which we anticipate here only to show the convention at work.
Fix \(I\in\Sphere\) and regard \(\Hh\) as a left vector space over
\(\C_I\). If
\(J\in\Sphere\) is orthogonal to \(I\) and \(h=h_1+h_2J\), with
\(h_1,h_2\colon\Omega\to\C_I\), then
\[
  \int_\gamma\dd z\,h(z)
  =\int_\gamma\dd z\,h_1(z)
   +\left(\int_\gamma\dd z\,h_2(z)\right)J .
\]
Hence the integral is exactly the componentwise complex integral to
which Cauchy's theorem and the residue theorem apply.

This is also the module convention underlying the Cauchy formulas for
left slice regular functions: in the planar formula the boundary
differential acts on the left of the boundary value
\cite[Theorem~27]{GhiloniPerotti2011}, while the corresponding
four-dimensional volume formula contains the ordered product
\(C_{\Sphere}(x,w)n(w)f(w)\)
\cite[Theorem~2.3 and Corollary~2.7]{GhiloniPerottiVolume}.
A closely related interpretation in terms of the real differential of a
slice regular primitive acting on quaternion-valued differential forms is
given in \cite[Remark~17]{PerottiLocalCauchy}.

The expression \(\int_\gamma h(z)\,\dd z\) would instead use right
multiplication by elements of \(\C_I\). The quaternionic algebra does
provide such an action, but it is different from the left action used
above. For example, if \(J\perp I\), then on the positively oriented
unit circle in \(\C_I\),
\[
  (2\pi I)^{-1}
  \int_{|z|=1}\dd z\,z^{-1}J=J,
  \qquad
  \int_{|z|=1}z^{-1}J\,\dd z=0.
\]
Thus reversing the order can even destroy the residue contribution.
Our convention is the one compatible with power series having
coefficients on the right and with the left slice-holomorphic
structure used throughout the paper.
\end{remark}

\subsection{Complex line integrals and normal flux}
\label{subsec:complex-normal-flux}
\begin{definition}\label{def:admissible-domain}
A bounded open set $D\subset\C$, not necessarily connected, is called
\emph{admissible} if
\begin{enumerate}[label=(\roman*),leftmargin=*]
  \item $\partial D$ is the union of finitely many pairwise disjoint
    piecewise $C^1$ Jordan curves $\gamma_1,\dots,\gamma_n$;
  \item $D$ lies locally on one side of its boundary: every
    $p\in\partial D$ has a neighbourhood $U$ such that exactly one
    connected component of $U\setminus\partial D$ is contained in $D$.
\end{enumerate}
\end{definition}

Condition~(ii) is meaningful because of (i): if $U$ is small enough, it
meets only one of the curves $\gamma_j$, so that $U\setminus\partial D$ has
exactly two components, each of which is contained either in $D$ or in
$\C\setminus\overline{D}$, and at least one of which is contained in $D$.
Condition~(ii) thus rules out the remaining case, namely that both
components lie in $D$.

Under (i) and (ii), each $\gamma_j$ carries a well-defined orientation
\emph{positive relative to $D$}: the one for which $D$ stays on the left as
the curve is traversed, so that an exterior boundary runs counterclockwise
and the boundary of a hole clockwise (see Figure \ref{fig:admissible-domain}). Unless stated otherwise, $\partial D$
is always understood as the cycle $\gamma_1+\dots+\gamma_n$ with these
orientations, and we write
\[
  \int_{\partial D}:=\sum_{j=1}^{n}\int_{\gamma_j}.
\]

\begin{remark}
Condition~(ii) does not follow from~(i). Let
$D=\{|z|<2\}\setminus\{|z|=1\}$ (see Figure \ref{fig:one-sidedness-fails}): its boundary is the union of the two
disjoint smooth Jordan curves $\{|z|=1\}$ and $\{|z|=2\}$, so (i) holds,
whereas (ii) fails at every point of the inner circle, both nearby
components lying in $D$. Indeed, $D$ has the two connected components
$U_1=\{|z|<1\}$ and $U_2=\{1<|z|<2\}$, and $\{|z|=1\}$ occurs once in
$\partial U_1$, counterclockwise, and once in $\partial U_2$, clockwise;
viewing $\partial D$ as a family of curves, each with a single orientation,
therefore discards one of the two copies. The failure is already visible at
the level of boundary values: the function equal to $0$ on $U_1$ and to $1$
on $U_2$ is holomorphic on $D$ and extends continuously to each component
separately, but has no well-defined boundary value on $\{|z|=1\}$.
\end{remark}
\begin{figure}[htbp]
\centering
\begin{tikzpicture}[scale=0.65]
  \fill[even odd rule,black!8]
    plot[smooth cycle,tension=0.7] coordinates
      {(-4,0) (-2.7,-2) (0,-2.7) (2.6,-2.2) (4.2,0) (2.8,1.9) (0,2.6) (-2.6,2.1)}
    (-2.0,-0.3) circle (0.95)
    plot[smooth cycle,tension=0.7] coordinates
      {(1.0,0.4) (1.7,1.3) (2.7,0.9) (2.8,-0.1) (1.8,-0.6)};

  \draw[thick,orient] plot[smooth cycle,tension=0.7] coordinates
      {(-4,0) (-2.7,-2) (0,-2.7) (2.6,-2.2) (4.2,0) (2.8,1.9) (0,2.6) (-2.6,2.1)};

  \draw[thick,orient] (-2.0,-0.3) ++(90:0.95) arc(90:-270:0.95);

  \draw[thick,orient] plot[smooth cycle,tension=0.7] coordinates
      {(1.0,0.4) (1.7,1.3) (2.7,0.9) (2.8,-0.1) (1.8,-0.6)};

  \node at (-2.0,-0.3) {$\gamma_2$};
  \node at (1.95,0.45) {$\gamma_3$};
  \node at (-0.15,2.15) {$\gamma_1$};
  \node at (-0.1,-1.6) {$D$};
\end{tikzpicture}
\caption{An admissible domain $D$ with $\partial D=\gamma_1+\gamma_2+\gamma_3$,
each component positively oriented relative to $D$.}
\label{fig:admissible-domain}
\end{figure}
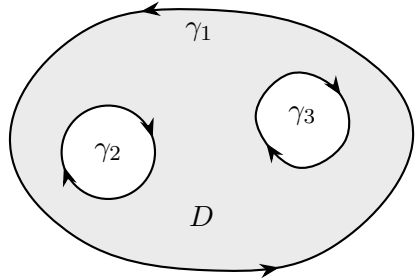
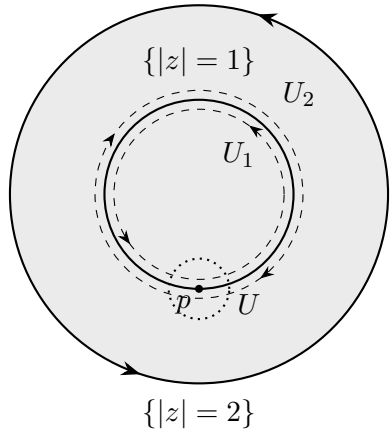
\begin{figure}[htbp]
\centering
\begin{tikzpicture}[scale=1.25]
  \fill[black!8] (0,0) circle (2);

  \draw[thick,orient] (0,0) ++(0:2) arc(0:360:2);
  \draw[thick] (0,0) circle (1);

  \draw[dashed,orientb] (0,0) ++(0:0.90) arc(0:360:0.90);
  \draw[dashed,orientb] (0,0) ++(0:1.10) arc(0:-360:1.10);

  \fill (0,-1) circle (1.2pt);
  \draw[dotted,thick] (0,-1) circle (0.32);
  \node[right] at (0.30,-1.16) {$U$};
  \node[below left] at (0.02,-0.96) {$p$};

  \node at (0.42,0.42) {$U_1$};
  \node at (1.06,1.06) {$U_2$};
  \node[below] at (0,-2.06) {$\{|z|=2\}$};
  \node[above] at (0,1.14) {$\{|z|=1\}$};
\end{tikzpicture}
\caption{The set $D=\{|z|<2\}\setminus\{|z|=1\}$ satisfies (i) but not (ii).
Both components of $U\setminus\partial D$ lie in $D$, for every neighbourhood
$U$ of every point $p$ of the inner circle; accordingly, $\{|z|=1\}$ has to be
traversed twice, counterclockwise as part of $\partial U_1$ and clockwise as
part of $\partial U_2$ (dashed).}
\label{fig:one-sidedness-fails}
\end{figure}

Let $p\in\partial D$ be a \emph{regular point}, that is, a point at which the
boundary curve through $p$ is differentiable with non-vanishing derivative;
all but finitely many boundary points are regular. At such a $p$ we denote by
\[
  \nu(p)=n_1(p)+\ci n_2(p)
\]
the outward unit normal and by $\tau(p)$ the unit tangent, both viewed as
complex numbers, the latter taken in the direction of the positive
orientation of $\partial D$.

\begin{lemma}\label{lem:complex-tangent-normal}
At every regular point of $\partial D$, we have that
  \(\tau=\ci\nu\).
Consequently, denoting by $s$ the arclength parameter,
\begin{equation}\label{eq:complex-dz-normal}
  \dd z=\ci\nu\dd s .
\end{equation}
\end{lemma}

\begin{proof}
Multiplication by $\ci$ is the rotation by $\pi/2$ in the counterclockwise
direction. Since $D$ lies on the left of its positively oriented boundary,
the outward normal points to the right of $\tau$, that is, $\nu$ is obtained
from $\tau$ by a clockwise rotation by $\pi/2$. Rotating back gives
the first equality (see Figure~\ref{fig:complex-tangent-normal}).

For \eqref{eq:complex-dz-normal}, let $\gamma$ be a parametrization by
arclength inducing the positive orientation. Then $|\gamma'(s)|=1$ and
$\gamma'(s)=\tau(\gamma(s))$, whence
\[
  \dd z=\gamma'(s)\dd s=\tau(\gamma(s))\dd s=\ci\,\nu(\gamma(s))\dd s .
\]
\end{proof}
\begin{figure}[htbp]
\centering
\begin{tikzpicture}[scale=1.5]
  \fill[black!8] (-1.5,-0.80) .. controls (-0.75,0.15) and (-0.35,-0.22) ..
                 (0,0) .. controls (0.45,0.14) and (0.85,0.75) ..
                 (1.5,0.72) -- (1.5,1.25) -- (-1.5,1.25) -- cycle;
  \draw[very thick,-{Stealth}] (-1.5,-0.80) .. controls (-0.75,0.15) and (-0.35,-0.22) ..
                 (0,0) .. controls (0.45,0.14) and (0.85,0.75) .. (1.5,0.72);

  \fill (0,0) circle (0.8pt);
  \draw[-{Stealth},thick] (0,0) -- (0.86,0.26) node[above right] {$\tau$};
  \draw[-{Stealth},thick] (0,0) -- (0.26,-0.86) node[below right] {$\nu$};
  \draw (0.13,-0.43) arc(-73:17:0.45);
  \node[right,font=\small] at (0.47,-0.13) {$\pi/2$};

  \node at (-0.75,0.85) {$D$};
  \node[above left] at (-0.02,0.02) {$p$};
  \node[below,font=\small] at (-1.15,-0.62) {$\partial D$};
\end{tikzpicture}
\caption{At a regular point $p$ of $\partial D$, the domain lies on the left
of the positive direction of travel, so that the outward normal $\nu$ points
to the right of the tangent $\tau$. Multiplication by $\ci$, the
counterclockwise rotation by $\pi/2$, therefore takes $\nu$ to $\tau$.}
\label{fig:complex-tangent-normal}
\end{figure}
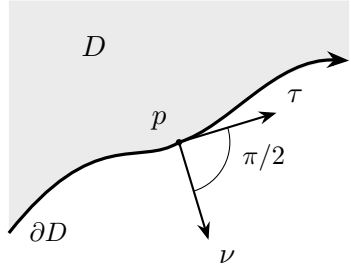
\begin{definition}\label{def:complex-flux}
Let $h$ be continuous on $\partial D$. The \emph{normalized (planar) flux} of $h$
through $\partial D$ is
\begin{equation}\label{eq:complex-flux-definition}
  \Flux{D}{h}:=\frac{1}{2\pi}\int_{\partial D}\nu(z)h(z)\dd s(z),
\end{equation}
where $\dd s$ is the arclength element.
\end{definition}

Substituting \eqref{eq:complex-dz-normal} into
\eqref{eq:complex-flux-definition} yields
\[
  \int_{\partial D}\dd z\,h(z)=\ci\int_{\partial D}\nu(z)h(z)\dd s(z),
\]
so that the flux can equivalently be computed as a contour integral:
\begin{equation}\label{eq:complex-flux-contour}
  \Flux{D}{h}=\frac{1}{2\pi\ci}\int_{\partial D}\dd z\,h(z).
\end{equation}
The factor $1/(2\pi)$ in \eqref{eq:complex-flux-definition} is chosen
precisely so that \eqref{eq:complex-flux-contour} carries the same
normalization as the Cauchy integral formula.



Recall the Wirtinger operator
\[
  \partial_{\bar z}
  :=\frac12\left(\frac{\partial}{\partial x}
                  +\ci\frac{\partial}{\partial y}\right),
  \qquad z=x+\ci y.
\]

\begin{proposition}[Green--Cauchy identity]\label{prop:complex-green-cauchy}
If $h$ is of class $C^1$ on a neighbourhood of $\overline D$, then
\begin{equation}\label{eq:complex-green-cauchy}
  \int_{\partial D}\nu(z)h(z)\dd s(z)=2\int_D\partial_{\bar z}h(z)\darea(z).
\end{equation}

In particular, $\Flux{D}{h}=0$ whenever $h$ is holomorphic near
$\overline D$.
\end{proposition}

\begin{proof}
Write $\nu=n_1+\ci n_2$. The divergence theorem, applied
componentwise to the $V$-valued fields $(h,0)$ and $(0,h)$, gives
\[
  \int_{\partial D}n_1h\,\dd s
  =\int_D\frac{\partial h}{\partial x}\,\darea,
  \qquad
  \int_{\partial D}n_2h\,\dd s
  =\int_D\frac{\partial h}{\partial y}\,\darea.
\]
Multiplying the second identity by $\ci$ and adding yields
\[
  \int_{\partial D}(n_1+\ci n_2)h\,\dd s
  =\int_D\left(\frac{\partial h}{\partial x}
          +\ci\frac{\partial h}{\partial y}\right)\darea
  =2\int_D\partial_{\bar z}h\,\darea.
\]
This proves \eqref{eq:complex-green-cauchy}. If $h$ is holomorphic,
the Cauchy--Riemann equations are precisely $\partial_{\bar z}h=0$,
component by component.
\end{proof}

\begin{remark}[The flux as a coefficient selector]
\label{rem:complex-circle-selector}
The normalization of Definition~\ref{def:complex-flux} is already explained
by the case of a disc. Let $a\in\C$, $r>0$, $v\in V$ and $k\in\mathbb Z$,
and consider the function $h(z)=(z-a)^kv$ on the circle
$C_r(a)=\partial B(a,r)$. Parametrizing $C_r(a)$ by $z=a+re^{\ci\theta}$ one
has $\nu=e^{\ci\theta}$ and $\dd s=r\dd\theta$, whence
\begin{equation}\label{eq:complex-circle-selector}
  \Flux{B(a,r)}{h}
  =\frac{r^{k+1}}{2\pi}\int_0^{2\pi}e^{\ci(k+1)\theta}\dd\theta\,v
  =\begin{cases}
     v,& k=-1,\\
     0,& k\neq-1.
   \end{cases}
\end{equation}
The flux thus selects the coefficient of $(z-a)^{-1}$, and does so
independently of $r$. Note that for $k\leq-1$ the function $h$ is not of
class $C^1$ on a neighbourhood of $\overline{B(a,r)}$, so
Proposition~\ref{prop:complex-green-cauchy} does not apply: the value $v$
obtained for $k=-1$ is exactly the obstruction measured by the residue.
\end{remark}

\subsection{The residue theorem in flux form}
\label{subsec:complex-residue-flux}

Suppose that $h$ is meromorphic near a point $a$, so that on a punctured
disc about $a$ it admits a Laurent expansion
\begin{equation}\label{eq:complex-laurent}
  h(z)=\sum_{n\geq-m}(z-a)^nc_n,
  \qquad c_n\in V,
\end{equation}
for some $m\in\mathbb N$. As usual, we set $\Res_ah:=c_{-1}$, and
$\Res_ah:=0$ when $h$ is holomorphic at $a$.

The next statement is the classical residue theorem, applied componentwise
and rewritten through \eqref{eq:complex-flux-contour}; for the scalar case
we refer to \cite[Chapter~10]{Rudin} or \cite[Chapter~4]{Ahlfors}. We
present it as a sum of residues along a contour, following
\cite[Section~9, equations~(50)--(51)]{deBoor}, since this is the form
needed for divided differences, and we include a proof because it is the argument, rather than the statement, that the
quaternionic setting will reuse.

\begin{theorem}[Residue theorem in normal form]
\label{thm:complex-residue-flux}
Let $D$ be an admissible domain in the sense of
Definition~\ref{def:admissible-domain}, and let $h$ be meromorphic on an
open neighbourhood of $\overline D$, with no pole on $\partial D$. Then
\begin{equation}\label{eq:complex-residue-flux}
  \boxed{\;
  \Flux{D}{h}
  =\frac{1}{2\pi\ci}\int_{\partial D}\dd z\,h(z)
  =\sum_{a\in\mathcal P(h)\cap D}\Res_a h,
  \;}
\end{equation}
where $\mathcal P(h)$ denotes the set of poles of $h$.
\end{theorem}

\begin{proof}
Since $\overline D$ is compact and poles are isolated, $h$ has finitely many
poles $a_1,\dots,a_N$ in $D$. Choose $\varepsilon_j>0$ so small that the
closed discs $\overline{B(a_j,\varepsilon_j)}$ are pairwise disjoint,
contained in $D$, and free of poles other than their centres, and set
\[
  D_\varepsilon:=D\setminus\bigcup_{j=1}^{N}\overline{B(a_j,\varepsilon_j)} .
\]
Then $D_\varepsilon$ is again as in Definition~\ref{def:admissible-domain},
and $h$ is holomorphic on a neighbourhood of $\overline{D_\varepsilon}$.

Each $C_j:=\partial B(a_j,\varepsilon_j)$ is positively oriented relative to
$B(a_j,\varepsilon_j)$, but bounds a hole of $D_\varepsilon$, hence carries
the opposite orientation as a component of $\partial D_\varepsilon$ (see
Figure~\ref{fig:complex-excision}). Since the flux is additive over
boundary components and changes sign with the orientation,
\begin{equation}\label{eq:complex-excision}
  \Flux{D_\varepsilon}{h}
  =\Flux{D}{h}-\sum_{j=1}^{N}\Flux{B(a_j,\varepsilon_j)}{h},
\end{equation}
and the left-hand side vanishes by
Proposition~\ref{prop:complex-green-cauchy}.

Finally, the Laurent series of $h$ at $a_j$ converges uniformly on $C_j$ and
may therefore be integrated term by term; by
Remark~\ref{rem:complex-circle-selector} every term gives $0$ except the one
of exponent $-1$, so that $\Flux{B(a_j,\varepsilon_j)}{h}=\Res_{a_j}h$.
Substituting into \eqref{eq:complex-excision} proves the second equality in
\eqref{eq:complex-residue-flux}; the first one is
\eqref{eq:complex-flux-contour}.
\end{proof}
\begin{figure}[htbp]
\centering
\begin{tikzpicture}[scale=0.65]
  \def\outerpath{
    plot[smooth cycle,tension=0.7] coordinates
      {(-4,0) (-2.7,-2) (0,-2.7) (2.6,-2.2) (4.2,0) (2.8,1.9) (0,2.6)
       (-2.6,2.1)}}

  \begin{scope}[even odd rule]
    \fill[black!8] \outerpath
      (-2.0,-0.3) circle (0.95)
      plot[smooth cycle,tension=0.7] coordinates
        {(1.0,0.4) (1.7,1.3) (2.7,0.9) (2.8,-0.1) (1.8,-0.6)}
      (-0.6,1.15) circle (0.34)
      (0.35,-1.35) circle (0.34)
      (3.05,-0.75) circle (0.34);
  \end{scope}

  \draw[thick,orient] \outerpath;
  \draw[thick,orient] (-2.0,-0.3) ++(90:0.95) arc(90:-270:0.95);
  \draw[thick,orient] plot[smooth cycle,tension=0.7] coordinates
      {(1.0,0.4) (1.7,1.3) (2.7,0.9) (2.8,-0.1) (1.8,-0.6)};

  \foreach \c in {(-0.6,1.15),(0.35,-1.35),(3.05,-0.75)}{
    \draw[thick,orient] \c ++(90:0.34) arc(90:-270:0.34);}
  \fill \pgfextra{} (-0.6,1.15) circle (1.3pt);
  \fill (0.35,-1.35) circle (1.3pt);
  \fill (3.05,-0.75) circle (1.3pt);

  \node[above] at (-0.6,1.54) {$a_1$};
  \node[below] at (0.35,-1.74) {$a_2$};
  \node[right] at (3.44,-0.75) {$a_3$};
  \node at (-0.6,-0.45) {$D_\varepsilon$};
  \node at (-2.0,-0.3) {$\gamma_2$};
  \node at (1.95,0.45) {$\gamma_3$};
  \node at (-0.15,2.15) {$\gamma_1$};
\end{tikzpicture}
\caption{The excised domain $D_\varepsilon$ obtained from the admissible
domain of Figure~\ref{fig:admissible-domain} by removing three closed discs
around the poles $a_1,a_2,a_3$ of $h$. The original boundary components keep
their orientations, while the three new circles bound holes of
$D_\varepsilon$ and are therefore traversed clockwise, which is the sign in
\eqref{eq:complex-excision}.}
\label{fig:complex-excision}
\end{figure}
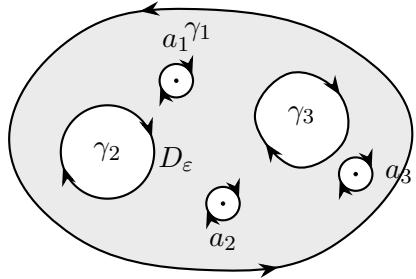

The following standard formula for the residue at a pole of higher order is
also the one-node confluent case of the divided-difference identities of
\cite[Section~1, equation~(14), and Section~9, equation~(51)]{deBoor}.

\begin{corollary}[Residue at a pole of higher order]
\label{cor:complex-higher-order-residue}
Let $m\geq1$ and suppose that $g(z):=(z-a)^mh(z)$ extends holomorphically
across $a$. Then
\begin{equation}\label{eq:complex-higher-order-residue}
  \Res_ah=\frac{1}{(m-1)!}\,g^{(m-1)}(a).
\end{equation}
\end{corollary}

Note that $m$ is not required to be the exact order of the pole of $h$ at
$a$: any larger exponent will do, and also the case in which $h$ is
holomorphic at $a$ is included.

\begin{proof}
Expanding $g$ in Taylor series at $a$ and dividing by $(z-a)^m$ gives the
Laurent expansion
\[
  h(z)=\sum_{k\geq0}\frac{1}{k!}(z-a)^{k-m}g^{(k)}(a) ,
\]
in which the exponent $-1$ occurs exactly for $k=m-1$. This proves
\eqref{eq:complex-higher-order-residue}; if $m$ exceeds the exact order of
the pole, the first Taylor coefficients of $g$ vanish, which does not affect
the one of index $m-1$.
\end{proof}

\begin{example}\label{ex:complex-higher-order}
Fix $a\in\C$ and $v\in V$, and let
\[
  h(z)=\frac{e^z}{(z-a)^3}v ,
\]
so that $g(z)=(z-a)^3 h(z)=e^zv$ and $m=3$. By
\eqref{eq:complex-higher-order-residue}, $\Res_ah=\frac{1}{2}e^av$, and
therefore, if $a\in D$ and $h$ has no other pole in $D$,
\[
  \frac{1}{2\pi}\int_{\partial D}\nu(z)\frac{e^z}{(z-a)^3}v\dd s(z)
  =\frac{1}{2}e^av .
\]
\end{example}

\subsection{A conjugate pair with unequal pole orders}
\label{subsec:complex-conjugate-pair}
A conjugate pair of non-real points \(\alpha\pm\ci\beta\) is the trace, on a
complex slice, of the conjugacy sphere \(\alpha+\Sphere\beta\) of
quaternions having real part \(\alpha\) and imaginary part of modulus
\(\beta\): accordingly, the functionals introduced
here are the ones that are symmetric in the pair, and they are the ones that
will later be attached to the sphere.
Let
\[
  p=\alpha+\ci\beta,\qquad \bar p=\alpha-\ci\beta,\qquad \beta>0,
\]
so that $p$ and $\bar p$ are distinct non-real conjugate points, and
introduce the real quadratic polynomial
\begin{equation}\label{eq:complex-characteristic-polynomial}
  \DeltaSphere{p}(z):=(z-p)(z-\bar p)=z^2-2\alpha z+\alpha^2+\beta^2 .
\end{equation}
Its coefficients $2\alpha=2\req p$ and $\alpha^2+\beta^2=|p|^2$ depend on
$p$ only through the conjugacy class $\{p,\bar p\}$.

From now on, $h$ is assumed to be meromorphic near both $p$ and $\bar p$;
the two pole orders need \emph{not} be equal, and a point at which $h$ is
holomorphic is regarded as a pole of order zero.

\begin{definition}[Residue of a conjugate pair]
\label{def:complex-paired-residue}
The \emph{paired residue} of $h$ at $\{p,\bar p\}$ is
\begin{equation}\label{eq:complex-paired-residue}
  \PRes_{\{p,\bar p\}}h:=\Res_ph+\Res_{\bar p}h .
\end{equation}
\end{definition}

The next proposition computes \eqref{eq:complex-paired-residue} by clearing
both poles at once, with a single exponent $M$ at the two nodes. This is the
finite principal-part counterpart of the two-point Laurent expansions of
\cite[Section~3, Theorem~2 and Proposition~3]{LopezTemme}, and it is a
direct application of Corollary~\ref{cor:complex-higher-order-residue}.

\begin{proposition}[A common regularizing exponent]
\label{prop:complex-common-regularization}
Let $M\geq1$ be an integer at least as large as both pole orders, and set
\begin{equation}\label{eq:complex-pair-regularization}
  g_M(z):=\DeltaSphere{p}(z)^Mh(z).
\end{equation}
Then $g_M$ extends holomorphically across $p$ and $\bar p$, and
\begin{equation}\label{eq:complex-pair-derivative-formula}
  \PRes_{\{p,\bar p\}}h
  =\frac{1}{(M-1)!}\left[
      \frac{\dd^{\,M-1}}{\dd z^{M-1}}\frac{g_M(z)}{(z-\bar p)^M}
      \bigg|_{z=p}
     +\frac{\dd^{\,M-1}}{\dd z^{M-1}}\frac{g_M(z)}{(z-p)^M}
      \bigg|_{z=\bar p}
   \right].
\end{equation}
Equivalently, if $\gamma_p$ and $\gamma_{\bar p}$ are small positively
oriented circles about the two nodes, enclosing no other pole of $h$, then
\begin{equation}\label{eq:complex-pair-contour-formula}
  \PRes_{\{p,\bar p\}}h
  =\frac{1}{2\pi\ci}\int_{\gamma_p+\gamma_{\bar p}}
    \dd z\,\frac{g_M(z)}{\DeltaSphere{p}(z)^M},
\end{equation}
or, in terms of fluxes,
\begin{equation}\label{eq:complex-pair-flux-formula}
  \PRes_{\{p,\bar p\}}h
  =\Flux{B(p,\varepsilon)}{h}+\Flux{B(\bar p,\varepsilon)}{h}
\end{equation}
for every $\varepsilon>0$ small enough that the closed discs
$\overline{B(p,\varepsilon)}$ and $\overline{B(\bar p,\varepsilon)}$ are
disjoint and contain no pole of $h$ other than their centres.
\end{proposition}

\begin{proof}
Away from the two nodes one has
$h(z)=g_M(z)\big/\big((z-p)^M(z-\bar p)^M\big)$, so that multiplication by
$(z-p)^M$ leaves $g_M(z)/(z-\bar p)^M$, which is holomorphic near $p$
because $p\neq\bar p$. Corollary~\ref{cor:complex-higher-order-residue},
applied with the exponent $M$, gives the first summand in
\eqref{eq:complex-pair-derivative-formula}; the same argument at $\bar p$
gives the second. Both \eqref{eq:complex-pair-contour-formula} and
\eqref{eq:complex-pair-flux-formula} follow from
Theorem~\ref{thm:complex-residue-flux} applied to each of the two circles,
once one observes that $g_M/\DeltaSphere{p}^M=h$ off the two nodes.
\end{proof}

The contour in \eqref{eq:complex-pair-contour-formula} is a union of two
curves; the two nodes need not belong to the same connected neighbourhood.
\begin{remark}[Looking ahead: a divided difference]
\label{rem:complex-pair-divided-difference}
Two features of \eqref{eq:complex-pair-contour-formula} will matter later.
The contour is a union of two curves, so the two nodes need not lie in a
common connected neighbourhood on which \(h\) is meromorphic; and the
right-hand side is a confluent divided difference
\[
  g_M\big[\underbrace{p,\ldots,p}_{M},
          \underbrace{\bar p,\ldots,\bar p}_{M}\big] ,
\]
as Section~\ref{sec:divided-differences} will show. It is this
coefficient-extraction formula, and not the Laurent expansion behind it,
that survives in the slice-regular theory.
\end{remark}

For higher-order poles, another elementary paired functional will be useful.
Write the Laurent expansions of $h$ at the two nodes as
\[
  h(z)=\sum_{n\geq-m_+}(z-p)^nc_n^+\ \text{ near }p,
  \qquad
  h(z)=\sum_{n\geq-m_-}(z-\bar p)^nc_n^-\ \text{ near }\bar p,
\]
with the convention that $c_n^{\pm}:=0$ whenever the corresponding term is
absent, and set $r_+:=\Res_ph=c_{-1}^+$ and $r_-:=\Res_{\bar p}h=c_{-1}^-$.

A conjugate pair carries two independent pieces of local data, whereas
\(\PRes\) returns a single quaternion. A second functional is therefore
needed, and it is obtained by weighting \(h\) before taking residues. The
weight must be symmetric in the pair, so that the result belongs to the
sphere rather than to a chosen representative, and the simplest one of
degree one is \(z-2\alpha=z-(p+\bar p)\), whose coefficient, like those of
\(\DeltaSphere{p}\), depends only on \(\{p,\bar p\}\). Its relevance lies in
the two identities
\begin{equation}\label{eq:complex-weight-at-nodes}
  z-2\alpha=(z-p)-\bar p=(z-\bar p)-p ,
\end{equation}
by which the weight takes the value \(-\bar p\) at \(p\) and \(-p\) at
\(\bar p\), exchanging the two nodes.

\begin{proposition}[A first paired moment]\label{prop:complex-first-moment}
Define
\begin{equation}\label{eq:complex-first-moment-definition}
  \mathcal M_{\{p,\bar p\}}(h)
  :=\Res_p\Bigl((z-2\alpha)h(z)\Bigr)
   +\Res_{\bar p}\Bigl((z-2\alpha)h(z)\Bigr).
\end{equation}
Then
\begin{equation}\label{eq:cx-first-moment}
  \mathcal M_{\{p,\bar p\}}(h)
  =-\bar p\,r_+-p\,r_-+c_{-2}^++c_{-2}^- .
\end{equation}
\end{proposition}

In few words, \(\mathcal M\) is the paired residue of \(h\) weighted by
\(z-2\alpha\): the same operation as \(\PRes\), performed after a
degree-one weighting. In the language of \eqref{eq:complex-pair-flux-formula},
where \(\PRes\) was the total flux through two small circles,
\(\mathcal M\) is the corresponding first moment about the midpoint
\(\alpha\) of the pair, which is why we call it so. Formula
\eqref{eq:cx-first-moment} shows what it returns: the two residues again,
now weighted by the exchanged nodes, together with the two coefficients of
index \(-2\), which \(\PRes\) did not see at all.

\begin{proof}
Near $p$, the first identity in \eqref{eq:complex-weight-at-nodes} gives
\begin{align*}
  (z-2\alpha)h(z)
  &=\bigl((z-p)-\bar p\bigr)\sum_{n\geq-m_+}(z-p)^nc_n^+\\
  &=\sum_{n\geq-m_+}(z-p)^{n+1}c_n^+
   -\bar p\sum_{n\geq-m_+}(z-p)^nc_n^+\\
   &=\sum_{k\geq-m_++1}(z-p)^k\bigl(c_{k-1}^+-\bar p\,c_k^+\bigr),
\end{align*}
where in the last line the first sum has been reindexed by $k=n+1$ and the
second by $k=n$, and where $c_{-m_+-1}^+=0$ by our convention. Reading off
the coefficient of index $k=-1$ we obtain
\[
  \Res_p\bigl((z-2\alpha)h(z)\bigr)=c_{-2}^+-\bar p\,c_{-1}^+
  =-\bar p\,r_++c_{-2}^+ .
\]
The second identity in \eqref{eq:complex-weight-at-nodes} yields in the same
way
\[
  (z-2\alpha)h(z)
  =\sum_{k\geq-m_-+1}(z-\bar p)^k\bigl(c_{k-1}^--p\,c_k^-\bigr)
  \quad\text{near }\bar p,
\]
whence $\Res_{\bar p}\bigl((z-2\alpha)h(z)\bigr)=-p\,r_-+c_{-2}^-$. Adding
the two identities proves \eqref{eq:cx-first-moment}.
\end{proof}

\begin{remark}\label{rem:complex-first-moment-correction}
If both poles are at most simple, then $c_{-2}^+=c_{-2}^-=0$ and the first
moment is determined by the two residues alone. At higher-order poles the
coefficients of index $-2$ are genuine correction terms: unlike the
situation of Example~\ref{ex:complex-higher-order}, where the flux ignored
every coefficient but the one of index $-1$, the weight $z-2\alpha$ makes
one further coefficient visible. This is what will produce the second
coefficient of a spherical Laurent expansion.
\end{remark}

\begin{example}[Orders two and one]\label{ex:complex-unequal-orders}
Consider the scalar rational function
\begin{equation}\label{eq:complex-unequal-orders-function}
  h(z)=\frac{z^2}{(z-\ci)^2(z+\ci)},
\end{equation}
with a pole of order two at $p=\ci$ and a simple pole at $\bar p=-\ci$.
Here $V=\C$.

By Corollary~\ref{cor:complex-higher-order-residue},
\[
  \Res_{\ci}h
  =\frac{\dd}{\dd z}\left(\frac{z^2}{z+\ci}\right)\bigg|_{z=\ci}
  =\left(\frac{2z}{z+\ci}-\frac{z^2}{(z+\ci)^2}\right)\bigg|_{z=\ci}
  =1-\frac14=\frac34 ,
\]
while at the simple pole
\[
  \Res_{-\ci}h=\frac{z^2}{(z-\ci)^2}\bigg|_{z=-\ci}
  =\frac{-1}{(-2\ci)^2}=\frac14 ,
\]
so that
\begin{equation}\label{eq:complex-unequal-paired-residue}
  \PRes_{\{\ci,-\ci\}}h=1 .
\end{equation}

A common regularizing exponent must satisfy $M\geq2$. Since
$\DeltaSphere{\ci}(z)=z^2+1$, the choice $M=2$ gives
\begin{equation}\label{eq:complex-unequal-regularization}
  g_2(z)=(z^2+1)^2h(z)=z^2(z+\ci),
\end{equation}
whence $g_2(\ci)=-2\ci$ and $g_2(-\ci)=0$. The regularization thus vanishes
at the node carrying the lower-order pole, although both $\Res_{-\ci}h$ and
the paired residue are nonzero. With the common
exponent $M=2$ the residue at $-\ci$ is read off from a derivative of $g_2$. Indeed, writing
$g_2(z)/(z-\ci)^2=(z+\ci)\,z^2/(z-\ci)^2$ and differentiating, the term in
which $(z+\ci)$ survives undifferentiated vanishes at $z=-\ci$, so that only
\[
  \frac{\dd}{\dd z}\left(\frac{g_2(z)}{(z-\ci)^2}\right)\bigg|_{z=-\ci}
  =\frac{z^2}{(z-\ci)^2}\bigg|_{z=-\ci}=\frac14
\]
remains, in accordance with
\eqref{eq:complex-pair-derivative-formula}. This is the planar precursor of
an exceptional point on a polar sphere.

For completeness, the coefficient of index $-2$ at $\ci$ is
\[
  c_{-2}^+=\frac{z^2}{z+\ci}\bigg|_{z=\ci}=\frac{\ci}{2},
\]
whereas $c_{-2}^-=0$. Since $\alpha=0$, the weight of
\eqref{eq:complex-first-moment-definition} is $z$ itself, and
\eqref{eq:cx-first-moment} gives
\[
  \mathcal M_{\{\ci,-\ci\}}(h)
  =\frac34\ci-\frac14\ci+\frac{\ci}{2}=\ci ,
\]
the two residues alone accounting for $\ci/2$ and the correction $c_{-2}^+$
for the other half.
\end{example}

\subsection{Reflection}
\label{subsec:complex-reflection}

Pairing two conjugate points does not by itself impose any relation between
the two germs of a meromorphic function. Such a relation appears only after a reflection symmetry has been
specified, and the symmetry involves the target space as much as the
domain.

\begin{definition}[Reflected function]\label{def:complex-reflection}
Let $\Omega\subset\C$ be invariant under complex conjugation, and let
$c\colon V\to V$ be an antilinear involution:
\[
  c(\lambda v)=\bar\lambda\,c(v),
  \qquad
  c(c(v))=v
\]
for $\lambda\in\C$ and $v\in V$. For a function $h\colon\Omega\to V$, define
\begin{equation}\label{eq:complex-reflected-function}
  h^\#(z):=c\bigl(h(\bar z)\bigr).
\end{equation}
\end{definition}

The involution $c$ amounts to a choice of real form of $V$: the fixed set
$V^c:=\{v\in V:c(v)=v\}$ is a real subspace with $V=V^c\oplus\ci V^c$, and
$c$ is the conjugation of $V$ with respect to $V^c$. There is no canonical
choice: specifying $c$ is precisely specifying which vectors of $V$ are to
be regarded as real, and \eqref{eq:complex-reflected-function} is the
Schwarz reflection principle relative to that choice. In the scalar case
$V=\C$, $V^c=\R$, $c(v)=\bar v$ is the familiar one.

\begin{example}
\label{ex:complex-reflection-swap}
Let $V=\C^2$ and $c(v_1,v_2):=(\bar v_2,\bar v_1)$, which is antilinear and
involutive. Its real form is
\[
  V^c=\{(w,\bar w):w\in\C\},
\]
a real two-dimensional subspace meeting the two coordinate lines only at the
origin: none of the coordinates of a vector of $V^c$ is real, in general.
For $h=(h_1,h_2)$ the reflection \eqref{eq:complex-reflected-function} reads
$h^{\#}=\bigl(\overline{h_2(\bar z)},\overline{h_1(\bar z)}\bigr)$, so that
$h^{\#}=h$ means
\[
  h_2(z)=\overline{h_1(\bar z)} ,
\]
that is, the two components are Schwarz reflections of one another. This
constrains neither component separately: $h_1$ may be arbitrary, and then
$h_2$ is determined. By contrast, the componentwise involution
$c'(v_1,v_2):=(\bar v_1,\bar v_2)$ has real form $\R^2$, and $h^{\#}=h$
forces each $h_j$ to be real on $\Omega\cap\R$. The two symmetries are
genuinely different conditions on the same $V$.
\end{example}

In this subsection the Laurent coefficients are denoted by $a_n$, the letter
$c$ being reserved for the involution.
\begin{lemma}[Reflection of Laurent coefficients]
\label{lem:complex-reflection}
If $h$ is meromorphic near $p$, then $h^\#$ is meromorphic near $\bar p$.
If
\(h(z)=\sum_{n=-m}^{\infty}(z-p)^n a_n\),
then
\begin{equation}\label{eq:complex-reflected-laurent}
  h^\#(z)
  =\sum_{n=-m}^{\infty}(z-\bar p)^n c(a_n)
\end{equation}
near $\bar p$. In particular,
\begin{equation}\label{eq:complex-reflected-residue}
  \Res_{\bar p}h^\#=c(\Res_p h).
\end{equation}
\end{lemma}

\begin{proof}
Set $w:=z-\bar p$, so that $\bar z-p=\bar w$. Substituting $\bar z$ into the
Laurent series of $h$ at $p$ and applying $c$ term by term, which is
legitimate because $c$ is continuous, we obtain
\[
  h^{\#}(z)=c\bigl(h(\bar z)\bigr)
  =c\Bigl(\sum_{n\geq-m}\bar w^{\,n}a_n\Bigr)
  =\sum_{n\geq-m}w^nc(a_n),
\]
the last equality by antilinearity of $c$, since $\overline{\bar w^{\,n}}
=w^n$ for every $n\in\mathbb Z$. This is
\eqref{eq:complex-reflected-laurent}, and reading off the coefficient of
index $-1$ gives \eqref{eq:complex-reflected-residue}.
\end{proof}

\begin{corollary}[Reflection-symmetric functions]
\label{cor:complex-reflection-symmetric}
If $h^\#=h$, then the pole orders at $p$ and $\bar p$ are equal and
\begin{equation}\label{eq:complex-reflection-pair}
  \Res_{\bar p}h=c(\Res_p h).
\end{equation}
Consequently,
\[
  \PRes_{\{p,\bar p\}}h
  =\Res_p h+c(\Res_p h).
\]
For scalar-valued functions with $c(v)=\bar v$, this is
$2\req(\Res_p h)$.
\end{corollary}

\begin{proof}
If $h=h^\#$, the Laurent expansion of $h$ at $\bar p$ is exactly the
reflected expansion in \eqref{eq:complex-reflected-laurent}. Its first
nonzero principal coefficient occurs at the same exponent as at $p$, and
the residue relation follows from
\eqref{eq:complex-reflected-residue}.
\end{proof}

\begin{remark}\label{rem:complex-reflection-unequal-orders}
Unequal pole orders at $p$ and $\bar p$ are therefore impossible for a
reflection-invariant function; Example~\ref{ex:complex-unequal-orders} is
deliberately not reflection invariant. In the slice-regular setting the
relevant involution acts on the intrinsic complex coordinates of a stem
function and leaves the quaternionic variable untouched. This is why no reflection symmetry will be assumed in
the quaternionic residue theorem, which must allow unequal slice pole
orders and even a removable member of a conjugate pair.
\end{remark}

\subsection{Tangent disks}
\label{subsec:complex-tangent-disks}
The construction of this subsection is used twice. In the plane it turns the
residue theorem into formulas for integrals over the imaginary axis, as
Subsection~\ref{subsec:complex-axis-integrals} will show; and in four
dimensions the circularizations of these disks close the imaginary
hyperplane, giving the integral formulas of
Section~\ref{sec:hyperplane}. The two arguments run in parallel, down to the
sign, so it is worth carrying out the planar one in detail.

Let
\[
  H_+:=\{z\in\C:\req z>0\}
\]
be the open right half-plane. For $R>0$, consider the disk
\begin{equation}\label{eq:complex-tangent-disk}
  D_R:=B(R,R)=\{z\in\C:|z-R|<R\}.
\end{equation}
Its centre is the real point $R$ and its radius is $R$. Hence
$\partial D_R$ is tangent to the imaginary axis at the origin and
$D_R\subset H_+$. Expanding the defining inequality gives
\begin{equation}\label{eq:complex-tangent-disk-cartesian}
  z\in D_R
  \quad\Longleftrightarrow\quad
  |z|^2<2R\req z.
\end{equation}
It follows at once that the disks increase with $R$ and that
\[
  \bigcup_{R>0}D_R=H_+,
\]
(see Figure \ref{fig:complex-tangent-disks}).
Indeed, if $z=x+\ci y$ with $x>0$, then $z\in D_R$ whenever
$R>|z|^2/(2x)$.

\begin{figure}[htbp]
\centering
\begin{tikzpicture}[scale=0.62]
  \def\Ra{1}\def\Rb{2}\def\Rc{3.5}

  \fill[black!4] (0,-4) rectangle (8.4,4);
  \node[black!55] at (7.2,-3.3) {$H_+$};

  \draw[->,gray!70] (-1.2,0) -- (8.6,0) node[below right] {$\R$};
  \draw[->,gray!70] (0,-4.2) -- (0,4.3) node[above left] {$\ci\R$};

  \fill[black!10] (\Rc,0) circle (\Rc);
  \draw[thick,orient] (\Rc,0) ++(0:\Rc) arc(0:360:\Rc);
  \draw[thick] (\Rb,0) circle (\Rb);
  \draw[thick] (\Ra,0) circle (\Ra);

  \fill (\Ra,0) circle (1.3pt);
  \fill (\Rb,0) circle (1.3pt);
  \fill (\Rc,0) circle (1.3pt);
  \fill (0,0) circle (1.7pt);
  \node[below left] at (0.05,-0.05) {$0$};

  \fill (1.2,2.2) circle (1.7pt);
  \node[right] at (1.34,2.3) {$z$};

  \node at (1.45,1.25) {$D_{R_1}$};
  \node at (3.35,2.05) {$D_{R_2}$};
  \node at (6.55,2.85) {$D_{R_3}$};
\end{tikzpicture}
\caption{The tangent disks $D_R=B(R,R)$ for $R_1<R_2<R_3$. Each of them is
tangent to the imaginary axis at the origin and contained in the right
half-plane $H_+$; they increase with $R$ and their union is $H_+$. The point
$z$ lies outside $D_{R_1}$ and $D_{R_2}$ and inside $D_{R_3}$, in accordance
with \eqref{eq:complex-tangent-disk-cartesian}: it belongs to $D_R$ exactly
when $R>|z|^2/(2\req z)$.}
\label{fig:complex-tangent-disks}
\end{figure}
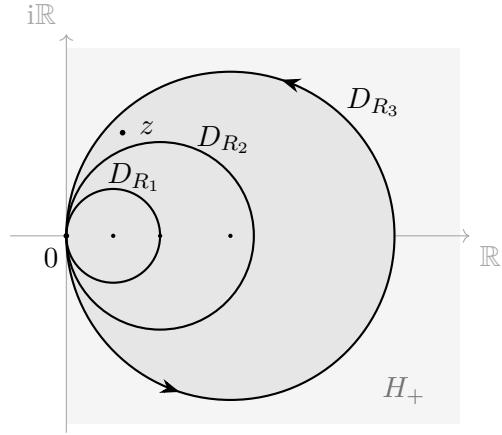

We now prove a limiting statement needed for integrals over the imaginary
axis. The proof is included in detail because the sign and the two different
parts of the circular boundary will have exact analogues in four dimensions.

\begin{lemma}[Tangent-disk limit]\label{lem:complex-tangent-disks}
Let $\Omega$ be an open set containing the closed right half-plane
$\overline{H_+}$, and let $h\colon\Omega\to V$ be meromorphic. Assume that
\begin{enumerate}[label=\textup{(\alph*)}]
  \item $h$ has finitely many poles in $H_+$ and no pole on $\ci\R$;
  \item there exist $C>0$ and $R_0>0$ such that
  \begin{equation}\label{eq:complex-quadratic-decay}
    \|h(z)\|\le \frac{C}{|z|^2}
    \qquad
    (\req z\ge0,\ |z|\ge R_0).
  \end{equation}
\end{enumerate}
Then $y\mapsto h(\ci y)$ belongs to $L^1(\R,V)$. Every pole of $h$ in
$H_+$ lies in $D_R$ for all sufficiently large $R$, and
\begin{equation}\label{eq:complex-tangent-limit}
  \boxed{\;
  \lim_{R\to\infty}\Flux{D_R}{h}
  =-\frac{1}{2\pi}\int_{-\infty}^{\infty}h(\ci y)\,\dd y
  =\sum_{a\in\mathcal P(h)\cap H_+}\Res_a h.
  \;}
\end{equation}
\end{lemma}

\begin{proof}
We divide the proof into four steps.

All limits below are understood for sufficiently large $R$ such that
$\partial D_R$ contains no pole. Step~4 shows that, beyond one fixed
threshold, every $R$ has this property.

\smallskip
\noindent\emph{Step 1: integrability on the imaginary axis.}
The function $y\mapsto h(\ci y)$ is continuous on every compact interval,
because there are no poles on $\ci\R$. For $|y|\ge R_0$,
\eqref{eq:complex-quadratic-decay} gives
\(\|h(\ci y)\|\le C/y^2\).
The function $y\mapsto y^{-2}$ is integrable away from the origin. Thus
$h(\ci\,\cdot)\in L^1(\R,V)$.

\smallskip
\noindent\emph{Step 2: the near and far arcs.}
For $|y|<R$, set
\[
  s_R(y):=\sqrt{R^2-y^2}.
\]
The two points of $\partial D_R$ with imaginary part $y$ are
\begin{equation}\label{eq:complex-near-far-parametrization}
  z_R^-(y)=R-s_R(y)+\ci y,
  \qquad
  z_R^+(y)=R+s_R(y)+\ci y.
\end{equation}
The minus sign describes the \emph{near arc}, adjacent to the imaginary
axis, and the plus sign the \emph{far arc}. Since the outward normal to a
circle is the radial unit vector,
\begin{equation}\label{eq:complex-near-far-normals}
  \nu_R^\pm(y)
  =\frac{z_R^\pm(y)-R}{R}
  =\pm\frac{s_R(y)}{R}+\ci\frac{y}{R}.
\end{equation}
Moreover,
\[
  \left|\frac{\dd z_R^\pm}{\dd y}\right|
  =\sqrt{1+\frac{y^2}{s_R(y)^2}}
  =\frac{R}{s_R(y)}.
\]
Hence
\begin{equation}\label{eq:complex-near-far-flux}
  2\pi\Flux{D_R}{h}
  =I_R^-+I_R^+,
\end{equation}
where
\begin{equation}\label{eq:complex-near-far-integrals}
  I_R^\pm
  :=
  \int_{-R}^{R}
  \frac{R}{s_R(y)}
  \nu_R^\pm(y)h(z_R^\pm(y))\,\dd y.
\end{equation}
The endpoints $y=\pm R$, where the two graph parametrizations meet, have
measure zero and do not affect the integrals.

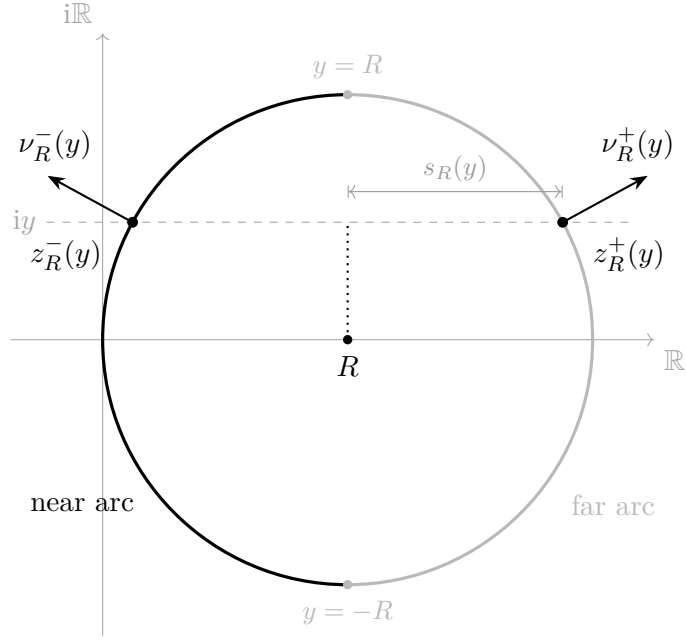
\begin{figure}[htbp]
\centering
\begin{tikzpicture}[scale=1.35]
  \def\Rad{2.4}
  \pgfmathsetmacro{\yy}{1.15}
  \pgfmathsetmacro{\ss}{sqrt(\Rad*\Rad-\yy*\yy)}
  \pgfmathsetmacro{\xm}{\Rad-\ss}
  \pgfmathsetmacro{\xp}{\Rad+\ss}

  \draw[->,gray!70] (-0.9,0) -- (5.4,0) node[below right] {$\R$};
  \draw[->,gray!70] (0,-2.9) -- (0,3.0) node[above left] {$\ci\R$};

  \draw[very thick,gray!55] (\Rad,0) ++(-90:\Rad) arc(-90:90:\Rad);
  \draw[very thick] (\Rad,0) ++(90:\Rad) arc(90:270:\Rad);

  \draw[dashed,gray!70] (-0.55,\yy) -- (5.15,\yy);
  \node[left,gray!70] at (-0.55,\yy) {$\ci y$};

  \fill (\xm,\yy) circle (1.5pt);
  \fill (\xp,\yy) circle (1.5pt);
  \draw[-{Stealth},thick] (\xm,\yy) -- ++({(\xm-\Rad)/\Rad*0.95},{\yy/\Rad*0.95});
  \draw[-{Stealth},thick] (\xp,\yy) -- ++({(\xp-\Rad)/\Rad*0.95},{\yy/\Rad*0.95});
  \node[above left] at (\xm-0.30,\yy+0.5) {$\nu_R^-(y)$};
  \node[above right] at (\xp+0.28,\yy+0.5) {$\nu_R^+(y)$};
  \node[below left] at (\xm-0.2,\yy-0.04) {$z_R^-(y)$};
  \node[below right] at (\xp+0.2,\yy-0.04) {$z_R^+(y)$};

  \fill (\Rad,0) circle (1.3pt);
  \node[below] at (\Rad,-0.06) {$R$};
  \draw[dotted,thick] (\Rad,0) -- (\Rad,\yy);
  \draw[|<->|,gray!80] (\Rad,\yy+0.30) -- node[above,font=\small] {$s_R(y)$} (\xp,\yy+0.30);

  \fill[gray!55] (\Rad,\Rad) circle (1.3pt);
  \fill[gray!55] (\Rad,-\Rad) circle (1.3pt);
  \node[above,font=\small,gray!55] at (\Rad,\Rad+0.06) {$y=R$};
  \node[below,font=\small,gray!55] at (\Rad,-\Rad-0.06) {$y=-R$};

  \node at (-0.2,-1.62) {near arc};
  \node[gray!55] at (5,-1.62) {far arc};
\end{tikzpicture}
\caption{The decomposition of $\partial D_R$ into the near arc (in black)
and the far arc (in grey), each a graph over the segment $|y|<R$. The
horizontal line at height $y$ meets them at the two points
\eqref{eq:complex-near-far-parametrization}, whose distance from the
vertical through the centre is $s_R(y)$; the outward normals
\eqref{eq:complex-near-far-normals} there have opposite real parts. As
$R\to\infty$ the near arc flattens onto the imaginary axis and $\nu_R^-$
tends to $-1$, whence the sign in \eqref{eq:complex-tangent-limit}.}
\label{fig:complex-near-far-arcs}
\end{figure}

\smallskip
\noindent\emph{Step 3: limit of the two arcs.}
Fix $A>0$. For $|y|\le A$ and $R>2A$, one has
\(R-s_R(y)
  =\frac{y^2}{R+s_R(y)}
  \le\frac{A^2}{R}\).
It follows uniformly for $|y|\le A$ that
\[
  z_R^-(y)\longrightarrow\ci y,
  \qquad
  \nu_R^-(y)\longrightarrow-1,
  \qquad
  \frac{R}{s_R(y)}\longrightarrow1.
\]
Because the compact segment $\{\ci y:|y|\le A\}$ contains no pole, $h$ is
uniformly continuous and bounded on a small neighbourhood of that segment.
Therefore
\begin{equation}\label{eq:complex-near-compact-limit}
  \lim_{R\to\infty}
  \int_{-A}^{A}
  \frac{R}{s_R(y)}
  \nu_R^-(y)h(z_R^-(y))\,\dd y
  =-\int_{-A}^{A}h(\ci y)\,\dd y.
\end{equation}
The minus sign has a geometric origin: as the near arc flattens onto the
imaginary axis, its outward normal tends to \(-1\), which is precisely the
outward normal of the right half-plane along its boundary.

We next control the tails of the near arc, uniformly in \(R\). Fix
\(A\geq R_0\). Since \(|z_R^-(y)|\geq|y|\), the decay assumption
\eqref{eq:complex-quadratic-decay} gives
\(\|h(z_R^-(y))\|\leq C/y^2\) for \(A<|y|<R\); taking norms under the
integral sign and using \(|\nu_R^-|=1\),
\begin{equation}\label{eq:complex-near-tail-estimate}
 \left\|
 \int_{A<|y|<R}
 \frac{R}{s_R(y)}\nu_R^-(y)h(z_R^-(y))\dd y
 \right\|
 \leq
 C\int_{A<|y|<R}\frac{R}{y^2\sqrt{R^2-y^2}}\dd y
 \leq\frac{2C}{A} ,
\end{equation}
where the last step rests on the elementary integral
\begin{equation}\label{eq:complex-tail-elementary-integral}
  \int_A^R\frac{R}{y^2\sqrt{R^2-y^2}}\dd y
  =\frac{1}{A}\sqrt{1-\frac{A^2}{R^2}}\leq\frac1A ,
\end{equation}
applied on each of the two half-intervals. Note that the bound
\eqref{eq:complex-near-tail-estimate} does not depend on \(R\), which is
what allows the two limits below to be taken in succession.
The same decay estimate on the imaginary axis gives
\begin{equation}\label{eq:complex-axis-tail-estimate}
  \int_{|y|>A}\|h(\ci y)\|\,\dd y
  \le\frac{2C}{A}.
\end{equation}
Equations \eqref{eq:complex-near-compact-limit},
\eqref{eq:complex-near-tail-estimate}, and
\eqref{eq:complex-axis-tail-estimate} show, first letting $R\to\infty$ and
then $A\to\infty$, that
\begin{equation}\label{eq:complex-near-full-limit}
  \lim_{R\to\infty}I_R^-
  =-\int_{-\infty}^{\infty}h(\ci y)\,\dd y.
\end{equation}

It remains to show that the far arc contributes nothing. For
$z_R^+(y)$, the circle equation gives
\[
  |z_R^+(y)|^2
  =2R\req z_R^+(y)
  =2R\bigl(R+s_R(y)\bigr)
  \ge2R^2.
\]
For large $R$, \eqref{eq:complex-quadratic-decay} therefore yields
\(\|h(z_R^+(y))\|\le\frac{C}{2R^2}\).
Since
\(\int_{-R}^{R}\frac{\dd y}{\sqrt{R^2-y^2}}=\pi\),
we obtain
\begin{equation}\label{eq:complex-far-arc-estimate}
  \|I_R^+\|
  \le\frac{C}{2R}
       \int_{-R}^{R}\frac{\dd y}{\sqrt{R^2-y^2}}
  =\frac{\pi C}{2R}
  \longrightarrow0.
\end{equation}
Combining \eqref{eq:complex-near-far-flux},
\eqref{eq:complex-near-full-limit}, and
\eqref{eq:complex-far-arc-estimate}, we conclude that
\begin{equation}\label{eq:complex-tangent-geometric-limit}
  \lim_{R\to\infty}\Flux{D_R}{h}
  =-\frac{1}{2\pi}\int_{-\infty}^{\infty}h(\ci y)\,\dd y.
\end{equation}

\smallskip
\noindent\emph{Step 4: the residues enclosed by $D_R$.}
Let $a=x+\ci y$ be a pole in $H_+$, so $x>0$. By
\eqref{eq:complex-tangent-disk-cartesian},
\[
  a\in D_R
  \quad\Longleftrightarrow\quad
  R>\frac{|a|^2}{2x}.
\]
There are only finitely many such poles. Consequently, for all sufficiently
large $R$, every pole in $H_+$ lies strictly inside $D_R$ and no pole lies on
$\partial D_R$. Theorem~\ref{thm:complex-residue-flux} then gives
\[
  \Flux{D_R}{h}
  =\sum_{a\in\mathcal P(h)\cap H_+}\Res_a h.
\]
The right-hand side is independent of $R$. Taking the limit and using
\eqref{eq:complex-tangent-geometric-limit} completes the proof.
\end{proof}

\begin{remark}[Orientation of the imaginary axis]
\label{rem:complex-imaginary-orientation}
The positively oriented tangent is $\tau=\ci\nu$. Since the limiting normal
on the near arc is $\nu=-1$, the limiting tangent is $-\ci$: the boundary of
the right half-plane is traversed \emph{downward}, from $+\ci\infty$ to
$-\ci\infty$. Thus an upward-oriented integral on the imaginary axis has the
opposite sign. This is the contour-integral explanation of the minus sign in
\eqref{eq:complex-tangent-limit}.
\end{remark}

\begin{remark}[The two half-planes]
\label{rem:complex-two-half-planes}
No assumption is made on the poles of $h$ in the left half-plane. The disks
$D_R$ are contained in $H_+$, and by
\eqref{eq:complex-tangent-disk-cartesian} a point $a$ with $\req a\leq0$
lies in no $D_R$ at all; such poles are therefore never enclosed, for any
$R$, and the right-hand side of \eqref{eq:complex-tangent-limit} involves
only $\mathcal P(h)\cap H_+$. In particular $h$ may have infinitely many
poles in $\req z<0$, since \eqref{eq:complex-quadratic-decay} is required on
$\req z\geq0$ only.

Suppose, however, that $h$ is meromorphic on $\C$, with finitely many poles
in all, and satisfies \eqref{eq:complex-quadratic-decay} on the whole plane.
Applying the lemma to the reflected family $B(-R,R)$, whose near arcs
approach the imaginary axis from the left, gives
\[
  \frac{1}{2\pi}\int_{-\infty}^{\infty}h(\ci y)\dd y
  =\sum_{a\in\mathcal P(h)\cap H_-}\Res_ah ,
\]
and comparison with \eqref{eq:complex-tangent-limit} yields
\begin{equation}\label{eq:complex-total-residue-vanishes}
  \sum_{a\in\mathcal P(h)}\Res_ah=0 .
\end{equation}
The integral on the imaginary axis can thus be read from either side, and
the choice of $H_+$ is a convention.

Identity \eqref{eq:complex-total-residue-vanishes} is the vanishing of the
residue at infinity. Indeed, since \(h\) is meromorphic on \(\C\) and has only
finitely many finite poles, infinity is an isolated singularity and one may
define
\[
  \Res_{\infty}h
  :=-\frac{1}{2\pi\ci}\int_{|z|=R}\dd z\,h(z),
\]
where \(R\) is larger than the modulus of every finite pole. The ordinary
residue theorem on the disk then gives
\[
  \sum_{a\in\mathcal P(h)}\Res_ah+\Res_{\infty}h=0.
\]
Under the present global decay assumption, the residue at infinity vanishes,
since
\[
  \Bigl\|\int_{|z|=R}\dd z\,h(z)\Bigr\|
  \leq2\pi R\cdot\frac{C}{R^2}
  \xrightarrow[R\to\infty]{}0 .
\]
Equivalently, the meromorphic \(1\)-form \(h(z)\,\dd z\) extends
holomorphically across the point at infinity, and
\eqref{eq:complex-total-residue-vanishes} is the global residue theorem on
the Riemann sphere; see \cite[Theorem~10.21]{Forster}.
The exponent is sharp: a decay of order $|z|^{-1}$ would leave the above
bound constant, and indeed $h(z)=1/z$ has a single pole, of residue $1$.
\end{remark}

\subsection{Rational and trigonometric integrals along the imaginary axis}
\label{subsec:complex-axis-integrals}

The tangent-disk lemma immediately gives a practical residue formula.

\begin{corollary}[Right-half-plane formula]\label{cor:complex-axis-formula}
Under the hypotheses of Lemma~\ref{lem:complex-tangent-disks},
\begin{equation}\label{eq:complex-axis-formula}
  \boxed{\;
  \int_{-\infty}^{\infty}h(\ci y)\,\dd y
  =-2\pi
  \sum_{a\in\mathcal P(h),\,\req a>0}\Res_a h.
  \;}
\end{equation}
Equivalently, if the imaginary axis is oriented upward,
\begin{equation}\label{eq:complex-axis-contour-formula}
  \int_{-\ci\infty}^{\ci\infty}\dd z\,h(z)
  =-2\pi\ci
  \sum_{a\in\mathcal P(h),\,\req a>0}\Res_a h.
\end{equation}
\end{corollary}

\begin{proof}
The first identity is just \eqref{eq:complex-tangent-limit} multiplied by
$-2\pi$. For the second, parametrize the upward imaginary axis by
$z=\ci y$, $-\infty<y<\infty$. Then $\dd z=\ci\,\dd y$, and
\eqref{eq:complex-axis-contour-formula} follows from
\eqref{eq:complex-axis-formula}.
\end{proof}

In particular, let $r=P/Q$ be a scalar rational function with no pole on
$\ci\R$. If
\[
  \deg Q\ge\deg P+2,
\]
then $r(z)=O(|z|^{-2})$ at infinity and
\eqref{eq:complex-axis-formula} applies. The same conclusion holds for a
$V$-valued rational function with a common scalar denominator and the
corresponding componentwise degree bound.

\begin{example}\label{ex:complex-rational-axis}
Let $a>0$ and first take
\[
  h_a(z)=\frac{1}{a^2-z^2}.
\]
The poles are $\pm a$, and only $a$ lies in the right half-plane. Since
\(a^2-z^2=-(z-a)(z+a)\),
we have
\[
  \Res_a h_a
  =\lim_{z\to a}\frac{z-a}{a^2-z^2}
  =-\frac{1}{2a}.
\]
On the imaginary axis,
\(h_a(\ci y)=\frac{1}{a^2+y^2}\).
Formula \eqref{eq:complex-axis-formula} therefore yields
\begin{equation}\label{eq:complex-rational-integral-one}
  \int_{-\infty}^{\infty}\frac{\dd y}{a^2+y^2}
  =-2\pi\left(-\frac{1}{2a}\right)
  =\frac{\pi}{a}.
\end{equation}

For a higher-order example, set
\[
  k_a(z)=\frac{1}{(a^2-z^2)^2}
        =\frac{1}{(z-a)^2(z+a)^2}.
\]
The point $a$ is now a pole of order two. After regularization,
\((z-a)^2k_a(z)=\frac{1}{(z+a)^2}\),
so
\[
  \Res_a k_a
  =
  \left.\frac{\dd}{\dd z}(z+a)^{-2}\right|_{z=a}
  =-\frac{2}{(2a)^3}
  =-\frac{1}{4a^3}.
\]
Since $k_a(\ci y)=(a^2+y^2)^{-2}$, we obtain
\begin{equation}\label{eq:complex-rational-integral-two}
  \int_{-\infty}^{\infty}
  \frac{\dd y}{(a^2+y^2)^2}
  =\frac{\pi}{2a^3}.
\end{equation}
\end{example}

Oscillatory factors can be treated without changing the closing domains.

\begin{proposition}\label{prop:complex-exponential-factor}
Let $h$ satisfy the hypotheses of Lemma~\ref{lem:complex-tangent-disks}, and
let $t\ge0$. Then
\begin{equation}\label{eq:complex-fourier-residue-formula}
  \boxed{\;
  \int_{-\infty}^{\infty}
  e^{-\ci ty}h(\ci y)\,\dd y
  =
  -2\pi
  \sum_{a\in\mathcal P(h),\,\req a>0}
  \Res_a\bigl(e^{-tz}h(z)\bigr).
  \;}
\end{equation}
If the Laurent principal part of $h$ at a pole $a$ of order $m$ is
\(\sum_{j=1}^{m}(z-a)^{-j}c_{-j}\),
then
\begin{equation}\label{eq:complex-exponential-residue-principal-part}
  \Res_a(e^{-tz}h(z))
  =
  e^{-ta}\sum_{k=0}^{m-1}\frac{(-t)^k}{k!}c_{-k-1}.
\end{equation}
\end{proposition}

\begin{proof}
Set $H_t(z):=e^{-tz}h(z)$. Since $t\geq0$, one has
$|e^{-tz}|=e^{-t\req z}\leq1$ on the whole closed half-plane
$\overline{H_+}$, so that $\|H_t(z)\|\leq\|h(z)\|$ there. Hence $H_t$ has
the same poles as $h$, none of them on $\ci\R$, and satisfies
\eqref{eq:complex-quadratic-decay} with the same constant $C$: the
hypotheses of Lemma~\ref{lem:complex-tangent-disks} hold for $H_t$ as they
do for $h$. Since $H_t(\ci y)=e^{-\ci ty}h(\ci y)$, applying
Corollary~\ref{cor:complex-axis-formula} to $H_t$ gives
\eqref{eq:complex-fourier-residue-formula}.

For the local formula, expand the holomorphic factor at $a$ as
\[
  e^{-tz}=e^{-ta}e^{-t(z-a)}
         =e^{-ta}\sum_{k\geq0}\frac{(-t)^k}{k!}(z-a)^k ,
\]
and multiply the Laurent series of $h$ at $a$ on the left by this
expansion. Since both series converge absolutely on a punctured disc about
$a$, the product may be rearranged, and the coefficient of $(z-a)^{-1}$
collects the pairs of exponents $k$ and $-j$ with $k-j=-1$; as $j$ ranges
over $1,\dots,m$ only, the sum is finite and equals
\[
  e^{-ta}\sum_{k=0}^{m-1}\frac{(-t)^k}{k!}\,c_{-k-1},
\]
which is \eqref{eq:complex-exponential-residue-principal-part}.
\end{proof}

At this point we would usually appeal to a Jordan-type lemma, which
compensates an insufficient decay of the integrand with the decay of the
phase along the closing arc. We do not need one: the decay
\eqref{eq:complex-quadratic-decay} suffices on its own, and since
\(|e^{-tz}|\leq1\) throughout \(\overline{H_+}\) the function \(H_t\)
inherits the hypotheses of Lemma~\ref{lem:complex-tangent-disks} with the
same constants as \(h\), so that the same tangent disks close the contour.

\begin{remark}\label{rem:complex-exponential-higher-order}
At a simple pole, \eqref{eq:complex-exponential-residue-principal-part}
reduces to
\(\Res_a(e^{-tz}h(z))=e^{-ta}\Res_a h\).
At a higher-order pole this shortcut is false: all coefficients in the
principal part can contribute. Equivalently, one must differentiate the
entire regularized product $e^{-tz}(z-a)^m h(z)$.
\end{remark}

Since
\begin{equation}\label{eq:complex-euler-oscillation}
  e^{-\ci ty}=\cos(ty)-\ci\sin(ty),
\end{equation}
formula \eqref{eq:complex-fourier-residue-formula} evaluates cosine and sine
integrals. The following elementary observation makes precise when the two
parts separate.

\begin{lemma}[Reflection and parity on the axis]
\label{lem:complex-axis-parity}
Let $h$ be scalar-valued and reflection invariant in the sense of
Definition~\ref{def:complex-reflection}, with $V=\C$ and $c(v)=\bar v$;
equivalently,
\begin{equation}\label{eq:complex-axis-reflection}
  h(-\ci y)=\overline{h(\ci y)},\qquad y\in\R.
\end{equation}
Write $h(\ci y)=u(y)+\ci v(y)$ with $u,v$ real valued. Then $u$ is even and
$v$ is odd, and, whenever the integrals converge absolutely,
\begin{equation}\label{eq:complex-axis-parity-formula}
  \int_{\R}e^{-\ci ty}h(\ci y)\dd y
  =\int_{\R}\bigl(u(y)\cos(ty)+v(y)\sin(ty)\bigr)\dd y\in\R .
\end{equation}
In particular, the transform reduces to a pure cosine integral when
$h(\ci\,\cdot\,)$ is real, and to a pure sine integral when it is purely
imaginary.
\end{lemma}

\begin{proof}
Comparing real and imaginary parts in \eqref{eq:complex-axis-reflection},
which reads $u(-y)+\ci v(-y)=u(y)-\ci v(y)$, gives the stated parities.
Multiplying $h(\ci y)$ by \eqref{eq:complex-euler-oscillation} and
separating,
\[
  e^{-\ci ty}h(\ci y)
  =\bigl(u(y)\cos(ty)+v(y)\sin(ty)\bigr)
   +\ci\bigl(v(y)\cos(ty)-u(y)\sin(ty)\bigr),
\]
the second bracket is an odd function of $y$, being in each summand the
product of an even and an odd factor. It therefore integrates to zero over
$\R$, which leaves \eqref{eq:complex-axis-parity-formula}.
\end{proof}

\begin{example}\label{ex:complex-cosine-integral}
Let $a>0$ and $t\geq0$, and apply
Proposition~\ref{prop:complex-exponential-factor} to
\[
  h_a(z)=\frac{1}{a^2-z^2},
\]
which decays as $|z|^{-2}$ and has the two simple poles $\pm a$, only $a$
lying in $H_+$. From $a^2-z^2=-(z-a)(z+a)$ we get
\[
  \Res_a\bigl(e^{-tz}h_a(z)\bigr)=-\frac{e^{-at}}{2a},
\]
whence
\[
  \int_{-\infty}^{\infty}\frac{e^{-\ci ty}}{a^2+y^2}\dd y
  =\frac{\pi}{a}e^{-at}.
\]
Since $h_a(\ci y)=(a^2+y^2)^{-1}$ is real,
Lemma~\ref{lem:complex-axis-parity} leaves the cosine integral alone:
\begin{equation}\label{eq:complex-cosine-integral}
  \int_{-\infty}^{\infty}\frac{\cos(ty)}{a^2+y^2}\dd y
  =\frac{\pi}{a}e^{-at},
  \qquad a>0,\ t\geq0 .
\end{equation}
The left-hand side being even in $t$, the case of an arbitrary real $t$ is
obtained by replacing $t$ with $|t|$; at $t=0$ one recovers
\eqref{eq:complex-rational-integral-one}.
\end{example}

\begin{example}
\label{ex:complex-cosine-double-pole}
Let now $k_a(z)=(a^2-z^2)^{-2}$, with a pole of order two at $a$. As in
Remark~\ref{rem:complex-exponential-higher-order}, the exponential factor
must be regularized together with $k_a$:
\[
  (z-a)^2e^{-tz}k_a(z)=\frac{e^{-tz}}{(z+a)^2},
\]
so that, by Corollary~\ref{cor:complex-higher-order-residue},
\[
  \Res_a\bigl(e^{-tz}k_a(z)\bigr)
  =\frac{\dd}{\dd z}\left(\frac{e^{-tz}}{(z+a)^2}\right)\bigg|_{z=a}
  =e^{-at}\left(-\frac{t}{4a^2}-\frac{2}{8a^3}\right)
  =-\frac{e^{-at}}{4a^3}(1+at).
\]
The axis value $k_a(\ci y)=(a^2+y^2)^{-2}$ is again real, and therefore
\begin{equation}\label{eq:complex-cosine-double-pole}
  \int_{-\infty}^{\infty}\frac{\cos(ty)}{(a^2+y^2)^2}\dd y
  =\frac{\pi}{2a^3}(1+at)e^{-at},
  \qquad t\geq0 ,
\end{equation}
which reduces to \eqref{eq:complex-rational-integral-two} at $t=0$. Note
that the two summands of the derivative contribute with the same sign: the
naive substitution $e^{-at}\Res_ak_a$ would produce only the second of them.
\end{example}

\begin{example}\label{ex:complex-sine-integral}
Take finally
\[
  h(z)=\frac{z}{(a^2-z^2)^2},\qquad a>0 ,
\]
whose only pole in $H_+$ is again $a$, of order two. Since
$h(\ci y)=\ci\,y(a^2+y^2)^{-2}$ is purely imaginary,
Lemma~\ref{lem:complex-axis-parity} guarantees that only the sine integral
survives. Setting
\[
  \phi_t(z):=(z-a)^2e^{-tz}h(z)=\frac{ze^{-tz}}{(z+a)^2}
\]
and differentiating,
\(\phi_t'(z)=e^{-tz}\left(\frac{1}{(z+a)^2}-\frac{tz}{(z+a)^2}
                          -\frac{2z}{(z+a)^3}\right)\),
where at $z=a$ the first and the third summand cancel, leaving
\(\Res_a\bigl(e^{-tz}h(z)\bigr)=\phi_t'(a)=-\frac{t}{4a}e^{-at}\).
Formula \eqref{eq:complex-fourier-residue-formula} therefore gives
\begin{equation}\label{eq:complex-sine-integral}
  \int_{-\infty}^{\infty}\frac{y\sin(ty)}{(a^2+y^2)^2}\dd y
  =\frac{\pi t}{2a}e^{-at},
  \qquad t\geq0 .
\end{equation}
Consistently with the parity of the left-hand side, this vanishes at $t=0$,
where \eqref{eq:complex-sine-integral} is the trivial identity.
\end{example}



\section{Divided differences and Hermite interpolation}
\label{sec:divided-differences}

Proposition~\ref{prop:complex-common-regularization} showed that, after a
common regularization at a conjugate pair $p,\bar p$, the sum of the two
residues is a contour integral of the form
\[
  \frac{1}{2\pi\ci}\int_\Gamma
  \frac{\dd\zeta}{(\zeta-p)^M(\zeta-\bar p)^M}\,g_M(\zeta).
\]
The purpose of the present section is to recognize this integral as a
classical \emph{confluent divided difference}.  This reformulation is useful
for two reasons.  First, a divided difference depends only on finitely many
derivatives of $g_M$ at the two nodes.  Second, when the nodes alternate
between $p$ and $\bar p$, Newton--Hermite interpolation produces exactly the
two polynomial families that occur in the quaternionic slice regular spherical expansions.

The final part of the section translates the same divided difference into an
alternating differential operator.  Later, in the quaternionic setting, this
operator becomes
\[
  \partial_s(\partial_c\partial_s)^{M-1}.
\]
Thus divided differences form the local analytic bridge between the ordinary
complex residue theorem and the differential evaluation of spherical
residues.

We include the necessary proofs so that the section requires only standard
 complex analysis.  General background on divided differences
and Hermite interpolation can be found in
\cite[Sections~1, 2 and~9]{deBoor}; for two-point Taylor expansions and their
Hermite interpretation, see \cite[Sections~2.1--2.3]{LopezTemme}.

The following conventions are in force throughout the section. $V$ is a
finite-dimensional complex vector space, with complex scalars acting on the
left, and $W\subset\C$ is open and not necessarily connected. A map
$g\colon W\to V$ is holomorphic when it is holomorphic componentwise in any
basis of $V$, and, as in Section~\ref{sec:complex-model}, all derivatives
and integrals are understood componentwise. Only
Subsection~\ref{subsec:dd-conjugate-operators} requires more, namely that
$W$ be invariant under complex conjugation; this is stated again there.

\subsection{Distinct and repeated nodes}
\label{subsec:dd-nodes}

We first define divided differences at pairwise distinct nodes
$x_0,\ldots,x_n\in W$, and then extend the definition to repeated nodes by
continuity.

\begin{definition}[Divided differences at distinct nodes]
\label{def:dd-distinct}
The divided differences of $g$ are defined recursively by
\begin{align}
  g[x_0]&:=g(x_0), \label{eq:dd-order-zero}\\
  g[x_0,\ldots,x_n]
  &:=
  (x_n-x_0)^{-1}
  \bigl(
    g[x_1,\ldots,x_n]
    -g[x_0,\ldots,x_{n-1}]
  \bigr),
  \qquad n\ge1.
  \label{eq:dd-recursion}
\end{align}
\end{definition}

This is the standard recursive convention (compare
\cite[Section~1, equations~(5), (6) and~(10)]{deBoor}). The denominators in
\eqref{eq:dd-recursion} are complex scalars and act on the left of the
vectors in $V$. The first divided difference is
\begin{equation}\label{eq:dd-first}
  g[x_0,x_1]
  =\frac{g(x_1)-g(x_0)}{x_1-x_0},
\end{equation}
and the second is
\begin{equation}  \label{eq:dd-second}
  g[x_0,x_1,x_2]
  =
  \frac{g[x_1,x_2]-g[x_0,x_1]}{x_2-x_0}
=
  \frac{1}{x_2-x_0}
  \left(
    \frac{g(x_2)-g(x_1)}{x_2-x_1}
    -
    \frac{g(x_1)-g(x_0)}{x_1-x_0}
  \right).
\end{equation}
A divided difference on $n+1$ nodes is said to have \emph{order $n$}.

Although the recursive formula singles out the first and last nodes, the
result is symmetric in all the nodes.  This fact is well known but not obvious from
\eqref{eq:dd-recursion}; in this paper we will see that it follows immediately from Theorem~\ref{thm:dd-contour}.

Suppose now that two nodes approach each other. By \eqref{eq:dd-first},
\[
  \lim_{y\to x}g[x,y]
  =\lim_{y\to x}\frac{g(y)-g(x)}{y-x}
  =g'(x),
\]
so that divided differences at repeated nodes can be defined by continuous
extension: holomorphy guarantees that all such limits exist and do not
depend on the way in which the nodes coalesce, as
Theorem~\ref{thm:dd-contour} will show. This confluent convention agrees
with \cite[Section~1, equation~(14), and Proposition~21]{deBoor}, where a
block of equal nodes likewise yields normalized derivatives.

\begin{definition}[Repeated-node notation]\label{def:dd-repeated}
For $m\geq1$, the symbol $x^{(m)}$ inside a divided difference means that
the node $x$ occurs $m$ times; a block of multiplicity zero is omitted.
\end{definition}

With this notation, the confluent limit of a single block of nodes is a
normalized derivative,
\begin{equation}\label{eq:dd-repeated-block}
  g[x^{(m)}]=\frac{g^{(m-1)}(x)}{(m-1)!},\qquad m\geq1,
\end{equation}
as will be proved at the end of Theorem~\ref{thm:dd-contour}. The first two
cases, $g[x]=g(x)$ and $g[x,x]=g'(x)$, are the ones already met above.
For a $V$-valued map, the derivative is taken componentwise.

\begin{remark}
Mixed repetitions combine values and derivatives.  For example, if $x\ne y$,
then
\begin{equation}\label{eq:dd-mixed-example}
  g[x,x,y]=
  \frac{g[x,y]-g[x,x]}{y-x}
  =
  \frac{g(y)-g(x)-(y-x)g'(x)}{(y-x)^2}.
\end{equation}
The numerator in the last expression is the error after subtracting the
first-order Taylor polynomial of $g$ at $x$.  This is the first indication of
the relation between repeated nodes and Hermite interpolation.
\end{remark}


\subsection{The contour formula}
\label{subsec:dd-contour}

The contour representation simultaneously proves symmetry, explains
confluence, and connects divided differences to residues.  We formulate it so
that the nodes may belong to different connected components of $W$.
In the scalar case this is the Cauchy contour representation of
\cite[Section~9, equation~(51)]{deBoor}, which also yields the confluent
case; the vector-valued version below is obtained componentwise.

\begin{theorem}[Cauchy formula for confluent divided differences]
\label{thm:dd-contour}
Let $g\colon W\to V$ be holomorphic and let $x_0,\ldots,x_n\in W$, with
repetitions allowed. Let $B$ be an admissible domain in the sense of
Definition~\ref{def:admissible-domain} with
\(\overline B\subset W\)
and
\(x_0,\ldots,x_n\in B\),
and let $\Gamma:=\partial B$ carry the orientation positive relative to $B$.
Then the divided difference $g[x_0,\ldots,x_n]$ is well defined, in the
confluent sense of Subsection~\ref{subsec:dd-nodes}, and
\begin{equation}\label{eq:cx-hermite}
  \boxed{\;
  g[x_0,\ldots,x_n]
  =
  \frac{1}{2\pi\ci}
  \int_\Gamma
  \frac{\dd\zeta}{\prod_{j=0}^{n}(\zeta-x_j)}\,g(\zeta) .
  \;}
\end{equation}
In particular the right-hand side does not depend on the choice of $B$, and
$g[x_0,\ldots,x_n]$ is a symmetric function of the nodes. Denoting by
$\xi_1,\ldots,\xi_s$ the distinct nodes among $x_0,\ldots,x_n$, one has
equivalently
\begin{equation}\label{eq:dd-contour-residues}
  g[x_0,\ldots,x_n]
  =\sum_{i=1}^{s}
   \Res_{\xi_i}\left(\frac{g(z)}{\prod_{j=0}^{n}(z-x_j)}\right).
\end{equation}
\end{theorem}

\begin{proof}
Write, for $x=(x_0,\ldots,x_n)\in B^{n+1}$,
\[
  \Psi(x):=\frac{1}{2\pi\ci}\int_\Gamma
           \frac{\dd\zeta}{\prod_{j=0}^{n}(\zeta-x_j)}\,g(\zeta) ,
\]
the integral being well defined because no $x_j$ lies on $\Gamma$. We prove
that $\Psi$ is continuous, that it agrees with the divided difference at
pairwise distinct nodes, and that these two facts together give the
statement.

\emph{Step 1: $\Psi$ is continuous on $B^{n+1}$.} Continuity is a local
property, so fix $\hat x\in B^{n+1}$ and choose a compact set $K\subset B$
containing a neighbourhood of each $\hat x_j$; it suffices to prove that
$\Psi$ is continuous at $\hat x$ among the tuples whose entries all lie in
$K$. Set $\delta:=\operatorname{dist}(K,\Gamma)>0$. For such tuples and for
$\zeta\in\Gamma$ each factor of the denominator has modulus at least
$\delta$, so the integrand is bounded in norm by
$\delta^{-n-1}\max_\Gamma\|g\|$, a constant, and depends continuously on
$(x,\zeta)$. Dominated convergence gives the continuity of $\Psi$ at
$\hat x$. The same estimate, applied with one node varying and the others
fixed, shows that $\Psi$ is holomorphic in each $x_j$ separately; this is
used in Proposition~\ref{prop:dd-node-derivative}.

\emph{Step 2: pairwise distinct nodes.} We argue by induction on $n$. For
$n=0$ the integrand has the single pole $x_0$ in $B$, with residue $g(x_0)$,
so Theorem~\ref{thm:complex-residue-flux} gives
$\Psi(x_0)=g(x_0)=g[x_0]$.

Let $n\geq1$ and assume the claim for $n-1$ nodes. Multiplying numerator and
denominator, the two kernels of order $n-1$ obtained by deleting
respectively the first and the last node can be written over the common
denominator $\prod_{j=0}^{n}(\zeta-x_j)$ as
\[
  \frac{1}{\prod_{j=1}^{n}(\zeta-x_j)}
  =\frac{\zeta-x_0}{\prod_{j=0}^{n}(\zeta-x_j)} ,
  \qquad
  \frac{1}{\prod_{j=0}^{n-1}(\zeta-x_j)}
  =\frac{\zeta-x_n}{\prod_{j=0}^{n}(\zeta-x_j)} ,
\]
whence, subtracting and using $(\zeta-x_0)-(\zeta-x_n)=x_n-x_0$,
\begin{equation}\label{eq:dd-kernel-recursion}
  \frac{1}{x_n-x_0}
  \left(
    \frac{1}{\prod_{j=1}^{n}(\zeta-x_j)}
    -\frac{1}{\prod_{j=0}^{n-1}(\zeta-x_j)}
  \right)
  =\frac{1}{\prod_{j=0}^{n}(\zeta-x_j)} .
\end{equation}
For $n=1$, for instance, this is the partial-fraction identity
\[
  \frac{1}{x_1-x_0}\left(\frac{1}{\zeta-x_1}-\frac{1}{\zeta-x_0}\right)
  =\frac{1}{(\zeta-x_0)(\zeta-x_1)} .
\]
Multiplying \eqref{eq:dd-kernel-recursion} by $\dd\zeta\,g(\zeta)/(2\pi\ci)$
and integrating over $\Gamma$ shows that $\Psi$ obeys the recursion
\eqref{eq:dd-recursion}; by the inductive hypothesis its two terms are
$g[x_1,\ldots,x_n]$ and $g[x_0,\ldots,x_{n-1}]$, so that
$\Psi(x)=g[x_0,\ldots,x_n]$.

\emph{Step 3: confluent nodes.} The tuples with pairwise distinct entries
form a dense subset of $B^{n+1}$, on which $\Psi$ agrees with the divided
difference by Step~2. Since $\Psi$ is continuous by Step~1, the divided
difference admits a continuous extension to all of $B^{n+1}$, this extension
is unique, it is independent of the way in which the nodes coalesce, and it
equals $\Psi$. This is precisely the confluent convention of
Subsection~\ref{subsec:dd-nodes}, and \eqref{eq:cx-hermite} is proved. The
kernel in $\Psi$ is symmetric in $x_0,\ldots,x_n$, hence so is the divided
difference; and since the left-hand side of \eqref{eq:cx-hermite} does not
involve $B$, neither does the right-hand side.

\emph{Step 4: the two consequences.} Applying
Theorem~\ref{thm:complex-residue-flux} to the meromorphic function
$z\mapsto g(z)\big/\prod_{j=0}^{n}(z-x_j)$, whose poles in $B$ are exactly
$\xi_1,\ldots,\xi_s$, gives \eqref{eq:dd-contour-residues}. When the nodes
are pairwise distinct each pole is simple, with residue
$\bigl(\prod_{k\neq j}(x_j-x_k)\bigr)^{-1}g(x_j)$, so that
\eqref{eq:dd-contour-residues} becomes (the Lagrange formula):
\begin{equation}\label{eq:dd-lagrange-form}
  g[x_0,\ldots,x_n]
  =\sum_{j=0}^{n}
   \left(\prod_{\substack{0\leq k\leq n\\k\neq j}}(x_j-x_k)\right)^{-1}
   g(x_j) ;
\end{equation}
when instead all $m=n+1$ nodes are equal to $x$, formula
\eqref{eq:cx-hermite} reduces to
\[
  g[x^{(m)}]=\frac{1}{2\pi\ci}\int_\Gamma
             \frac{\dd\zeta}{(\zeta-x)^m}\,g(\zeta)
           =\frac{g^{(m-1)}(x)}{(m-1)!}
\]
by the Cauchy formula for derivatives, which proves
\eqref{eq:dd-repeated-block}.
\end{proof}

\begin{remark}
\label{rem:dd-disconnected}
If two nodes lie in different connected
components of $W$, one takes a small Jordan domain around each and lets
$\Gamma$ be the union of their boundaries. This is why $B$ was allowed to be
an arbitrary admissible domain. What makes it work is that
\eqref{eq:cx-hermite} evaluates $g$ only on $\Gamma$, so that a divided
difference is determined by the germs of $g$ at the nodes and by nothing
else; the recursion \eqref{eq:dd-recursion}, which manipulates values at
distinct points as if they belonged to one expansion, gives no hint of this.

This point will be essential later in the quaternionic setting for circular product domains, where the two
conjugate nodes representing a non-real sphere may lie in two disconnected
half-slices: the paired residue is then computed from the two local germs,
with no analytic continuation along a path from one node to the other.
\end{remark}

The contour formula also makes differentiation with respect to a node
transparent. The identity below is the repeated-block version of
\cite[Proposition~22]{deBoor}; by symmetry, the moving block may be placed
first, as we do here.

\begin{proposition}[Differentiating a repeated block]
\label{prop:dd-node-derivative}
Let $g\colon W\to V$ be holomorphic, let $m\geq1$, and let
$y_1,\ldots,y_r\in W$ be fixed nodes, repetitions allowed. Then
$x\mapsto g[x^{(m)},y_1,\ldots,y_r]$ is holomorphic on $W$ and
\begin{equation}\label{eq:cx-dd-derivative}
  \boxed{\;
  \frac{\partial}{\partial x}g[x^{(m)},y_1,\ldots,y_r]
  =m\,g[x^{(m+1)},y_1,\ldots,y_r] .
  \;}
\end{equation}
\end{proposition}

\begin{proof}
Fix $x\in W$ and choose an admissible $B$ as in
Theorem~\ref{thm:dd-contour} containing $x$, a neighbourhood of $x$, and all
the $y_j$; set $\Gamma=\partial B$ and
$Q(\zeta):=\prod_{j=1}^{r}(\zeta-y_j)$, so that \eqref{eq:cx-hermite} reads
\[
  g[x^{(m)},y_1,\ldots,y_r]
  =\frac{1}{2\pi\ci}\int_\Gamma
   \frac{\dd\zeta}{(\zeta-x)^mQ(\zeta)}\,g(\zeta) .
\]
As $x$ varies in that neighbourhood the contour stays fixed and at a
positive distance from $x$, so the integrand and its $x$-derivative are
continuous in $(x,\zeta)$ and bounded on $\Gamma$; differentiation under the
integral sign is therefore permitted, and gives, by
$\partial_x(\zeta-x)^{-m}=m(\zeta-x)^{-m-1}$,
\[
  \frac{\partial}{\partial x}g[x^{(m)},y_1,\ldots,y_r]
  =\frac{m}{2\pi\ci}\int_\Gamma
   \frac{\dd\zeta}{(\zeta-x)^{m+1}Q(\zeta)}\,g(\zeta)
  =m\,g[x^{(m+1)},y_1,\ldots,y_r] ,
\]
the last equality being \eqref{eq:cx-hermite} again, now with $m+1$ copies
of $x$. 
\end{proof}

\subsection{Newton--Hermite interpolation}
\label{subsec:dd-newton-hermite}

Divided differences are the coefficients of the Newton interpolation
polynomial, and repeating a node turns interpolation of values into
interpolation of derivatives, that is, Hermite interpolation. To state this
precisely, fix nodes $a_0,\ldots,a_{N-1}\in W$, repetitions allowed, and
associate with them the Newton products
\begin{equation}\label{eq:dd-newton-polynomials}
  P_0(z):=1,
  \qquad
  P_j(z):=\prod_{\ell=0}^{j-1}(z-a_\ell),
  \quad 1\leq j\leq N ,
\end{equation}
so that $P_j$ collects the factors of the first $j$ nodes and has degree
exactly $j$.The next theorem is the holomorphic Newton--Hermite expansion with
remainder. In the scalar case, the algebraic remainder formula appears in
\cite[Section~1, Example~6, equation~(16)]{deBoor}, while the
contour-integral derivation used below, based on the same telescoping
identity and Cauchy's formula, is given in
\cite[Section~9, equations~(49)--(51)]{deBoor}. The \(V\)-valued version
follows componentwise.

\begin{theorem}[Newton formula with remainder]
\label{thm:dd-newton}
Let $g\colon W\to V$ be holomorphic, let $a_0,\ldots,a_{N-1}\in W$ be nodes,
with repetitions allowed, and let $P_0,\ldots,P_N$ be the associated Newton
products \eqref{eq:dd-newton-polynomials}. Then, for every $z\in W$,
\begin{equation}\label{eq:sres-newton}
  \boxed{\;
  g(z)
  =
  \sum_{j=0}^{N-1}
  P_j(z)\,g[a_0,\ldots,a_j]
  +
  P_N(z)\,g[a_0,\ldots,a_{N-1},z],
  \;}
\end{equation}
the divided difference in the remainder being interpreted confluently when
$z$ coincides with one of the nodes.
\end{theorem}

\begin{proof}
Fix \(z\in W\) and choose a contour \(\Gamma\) as in
Theorem~\ref{thm:dd-contour}, enclosing the nodes and \(z\); if \(W\) is
disconnected, \(\Gamma\) is allowed to be the boundary of a finite union of
Jordan domains. Set
\[
  R_j(\zeta):=
  \frac{P_j(z)}{(\zeta-z)P_j(\zeta)},
  \qquad 0\leq j\leq N .
\]
Since \(P_{j+1}(w)=(w-a_j)P_j(w)\), a direct computation gives
\(R_j(\zeta)
  =
  \frac{P_j(z)}{P_{j+1}(\zeta)}
  +R_{j+1}(\zeta)\).
Iterating this identity from \(j=0\) to \(j=N-1\), and using \(P_0=1\),
yields
\[
  \frac{1}{\zeta-z}
  =
  \sum_{j=0}^{N-1}
  \frac{P_j(z)}{P_{j+1}(\zeta)}
  +
  \frac{P_N(z)}{(\zeta-z)P_N(\zeta)}.
\]
Multiplying by \(\dd\zeta\,g(\zeta)/(2\pi\ci)\) and integrating over
\(\Gamma\), the Cauchy formula gives \(g(z)\) on the left, while
Theorem~\ref{thm:dd-contour} identifies the terms on the right with the
divided differences appearing in \eqref{eq:sres-newton}. The same contour
formula already includes repeated nodes and the case in which \(z\)
coincides with one of them.
\end{proof}

The truncation \(H_{N-1}\) is the Hermite interpolant associated with the
repeated nodes. The standard scalar characterization in terms of prescribed
jets, together with the divisibility argument proving uniqueness, can be
found in
\cite[Proposition~7, equations~(8)--(9), and Definition~37]{deBoor}
(see also \cite{Hermite}).

\begin{corollary}[Hermite interpolation of jets]
\label{cor:dd-hermite-jets}
Let $\xi_1,\ldots,\xi_s$ be the distinct nodes among $a_0,\ldots,a_{N-1}$,
occurring with multiplicities $m_1,\ldots,m_s$, so that
$m_1+\cdots+m_s=N$. Then $H_{N-1}$ is the unique $V$-valued polynomial of
degree at most $N-1$ such that
\begin{equation}\label{eq:dd-hermite-conditions}
  H_{N-1}^{(r)}(\xi_i)=g^{(r)}(\xi_i),
  \qquad 0\leq r\leq m_i-1,\quad 1\leq i\leq s .
\end{equation}
\end{corollary}

\begin{proof}
By Theorems~\ref{thm:dd-newton} and~\ref{thm:dd-contour},
\[
  g(z)-H_{N-1}(z)
  =
  \prod_{i=1}^{s}(z-\xi_i)^{m_i}\,
  g[a_0,\ldots,a_{N-1},z],
\]
where the last factor is holomorphic also at the confluent nodes. Hence
\(g-H_{N-1}\) vanishes at \(\xi_i\) to order at least \(m_i\), which gives
the prescribed jets. If two polynomials of degree at most \(N-1\) have
these jets, their difference is divisible, componentwise, by the
degree-\(N\) polynomial
\(\prod_i(z-\xi_i)^{m_i}\), and therefore vanishes identically.
\end{proof}

\subsection{Alternating conjugate nodes}
\label{subsec:dd-alternating}

We now specialize the Newton construction to the conjugate pair
\(p=\alpha+\ci\beta, \bar p=\alpha-\ci\beta\), with \(\beta>0\),
of Subsection~\ref{subsec:complex-conjugate-pair}, taking the nodes in
alternating order,
\begin{equation}\label{eq:dd-alternating-sequence}
  p,\ \bar p,\ p,\ \bar p,\ \ldots
\end{equation}
The reason for this choice is that consecutive pairs of factors then
assemble into the characteristic polynomial
\begin{equation}\label{eq:dd-delta}
  \DeltaSphere{p}(z):=(z-p)(z-\bar p) ,
\end{equation}
the Newton products \eqref{eq:dd-newton-polynomials} being
\[
  1,\quad z-p,\quad \DeltaSphere{p}(z),\quad
  \DeltaSphere{p}(z)(z-p),\quad \DeltaSphere{p}(z)^2,\ \ldots
\]
Only two shapes occur, $\DeltaSphere{p}^n$ and $\DeltaSphere{p}^n(z-p)$,
and the expansion is therefore organized by powers of a single polynomial
symmetric in the pair. For the equivalent two-point Taylor organization and
its Hermite interpretation, compare \cite[Section~2, Theorem~1, and
Remark~1]{LopezTemme}.

\begin{corollary}[Newton formula at two alternating nodes]
\label{cor:dd-alternating-newton}
Let $g\colon W\to V$ be holomorphic with $p,\bar p\in W$, fix $M\geq1$, and
set, for $0\leq n\leq M-1$,
\begin{align}
  A_n&:=g[p^{(n+1)},\bar p^{(n)}], \label{eq:dd-A-coefficients}\\
  B_n&:=g[p^{(n+1)},\bar p^{(n+1)}]. \label{eq:dd-B-coefficients}
\end{align}
Then, for every $z\in W$,
\begin{equation}\label{eq:dd-alternating-newton}
  \boxed{\;
  g(z)
  =\sum_{n=0}^{M-1}\DeltaSphere{p}(z)^n\bigl(A_n+(z-p)B_n\bigr)
   +\DeltaSphere{p}(z)^Mg[p^{(M)},\bar p^{(M)},z] .
  \;}
\end{equation}
\end{corollary}

\begin{proof}
Apply Theorem~\ref{thm:dd-newton} with $N=2M$ and the nodes
$a_0=p,a_1=\bar p,a_2=p,\ldots$ of \eqref{eq:dd-alternating-sequence}. The
first $2n$ of them consist of $n$ copies of each node and the first $2n+1$
of $n+1$ copies of $p$ and $n$ of $\bar p$, so that
\begin{equation}\label{eq:dd-alternating-products}
  P_{2n}(z)=\DeltaSphere{p}(z)^n,
  \qquad
  P_{2n+1}(z)=\DeltaSphere{p}(z)^n(z-p) ,
\end{equation}
and the divided differences accompanying them are those on the first $2n+1$
and $2n+2$ nodes, namely $A_n$ and $B_n$. Grouping the terms of
\eqref{eq:sres-newton} in consecutive pairs gives the sum in
\eqref{eq:dd-alternating-newton}, while the remainder term, with
$P_{2M}=\DeltaSphere{p}^M$, is the last one.
\end{proof}

Each power of $\DeltaSphere{p}$ carries two coefficients, the interpolant on
a pair of distinct nodes being affine rather than constant, and the Newton
construction produces them on the basis $\{1,z-p\}$. That basis involves the
choice of $p$ as first node; passing to $\{1,z\}$ by
\begin{equation}\label{eq:dd-change-two-families}
  A_n+(z-p)B_n=zB_n+(A_n-pB_n)
\end{equation}
removes the dependence, and is the form used in
Proposition~\ref{prop:dd-two-point-principal-part}.

The interpolation formula can now be applied directly to a meromorphic
function, and it extracts its paired principal part in finite terms: no
infinite series is involved, and, by Remark~\ref{rem:dd-disconnected}, no
connectedness assumption either. The two-point Laurent framework behind this
is classical, see \cite[Section~3, Theorem~2 and
Proposition~3]{LopezTemme}; what the proposition below adds is the use of a
single regularizing exponent at the two nodes, which makes it insensitive to
unequal pole orders.

\begin{proposition}
\label{prop:dd-two-point-principal-part}
Let $h$ be $V$-valued and meromorphic near $\{p,\bar p\}$, and let $M\geq1$
be at least as large as both pole orders, so that
\[
  g_M(z):=\DeltaSphere{p}(z)^Mh(z)
\]
extends holomorphically across the two nodes. Let $A_n,B_n$ be the divided
differences \eqref{eq:dd-A-coefficients} and \eqref{eq:dd-B-coefficients}
of $g_M$, and set
\begin{equation}\label{eq:dd-C-coefficients}
  C_n:=A_n-pB_n .
\end{equation}
Then, on a punctured neighbourhood of the two nodes,
\begin{equation}\label{eq:dd-two-point-principal-part}
  \boxed{\;
  h(z)
  =\sum_{n=0}^{M-1}\DeltaSphere{p}(z)^{n-M}\bigl(zB_n+C_n\bigr)+R_M(z) ,
  \;}
\end{equation}
where
\begin{equation}\label{eq:dd-two-point-remainder}
  R_M(z):=g_M[p^{(M)},\bar p^{(M)},z]
\end{equation}
is holomorphic near both $p$ and $\bar p$. In particular the coefficient of
$z$ carried by $\DeltaSphere{p}^{-1}$ is
\begin{equation}\label{eq:dd-balanced-coefficient}
  B_{M-1}=g_M[p^{(M)},\bar p^{(M)}] .
\end{equation}
\end{proposition}

\begin{proof}
Apply \eqref{eq:dd-alternating-newton} to $g_M$, rewrite each summand by
means of \eqref{eq:dd-change-two-families}, and divide throughout by $\DeltaSphere{p}(z)^M$, a legitimate operation
on the punctured neighbourhood, where that polynomial has no zero. The
resulting identity is \eqref{eq:dd-two-point-principal-part}, with the
remainder term of \eqref{eq:dd-alternating-newton} becoming exactly
$R_M(z)$; the latter is holomorphic near the nodes by
Theorem~\ref{thm:dd-contour}, confluences included. Finally, only the
summand with $n=M-1$ carries the power $\DeltaSphere{p}^{-1}$, and its
coefficient of $z$ is $B_{M-1}$, which is \eqref{eq:dd-B-coefficients} with
$n=M-1$.
\end{proof}

\begin{corollary}[Paired residue as a divided difference]
\label{cor:dd-paired-residue}
Under the hypotheses of
Proposition~\ref{prop:dd-two-point-principal-part},
\begin{equation}\label{eq:cx-sres-dd}
  \boxed{\;
  \PRes_{\{p,\bar p\}}h=g_M[p^{(M)},\bar p^{(M)}] .
  \;}
\end{equation}
\end{corollary}

\begin{proof}
Let $B$ be the union of two disjoint discs, one about each node, small
enough to contain no other pole of $h$, and let $\Gamma=\partial B$. Since
$g_M=\DeltaSphere{p}^Mh$, the kernel of \eqref{eq:cx-hermite} with $M$
copies of each node cancels against $g_M$, leaving
\[
  g_M[p^{(M)},\bar p^{(M)}]
  =\frac{1}{2\pi\ci}\int_\Gamma
   \frac{\dd\zeta}{(\zeta-p)^M(\zeta-\bar p)^M}\,g_M(\zeta)
  =\frac{1}{2\pi\ci}\int_\Gamma\dd\zeta\,h(\zeta) ,
\]
and the last integral is $\Res_ph+\Res_{\bar p}h$ by
Theorem~\ref{thm:complex-residue-flux}.
\end{proof}

\begin{remark}
\label{rem:dd-nonminimal}
As in the previous section, the exponent in \eqref{eq:cx-sres-dd} need not be minimal: the computation
just made returns $\int_\Gamma\dd\zeta\,h(\zeta)$ whatever the exponent,
since the kernel always cancels against the regularization, so that
\[
  g_{M'}[p^{(M')},\bar p^{(M')}]=g_M[p^{(M)},\bar p^{(M)}]
  \qquad\text{for every }M'\geq M .
\]
Note also that the argument divides by $\DeltaSphere{p}$ only away from its
two zeros, and never by a value of $g_M$; it therefore remains valid when
$g_M(p)=0$ or $g_M(\bar p)=0$, in particular when the pole order drops at
one of the two nodes.
\end{remark}

\subsection{Conjugate-node differential operators}
\label{subsec:dd-conjugate-operators}

The contour formula expresses a balanced divided difference through the jets
of $g$ at the two nodes. We now show that the same quantity is produced by
alternating two elementary operations: taking a difference between conjugate
nodes, and differentiating with respect to the moving one.

Throughout this subsection $W$ is assumed to be invariant under complex
conjugation. For $F\colon W\to V$ and $z\in W\setminus\R$, define the
conjugate-node operator
\begin{equation}\label{eq:dd-conjugate-operator}
  (\mathcal SF)(z):=(z-\bar z)^{-1}\bigl(F(z)-F(\bar z)\bigr).
\end{equation}
 On holomorphic
functions $\mathcal S$ is nothing but the first divided difference at the
conjugate pair,
\begin{equation}\label{eq:dd-S-as-divided-difference}
  (\mathcal Sg)(z)=g[z,\bar z] .
\end{equation}

Its output, however, depends on $z$ and $\bar z$ alike and is no longer
holomorphic, so that differentiating it calls for the Wirtinger operator
conjugate to the one used in Section~\ref{sec:complex-model},
\begin{equation}\label{eq:dd-wirtinger}
  \partial_z:=\frac12\left(\frac{\partial}{\partial x}
                           -\ci\frac{\partial}{\partial y}\right),
  \qquad z=x+\ci y ,
\end{equation}
which reduces to the ordinary complex derivative on holomorphic functions
but makes sense for every real-differentiable one. We set
\begin{equation}\label{eq:dd-D-operator}
  \mathcal D:=\partial_z\circ\mathcal S ,
\end{equation}
compositions being read from right to left. The theorem below identifies the
alternating compositions of $\mathcal S$ and $\partial_z$ with the balanced
divided differences of Subsection~\ref{subsec:dd-alternating}; it rests on
the differentiation rule of \cite[Proposition~22]{deBoor}, and is the
complex one-variable model for the spherical-coefficient formula of
\cite[Theorem~4.1]{AltavillaSphericalCoefficients}.

\begin{theorem}[Operator form of balanced divided differences]
\label{prop:cx-operators}
Let $g\colon W\to V$ be holomorphic. For every $z\in W\setminus\R$ and every
$n\geq0$,
\begin{align}
  (\mathcal D^ng)(z)
  &=n!\,g[z^{(n+1)},\bar z^{(n)}], \label{eq:sres-operator-even}\\
  (\mathcal S\mathcal D^ng)(z)
  &=n!\,g[z^{(n+1)},\bar z^{(n+1)}]. \label{eq:sres-operator-odd}
\end{align}
As usual, the block $\bar z^{(0)}$ in the first formula is omitted.
\end{theorem}

\begin{proof}
The two identities are proved together, by an induction that alternates
between them: \eqref{eq:sres-operator-even} for $n$ yields
\eqref{eq:sres-operator-odd} for $n$ by applying $\mathcal S$, and the
latter yields \eqref{eq:sres-operator-even} for $n+1$ by applying
$\partial_z$. Each of the two steps has its own difficulty, namely that
$\mathcal S$ evaluates at $\bar z$, thereby interchanging the multiplicities
of the two nodes, and that $\partial_z$ must treat $z$ and $\bar z$ as
independent variables.

\emph{A recursion for balanced blocks.} Let $z\neq w$ be two nodes and
$n\geq0$. Arrange the $2n+2$ nodes $z^{(n+1)},w^{(n+1)}$ so that the first
is $w$ and the last is $z$: deleting the first leaves $n+1$ copies of $z$
and $n$ of $w$, deleting the last leaves $n$ copies of $z$ and $n+1$ of $w$,
so that the recursion \eqref{eq:dd-recursion}, combined with the symmetry
granted by Theorem~\ref{thm:dd-contour}, reads
\begin{equation}\label{eq:dd-balanced-recursion}
  \frac{g[z^{(n+1)},w^{(n)}]-g[z^{(n)},w^{(n+1)}]}{z-w}
  =g[z^{(n+1)},w^{(n+1)}] .
\end{equation}

\emph{Base of the induction.} For $n=0$ the first identity is trivial, since
$(\mathcal D^0g)(z)=g(z)=g[z]$.

\emph{From \eqref{eq:sres-operator-even} to \eqref{eq:sres-operator-odd}.}
Assume \eqref{eq:sres-operator-even} for some $n\geq0$. Evaluating it at
$\bar z$ interchanges the two nodes, and symmetry rewrites the result as
$g[z^{(n)},\bar z^{(n+1)}]$; hence
\[
  (\mathcal S\mathcal D^ng)(z)
  =\frac{(\mathcal D^ng)(z)-(\mathcal D^ng)(\bar z)}{z-\bar z}
  =\frac{n!}{z-\bar z}
   \Bigl(g[z^{(n+1)},\bar z^{(n)}]-g[z^{(n)},\bar z^{(n+1)}]\Bigr) ,
\]
which is $n!\,g[z^{(n+1)},\bar z^{(n+1)}]$ by
\eqref{eq:dd-balanced-recursion} with $w=\bar z$. This is
\eqref{eq:sres-operator-odd} for $n$.

\emph{From \eqref{eq:sres-operator-odd} to \eqref{eq:sres-operator-even}.}
Differentiating the identity just proved requires separating the two
conjugate nodes, so introduce the function of two independent variables
\[
  \Theta_n(z,w):=g[z^{(n+1)},w^{(n+1)}] .
\]
By Theorem~\ref{thm:dd-contour}, $\Theta_n$ is holomorphic in each variable
separately, confluences included, and continuous in the pair; in particular
$\partial_{\bar z}\Theta_n=\partial_{\bar w}\Theta_n=0$. For its diagonal
restriction $z\mapsto\Theta_n(z,\bar z)$, which is the function appearing in
\eqref{eq:sres-operator-odd}, the Wirtinger chain rule therefore gives
\begin{equation}\label{eq:dd-wirtinger-diagonal}
  \partial_z\bigl(\Theta_n(z,\bar z)\bigr)
  =(\partial_1\Theta_n)(z,\bar z)\,\partial_zz
   +(\partial_2\Theta_n)(z,\bar z)\,\partial_z\bar z
  =(\partial_1\Theta_n)(z,\bar z) ,
\end{equation}
since $\partial_zz=1$ and $\partial_z\bar z=0$; it is the holomorphy of
$\Theta_n$ in the second variable that removes the second summand
altogether, rather than replacing it by a $\partial_{\bar w}$ term. Now
Proposition~\ref{prop:dd-node-derivative}, applied to the block of $n+1$
copies of the first node, gives
$\partial_1\Theta_n(z,w)=(n+1)g[z^{(n+2)},w^{(n+1)}]$, whence
\[
  (\mathcal D^{n+1}g)(z)
  =\partial_z(\mathcal S\mathcal D^ng)(z)
  =n!\,(n+1)\,g[z^{(n+2)},\bar z^{(n+1)}]
  =(n+1)!\,g[z^{(n+2)},\bar z^{(n+1)}] ,
\]
which is \eqref{eq:sres-operator-even} with $n+1$ in place of $n$. The
induction is complete.
\end{proof}

\begin{remark}
\label{rem:dd-first-operator-identities}
The first few cases make the alternating pattern visible:
\[
 \begin{aligned}
  g(z)&=g[z], &\qquad
  \mathcal S\mathcal Dg(z)&=g[z,z,\bar z,\bar z],\\
  \mathcal Sg(z)&=g[z,\bar z], &\qquad
  \mathcal D^2g(z)&=2!\,g[z,z,z,\bar z,\bar z].\\
  \mathcal Dg(z)&=g[z,z,\bar z],
 \end{aligned}
\]
The multiplicities of the two nodes run through
$(1,0),(1,1),(2,1),(2,2),(3,2),\ldots$: each $\mathcal S$ balances them,
each $\partial_z$ adds one copy of the moving node and, with it, the
integer factor of Proposition~\ref{prop:dd-node-derivative}. The factorials
therefore accumulate only at every other step.
\end{remark}

\begin{remark}[Extension to a real node]\label{rem:dd-real-extension}
The denominator $z-\bar z$ in \eqref{eq:dd-conjugate-operator} vanishes on
$\R$, but the singularity is removable for functions arising from
holomorphic data. Indeed, the right-hand sides of
\eqref{eq:sres-operator-even} and \eqref{eq:sres-operator-odd} are divided
differences, hence continuous at confluent nodes by
Theorem~\ref{thm:dd-contour}; so if $g$ is holomorphic near a real point
$x$, letting $z\to x$ makes all their nodes coalesce at $x$, and
\eqref{eq:dd-repeated-block} applies to the resulting blocks of $2n+1$ and
$2n+2$ equal nodes:
\[
  \lim_{z\to x}(\mathcal D^ng)(z)=\frac{n!}{(2n)!}g^{(2n)}(x),
  \qquad
  \lim_{z\to x}(\mathcal S\mathcal D^ng)(z)
  =\frac{n!}{(2n+1)!}g^{(2n+1)}(x) .
\]
The non-real case is the one needed for conjugate polar nodes, but these
limits will matter on circular domains that do meet the real axis.
\end{remark}

Combining the operator theorem with
Corollary~\ref{cor:dd-paired-residue} gives the precise complex formula that
will be transferred to slice regular functions.

\begin{corollary}
\label{cor:dd-residue-operator}
Let $h,p,\bar p,M$ and $g_M=\DeltaSphere{p}^Mh$ be as in
Proposition~\ref{prop:dd-two-point-principal-part}. Then
\begin{equation}\label{eq:dd-paired-residue-operator}
  \boxed{\;
  \PRes_{\{p,\bar p\}}h
  =g_M[p^{(M)},\bar p^{(M)}]
  =\frac{[(\mathcal S\mathcal D^{M-1})g_M](p)}{(M-1)!} ,
  \;}
\end{equation}
and the same value is obtained by evaluating the differential expression at
$\bar p$ instead.
\end{corollary}

\begin{proof}
The first equality is \eqref{eq:cx-sres-dd}, and the second is
\eqref{eq:sres-operator-odd} with $n=M-1$ and $z=p$, divided by $(M-1)!$.
Evaluating at $\bar p$ turns the divided difference into
$g_M[\bar p^{(M)},p^{(M)}]$, which has the same value by symmetry.
\end{proof}

\begin{remark}
\label{rem:dd-representative-invariance}
Corollary~\ref{cor:dd-residue-operator} proves invariance under exchanging
the two nodes of one conjugate pair. In the
quaternionic setting, replacing $p$ by another point of the same sphere
changes the complex slice, and with it the complex structure, so that
independence of that wider choice does not follow from the symmetry of a
two-node divided difference. It will be proved in
Section~\ref{sec:spherical-residues}, by identifying every slice computation
with one and the same quaternionic spherical Laurent coefficient.
\end{remark}

\begin{example}[The unequal-order model revisited]
\label{ex:dd-unequal-orders}
Return to the function of Example~\ref{ex:complex-unequal-orders},
\[
  h(z)=\frac{z^2}{(z-\ci)^2(z+\ci)} ,
\]
whose conjugate pair is $p=\ci$, $\bar p=-\ci$. A common regularizing
exponent is $M=2$, and
\[
  g_2(z)=(z^2+1)^2h(z)=z^2(z+\ci)=z^3+\ci z^2 ,
\]
so that the balanced divided difference $g_2[\ci^{(2)},(-\ci)^{(2)}]$ is one
of order three. Now the divided difference of order $N$ of a polynomial of degree $N$ equals
its leading coefficient, whatever the nodes. Indeed, by
Corollary~\ref{cor:dd-hermite-jets} the Hermite interpolant on $N+1$ nodes
is the polynomial itself, so that
\eqref{eq:sres-newton} has vanishing remainder and expresses it as
$\sum_{j=0}^{N}P_j(z)\,g[a_0,\ldots,a_j]$; of the Newton products only the
last, monic of degree $N$, has degree $N$, so the coefficient of $z^N$ is
$g[a_0,\ldots,a_N]$. Since $g_2$ is monic of degree three,
\begin{equation}\label{eq:dd-unequal-balanced}
  g_2[\ci^{(2)},(-\ci)^{(2)}]=1 ,
\end{equation}
in agreement with the paired residue
\eqref{eq:complex-unequal-paired-residue}.

The differential formula gives the same value. For non-real $z$,
\[
  (\mathcal Sg_2)(z)
  =\frac{z^3-\bar z^3+\ci(z^2-\bar z^2)}{z-\bar z}
  =z^2+z\bar z+\bar z^2+\ci(z+\bar z) ,
\]
and, treating $z$ and $\bar z$ as independent variables,
\[
  (\mathcal Dg_2)(z)=\partial_z(\mathcal Sg_2)(z)=2z+\bar z+\ci .
\]
Applying $\mathcal S$ once more means evaluating this last expression at
$\bar z$, that is, interchanging $z$ and $\bar z$ in it rather than
conjugating its value:
\[
  (\mathcal S\mathcal Dg_2)(z)
  =\frac{(2z+\bar z+\ci)-(2\bar z+z+\ci)}{z-\bar z}=1 ,
\]
identically in $z$, so that in particular
$[(\mathcal S\mathcal D)g_2](\pm\ci)=1$ as
\eqref{eq:dd-paired-residue-operator} predicts. Observe that
$g_2(-\ci)=0$ while the differential expression evaluated at $-\ci$ returns
the nonzero paired residue: the computation uses differences and derivatives
of $g_2$, and at no point divides by one of its values.

The example also illustrates Remark~\ref{rem:dd-nonminimal}. The nonminimal
exponent $M'=3$ gives $g_3(z)=(z^2+1)g_2(z)$, again monic, of degree five,
whence by the same argument
\[
  g_3[\ci^{(3)},(-\ci)^{(3)}]=1 ,
  \qquad
  \frac{[(\mathcal S\mathcal D^2)g_3](\pm\ci)}{2!}=1 .
\]
\end{example}

\section{Preliminaries on slice regular functions}
\label{sec:slice-preliminaries}

This section collects the part of slice regular function theory used in the
rest of the paper, in the form in which it will be used. Three requirements
shape the presentation: the language must work equally well on circular sets
meeting the real axis and on circular sets without real points; quaternionic
statements must reduce to vector-valued complex statements on one full slice;
and the differential operators must be the ones that turn the
divided-difference formula of Section~\ref{sec:divided-differences} into an
intrinsic quaternionic identity.

The stem-function approach meets all three, and we adopt it from the outset.
The basic construction is in \cite{GhiloniPerotti2011,GPS-algebra} and the
theory of singularities in \cite{GPSSingularities}; for the classical theory
on slice domains we refer to
\cite{GentiliStruppa2007,GentiliStoppatoStruppa}, and for the phenomena
specific to domains without real points to \cite{AltavillaNoReal}. Throughout
the paper, functions are \emph{left} slice regular, quaternionic coefficients
are written on the right, and complex scalar kernels and differentials
therefore act from the left.

\subsection{Quaternions, slices, and circular sets}
\label{subsec:slice-circular-sets}

Our algebraic conventions are standard; see
\cite[Chapter~1]{GentiliStoppatoStruppa} or
\cite[Section~2.1]{Sudbery}. The quaternion algebra is the
four-dimensional real vector space
\[
  \Hh=\R+\R\qi+\R\qj+\R\qk ,
\]
whose multiplication is determined by \(\qi^2=\qj^2=\qk^2=-1\), by
\(\qi\qj=\qk\), \(\qj\qk=\qi\), \(\qk\qi=\qj\), and by the sign change that
follows when two distinct imaginary units are interchanged, as in
\(\qj\qi=-\qk\). For \(q=x_0+x_1\qi+x_2\qj+x_3\qk\) we write
\[
  \bar q=x_0-x_1\qi-x_2\qj-x_3\qk,
  \qquad
  \req q=\frac{q+\bar q}{2}=x_0,
  \qquad
  \imq q=\frac{q-\bar q}{2},
\]
and \(|q|^2=q\bar q=\bar qq=\sum_{j=0}^3x_j^2\), so that every nonzero
quaternion is invertible, with \(q^{-1}=\bar q/|q|^2\). Conjugation reverses
the order of a product: \(\overline{pq}=\bar q\bar p\).

The purely imaginary unit quaternions form the two-sphere
\[
  \Sphere=\{I\in\Hh:I^2=-1\}\subset \imq\Hh\simeq\R^3 ,
\]
 and each \(I\in\Sphere\)
spans a copy of the complex plane
\[
  \C_I=\R+I\R ,
\]
in which multiplication is commutative and is identified with complex
multiplication through \(x+\ci r\mapsto x+Ir\). Here and throughout, \(\ci\)
is the ordinary complex imaginary unit. We underline the difference between the complex unit \(\ci\) with the quaternionic one \(\qi\).

A non-real quaternion can be written in two ways. It has a unique expression
\begin{equation}\label{eq:slice-positive-coordinates}
  q=x+Ir,
  \qquad
  x=\req q,\quad r=|\imq q|>0,\quad I=\frac{\imq q}{|\imq q|} ,
\end{equation}
but once a slice \(\C_I\) has been fixed it is more convenient to allow
\emph{signed slice coordinates}
\begin{equation}\label{eq:slice-signed-coordinates}
  q=x+Ir,\qquad x,r\in\R ,
\end{equation}
where the sign matters: the pairs \((I,r)\) and \((-I,-r)\) describe the
same quaternion, and quaternionic conjugation sends \(x+Ir\) to \(x-Ir\).

For \(x\in\R\) and \(r\geq0\), the set
\begin{equation}\label{eq:slice-conjugacy-sphere}
  x+\Sphere r:=\{x+Jr:J\in\Sphere\}
\end{equation}
consists of the quaternions with real part \(x\) and imaginary part of
modulus \(r\); it is a two-sphere in the affine space \(x+\imq\Hh\) when
\(r>0\), and the single real point \(x\) when \(r=0\). We call it the
\emph{conjugacy sphere} through any of its points.

Circularization is the standard construction of
\cite[Section~3.2, Definition~5]{GhiloniPerotti2011} (see also
\cite[Section~2]{GPSSingularities}). Let \(E\subset\C\) be open and invariant
under complex conjugation. Its \emph{circularization} is
\begin{equation}\label{eq:slice-circularization}
  \Omega_E:=\{x+Ir:x+\ci r\in E,\ I\in\Sphere\} ,
\end{equation}
the variable \(r\) being signed; since \(E\) is conjugation invariant, one
may equally restrict to \(r\geq0\),
\begin{equation}\label{eq:slice-circularization-positive}
  \Omega_E=\{x+Ir:x+\ci r\in E,\ r\geq0,\ I\in\Sphere\} .
\end{equation}
A subset of \(\Hh\) of this form is called \emph{circular}; equivalently, it
contains the whole conjugacy sphere through each of its points.

\begin{proposition}
\label{prop:slice-circularization-properties}
Let \(E\subset\C\) be open and invariant under conjugation. Then
\(\Omega_E\) is open in \(\Hh\) and, for every \(I\in\Sphere\),
\begin{equation}\label{eq:slice-full-section}
  \Omega_E\cap\C_I=E_I:=\{x+Ir:x+\ci r\in E\} ,
\end{equation}
where \(r\) ranges over all of \(\R\), so that \(E_I\) is the \emph{full}
slice section and not merely its half with \(r>0\). Moreover a conjugacy sphere
\(x+\Sphere r\) with \(r>0\) meets \(\C_I\) exactly at the two conjugate
points \(x\pm Ir\).
\end{proposition}

\begin{proof}
Openness follows from \(\Omega_E=\pi^{-1}(E)\), where
\(\pi(q):=\req q+\ci|\imq q|\) is continuous and the equality holds by
\eqref{eq:slice-circularization-positive} together with the conjugation
invariance of \(E\). For \eqref{eq:slice-full-section}, one inclusion is the
definition; conversely, a point \(q=x+Ir\in\Omega_E\cap\C_I\) with signed
\(r\) is represented by the pair \((I,r)\) or by \((-I,-r)\), and in either
case \(x+\ci r\in E\) because \(E\) is conjugation invariant, so that
\(q\in E_I\). The last assertion holds because \(x+Jr\in\C_I\) forces
\(Jr\in I\R\), hence \(J=\pm I\) when \(r>0\).
\end{proof}

Among connected circular open sets, two types are standard. A
\emph{slice domain} is one which meets \(\R\) and whose intersection with
every slice is a standard complex domain; a \emph{product domain} is one disjoint from \(\R\).
The planar profile of a product domain splits as
\begin{equation}\label{eq:slice-product-profile}
  E=E^+\sqcup E^-,
  \qquad
  E^+\subset\{z\in\C:\imq z>0\},
  \qquad
  E^-=\{\bar z:z\in E^+\} ,
\end{equation}
with \(E^+\) connected, and \(\Omega_E\) is then homeomorphic to
\(E^+\times\Sphere\). No boundary regularity is assumed at this stage.

\begin{remark}[Full slices on a product domain]
\label{rem:slice-two-components}
The splitting \eqref{eq:slice-product-profile} has a consequence that will
matter repeatedly. On a product domain the full slice
\(E_I=\Omega_E\cap\C_I\) is disconnected: it has one component over \(E^+\)
and one over \(E^-\), and quaternionic conjugation exchanges them. Data given
on one component alone do not determine a slice function, so that statements
which on a slice domain involve a single connected region here require both
conjugate components of one full slice. This is one of the reasons for the care taken
in Section~\ref{sec:divided-differences} to allow disconnected domains, and
it is why a conjugate pair, rather than a point, is the basic local object of
this paper.
\end{remark}

\subsection{Stem functions and slice regularity}
\label{subsec:slice-stems}

The complexified quaternion algebra is defined as
\[
  \Hh_\C:=\Hh\otimes_\R\C=\{a+\ci b:a,b\in\Hh\} ,
\]
where the imaginary unit \(\ci\) is central, that is, it commutes with every
quaternion. A bar denotes conjugation with respect to this central complex
structure,
\begin{equation}\label{eq:slice-central-conjugation}
  \overline{a+\ci b}=a-\ci b ,
\end{equation}
an operation not to be confused with quaternionic conjugation, whose
extension to \(\Hh_\C\) is given in
Subsection~\ref{subsec:slice-product-conjugation}.

The next definitions are those of
\cite[Definitions~4--5]{GhiloniPerotti2011}.

\begin{definition}[Stem and slice functions]
\label{def:slice-stem-function}
Let \(E\subset\C\) be open and invariant under conjugation. A map
\[
  F=A+\ci B:E\longrightarrow\Hh_\C,
  \qquad A,B:E\longrightarrow\Hh ,
\]
is a \emph{stem function} if
\begin{equation}\label{eq:slice-stem-symmetry}
  F(\bar z)=\overline{F(z)} ,
\end{equation}
or equivalently if
\begin{equation}\label{eq:slice-stem-parity}
  A(\bar z)=A(z),\qquad B(\bar z)=-B(z) .
\end{equation}
The \emph{slice function induced by \(F\)} is the map
\(f=\mathcal I(F):\Omega_E\to\Hh\) given by
\begin{equation}\label{eq:slice-induction}
  f(x+Ir)=A(x+\ci r)+IB(x+\ci r) .
\end{equation}
\end{definition}

The parity relations \eqref{eq:slice-stem-parity} are exactly what makes
\eqref{eq:slice-induction} well posed: the only ambiguity in signed
coordinates is \(x+Ir=x+(-I)(-r)\), and the second representation gives
\(A(x-\ci r)-IB(x-\ci r)=A(x+\ci r)+IB(x+\ci r)\), the same value as the
first; on the real axis \(B\) is odd, hence \(B(x)=0\), and the value
\(A(x)\) does not depend on \(I\).

The construction can be reversed on each conjugacy sphere. The reconstruction
identities and the representation formula below are classical (a detailed proof can be found in 
\cite[Propositions~5--6]{GhiloniPerotti2011}).

\begin{proposition}[Reconstruction and representation formula]
\label{prop:slice-reconstruction}
Let \(f=\mathcal I(A+\ci B)\), let \(I\in\Sphere\), and let \(x+\ci r\in E\).
Then
\begin{align}
  A(x+\ci r)&=\frac{f(x+Ir)+f(x-Ir)}{2},
  \label{eq:slice-reconstruct-A}\\
  B(x+\ci r)&=-\frac{I}{2}\bigl(f(x+Ir)-f(x-Ir)\bigr),
  \label{eq:slice-reconstruct-B}
\end{align}
and consequently, for every \(J\in\Sphere\),
\begin{equation}\label{eq:slice-representation-formula}
  f(x+Jr)=\frac{1-JI}{2}f(x+Ir)+\frac{1+JI}{2}f(x-Ir) .
\end{equation}
In particular a slice function is determined by its restriction to one full
slice.
\end{proposition}


\begin{remark}
\label{rem:slice-full-slice-needed}
On a slice domain, the two points \(x+Ir\) and \(x-Ir\) lie in the same
connected slice section, so there is no ambiguity. On a product domain, they
belong to two disconnected components. Formula
\eqref{eq:slice-representation-formula} requires one value from each
component. 
\end{remark}

We next introduce regularity. A map into \(\Hh_\C\) is called holomorphic if
all its components in a real quaternionic basis are complex holomorphic.

This definition and its Cauchy--Riemann characterization agree with
\cite[Definitions~7--8 and Proposition~8]{GhiloniPerotti2011}.

\begin{definition}[Slice regularity]
\label{def:slice-regularity}
A slice function \(f=\mathcal I(F)\) on \(\Omega_E\) is \emph{slice
regular} if its stem function \(F\) is holomorphic on \(E\).
\end{definition}

Write \(F=A+\ci B\), with \(A,B\) of class \(C^1\). The condition
\(\partial_{\bar z}F=0\) is equivalent, component by component, to the
Cauchy--Riemann equations
\begin{equation}\label{eq:slice-stem-CR}
  \partial_xA=\partial_rB,
  \qquad
  \partial_rA=-\partial_xB.
\end{equation}
These are ordinary Cauchy--Riemann equations for four complex-valued
coordinates. Quaternionic multiplication does not enter them.

The simplest examples are polynomials with coefficients on the right. If
\(a_0,\ldots,a_N\in\Hh\), then
\[
  F(z)=\sum_{n=0}^N z^na_n
\]
is a holomorphic stem, because the central variable \(z\) commutes with all
quaternionic coefficients and
\(F(\bar z)=\overline{F(z)}\). Its induced function is
\[
  f(q)=\sum_{n=0}^N q^na_n,
\]
which is therefore slice regular on \(\Hh\). 

To relate this definition to usual complex analysis, fix \(I\in\Sphere\) and
endow the real vector space \(\Hh\) with the complex structure
\[
  L_I(a):=Ia.
\]
We denote this two-dimensional complex vector space by \((\Hh,L_I)\).
Complex scalars act on it from the left:
\begin{equation}\label{eq:slice-complex-scalar-action}
  (u+\ci v)\cdot a:=(u+Iv)a,
  \qquad u,v\in\R,\quad a\in\Hh.
\end{equation}
Define the evaluation map
\begin{equation}\label{eq:slice-evaluation-map}
  \Phi_I:\Hh_\C\longrightarrow(\Hh,L_I),
  \qquad
  \Phi_I(a+\ci b)=a+Ib.
\end{equation}

The complex reduction to a fixed slice, and the associated splitting into
complex-valued components, is the standard splitting mechanism; compare
\cite[Lemma~2.7]{GPSSingularities}.

\begin{lemma}
\label{lem:slice-evaluation-linear}
The map \(\Phi_I\) is complex linear. If \(f=\mathcal I(F)\), then
\begin{equation}\label{eq:slice-restriction-as-evaluation}
  f(x+Ir)=\Phi_I\bigl(F(x+\ci r)\bigr) ,
\end{equation}
so that if \(f\) is slice regular, its restriction \(f_I\) to the full slice
\(E_I\) is a holomorphic map with values in \((\Hh,L_I)\).
\end{lemma}

\begin{proof}
For \(\lambda=u+\ci v\in\C\) and \(w=a+\ci b\in\Hh_\C\), centrality of
\(\ci\) gives \(\lambda w=(ua-vb)+\ci(va+ub)\), whence
\[
  \Phi_I(\lambda w)=ua-vb+I(va+ub)=(u+Iv)(a+Ib)=\lambda\cdot\Phi_I(w) ,
\]
the middle equality using only \(u,v\in\R\) and \(I^2=-1\). Multiplication
is not preserved: if \(J\in\Sphere\) is orthogonal to \(I\), so that
\(JI=-IJ\), then \(\ci\) commutes with \(J\) in \(\Hh_\C\), yet
\[
  \Phi_I(J\ci)=IJ,
  \qquad
  \Phi_I(J)\Phi_I(\ci)=JI=-IJ .
\]
Finally, \eqref{eq:slice-restriction-as-evaluation} is
\eqref{eq:slice-induction} itself, and the composition of a holomorphic \(F\)
with the complex-linear \(\Phi_I\) is holomorphic.
\end{proof}

The failure of multiplicativity just observed is the reason why the product
of slice functions, introduced in
Subsection~\ref{subsec:slice-product-conjugation}, cannot be the pointwise
quaternionic one. As for the last assertion of the lemma, it can also be
verified in coordinates: from \(f_I=A+IB\),
\[
  (\partial_x+I\partial_r)f_I
  =(\partial_xA-\partial_rB)+I(\partial_xB+\partial_rA) ,
\]
which vanishes by \eqref{eq:slice-stem-CR} — precisely the Cauchy--Riemann
equation for a map with values in \((\Hh,L_I)\).

When a classical complex theorem has to be applied componentwise, the
following scalar decomposition is convenient. Write
\((e_0,e_1,e_2,e_3)=(1,\qi,\qj,\qk)\); every stem then has a unique
expression
\begin{equation}\label{eq:slice-component-decomposition}
  F=\sum_{\ell=0}^3F_\ell e_\ell,
  \qquad F_\ell:E\longrightarrow\C ,
\end{equation}
and the stem symmetry \eqref{eq:slice-stem-symmetry} amounts to
\(F_\ell(\bar z)=\overline{F_\ell(z)}\) for each \(\ell\), that is, to the
reflection symmetry of Definition~\ref{def:complex-reflection} for the four
scalar coordinates separately. Setting \(f_\ell=\mathcal I(F_\ell)\) gives
\begin{equation}\label{eq:slice-function-component-decomposition}
  f=\sum_{\ell=0}^3f_\ell e_\ell ,
\end{equation}
with the quaternionic basis elements on the right.

\begin{definition}[Slice-preserving functions]\label{def:slice-preserving}
A slice function is \emph{slice preserving} if it maps every slice
\(\C_I\) into itself, or equivalently if its stem is complex valued, that
is, of the form \(F=A+\ci B\) with \(A,B\) real valued.
\end{definition}

Examples of slice preserving (regular) functions are polynomials and power series with real coefficients.

Each \(f_\ell\) in \eqref{eq:slice-function-component-decomposition} is
therefore slice preserving. If \(F\) is holomorphic, so are all the
\(F_\ell\), and all the \(f_\ell\) are slice regular; consequently every
complex integral, derivative, Laurent expansion or residue used below may be
computed componentwise in
\eqref{eq:slice-function-component-decomposition}.

\subsection{Slice product, conjugation, and symmetrization}
\label{subsec:slice-product-conjugation}

The pointwise product of two slice functions need not be a slice function,
and the natural product is instead transported from the complexified
algebra. It is the standard slice product of \cite{GentiliStoppatoStruppa},
\cite[Section~5, Definition~9]{GhiloniPerotti2011} and
\cite[Section~1.2]{GPS-algebra}.

\begin{definition}[Slice product]
\label{def:slice-product}
If \(f=\mathcal I(F)\) and \(g=\mathcal I(G)\) are slice functions on the
same circular set, their \emph{slice product} is
\begin{equation}\label{eq:slice-product}
  f*g:=\mathcal I(FG) ,
\end{equation}
the product \(FG\) being computed pointwise in \(\Hh_\C\).
\end{definition}

Since multiplication in \(\Hh_\C\) is associative, so is \(*\); it is
distributive and real bilinear, but not commutative in general. Holomorphy
of \(F\) and \(G\) passes to \(FG\), so the slice product of two slice
regular functions is slice regular. The simplification used throughout this
paper occurs when one factor is slice preserving, in which case the slice
product reduces to the pointwise one.

\begin{proposition}
\label{prop:slice-preserving-factor}
Let \(h=\mathcal I(H)\) be slice preserving and let \(f=\mathcal I(F)\) be
any slice function on the same circular set. Then
\begin{equation}\label{eq:slice-intrinsic-factor}
  h*f=f*h,
  \qquad
  (h*f)(q)=h(q)f(q) .
\end{equation}
\end{proposition}


Two conjugations are at work here, and it is worth keeping their roles
apart. Complex conjugation \eqref{eq:slice-central-conjugation} changes the
sign of the central unit \(\ci\) and is what the stem symmetry is about;
quaternionic conjugation changes the signs of \(\qi,\qj,\qk\). We extend the
latter to \(\Hh_\C\) \emph{complex linearly}, setting
\begin{equation}\label{eq:slice-stem-conjugate}
  F^c:=\sum_{\ell=0}^3F_\ell\bar e_\ell
\end{equation}
for \(F=\sum_{\ell=0}^3F_\ell e_\ell\), so that the complex coefficients
\(F_\ell\) are left untouched. The operation is an anti-involution,
\((FG)^c=G^cF^c\), and it maps stems to stems: the two conjugations commute,
whence
\[
  F^c(\bar z)=\bigl(F(\bar z)\bigr)^c
             =\bigl(\overline{F(z)}\bigr)^c=\overline{F^c(z)} .
\]

\begin{definition}[Slice conjugate and symmetrization]
\label{def:slice-conjugate-symmetrization}
For \(f=\mathcal I(F)\), the \emph{slice conjugate} and the
\emph{symmetrization} of \(f\) are
\begin{equation}\label{eq:slice-conjugate-symmetrization}
  f^c:=\mathcal I(F^c),
  \qquad
  f^s:=f*f^c .
\end{equation}
\end{definition}

These are the standard operations of the slice algebra. Note that \(f^c\) is not the pointwise
quaternionic conjugate: at \(q=x+Ir\), where \(f(q)=A+IB\),
\[
  f^c(q)=\bar A+I\bar B
  \qquad\text{while}\qquad
  \overline{f(q)}=\bar A-\bar BI ,
\]
the discrepancy being due to the reversal of order under quaternionic
conjugation.

The symmetrization has a particularly simple stem. Multiplying out
\eqref{eq:slice-component-decomposition}, the mixed terms cancel because
distinct quaternionic units anticommute, leaving
\begin{equation}\label{eq:slice-null-quadric}
  FF^c=F^cF=Q(F),
  \qquad
  Q(F):=\sum_{\ell=0}^3F_\ell^2 .
\end{equation}
This is complex scalar valued and satisfies the stem symmetry, so that
\(f^s=f^c*f\) is slice preserving. What matters for us is that \(Q(F)\) is a
\emph{complex} quadratic form and not a Hermitian norm: the
equation \(w_0^2+w_1^2+w_2^2+w_3^2=0\) defines the null quadric of
\(\C^4\), and a stem may take all its values there without ever vanishing.

\begin{example}[A zero divisor on a product domain]
\label{ex:slice-zero-divisor-preliminary}
Let \(\Omega_E\) be a product domain with profile
\(E=E^+\sqcup E^-\) as in \eqref{eq:slice-product-profile}, fix
\(K\in\Sphere\), and consider the locally constant, hence holomorphic, map
\begin{equation}\label{eq:slice-zero-divisor-stem}
  H(z)=
  \begin{cases}
    1+\ci K,&z\in E^+,\\
    1-\ci K,&z\in E^- .
  \end{cases}
\end{equation}
Its two values are exchanged by central complex conjugation, so \(H\) is a
stem, and evaluating it at \(q=x+Ir\) with \(r>0\) gives \(1+IK\); the
induced slice regular function is therefore
\[
  h(q)=1+\mathcal J(q)K ,
  \qquad \mathcal J(q):=\frac{\imq q}{|\imq q|} ,
\]
which vanishes nowhere except on one slice, as we now see. On \(E^+\) the
conjugate stem is \(H^c=1-\ci K\), and since \((\ci K)^2=\ci^2K^2=1\),
\[
  HH^c=(1+\ci K)(1-\ci K)=0 ;
\]
the same computation holds on \(E^-\). Hence
\[
  h^s=h*h^c\equiv0
  \qquad\text{while}\qquad
  h\not\equiv0,\quad h^c\not\equiv0 ,
\]
so that \(h\) is a genuine zero divisor of the slice algebra. Concretely, on
the single slice \(\C_K\) the function \(h\) vanishes identically on the
component over \(E^+\) and equals \(2\) on the component over \(E^-\),
whereas on every other slice it vanishes nowhere.

This  example is revisited geometrically in
Section~\ref{sec:examples}.
\end{example}

\subsection{Slice and spherical derivatives}
\label{subsec:slice-derivatives}

Let \(f=\mathcal I(A+\ci B)\) be a slice function induced by a \(C^1\)
stem. On a fixed slice \(\C_I\), with signed coordinates \(q=x+Ir\), define
the \emph{slice derivative} by
\begin{equation}\label{eq:slice-derivative-definition}
  (\partial_cf)_I(x+Ir)
  :=\frac12(\partial_x-I\partial_r)f_I(x+Ir).
\end{equation}
A direct calculation gives
\begin{equation}\label{eq:slice-derivative-stem-coordinates}
  \partial_cf=\mathcal I(\partial_zF),
\end{equation}
where
\begin{equation}\label{eq:slice-dz-stem}
  \partial_zF
  =\frac12\bigl(\partial_xA+\partial_rB\bigr)
   +\frac{\ci}{2}\bigl(\partial_xB-\partial_rA\bigr).
\end{equation}
In particular, the
definition is compatible with changing the signed representation of a
non-real quaternion.

If \(f\) is slice regular, the Cauchy--Riemann equations
\eqref{eq:slice-stem-CR} reduce \eqref{eq:slice-dz-stem} to
\(\partial_zF=\partial_xF\).
Consequently,
\begin{equation}\label{eq:slice-derivative-holomorphic}
  (\partial_cf)_I=\partial_xf_I
\end{equation}
on every slice, and \(\partial_cf\) is again slice regular.

The second operator compares  two conjugate points on the same sphere.
Both points belong to \(\Omega_E\), because this set is circular.

The spherical value and spherical derivative were introduced in this form in
\cite[Definition~6]{GhiloniPerotti2011} (compare
\cite{GentiliStoppatoStruppa}).

\begin{definition}[Spherical value and spherical derivative]
\label{def:slice-spherical-derivative}
For \(q\in\Omega_E\setminus\R\), define
\begin{align}
  f_s^\circ(q)
  &:=\frac{f(q)+f(\bar q)}2,
  \label{eq:slice-spherical-value}\\
  \partial_sf(q)
  &:=(q-\bar q)^{-1}\bigl(f(q)-f(\bar q)\bigr).
  \label{eq:slice-spherical-derivative}
\end{align}
\end{definition}

The notation \(f_s^\circ\) for the spherical value should not be confused
with the superscript in the symmetrization \(f^s\). If \(q=x+Ir\), with
signed \(r\neq0\), then
\[
  f(q)=A(x+\ci r)+IB(x+\ci r),
  \qquad
  f(\bar q)=A(x+\ci r)-IB(x+\ci r).
\]
It follows that
\begin{equation}\label{eq:slice-spherical-stem-coordinates}
  f_s^\circ(x+Ir)=A(x+\ci r),
  \qquad
  \partial_sf(x+Ir)=\frac{B(x+\ci r)}{r}.
\end{equation}
For the second equality, observe that
\[
  q-\bar q=2Ir,
  \qquad
  f(q)-f(\bar q)=2IB(x+\ci r),
\]
and multiply the latter expression on the left by
\((2Ir)^{-1}=-I/(2r)\).

Formula \eqref{eq:slice-spherical-stem-coordinates} has three immediate
consequences. First, both \(f_s^\circ\) and \(\partial_sf\) are constant
on each conjugacy sphere. Second,
\begin{equation}\label{eq:slice-spherical-decomposition}
  f(q)=f_s^\circ(q)+\imq q\,\partial_sf(q).
\end{equation}
Indeed, \(\imq q=Ir\), so the right-hand side is \(A+IB\). Third, although
\(\partial_sf\) is defined by a quotient, it has no singularity forced by
the real axis. Since the stem symmetry gives \(B(x)=0\), differentiability
yields
\[
  \lim_{r\to0}\frac{B(x+\ci r)}r=\partial_rB(x).
\]
For a slice regular function, \(\partial_rB(x)=\partial_xA(x)\) by
\eqref{eq:slice-stem-CR}; hence the continuous real-axis extension satisfies
\begin{equation}\label{eq:slice-derivatives-real-axis}
  \partial_sf(x)=\partial_cf(x),
  \qquad x\in E\cap\R.
\end{equation}
We shall mainly use \(\partial_s\) on nonreal spheres, where
\eqref{eq:slice-spherical-derivative} is unambiguous.

The corresponding properties of the slice and spherical derivatives are
standard; see \cite[Definitions~7--8 and
Proposition~8]{GhiloniPerotti2011} and
\cite{GentiliStoppatoStruppa}.

\begin{proposition}
\label{prop:slice-derivative-properties}
The operators \(\partial_c\) and \(\partial_s\) are additive and right
\(\Hh\)-linear. More precisely, for a constant \(a\in\Hh\),
\begin{equation}\label{eq:slice-derivative-right-linearity}
  \partial_c(fa)=(\partial_cf)a,
  \qquad
  \partial_s(fa)=(\partial_sf)a.
\end{equation}
If \(f\) is slice regular, then \(\partial_cf\) is slice regular. In
contrast, \(\partial_sf\) need not be slice regular.
\end{proposition}



The contrast in the last assertion is already visible for \(f(q)=q^2\):
\[
  \partial_sf(q)=q+\bar q=2\req q.
\]
Its restriction to a slice is \(z+\bar z\), which is not holomorphic since
its \(\partial_{\bar z}\)-derivative equals \(1\).

The preceding example explains a convention needed for the differential
residue formula. In an expression such as
\[
  \partial_s(\partial_c\partial_s)^nf,
\]
operators act from right to left. Thus one first applies \(\partial_s\), then
\(\partial_c\), and repeats. The intermediate functions need not be slice
regular, but they remain differentiable slice functions wherever the formula
is evaluated, hence \(\partial_c\) is then understood through the Wirtinger
formula \eqref{eq:slice-derivative-definition}.

We can now make the connection with Section~\ref{sec:divided-differences}
precise. On a fixed full slice \(\C_I\), set \(z=x+Ir\) and
\(\bar z=x-Ir\). Then
\begin{equation}\label{eq:slice-spherical-as-divided-difference}
  (\partial_sf)_I(z)
  =(z-\bar z)^{-1}\bigl(f_I(z)-f_I(\bar z)\bigr).
\end{equation}
The right-hand side is exactly the conjugate-node operator \(\mathcal S\) of
Section~\ref{sec:divided-differences}. Moreover,
\begin{equation}\label{eq:slice-cullen-as-wirtinger}
  (\partial_ch)_I=\partial_z h_I
  :=\frac12(\partial_x-I\partial_r)h_I
\end{equation}
for every differentiable slice function \(h\). Therefore, on a fixed slice,
\begin{equation}\label{eq:slice-operator-translation}
  \bigl[\partial_s(\partial_c\partial_s)^nf\bigr]_I
  =\mathcal S(\partial_z\mathcal S)^nf_I.
\end{equation}
The variables \(z\) and \(\bar z\) in the calculation of
Section~\ref{sec:divided-differences} are treated as independent Wirtinger
variables; equation \eqref{eq:slice-cullen-as-wirtinger} is exactly the same
convention in signed slice coordinates. We follow the operator conventions
used for spherical coefficients in \cite{AltavillaSphericalCoefficients}.

\subsection{Semiregular functions and restrictions to slices}
\label{subsec:slice-semiregular}

It remains to describe the singular functions to which the residue theorem
will apply. We define them at the level of the stem: this avoids any use of
slice reciprocals, and it continues to make sense on product domains, where
a symmetrization may vanish identically as in
Example~\ref{ex:slice-zero-divisor-preliminary}.

\begin{definition}[Meromorphic stems and semiregular functions]
\label{def:slice-semiregular}
Let \(E\subset\C\) be open and invariant under conjugation. Write a stem as
in \eqref{eq:slice-component-decomposition}:
\[
  F=\sum_{\ell=0}^3F_\ell e_\ell.
\]
We call \(F\) \emph{meromorphic} if every complex scalar coordinate \(F_\ell\) is
meromorphic on \(E\) and the stem symmetry holds wherever \(F\) is finite.
The induced slice function \(f=\mathcal I(F)\), defined away from the
circularization of the poles of \(F\), is called \emph{semiregular}.
\end{definition}

This is the standard notion of semiregularity; for its local interpretation
see \cite{GentiliStoppatoStruppa} and \cite[Definition~11.1 and Remark~11.2]{GPSSingularities}.

The union \(P_F\) of the poles of the four scalar coordinates is a
conjugation-invariant subset of \(E\) that is locally finite, in the sense
that every compact subset of \(E\) contains only finitely many of its
points. Its circularization is easy to describe: a real pole
\(a\in P_F\cap\R\) contributes the isolated real point \(a\), while a
non-real conjugate pair
\[
  p=\alpha+\ci\beta,\qquad \bar p=\alpha-\ci\beta,\qquad \beta>0 ,
\]
contributes the two-sphere
\begin{equation}\label{eq:slice-polar-sphere}
  \Sigma=\alpha+\Sphere\beta ,
\end{equation}
which we call a \emph{polar sphere} of the meromorphic stem. The name is
kept even though the restriction of \(f\) to a particular slice may have a
smaller pole order, or a removable singularity, at one exceptional point of
\(\Sigma\); this is the phenomenon analysed in
Proposition~\ref{prop:slice-exceptional-points} below.

Attached to \(\Sigma\) is the intrinsic characteristic polynomial
\begin{equation}\label{eq:slice-characteristic-polynomial}
  \DeltaSphere{\Sigma}(q):=q^2-2\alpha q+\alpha^2+\beta^2 ,
\end{equation}
slice preserving because its coefficients are real, and factoring as a slice
product
\begin{equation}\label{eq:slice-characteristic-star-factorization}
  \DeltaSphere{\Sigma}(q)=(q-y)*(q-\bar y)
\end{equation}
for every \(y=\alpha+J\beta\in\Sigma\): the stem variable \(z\) is central,
so the right-hand side expands to
\(z^2-z(y+\bar y)+y\bar y=z^2-2\alpha z+\alpha^2+\beta^2\), independently of
which representative \(y\) is chosen. Both the polynomial and the
description of its zero set in
Lemma~\ref{lem:slice-characteristic-zero-set} below are standard (see for instance
\cite[Example~1.22]{GPS-algebra} and
\cite[Section~2, equation~(4)]{GPSSingularities}).

\begin{lemma}
\label{lem:slice-characteristic-zero-set}
The zero set of \(\DeltaSphere{\Sigma}\) is exactly \(\Sigma\).
\end{lemma}


The polynomial \(\DeltaSphere{\Sigma}\) is thus the natural device for
removing the two stem poles at \(p\) and \(\bar p\) at once. The
regularization that follows is the local form of the singularity results of
\cite[Theorems~9.4--9.5 and Remark~11.2]{GPSSingularities}.

\begin{proposition}[Regularization of real points and polar spheres]
\label{prop:slice-regularization}
Let \(f=\mathcal I(F)\) be semiregular.
\begin{enumerate}[label=\textup{(\roman*)}]
   \item If \(a\in\R\) is a pole of \(F\), there exists an integer
        \(L\geq1\) such that
         \begin{equation}\label{eq:slice-real-regularization}
          h_L(q):=(q-a)^L f(q)
        \end{equation}
        extends slice regularly across \(a\).
  \item If \(\Sigma=\alpha+\Sphere\beta\) is the circularization of the
        conjugate stem poles \(p=\alpha+\ci\beta\) and \(\bar p\), there
        exists an integer \(M\geq1\) such that
        \begin{equation}\label{eq:slice-spherical-regularization}
          g_M(q):=\DeltaSphere{\Sigma}(q)^M f(q)
        \end{equation}
        extends slice regularly across the whole sphere \(\Sigma\).
\end{enumerate}

\end{proposition}
The whole point is that the statement is about the stem, not about \(f\).
Through \eqref{eq:slice-component-decomposition}, a meromorphic stem is a
quadruple of scalar meromorphic functions, that is, a meromorphic curve in
\(\C^4\), and the orders of its poles are four ordinary integers; clearing
them all at once is then a matter of taking a maximum, with no quaternionic
multiplication involved anywhere. The slice-preserving character of the
clearing factor, guaranteed by
Proposition~\ref{prop:slice-preserving-factor}, is what carries the
conclusion back to \(f\).
\begin{proof}
Write \(F=\sum F_\ell e_\ell\). Near a real pole \(a\), each scalar
meromorphic function \(F_\ell\) has a pole of some finite order
\(L_\ell\), with \(L_\ell=0\) if it is holomorphic at \(a\). Choose
\[
  L\geq\max_{0\leq\ell\leq3}L_\ell.
\]
Then \((z-a)^LF_\ell(z)\) is holomorphic at \(a\) for every \(\ell\),
so \((z-a)^LF(z)\) is a holomorphic stem near \(a\). Its induced slice
function is \((q-a)^Lf(q)\), because \((q-a)^L\) is slice preserving. This
proves part~\textup{(i)}.

For part~\textup{(ii)}, choose two small conjugate disks around \(p\) and
\(\bar p\) containing no other stem pole. Let \(m_{\ell,p}\) and
\(m_{\ell,\bar p}\) be the pole orders of \(F_\ell\) at the two points,
taking the value zero for a removable point. Choose
\[
  M\geq
  \max_{0\leq\ell\leq3}
  \{m_{\ell,p},m_{\ell,\bar p}\}.
\]
 The scalar polynomial
\[
  \delta(z):=(z-p)(z-\bar p)=z^2-2\alpha z+\alpha^2+\beta^2 ,
\]
which is the \(\DeltaSphere{p}\) of \eqref{eq:complex-characteristic-polynomial},
has a simple zero at each of \(p\) and \(\bar p\).
Therefore
\(\delta^MF_\ell\) is holomorphic at both points for every \(\ell\).
It follows that \(\delta^MF\) extends to a holomorphic stem on a
neighbourhood of the pair \(\{p,\bar p\}\). The induced slice function is
exactly \(\DeltaSphere{\Sigma}^Mf\), again by
Proposition~\ref{prop:slice-preserving-factor}. Its domain contains the
circularization of that neighbourhood and hence the entire sphere \(\Sigma\).
This proves the claim.
\end{proof}

\begin{remark}
\label{rem:slice-nonminimal-regularization}
As usual in these notes, neither \(L\) nor \(M\) in Proposition~\ref{prop:slice-regularization} is
required to be minimal. If \(M\) works, then every \(M+k\), \(k\geq0\),
also works, because
\[
  g_{M+k}=\DeltaSphere{\Sigma}^k g_M.
\]
For \(k\geq1\), this new regularization vanishes identically on \(\Sigma\).
Even a minimal regularization may vanish at an exceptional point. Thus no
later argument may divide by \(g_M(y)\), assume that it is nonzero, or infer
information from the invertibility of a symmetrization. The residue formulas
will instead use derivatives and divided differences of \(g_M\).
\end{remark}

The following local analysis explains why exceptional points occur, and why
there is at most one of them on each polar sphere.

\begin{proposition}[Exceptional points of a polar sphere]
\label{prop:slice-exceptional-points}
Let \(f=\mathcal I(F)\) be semiregular and suppose that \(F\) has a pole of
order \(m\) at \(p=\alpha+\ci\beta\), with leading Laurent term
\begin{equation}\label{eq:slice-stem-leading-laurent}
  F(z)=(z-p)^{-m}(c_1+\ci c_2)+O\bigl((z-p)^{-m+1}\bigr),
  \qquad c_1,c_2\in\Hh,\ c_1+\ci c_2\neq0 .
\end{equation}
Then, at the point \(y=\alpha+I\beta\) of \(\Sigma=\alpha+\Sphere\beta\),
the restriction \(f_I\) has leading coefficient
\begin{equation}\label{eq:slice-exceptional-leading-coefficient}
  \Phi_I(c_1+\ci c_2)=c_1+Ic_2 ,
\end{equation}
and at the conjugate point \(\alpha-I\beta\) it has leading coefficient
\(c_1-Ic_2\). The coefficient
\eqref{eq:slice-exceptional-leading-coefficient} vanishes for at most one
\(I\in\Sphere\), namely for \(I=-c_1c_2^{-1}\) when \(c_2\neq0\) and
\(|c_1|=|c_2|\) with \(c_1c_2^{-1}\) purely imaginary, and for no \(I\)
otherwise. In particular the two coefficients never vanish simultaneously,
unless \(c_1=c_2=0\).
\end{proposition}

\begin{proof}
The leading coefficient of \(f_I\) at \(y\) is the image under \(\Phi_I\) of
the leading coefficient of \(F\) at \(p\), by
Lemma~\ref{lem:slice-evaluation-linear}; the stem symmetry turns
\(c_1+\ci c_2\) into \(c_1-\ci c_2\) at \(\bar p\), whence the second
assertion. If \(c_2=0\) then \(c_1\neq0\) and \(c_1+Ic_2=c_1\) never
vanishes. If \(c_2\neq0\), the equation \(c_1+Ic_2=0\) determines
\(I=-c_1c_2^{-1}\) uniquely, and this quaternion lies in \(\Sphere\)
precisely when it is purely imaginary of unit norm, that is, when
\(|c_1|=|c_2|\) and \(c_1c_2^{-1}\in\imq\Hh\); otherwise no point of
\(\Sigma\) is exceptional. Finally, \(c_1+Ic_2=c_1-Ic_2=0\) forces
\(Ic_2=0\) and \(c_1=0\), hence \(c_2=0\) as well.
\end{proof}

\begin{example}[A polar sphere with one exceptional point]
\label{ex:slice-exceptional-point}
Fix \(K\in\Sphere\) and consider, on \(E=\C\setminus\{\pm\ci\}\), the stem
\[
  F(z):=\frac{z-K}{(z^2+1)^2} .
\]
It is a stem because \(z\mapsto z-K\) is one, its components being
\(A(x+\ci r)=x-K\) and \(B(x+\ci r)=r\), and because the scalar factor
\((z^2+1)^{-2}\) has real coefficients. Its poles are \(\pm\ci\), both of
order two, so the polar sphere is the unit sphere
\(\Sigma=\Sphere\), with \(\DeltaSphere{\Sigma}(q)=q^2+1\), and the induced
semiregular function is
\[
  f(q)=(q^2+1)^{-2}(q-K) .
\]

Restricting to a slice \(\C_I\) and factoring
\((q^2+1)^2=(q-I)^2(q+I)^2\) there, the order of the pole of \(f_I\) at
\(\pm I\) is two minus the order of vanishing of the numerator. Now
\(q-K\) vanishes only at \(q=K\). Hence \(f_I\) has two poles of order two,
at \(I\) and at \(-I\), for every \(I\) with \(\C_I\neq\C_K\); whereas on
the single slice \(\C_K\) it has a pole of order two at \(-K\) and only a
simple pole at \(K\) (see Figure~\ref{fig:slice-exceptional-point}). The point \(K\) is exceptional, and
Proposition~\ref{prop:slice-exceptional-points} confirms that it is the only
one: the leading coefficient of \(F\) at \(\ci\) is
\((2\ci)^{-2}(\ci-K)=(K-\ci)/4\), that is \(c_1=K/4\) and \(c_2=-1/4\), so
that \(c_1+Ic_2=(K-I)/4\) vanishes exactly for \(I=K\).

Two features are worth recording. First, the drop is invisible on the stem,
whose two poles are both of order two, and invisible on \(\Sigma\) as a set,
which is a polar sphere all the same: what varies from point to point is the
order of the pole of the restriction. Second, the minimal regularization
\(g_2=\DeltaSphere{\Sigma}^2f=q-K\) vanishes precisely at the exceptional
point, which is the situation anticipated in
Remark~\ref{rem:slice-nonminimal-regularization}: no argument may divide by
\(g_2(K)\). On the slice \(\C_K\), the conjugate pair \(\{K,-K\}\) carries
poles of orders one and two, exactly as the pair \(\{\ci,-\ci\}\) does in
Example~\ref{ex:complex-unequal-orders}.
\end{example}

\begin{figure}[htbp]
\centering
\begin{tikzpicture}[scale=1.2]
  \def\bet{1.4}
  \def\jx{-0.55}\def\jy{-0.45}
  \pgfmathsetmacro{\Kx}{\bet*\jx*sin(40)}
  \pgfmathsetmacro{\Ky}{\bet*cos(40)+\bet*\jy*sin(40)}
  \def\halfw{2.6}   

  \fill[black!5]
      ({-\halfw+\Kx},{\Ky}) -- ({\halfw+\Kx},{\Ky})
   -- ({\halfw-\Kx},{-\Ky}) -- ({-\halfw-\Kx},{-\Ky}) -- cycle;
  \draw[gray!45]
      ({-\halfw+\Kx},{\Ky}) -- ({\halfw+\Kx},{\Ky})
   -- ({\halfw-\Kx},{-\Ky}) -- ({-\halfw-\Kx},{-\Ky}) -- cycle;
  \node[gray!70] at ({-\halfw+\Kx+0.5},{\Ky-0.22}) {$\C_K$};

  \draw[->,gray!70] (-2.9,0) -- (3.1,0) node[below right] {$\R$};
  \draw[->,gray!70] (0,-2.0) -- (0,2.15) node[above right] {$\qi\R$};
  \draw[->,gray!70] (0,0) -- (-1.35,-1.10) node[below left] {$\qj\R$};

  \draw[thick] plot[domain=180:360,samples=60]
    ({\bet*\jx*sin(\x)},{\bet*cos(\x)+\bet*\jy*sin(\x)});
  \draw[thick] plot[domain=0:180,samples=60]
    ({\bet*\jx*sin(\x)},{\bet*cos(\x)+\bet*\jy*sin(\x)});
  \node at (0.92,1.30) {$\Sigma$};

  \draw[dotted,thick] (\Kx,\Ky) -- (-\Kx,-\Ky);
  \fill (\Kx,\Ky) circle (1.6pt);
  \fill (-\Kx,-\Ky) circle (1.6pt);
  \node[left] at (\Kx,\Ky+0.15) {$K$};
  \node[right] at (-\Kx+0.12,-\Ky+0.1) {$-K$};

  \node[right,align=left,font=\small] at ({\Kx-2},{\Ky+0.30})
    {simple pole};
  \node[left,align=right,font=\small] at ({-\Kx+2.5},{-\Ky-0.30})
    {pole of order two};
\end{tikzpicture}
\caption{The polar sphere $\Sigma=\Sphere$ of
Example~\ref{ex:slice-exceptional-point}, drawn in $\R+\qi\R+\qj\R$ with the
fourth coordinate suppressed, so that the two-sphere appears as the circle
$\Sigma\cap(\qi\R+\qj\R)$. The shaded
parallelogram is the slice $\C_K$, spanned by the real axis and the
direction of $K$; its horizontal edges pass through the conjugate pair
$\pm K$, where the restriction $f_K$ has a simple pole and a pole of order
two respectively. On every other slice both intersection points carry a pole
of order two.}
\label{fig:slice-exceptional-point}
\end{figure}
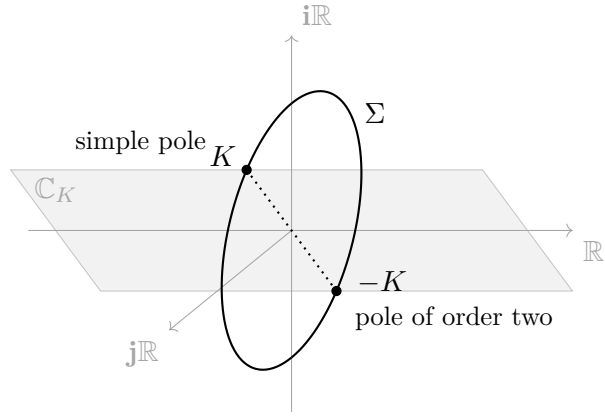

When the leading coefficient vanishes at \(y\), the pole order of \(f_I\)
there drops below \(m\), and the whole principal part may cancel, leaving a
removable singularity; Example~\ref{ex:slice-exceptional-point} realizes the
first of these two behaviours. 

It remains to record the slice-wise meromorphic reduction, which agrees with
\cite[Lemma~2.7 and Theorem~4.9]{GPSSingularities}, and to fix the
point-residue convention.

\begin{proposition}
\label{prop:slice-meromorphic-restriction}
Let \(f\) be semiregular on \(\Omega_E\) and fix \(I\in\Sphere\). Then the
restriction \(f_I\) is meromorphic on the full slice \(E_I\), with values in
\((\Hh,L_I)\), and for every \(J\in\Sphere\) orthogonal to \(I\) it has a
unique splitting
\begin{equation}\label{eq:slice-meromorphic-splitting}
  f_I=u+vJ ,
\end{equation}
with \(u,v\) meromorphic on \(E_I\) and \(\C_I\) valued.
\end{proposition}

\begin{proof}
This is the standard meromorphic splitting on a complex slice; see
\cite[Lemma~2.7 and Theorem~4.9]{GPSSingularities}. Indeed, applying the
complex-linear map \(\Phi_I\) to the Laurent expansions of the stem gives
\((\Hh,L_I)\)-valued Laurent expansions of \(f_I\). Some coefficients may
vanish, thereby lowering the pole order, but no essential singularity can
be created. The unique decomposition \eqref{eq:slice-meromorphic-splitting}
then follows from the complex direct sum
\(\Hh=\C_I\oplus\C_IJ \).
\end{proof}

Let \(p\in E_I\) be an isolated singular point of \(f_I\). By
Proposition~\ref{prop:slice-meromorphic-restriction}, \(f_I\) is a
meromorphic map with values in the two-dimensional complex vector space
\((\Hh,L_I)\), so it is exactly of the kind considered in
Section~\ref{sec:complex-model} with \(V=(\Hh,L_I)\): on a punctured disc
around \(p\) it admits a Laurent expansion
\begin{equation}\label{eq:slice-point-laurent-expansion}
  f_I(z)=\sum_{n\geq-m}(z-p)^na_n,
  \qquad a_n\in\Hh ,
\end{equation}
with the powers \((z-p)^n\in\C_I\) acting from the left and the quaternionic
coefficients on the right. Convergence and termwise integration hold
componentwise in any complex basis of \((\Hh,L_I)\), as in
\eqref{eq:complex-laurent}. If the singularity is removable we set \(a_n=0\) for
\(n<0\).

\begin{definition}
\label{def:slice-point-residue}
The point residue of \(f_I\) at \(p\) is
\begin{equation}\label{eq:slice-point-residue}
  \Res_p^If_I:=a_{-1} .
\end{equation}
\end{definition}

 Equivalently, if \(\gamma_p\) is
a sufficiently small positively oriented circle of \(\C_I\) around \(p\),
\begin{equation}\label{eq:slice-point-residue-integral}
  \Res_p^If_I=(2\pi I)^{-1}\int_{\gamma_p}\dd z\,f_I(z) ,
\end{equation}
both \(\dd z\) and \((2\pi I)^{-1}\) acting from the left. Indeed,
integrating \eqref{eq:slice-point-laurent-expansion} term by term on a
circle inside the annulus of convergence, and using
\[
  \int_{\gamma_p}\dd z\,(z-p)^n
  =\begin{cases}
     2\pi I,& n=-1,\\
     0,& n\neq-1,
   \end{cases}
\]
leaves \((2\pi I)a_{-1}\), the coefficients having stayed on the right
throughout. The factor \((2\pi I)^{-1}\) is an element of \(\C_I\) and must
be applied on the left, as a complex scalar: applying it on the right would
return \(a_{-1}\) conjugated by \(2\pi I\), a different quaternion in
general.

At a real pole \(a\) the point residue does not depend on the slice: the
Laurent expansion of the stem at a real centre has coefficients fixed by
central complex conjugation, hence lying in \(\Hh\), and every \(\Phi_I\)
evaluates them identically. We may therefore write \(\Res_af\) with no slice
label. At a non-real polar sphere, by contrast, a single point residue can
depend on the slice and on the representative; the invariant quantity will
be the sum of the two point residues at \(\alpha\pm I\beta\), a removable
point contributing zero. Section~\ref{sec:spherical-residues} identifies
that sum with the spherical residue.

All the ingredients of that construction are now in place: the
regularization \(g_M=\DeltaSphere{\Sigma}^Mf\) is slice regular across the
whole sphere, its restriction to a full slice is vector-valued holomorphic,
and \eqref{eq:slice-operator-translation} turns the conjugate-node
divided-difference formula into
\(\partial_s(\partial_c\partial_s)^{M-1}g_M\). None of this requires an
inverse of \(f\), of \(g_M\), or of a symmetrization.
\section{Spherical residues and their differential evaluation}
\label{sec:spherical-residues}
Since a semiregular function is induced by a stem, its singular set is
circular, and a non-real singularity belongs 
to a whole conjugacy sphere
\[
  \Sigma=\alpha+\Sphere\beta,\qquad \alpha\in\R,\ \beta>0 ,
\]
which every full slice \(\C_I\) meets in the conjugate pair
\(\alpha\pm I\beta\). This section identifies the single Laurent coefficient
that measures such a pair as a whole, and evaluates it by differentiating a
slice regular function.

That coefficient will be obtained in three equivalent forms, each doing a
different job. It is
\begin{enumerate}[label=\textup{(\roman*)}]
  \item the coefficient of \(q\DeltaSphere{\Sigma}^{-1}\) in the paired
        spherical Laurent expansion, which shows that it belongs to the
        sphere and not to a chosen slice;
  \item on any full slice, the sum of the two ordinary residues at the
        conjugate intersection points, which connects it to the classical
        residue theorem;
  \item after multiplication by a sufficiently high power of
        \(\DeltaSphere{\Sigma}\), a balanced Hermite divided difference,
        equivalently an alternating expression in \(\partial_s\) and
        \(\partial_c\), which makes it computable from a finite jet.
\end{enumerate}

\subsection{Cassini neighbourhoods and spherical Laurent expansions}
\label{subsec:sres-laurent}

Recall from \eqref{eq:slice-characteristic-polynomial} that the
characteristic polynomial of \(\Sigma=\alpha+\mathbb{S}\beta\) is
\begin{equation}\label{eq:sres-characteristic-polynomial}
  \DeltaSphere{\Sigma}(q)
  =q^2-2\alpha q+\alpha^2+\beta^2.
\end{equation}

Cassini balls and the expansions they support are well established in slice
analysis. The spherical series were introduced by Stoppato
\cite{StoppatoSeries} and extended to singularities and to real alternative
\(*\)-algebras in \cite[Sections~7 and~9]{GPSSingularities}; the form used
here, together with the sets on which it converges, is the one of
\cite[Theorem~1.7 and Definition~1.8]{AltavillaSphericalCoefficients}.

\begin{definition}[Cassini balls and shells]
\label{def:sres-cassini-sets}
For \(R>0\) and \(0\leq\rho<R\), define
\begin{align}
  \mathcal U(\Sigma,R)
  &:=\{q\in\Hh:|\DeltaSphere{\Sigma}(q)|<R^2\},
  \label{eq:sres-cassini-ball}\\
  \mathcal U(\Sigma;\rho,R)
  &:=\{q\in\Hh:\rho^2<|\DeltaSphere{\Sigma}(q)|<R^2\}.
  \label{eq:sres-cassini-shell}
\end{align}
The punctured Cassini ball is
\(\mathcal U(\Sigma;0,R)=\mathcal U(\Sigma,R)\setminus\Sigma\).
Here \(|\cdot|\) is the Euclidean norm of a quaternion. Since
\(\DeltaSphere{\Sigma}\) is slice preserving, these sets are circular.
\end{definition}

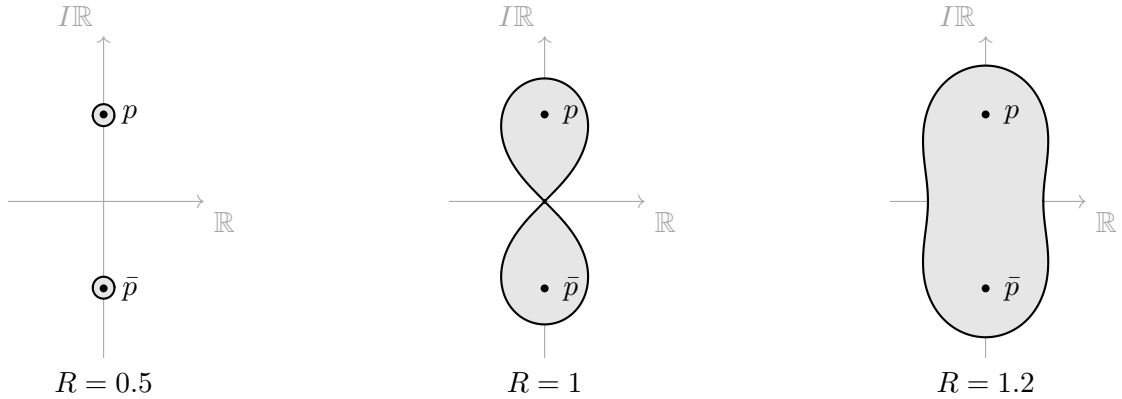
\begin{figure}[htbp]
\centering
\begin{tikzpicture}[scale=1.15]
  \def\Rq{0.0625}          
  \pgfmathsetmacro{\Dl}{acos(sqrt(1-\Rq))/2}
  \draw[->,gray!70] (-1.1,0) -- (1.15,0) node[below right] {$\R$};
  \draw[->,gray!70] (0,-1.8) -- (0,1.9) node[above left] {$I\R$};
  \foreach \c in {90,270}{
    \filldraw[black!10,draw=black,thick]
      plot[domain={\c-\Dl}:{\c+\Dl},samples=90,variable=\t]
        ({sqrt(max(0,-cos(2*\t)+sqrt(max(0,cos(2*\t)*cos(2*\t)-1+\Rq))))*cos(\t)},
         {sqrt(max(0,-cos(2*\t)+sqrt(max(0,cos(2*\t)*cos(2*\t)-1+\Rq))))*sin(\t)})
   -- plot[domain={\c+\Dl}:{\c-\Dl},samples=90,variable=\t]
        ({sqrt(max(0,-cos(2*\t)-sqrt(max(0,cos(2*\t)*cos(2*\t)-1+\Rq))))*cos(\t)},
         {sqrt(max(0,-cos(2*\t)-sqrt(max(0,cos(2*\t)*cos(2*\t)-1+\Rq))))*sin(\t)})
   -- cycle;}
  \fill (0,1) circle (1.3pt);  \node[right] at (0.10,1.0) {$p$};
  \fill (0,-1) circle (1.3pt); \node[right] at (0.10,-1.0) {$\bar p$};
  \node[below] at (0,-1.85) {$R=0.5$};
\end{tikzpicture}
\hfill
\begin{tikzpicture}[scale=1.15]
  \def\Rq{1.0}
  \pgfmathsetmacro{\Dl}{acos(sqrt(max(0,1-\Rq)))/2}
  \draw[->,gray!70] (-1.1,0) -- (1.15,0) node[below right] {$\R$};
  \draw[->,gray!70] (0,-1.8) -- (0,1.9) node[above left] {$I\R$};
  \foreach \c in {90,270}{
    \filldraw[black!10,draw=black,thick]
      plot[domain={\c-\Dl}:{\c+\Dl},samples=120,variable=\t]
        ({sqrt(max(0,-cos(2*\t)+sqrt(max(0,cos(2*\t)*cos(2*\t)-1+\Rq))))*cos(\t)},
         {sqrt(max(0,-cos(2*\t)+sqrt(max(0,cos(2*\t)*cos(2*\t)-1+\Rq))))*sin(\t)})
   -- plot[domain={\c+\Dl}:{\c-\Dl},samples=120,variable=\t]
        ({sqrt(max(0,-cos(2*\t)-sqrt(max(0,cos(2*\t)*cos(2*\t)-1+\Rq))))*cos(\t)},
         {sqrt(max(0,-cos(2*\t)-sqrt(max(0,cos(2*\t)*cos(2*\t)-1+\Rq))))*sin(\t)})
   -- cycle;}
  \fill (0,1) circle (1.3pt);  \node[right] at (0.10,1.0) {$p$};
  \fill (0,-1) circle (1.3pt); \node[right] at (0.10,-1.0) {$\bar p$};
  \node[below] at (0,-1.85) {$R=1$};
\end{tikzpicture}
\hfill
\begin{tikzpicture}[scale=1.15]
  \def\Rq{2.0736}          
  \draw[->,gray!70] (-1.1,0) -- (1.15,0) node[below right] {$\R$};
  \draw[->,gray!70] (0,-1.8) -- (0,1.9) node[above left] {$I\R$};
  \filldraw[black!10,draw=black,thick]
    plot[domain=0:360,samples=240,variable=\t]
      ({sqrt(max(0,-cos(2*\t)+sqrt(max(0,cos(2*\t)*cos(2*\t)-1+\Rq))))*cos(\t)},
       {sqrt(max(0,-cos(2*\t)+sqrt(max(0,cos(2*\t)*cos(2*\t)-1+\Rq))))*sin(\t)})
   -- cycle;
  \fill (0,1) circle (1.3pt);  \node[right] at (0.10,1.0) {$p$};
  \fill (0,-1) circle (1.3pt); \node[right] at (0.10,-1.0) {$\bar p$};
  \node[below] at (0,-1.85) {$R=1.2$};
\end{tikzpicture}
\caption{The Cassini ball $\mathcal U(\Sigma,R)\cap\C_I$ for
$\Sigma=\Sphere$, so that $\DeltaSphere{\Sigma}(z)=z^2+1$ and the nodes are
$p=I$, $\bar p=-I$. For $R<\beta=1$ the sublevel set has two components,
one around each node; at $R=\beta$ they meet at the origin; for $R>\beta$ it
is connected, and a contour drawn in it no longer separates the two nodes.
}
\label{fig:cassini-balls}
\end{figure}

The square on \(R\) is a convenient normalization, \(\DeltaSphere{\Sigma}\)
having degree two. On a slice \(\C_I\), writing \(p=\alpha+I\beta\) and
\(\bar p=\alpha-I\beta\) for the two nodes, the characteristic polynomial
factors as
\begin{equation}\label{eq:sres-delta-slice-factorization}
  \DeltaSphere{\Sigma}(z)=(z-p)(z-\bar p),
  \qquad z\in\C_I ,
\end{equation}
so that a Cassini ball is the circular analogue of a two-centre
neighbourhood.

For \(0<\rho<\beta\) the planar set \(\mathcal U(\Sigma,\rho)\cap\C_I\) has
exactly two connected components, one containing each node. It has at most
two, because every component must contain a zero of
\(\DeltaSphere{\Sigma}\) — otherwise the maximum-modulus principle applied
to \(1/\DeltaSphere{\Sigma}\) on that component would be violated — and the
only zeros are \(p\) and \(\bar p\); and it has at least two, because a path
joining the two nodes has to cross the real axis, where
\[
  |\DeltaSphere{\Sigma}(x)|=(x-\alpha)^2+\beta^2\geq\beta^2>\rho^2 ,
\]
so that no such path stays inside the set. We write
\begin{equation}\label{eq:sres-cassini-contour}
  \Gamma_{I,\rho}:=\partial\bigl(\mathcal U(\Sigma,\rho)\cap\C_I\bigr) ,
\end{equation}
each component of the boundary being positively oriented in \(\C_I\) (see
Figure~\ref{fig:cassini-balls}).

Before proving the expansion we isolate the contour computation that
extracts its two coefficient families, and which explains where the
apparently special weights \(1\) and \(z-2\alpha\) come from.

\begin{lemma}[The two coefficient-extraction moments]
\label{lem:sres-two-moments}
Fix \(I\in\Sphere\), let \(p=\alpha+I\beta\) and \(\bar p=\alpha-I\beta\),
and let \(\Gamma_{I,\rho}\) be as in \eqref{eq:sres-cassini-contour}. For
\(c,d\in\Hh\) and \(j,k\in\mathbb Z\), set
\(\Lambda_j(z):=\DeltaSphere{\Sigma}(z)^j(zc+d)\). Then
\begin{align}
  (2\pi I)^{-1}\int_{\Gamma_{I,\rho}}
    \dd z\,\DeltaSphere{\Sigma}(z)^{-k-1}\Lambda_j(z)
  &=\begin{cases} c,&j=k,\\ 0,&j\neq k,\end{cases}
  \label{eq:sres-first-moment}\\
  (2\pi I)^{-1}\int_{\Gamma_{I,\rho}}
    \dd z\,(z-2\alpha)\DeltaSphere{\Sigma}(z)^{-k-1}\Lambda_j(z)
  &=\begin{cases} d,&j=k,\\ 0,&j\neq k.\end{cases}
  \label{eq:sres-second-moment}
\end{align}
The integrands have no poles other than \(p\) and \(\bar p\), and all
complex scalars, \(\dd z\) and \((2\pi I)^{-1}\) included, act from the
left.
\end{lemma}

\begin{proof}
Both integrands carry the factor
\(\DeltaSphere{\Sigma}^{j-k-1}\), so three cases arise.

If \(j>k\) the exponent is non-negative and the integrands are polynomials
with values in \((\Hh,L_I)\); their integrals vanish by Cauchy's theorem.

If \(j=k\), the two integrands are \(w(z)(zc+d)\big/(z-p)(z-\bar p)\) with
\(w\equiv1\) and \(w(z)=z-2\alpha\) respectively, by
\eqref{eq:sres-delta-slice-factorization}. Both poles being simple, the sum
of the residues is
\[
  (p-\bar p)^{-1}\bigl[w(p)(pc+d)-w(\bar p)(\bar pc+d)\bigr] ,
\]
and the residue theorem, applied componentwise in \((\Hh,L_I)\), identifies
the normalized integral with it. For \(w\equiv1\) the bracket is
\((p-\bar p)c\), giving \(c\). For \(w(z)=z-2\alpha\), the relation
\(p+\bar p=2\alpha\) turns \(w(p)\) into \(-\bar p\) and \(w(\bar p)\) into
\(-p\), so that the bracket becomes
\[
  -\bar p(pc+d)+p(\bar pc+d)=(p\bar p-\bar pp)c+(p-\bar p)d=(p-\bar p)d ,
\]
the first term vanishing because \(p\) and \(\bar p\) commute in the plane
\(\C_I\); the integral is then \(d\).

If \(j<k\), put \(r=k+1-j\geq2\). The first integrand is
\((zc+d)\DeltaSphere{\Sigma}(z)^{-r}=O(|z|^{1-2r})=O(|z|^{-3})\) at
infinity and the second, whose numerator is quadratic,
\(O(|z|^{2-2r})=O(|z|^{-2})\). Either decay makes the residue at infinity
vanish, as in Remark~\ref{rem:complex-two-half-planes}; since \(p\) and
\(\bar p\) are the only finite poles, the residues there sum to zero, and
the integrals vanish.
\end{proof}

The local expansion below is known: it is the quaternionic counterpart of a
two-point Laurent series, and in the form stated it can be found in \cite{StoppatoSeries} and
\cite[Theorem~7.2 and Remark~7.4]{GPSSingularities}. We reprove it because
the argument, rather than the statement, is what the rest of the paper
needs: the coefficients are produced by the confluent Newton--Hermite
expansion of Section~\ref{sec:divided-differences}, so that the differential
formula of Subsection~\ref{subsec:sres-differential} comes out of the same
construction instead of having to be matched to it afterwards.

\begin{theorem}[Paired spherical Laurent expansion]
\label{thm:sres-paired-laurent}
Let \(f\) be semiregular on a punctured circular neighbourhood of the
non-real sphere \(\Sigma=\alpha+\Sphere\beta\), and let \(M\geq1\) be an
integer for which
\begin{equation}\label{eq:sres-regularization}
  g_M(q):=\DeltaSphere{\Sigma}(q)^Mf(q)
\end{equation}
extends slice regularly through the whole sphere. Then there is an \(R>0\)
such that
\begin{equation}\label{eq:sres-paired-laurent}
  \boxed{\;
  f(q)=\sum_{k=-M}^{\infty}\DeltaSphere{\Sigma}(q)^k(qu_k+v_k),
  \qquad u_k,v_k\in\Hh ,
  \;}
\end{equation}
on \(\mathcal U(\Sigma;0,R)\), the series converging normally on compact
subsets and its coefficients being unique. For every \(I\in\Sphere\), every
\(0<\rho<\min\{R,\beta\}\) and every \(k\geq-M\),
\begin{align}
  u_k&=(2\pi I)^{-1}\int_{\Gamma_{I,\rho}}
       \dd z\,\DeltaSphere{\Sigma}(z)^{-k-1}f_I(z),
  \label{eq:sres-u-contour}\\
  v_k&=(2\pi I)^{-1}\int_{\Gamma_{I,\rho}}
       \dd z\,(z-2\alpha)\DeltaSphere{\Sigma}(z)^{-k-1}f_I(z) .
  \label{eq:sres-v-contour}
\end{align}
In particular the coefficients depend neither on \(\rho\), nor on the slice,
nor on the representative chosen in \(\Sigma\), and a larger admissible
exponent only prefixes the expansion with vanishing pairs.
\end{theorem}

\begin{proof}
\emph{Step 1: expansion on one full slice.} Fix \(I\in\Sphere\), put
\(p=\alpha+I\beta\), and regard \(G:=(g_M)_I\) as a holomorphic map with
values in \((\Hh,L_I)\). Shrinking the neighbourhood if necessary, choose
\(R<\beta\) with \(G\) holomorphic near the closure of
\(\mathcal U(\Sigma,R)\cap\C_I\), and set
\[
  A_n:=G[p^{(n+1)},\bar p^{(n)}],
  \qquad
  B_n:=G[p^{(n+1)},\bar p^{(n+1)}],
\]
a block of zero repetitions being omitted. For every truncation order
\(N\geq1\) — an index unrelated to the fixed exponent \(M\) —
Corollary~\ref{cor:dd-alternating-newton}, applied to \(G\) with the first
\(2N\) nodes of the alternating sequence
\eqref{eq:dd-alternating-sequence}, gives
\begin{equation}\label{eq:sres-finite-newton}
  G(z)=\sum_{n=0}^{N-1}\DeltaSphere{\Sigma}(z)^n\bigl(A_n+(z-p)B_n\bigr)
       +\DeltaSphere{\Sigma}(z)^NG[p^{(N)},\bar p^{(N)},z] .
\end{equation}

The remainder is controlled by the contour formula \eqref{eq:cx-hermite}.
Let \(K\subset\mathcal U(\Sigma,R)\cap\C_I\) be compact and choose \(\tau\)
with \(\max_K|\DeltaSphere{\Sigma}|<\tau^2<R^2\) and \(\Gamma_{I,\tau}\)
disjoint from \(K\). Then
\[
  \DeltaSphere{\Sigma}(z)^NG[p^{(N)},\bar p^{(N)},z]
  =\frac{1}{2\pi I}\int_{\Gamma_{I,\tau}}
   \frac{\DeltaSphere{\Sigma}(z)^N\dd\zeta}
        {\DeltaSphere{\Sigma}(\zeta)^N(\zeta-z)}\,G(\zeta) ,
\]
division by slice scalars meaning left multiplication by their inverses.
As \(|\DeltaSphere{\Sigma}(\zeta)|=\tau^2\) on the contour, which is compact
and at positive distance from \(K\), there is a \(C_K\) independent of
\(N\) with
\begin{equation}\label{eq:sres-newton-remainder-bound}
  \sup_{z\in K}\bigl\|\DeltaSphere{\Sigma}(z)^N
    G[p^{(N)},\bar p^{(N)},z]\bigr\|
  \leq C_K\left(\frac{\max_K|\DeltaSphere{\Sigma}|}{\tau^2}\right)^N ,
\end{equation}
and the ratio in round brackets is smaller than one. Letting \(N\to\infty\) in
\eqref{eq:sres-finite-newton} therefore gives
\begin{equation}\label{eq:sres-G-two-point-series}
  G(z)=\sum_{n\geq0}\DeltaSphere{\Sigma}(z)^n\bigl(A_n+(z-p)B_n\bigr) ,
\end{equation}
normally convergent on compact subsets. Rewriting each summand as
\(A_n+(z-p)B_n=zB_n+(A_n-pB_n)\), setting
\begin{equation}\label{eq:sres-coefficients-from-newton}
  u_{n-M}:=B_n,\qquad v_{n-M}:=A_n-pB_n ,
\end{equation}
and dividing \eqref{eq:sres-G-two-point-series} by
\(\DeltaSphere{\Sigma}^M\) yields \eqref{eq:sres-paired-laurent} on the two
components of the punctured full slice; normal convergence survives the
division because \(|\DeltaSphere{\Sigma}|\) is bounded below on compact
subsets of the punctured ball.

\emph{Step 2: from the slice to the circular set.} Each summand of
\eqref{eq:sres-paired-laurent} is a slice function, \(\DeltaSphere{\Sigma}\)
being slice preserving, the variable \(q\) slice
regular, and the coefficients acting on the right. For \(q=x+Jr\), the
representation formula \eqref{eq:slice-representation-formula} recovers the
value of every partial sum from its values at \(x\pm Ir\), with coefficients
\((1\mp JI)/2\) of norm at most \(1\). Uniform convergence on a compact
subset of the full slice thus implies uniform convergence on its
circularization. Since the series and \(f\) agree at \(x\pm Ir\), the same
formula makes them agree at \(x+Jr\), which proves
\eqref{eq:sres-paired-laurent} on all of \(\mathcal U(\Sigma;0,R)\).

\emph{Step 3: extraction and uniqueness.} Fix \(k\geq-M\) and a contour
\(\Gamma_{I,\rho}\) inside the region of convergence. Normal convergence
allows termwise integration, and Lemma~\ref{lem:sres-two-moments}
annihilates every term of \eqref{eq:sres-paired-laurent} but the one of
index \(k\), returning \(u_k\) from the first moment and \(v_k\) from the
second: this is
\eqref{eq:sres-u-contour}--\eqref{eq:sres-v-contour}. Applying the same
two moments to the difference of two expansions of \(f\) shows that they
have the same coefficients.

\emph{Step 4: the auxiliary choices.} Independence of \(\rho\) is now
automatic, both integrals having been identified with a coefficient that
does not involve \(\rho\). Repeating the construction on another slice
produces another series of the form \eqref{eq:sres-paired-laurent}; Step~2
lifts both to the same quaternionic function \(f\), and the uniqueness just
proved forces the two families of coefficients to agree. As
\(\DeltaSphere{\Sigma}\) depends only on \(\Sigma\), no representative of
the sphere survives in the pairs \((u_k,v_k)\). Finally, for an admissible
\(M'\geq M\) the function \(f\) still has only one expansion of the
displayed form, so reindexing the Taylor expansion of
\(g_{M'}=\DeltaSphere{\Sigma}^{M'-M}g_M\) returns the same coefficients and
sets \(u_k=v_k=0\) for \(-M'\leq k<-M\).
\end{proof}

It is worth recording how the paired coefficients relate to the standard
spherical basis. Fix \(y\in\Sigma\) and set, for \(k\in\mathbb Z\),
\begin{equation}\label{eq:sres-spherical-basis}
  \mathcal S_{y,2k}(q):=\DeltaSphere{\Sigma}(q)^k ,
  \qquad
  \mathcal S_{y,2k+1}(q):=\DeltaSphere{\Sigma}(q)^k(q-y).
\end{equation}
 Writing
\(f(q)=\sum_{n\in\mathbb Z}\mathcal S_{y,n}(q)s_n\) and splitting each
paired summand as \(qu_k+v_k=(q-y)u_k+(v_k+yu_k)\), with all coefficients on
the right, gives the equalities
\begin{equation}\label{eq:sres-coefficient-dictionary}
  \boxed{\;
  s_{2k}=v_k+yu_k,
  \qquad
  s_{2k+1}=u_k .
  \;}
\end{equation}
The asymmetry is the point: \(s_{2k+1}\) coincides with \(u_k\) and is
therefore attached to the sphere, whereas \(s_{2k}\) moves with the chosen
representative. The pair \((u_k,v_k)\) is free of that choice.

\begin{remark}[Spherical order and regularizing exponents]
\label{rem:sres-spherical-order}
The spherical order is conventionally the smallest \emph{even} integer
\(2M_0\) such that \(s_n=0\) for every \(n<-2M_0\), equivalently the
smallest \(M_0\) with \((u_k,v_k)=(0,0)\) for every \(k<-M_0\); so the least
exponent for which \(\DeltaSphere{\Sigma}^{M_0}f\) extends regularly
corresponds to spherical order \(2M_0\), as in
\cite[Definition~9.2]{GPSSingularities}.

This even order is not the order of the pole at one point of one slice.
Generically the latter equals \(M_0\), but at an exceptional point it drops,
possibly to zero, as in Example~\ref{ex:slice-exceptional-point}. 
\end{remark}

\subsection{Real and spherical residues}
\label{subsec:sres-definitions}

We can now define the two residues. Each is a Laurent coefficient, read off
with the normalization already fixed for point residues in
Definition~\ref{def:slice-point-residue}.

\begin{definition}[Real and spherical residues]
\label{def:sres-residues}
Let \(f\) be semiregular.
\begin{enumerate}[label=\textup{(\roman*)}]
  \item At a real pole \(a\), write
        \(f(q)=\sum_{n\geq-L}(q-a)^nc_n\) with \(c_n\in\Hh\) on the right,
        and set
        \begin{equation}\label{eq:sres-real-residue-definition}
          \Res_af:=c_{-1} .
        \end{equation}
  \item At a non-real polar sphere \(\Sigma\), take the unique paired
        expansion \eqref{eq:sres-paired-laurent} and set
        \begin{equation}\label{eq:sres-spherical-residue-definition}
          \boxed{\;\SRes_\Sigma f:=u_{-1} .\;}
        \end{equation}
\end{enumerate}
Both are declared to be zero at a regular point, respectively a regular
sphere.
\end{definition}

At a real pole nothing new happens: the residue is the familiar
Taylor coefficient of the regularized function.

\begin{proposition}[Differential evaluation at a real pole]
\label{prop:sres-real-differential}
Let \(a\in\R\) and let \(L\geq1\) be any exponent for which
\begin{equation}\label{eq:sres-real-regularized-function}
  h_L(q):=(q-a)^Lf(q)
\end{equation}
extends slice regularly across \(a\). Then, for every \(I\in\Sphere\),
\begin{equation}\label{eq:sres-real-differential}
  \boxed{\;
  \Res_af=\frac{\partial_c^{L-1}h_L(a)}{(L-1)!}=\Res_a^If_I .
  \;}
\end{equation}
In particular the middle expression depends neither on \(L\) nor on the
slice.
\end{proposition}

\begin{proof}
Being slice regular near the real point \(a\), the function \(h_L\) has a
Taylor expansion
\begin{equation}\label{eq:sres-real-Taylor}
  h_L(q)=\sum_{n\geq0}(q-a)^nb_n,
  \qquad b_n=\frac{\partial_c^nh_L(a)}{n!} ,
\end{equation}
whose coefficients are quaternions independent of the slice, the centre
being real; this was already observed in Subsection~\ref{subsec:slice-semiregular}.
Dividing by the slice-preserving factor \((q-a)^L\) away from \(a\) turns
\eqref{eq:sres-real-Taylor} into the Laurent expansion
\(f(q)=\sum_{n\geq0}(q-a)^{n-L}b_n\) of
Definition~\ref{def:sres-residues}(i), so that \(c_{n-L}=b_n\); the
exponent \(-1\) occurs for \(n=L-1\), whence the first equality in
\eqref{eq:sres-real-differential}. On any slice the same series is the
vector-valued Laurent expansion of \(f_I\), which gives the second. Both
independence assertions follow, every admissible \(L\) and every slice
extracting the same coefficient of one and the same expansion.
\end{proof}

At a non-real sphere the corresponding statement cannot involve one point
alone: it needs both intersection points of a full slice.

\begin{proposition}[Spherical residue as a pair of point residues]
\label{prop:sres-paired-point-residues}
Let \(\Sigma=\alpha+\Sphere\beta\) be a non-real polar sphere of \(f\) and
let \(I\in\Sphere\). With \(p=\alpha+I\beta\) and \(\bar p=\alpha-I\beta\),
\begin{equation}\label{eq:sres-paired-point-residues}
  \boxed{\;
  \SRes_\Sigma f=\Res_p^If_I+\Res_{\bar p}^If_I ,
  \;}
\end{equation}
a removable point contributing zero to the sum.
\end{proposition}

\begin{proof}
Taking \(k=-1\) in \eqref{eq:sres-u-contour}, where the kernel becomes
\(\DeltaSphere{\Sigma}^0=1\), gives
\begin{equation}\label{eq:sres-u-minus-one-contour}
  \SRes_\Sigma f
  =(2\pi I)^{-1}\int_{\Gamma_{I,\rho}}\dd z\,f_I(z) .
\end{equation}
The contour consists of two positively oriented curves, one around each
node, and applying the vector-valued residue theorem to each of them
separately produces, with the normalization
\eqref{eq:slice-point-residue-integral}, the two point residues. If \(f_I\)
happens to extend holomorphically at one node, the corresponding integral
vanishes by Cauchy's theorem and that node contributes zero, so exceptional
points need no separate treatment.
\end{proof}

\begin{remark}[Exceptional points redistribute]
\label{rem:sres-exceptional-redistribution}
Formula \eqref{eq:sres-paired-point-residues} holds for every slice, the
exceptional one included, and it is worth being explicit about what this
does and does not say. The individual summands \(\Res_p^If_I\) and
\(\Res_{\bar p}^If_I\) genuinely depend on \(I\), and on an exceptional
slice one of them may vanish while on every other slice neither does; what
does not depend on \(I\) is their sum. The spherical residue is defined by
\eqref{eq:sres-spherical-residue-definition}, before any slice is chosen,
and Proposition~\ref{prop:sres-paired-point-residues} asserts that every
slice recovers it: an exceptional point merely redistributes the value
between the two nodes of its own slice, as
Example~\ref{ex:sres-exceptional-point} will show.
\end{remark}

\begin{remark}[A vanishing residue does not remove the sphere]
\label{rem:sres-zero-does-not-remove}
A residue records one Laurent coefficient and no more, here as in one
complex variable. So \(\SRes_\Sigma f=0\) says only that the two point
residues cancel on every full slice, and leaves the other pairs of the
principal part untouched: \(f\) need not extend through \(\Sigma\). In the
other direction, a removable exceptional point does not make \(\Sigma\)
removable either, the conjugate point being free to carry the singularity —
which is exactly what happens in
Example~\ref{ex:slice-exceptional-point}.
\end{remark}

\subsection{Differential evaluation of the spherical residue}
\label{subsec:sres-differential}

We now reach the main local statement. The proof is short because the conjugate-node calculus of
Section~\ref{sec:divided-differences} was built for exactly this purpose:
the balanced divided difference of \eqref{eq:cx-sres-dd} and its operator
form \eqref{eq:dd-paired-residue-operator} give the formula on any one
fixed slice. What they do not give is that the value is the same on all of
them — the invariance left open in
Remark~\ref{rem:dd-representative-invariance} — and that comes from the
paired expansion of Theorem~\ref{thm:sres-paired-laurent}, which is
quaternionic from the start. The differential evaluation of the
spherical coefficients of a slice regular function goes back to
\cite[Theorem~4.1]{AltavillaSphericalCoefficients}; the theorem below
identifies the relevant coefficient of the regularization \(g_M\) with the
spherical residue, for the class of domains considered here.

\begin{theorem}[Differential formula for the spherical residue]
\label{thm:sres-differential}
Let \(f\) be semiregular near the non-real polar sphere
\(\Sigma=\alpha+\Sphere\beta\), let \(M\geq1\) be any exponent for which
\(g_M=\DeltaSphere{\Sigma}^Mf\) extends slice regularly through \(\Sigma\),
and let \(y=\alpha+I\beta\) be any point of \(\Sigma\), with conjugate
\(\bar y=\alpha-I\beta\). Then
\begin{equation}\label{eq:sres-main-differential}
  \boxed{\;
  \SRes_\Sigma f
  =(g_M)_I[y^{(M)},\bar y^{(M)}]
  =\frac{[\partial_s(\partial_c\partial_s)^{M-1}g_M](y)}{(M-1)!} .
  \;}
\end{equation}
The repeated-node divided difference is taken in \((\Hh,L_I)\), with complex
scalars acting from the left; for \(M=1\) the last expression reduces to
\(\partial_sg_1(y)\).
\end{theorem}

\begin{proof}
On the full slice \(\C_I\), the restriction \(f_I\) is vector-valued
meromorphic by
Proposition~\ref{prop:slice-meromorphic-restriction}, whereas
\(G:=(g_M)_I\) is holomorphic at both nodes and satisfies
\begin{equation}\label{eq:sres-slice-regularization-identity}
  G(z)=\DeltaSphere{\Sigma}(z)^Mf_I(z)=(z-y)^M(z-\bar y)^Mf_I(z) .
\end{equation}
This is the setting of
Corollary~\ref{cor:dd-paired-residue}, applied to \(f_I\) with nodes \(y\)
and \(\bar y\), which gives
\[
  (g_M)_I[y^{(M)},\bar y^{(M)}]=\Res_y^If_I+\Res_{\bar y}^If_I ,
\]
a quantity identified with \(\SRes_\Sigma f\) by
Proposition~\ref{prop:sres-paired-point-residues}. This is the first
equality.

For the second, \eqref{eq:dd-paired-residue-operator} applied to \(G\) gives
\[
  G[y^{(M)},\bar y^{(M)}]
  =\frac{[\mathcal S(\partial_z\mathcal S)^{M-1}G](y)}{(M-1)!} ,
\]
and on \(\C_I\) the conjugate-node operator \(\mathcal S\) is the
restriction of \(\partial_s\), while \(\partial_z\) is the restriction of
\(\partial_c\); the translation \eqref{eq:slice-operator-translation}
therefore rewrites the numerator as
\([\partial_s(\partial_c\partial_s)^{M-1}g_M](y)\).
\end{proof}

Theorem~\ref{thm:sres-paired-laurent} already grants that the paired
coefficients survive the auxiliary choices. The three independences at work
in \eqref{eq:sres-main-differential} are nevertheless worth isolating, their
logical sources being different, and one of them closing a gap deliberately
left open in Section~\ref{sec:divided-differences}.

\begin{proposition}[Independence of the auxiliary choices]
\label{prop:sres-three-independences}
The value in \eqref{eq:sres-main-differential} does not depend on the
admissible regularizing exponent \(M\), on the representative
\(y\in\Sigma\), or on the slice \(\C_I\) through it. This remains true if
\(g_M(y)=0\), if one of the two point singularities is removable, or if
\(f^s\) or \(g_M^s\) vanishes identically.
\end{proposition}

\begin{proof}
\emph{The exponent.} On a small union \(\Gamma\) of two contours around the
nodes, the Hermite contour formula \eqref{eq:cx-hermite} gives
\begin{equation}\label{eq:sres-exponent-independent-integrand}
  (g_M)_I[y^{(M)},\bar y^{(M)}]
  =\frac{1}{2\pi I}\int_\Gamma
   \frac{\dd z\,(g_M)_I(z)}{(z-y)^M(z-\bar y)^M}
  =\frac{1}{2\pi I}\int_\Gamma\dd z\,f_I(z) ,
\end{equation}
the kernel cancelling against the regularization whatever the exponent, as
in Remark~\ref{rem:dd-nonminimal}; any other admissible \(M'\) therefore
leaves the same final integrand. This is why a nonminimal regularization
does no harm even when it vanishes identically on \(\Sigma\).

\emph{The representative and the slice.} Within one slice, exchanging \(y\)
with \(\bar y\) swaps the two node blocks, and divided differences are
symmetric. Another point of \(\Sigma\) generally lies on another slice,
where node symmetry no longer suffices — the gap recorded in
Remark~\ref{rem:dd-representative-invariance} — and there
Theorem~\ref{thm:sres-paired-laurent} closes it: whatever \(y\) and \(I\),
the quantity computed is the coefficient \(u_{-1}\), unique and attached to
\(\Sigma\) alone.

Observe finally that the whole argument used nothing but multiplication by
the intrinsic polynomial \(\DeltaSphere{\Sigma}\), differentiation, divided
differences and contour integration, and never divided by a value of
\(g_M\), by \(f^s\) or by \(g_M^s\). Vanishing regularizations, removable
exceptional points and zero divisors are therefore all covered.
\end{proof}

The same method evaluates every paired coefficient, and not only the one of
index \(-1\).

\begin{corollary}
\label{cor:sres-all-coefficients}
Under the hypotheses of Theorem~\ref{thm:sres-differential}, let
\(k\geq-M\), put \(n=M+k\geq0\), and let \(y\in\Sigma\). Then
\begin{align}
  u_k&=(g_M)_I[y^{(n+1)},\bar y^{(n+1)}]
      =\frac{[\partial_s(\partial_c\partial_s)^ng_M](y)}{n!} ,
  \label{eq:sres-all-u}\\
  v_k&=(g_M)_I[y^{(n+1)},\bar y^{(n)}]-yu_k
      =\frac{[(\partial_c\partial_s)^ng_M](y)}{n!}-yu_k .
  \label{eq:sres-all-v}
\end{align}
For \(n=0\) the block \(\bar y^{(0)}\) is omitted, so that the first term of
\eqref{eq:sres-all-v} is simply \(g_M(y)\).
\end{corollary}

\begin{proof}
Formula \eqref{eq:sres-coefficients-from-newton} in the proof of
Theorem~\ref{thm:sres-paired-laurent}, written on the slice through \(y\)
and with \(n=M+k\), reads
\[
  u_k=B_n=(g_M)_I[y^{(n+1)},\bar y^{(n+1)}] ,
  \qquad
  v_k=A_n-yB_n=(g_M)_I[y^{(n+1)},\bar y^{(n)}]-yu_k ,
\]
the coefficients on the left being the same for every choice of \(y\) and
of the slice, by Theorem~\ref{thm:sres-paired-laurent}. The two operator
identities of Proposition~\ref{prop:cx-operators}, transported by
\eqref{eq:slice-operator-translation}, then give
\[
  [\partial_s(\partial_c\partial_s)^ng_M](y)
   =n!\,(g_M)_I[y^{(n+1)},\bar y^{(n+1)}] ,
  \qquad
  [(\partial_c\partial_s)^ng_M](y)
   =n!\,(g_M)_I[y^{(n+1)},\bar y^{(n)}] ,
\]
whence both formulas. Note that in \eqref{eq:sres-all-v} each of the two
terms on the right may depend on \(y\), while their difference does not.
\end{proof}


A single example brings all of this together. It shows the two point
residues changing from slice to slice while their sum does not, and a
regularization that vanishes identically on the sphere without preventing
the differential formula from returning the residue.

\begin{example}[One removable point on a polar sphere]
\label{ex:sres-exceptional-point}
Let \(\Sigma=2\Sphere\), so that \(\DeltaSphere{\Sigma}(q)=q^2+4\), and
consider
\begin{equation}\label{eq:sres-exceptional-function}
  f(q)=(q^2+4)^{-1}(q\qj+2\qk) .
\end{equation}
 Its paired expansion has the
single term \(\DeltaSphere{\Sigma}(q)^{-1}(q\qj+2\qk)\), whence
\begin{equation}\label{eq:sres-exceptional-residue-value}
  u_{-1}=\qj,\qquad v_{-1}=2\qk,\qquad \SRes_\Sigma f=\qj .
\end{equation}
With the minimal exponent \(M=1\) the regularization is
\(g_1(q)=q\qj+2\qk\), and right linearity of \(\partial_s\) gives
\(\partial_sg_1(y)=\partial_s(q)(y)\,\qj=\qj\) at every \(y\in\Sigma\), in
agreement with Theorem~\ref{thm:sres-differential}.

Now the two point residues on an arbitrary slice. Fix \(I\in\Sphere\) and
factor \(z^2+4=(z-2I)(z+2I)\) on \(\C_I\); both poles being simple,
\begin{equation}\label{eq:sres-exceptional-residues}
  \Res_{2I}^If_I=(4I)^{-1}(2I\qj+2\qk)=\tfrac12(\qj-I\qk),
  \qquad
  \Res_{-2I}^If_I=(-4I)^{-1}(-2I\qj+2\qk)=\tfrac12(\qj+I\qk) .
\end{equation}
Each summand moves with \(I\); their sum is \(\qj\) for every \(I\), as
Proposition~\ref{prop:sres-paired-point-residues} requires.

The slice \(\C_{\qi}\) is the exceptional one. There \(\qi\qj=\qk\) forces
the factorization
\begin{equation}\label{eq:sres-exceptional-cancellation}
  z\qj+2\qk=(z+2\qi)\qj ,
  \qquad z\in\C_{\qi} ,
\end{equation}
so that the factor \(z+2\qi\) cancels and
\begin{equation}\label{eq:sres-exceptional-reduced-slice}
  f_{\qi}(z)=(z-2\qi)^{-1}\qj .
\end{equation}
The singularity at \(-2\qi\) is therefore removable, and the two residues
are \(\qj\) and \(0\) instead of the generic pair
\eqref{eq:sres-exceptional-residues} — a redistribution of the same total,
as in Remark~\ref{rem:sres-exceptional-redistribution}. Note also that
\(g_1(-2\qi)=0\) while \(\partial_sg_1(-2\qi)=\qj\): the exceptional point
is a perfectly legitimate evaluation point for the differential formula,
which reads differences and derivatives of \(g_1\) and never its values.

The nonminimal exponent \(M=2\) makes the last point sharper, since
\[
  g_2(q)=(q^2+4)(q\qj+2\qk)
\]
vanishes at every point of \(\Sigma\). We verify
\eqref{eq:sres-main-differential} directly. Treating the two conjugate nodes
\(z,w\) as independent variables on a fixed slice,
\[
  (\mathcal Sg_2)(z,w)=(z-w)^{-1}\bigl(g_2(z)-g_2(w)\bigr)
  =(z^2+zw+w^2+4)\qj+2(z+w)\qk ,
\]
and differentiating with respect to the first node leaves
\((2z+w)\qj+2\qk\). Restricting to \(w=\bar z\) and applying \(\mathcal S\)
once more,
\[
  (\mathcal S\partial_z\mathcal Sg_2)(z)
  =(z-\bar z)^{-1}\bigl[(2z+\bar z)-(2\bar z+z)\bigr]\qj=\qj ,
\]
identically in \(z\). Hence
\(\partial_s\partial_c\partial_sg_2\equiv\qj\), as \(M=2\) requires, even
though \(g_2\) vanishes on all of \(\Sigma\).
\end{example}
\subsection{Componentwise interpretation}
\label{subsec:sres-componentwise}

The stem description yields a scalar formula for the two residues which, like
everything else in this section, survives on the null quadric. Recall the
component decomposition \eqref{eq:slice-component-decomposition},
\[
  F=\sum_{\ell=0}^3F_\ell e_\ell,
  \qquad (e_0,e_1,e_2,e_3)=(1,\qi,\qj,\qk) .
\]

\begin{corollary}[Residues of the intrinsic meromorphic coordinates]
\label{cor:sres-componentwise}
Let \(f=\mathcal I(F)\) and let \(w=\alpha+\ci\beta\), \(\beta>0\),
correspond to \(\Sigma=\alpha+\Sphere\beta\). Then
\begin{equation}\label{eq:sres-componentwise}
  \boxed{\;
  \SRes_\Sigma f
  =\sum_{\ell=0}^3\bigl(\Res_wF_\ell+\Res_{\bar w}F_\ell\bigr)e_\ell
  =2\sum_{\ell=0}^3\req\bigl(\Res_wF_\ell\bigr)e_\ell ,
  \;}
\end{equation}
while at a real pole \(a\),
\begin{equation}\label{eq:sres-componentwise-real}
  \Res_af=\sum_{\ell=0}^3\bigl(\Res_aF_\ell\bigr)e_\ell ,
\end{equation}
every scalar residue \(\Res_aF_\ell\) being real. Both formulas hold even
when \(Q(F)=\sum_{\ell=0}^3F_\ell^2\) vanishes identically.
\end{corollary}

\begin{proof}
Each coordinate satisfies \(F_\ell(\bar z)=\overline{F_\ell(z)}\), so
conjugating the Laurent expansion of \(F_\ell\) at \(w\) gives the expansion
at \(\bar w\) with conjugated coefficients; in particular
\begin{equation}\label{eq:sres-coordinate-reflection}
  \Res_{\bar w}F_\ell=\overline{\Res_wF_\ell} ,
\end{equation}
and the paired scalar residue \(\Res_wF_\ell+\Res_{\bar w}F_\ell\) equals
\(2\req(\Res_wF_\ell)\), a real number.

Fix now a slice \(\C_I\) and the two nodes \(\alpha\pm I\beta\). Since
\(\Phi_I\) is complex linear, it carries the Laurent expansion of \(F\) at
\(w\) to that of \(f_I\) at \(\alpha+I\beta\) coordinate by coordinate, and
similarly at the conjugate node; adding the two and using
\eqref{eq:sres-coordinate-reflection}, the coefficients of \(e_\ell\) in
the sum are real, hence unaffected by \(\Phi_I\). The sum of the two point
residues of \(f_I\) is therefore the right-hand side of
\eqref{eq:sres-componentwise}, which
Proposition~\ref{prop:sres-paired-point-residues} identifies with
\(\SRes_\Sigma f\). At a real centre \(a=\bar a\), conjugating the expansion
shows directly that every scalar coefficient is real, and the same
componentwise argument gives \eqref{eq:sres-componentwise-real}.
\end{proof}
Nothing in the previous proof involved \(Q(F)\), an inverse of it, or a symmetrization, so
the formulas remain valid on the null quadric \(Q(F)=0\), including the
zero-divisor behaviour of
Example~\ref{ex:slice-zero-divisor-preliminary}. The reading of a slice
regular function as an intrinsic holomorphic curve in \(\C^4\) is developed
in several papers (see \cite{Mongodi,AltavillaMongodiRoots,AltavillaMongodiExponential,
BisiWinkelmannPicard,AltavillaDeFabritiisEquivalence}).

It is worth closing the section by comparing the spherical residue with the
ordinary residue at one point, which is simply the vector-valued form of the
familiar complex formula.

\begin{proposition}[Differential formula at one point of a slice]
\label{prop:sres-single-point-formula}
Fix \(I\in\Sphere\) and let \(f_I\) have a pole of order \(m\geq1\) at
\(p\in\C_I\). Then
\begin{equation}\label{eq:sres-single-point-formula}
  \Res_p^If_I
  =\frac{1}{(m-1)!}\,
   \partial_z^{m-1}\bigl((z-p)^mf_I(z)\bigr)\Big|_{z=p} ,
\end{equation}
where the scalar factor \((z-p)^m\) acts from the left and \(\partial_z\)
is the restriction of \(\partial_c\) to \(\C_I\).
\end{proposition}

\begin{proof}
The function \((z-p)^mf_I(z)\) is holomorphic at \(p\), with Taylor
expansion \(\sum_{n\geq0}(z-p)^n\partial_z^n\bigl((z-p)^mf_I\bigr)(p)/n!\)
in \((\Hh,L_I)\). Dividing by \((z-p)^m\), the coefficient of
\((z-p)^{-1}\) is the term of index \(n=m-1\), which by
Definition~\ref{def:slice-point-residue} is \(\Res_p^If_I\); and
\eqref{eq:slice-cullen-as-wirtinger} identifies \(\partial_z\) with
\(\partial_c\) on the slice.
\end{proof}

\begin{remark}[Two different coefficient functionals]
\label{rem:sres-point-versus-spherical}
The two formulas differ in every ingredient, and it is worth reading them
side by side. Where \eqref{eq:sres-single-point-formula} fixes a single node
\(p\) on a single slice, Theorem~\ref{thm:sres-differential} works with the
conjugate pair, hence with the whole sphere. Where the point formula clears
the singularity with the one-centre factor \((z-p)^m\), which lives on that
slice alone, the spherical one uses the intrinsic \(\DeltaSphere{\Sigma}^M\),
which is defined on all of \(\Hh\). Where the point formula differentiates
with \(\partial_c\) alone, the spherical one alternates \(\partial_c\) with
\(\partial_s\), the second operator being precisely what reaches across the
sphere. And where the point formula returns a quantity that changes with the
slice, the spherical one returns the invariant
\(u_{-1}=\Res_p^If_I+\Res_{\bar p}^If_I\).

The two are therefore not interchangeable, and it is the second that enters
the four-dimensional residue theorem: what a polar sphere contributes there
is the whole quaternion \(u_{-1}\), never one of its two slice-dependent
halves.
\end{remark}

\section{The weighted flux on circular sets}
\label{sec:weighted-flux}

So far the spherical residue has been a local object, read off from an
expansion or from a jet. It remains to produce the global boundary integral
that adds up such coefficients, and the planar identity
\[
  \frac{1}{2\pi}\int_{\partial D}\nu(z)h(z)\dd s(z)
  =\frac{1}{2\pi\ci}\int_{\partial D}\dd z\,h(z)
\]
of \eqref{eq:complex-flux-contour} tells us what to look for. Passing to
four dimensions changes the geometry in one respect: the boundary of a
circular set has, besides the direction inherited from the planar profile,
the two directions tangent to the conjugacy spheres. Both peculiarities of
the quaternionic flux come from them. Their area contributes a second factor
\(2\pi\), which is why the normalization is \(4\pi^2\); and they collapse as
the spheres shrink to real points, at a rate that the weight
\(|\imq q|^{-2}\) exactly compensates.

 We say first what a quaternion-valued flux is, then describe the boundary of
a circular set, and then prove that the four-dimensional integral reduces
exactly to the planar flux on any one full slice — a reduction that needs no
slice regularity, the boundary datum having only to be a continuous slice
function. Two direct computations follow, on a ball and on a tube around a
non-real sphere, which confirm the normalizing constant independently and
show that the flux detects \(u_{-1}\) rather than \(v_{-1}\), so justifying
the choice made in Definition~\ref{def:sres-residues}. A divergence
computation at the end identifies the correction term that appears on
boundaries that are not circular.

Let \(U\subset\R^4\) be a bounded open set with smooth boundary. The flux of
a vector field \(X\colon\overline U\to\R^4\) across \(\partial U\) is the
real number
\begin{equation}\label{eq:weighted-ordinary-flux}
  \int_{\partial U}X(q)\mathbin{\boldsymbol{\cdot}}n(q)\dsigma(q) ,
\end{equation}
where \(\mathbin{\boldsymbol{\cdot}}\) is the Euclidean scalar product of
\(\R^4\), \(n\) the outward Euclidean unit normal and \(\dsigma\) the
positive three-dimensional surface measure, so that the orientation of the
boundary is recorded by the normal and not by the measure; the divergence
theorem then reads
\begin{equation}\label{eq:weighted-ordinary-divergence-theorem}
  \int_{\partial U}X\mathbin{\boldsymbol{\cdot}}n\dsigma
  =\int_U\operatorname{div}X\dd V_4 .
\end{equation}

The boundary integral used below is also built from the outward normal, but
it is not this scalar flux. Identifying vectors of \(\R^4\) with
quaternions, we replace the scalar product by quaternionic multiplication,
which is neither symmetric nor real valued: writing
\[
  n(q)f(q)=\sum_{\mu=0}^3\lambda_\mu(q)e_\mu ,
  \qquad (e_0,e_1,e_2,e_3)=(1,\qi,\qj,\qk) ,
\]
the four coefficients \(\lambda_\mu\) are real valued and bilinear in the
components of \(n\) and \(f\), so that integrating \(nf\) means integrating
each \(\lambda_\mu\) against \(\dsigma\) in the ordinary sense and
reassembling the four numbers into a quaternion. This is what makes the
flux defined below quaternion-valued, and the order of the two factors,
discussed in Remark~\ref{rem:weighted-left-normal}, is part of its
definition.

\subsection{Admissible circular sets and their boundaries}
\label{subsec:weighted-admissible-geometry}
Rather than describing the four-dimensional boundaries directly, we specify
the planar profiles and let circularization produce them.

\begin{definition}[Admissible planar profile]
\label{def:weighted-admissible-profile}
A bounded open set \(D\subset\C\) is an \emph{admissible planar profile} if
\begin{enumerate}[label=\textup{(\roman*)}]
  \item \(D\) is invariant under complex conjugation;
  \item \(\partial D\) is a finite disjoint union of embedded \(C^2\)
        Jordan curves;
  \item \(D\) lies locally on one side of its boundary, in the sense of
        Definition~\ref{def:admissible-domain}(ii);
  \item every intersection of \(\partial D\) with \(\R\) is transverse.
\end{enumerate}
Each boundary component carries the orientation positive relative to \(D\),
so that exterior boundaries run counterclockwise and boundaries of holes
clockwise, while \(\dd s\) always denotes positive arclength, regardless of
that orientation.
\end{definition}

Condition~(iii) is the one already imposed in
Subsection~\ref{subsec:complex-normal-flux}, and it excludes in particular a
curve merely drawn inside the open set. Condition~(iv) says that
\(\partial D\) is nowhere tangent to \(\R\); combined with~(i) it forces
more, namely that the tangent at a real boundary point be vertical, since
the tangent line is preserved by the reflection across \(\R\). This last
condition is what makes the real points of the boundary harmless. Vertical
tangency lets the profile be written near such a point as a graph over the
vertical direction, whose circularization turns out to be \(C^2\) across the
real axis even though the radial coordinate \(|v|\) is not differentiable
there; and transversality keeps the real intersections isolated, hence
finite in number on a compact boundary. Both facts are established in
Proposition~\ref{prop:weighted-boundary-geometry}, and without~(iv) the
four-dimensional boundary could fail to be a hypersurface at the points
where it meets \(\R\).

\begin{figure}[htbp]
\centering
\begin{tikzpicture}[scale=0.95]
  \def\outerpath{
    plot[smooth cycle,tension=0.8] coordinates
      {(-3.2,0) (-3.05,-0.95) (-2.2,-1.75) (-0.9,-2.05) (0.6,-1.95)
       (1.9,-1.5) (2.75,-0.85) (2.95,0) (2.75,0.85) (1.9,1.5)
       (0.6,1.95) (-0.9,2.05) (-2.2,1.75) (-3.05,0.95)}}

  \begin{scope}[even odd rule]
    \fill[black!8] \outerpath
      (1.55,0.95) circle (0.42)
      (1.55,-0.95) circle (0.42)
      (-1.55,0) ellipse (0.75 and 0.50);
  \end{scope}
  \draw[thick,orient] \outerpath;
  \draw[thick,orient] (1.55,0.95) ++(90:0.42) arc(90:-270:0.42);
  \draw[thick,orient] (1.55,-0.95) ++(90:0.42) arc(90:-270:0.42);
  \draw[thick,orient] (-1.55,0) ++(0:0.75) arc(0:-360:0.75 and 0.50);

  \def\secondup{
    plot[smooth cycle,tension=0.75] coordinates
      {(4.15,1.35) (5.05,0.85) (5.75,1.15) (5.55,1.75) (4.75,1.95)}}
  \def\seconddn{
    plot[smooth cycle,tension=0.75] coordinates
      {(4.15,-1.35) (4.75,-1.95) (5.55,-1.75) (5.75,-1.15) (5.05,-0.85)}}
  \fill[black!8] \secondup;
  \fill[black!8] \seconddn;
  \draw[thick,orient] \secondup;
  \draw[thick,orient] \seconddn;

  \draw[->,gray!70] (-3.9,0) -- (6.5,0) node[below right] {$\R$};
  \draw[->,gray!70] (0,-2.5) -- (0,2.6) node[above left] {$\ci\R$};

  \fill (-3.2,0) circle (1.4pt);
  \fill (2.95,0) circle (1.4pt);
  \fill (-2.3,0) circle (1.4pt);
  \fill (-0.8,0) circle (1.4pt);

  \node at (0.0,1.0) {$D$};
  \node at (4.95,1.40) {$D^+$};
  \node at (4.95,-1.40) {$D^-$};
\end{tikzpicture}
\caption{An admissible planar profile: conjugation invariant, with $C^2$
boundary meeting $\R$ transversally at the four marked
points. The left-hand hole crosses the real axis and is itself symmetric;
the two round holes on the right are exchanged by conjugation and meet the
axis nowhere; the second connected component, disjoint from $\R$, is a
product domain with conjugate profile halves $D^+$ and $D^-$. Every boundary
component is positively oriented relative to the profile.}
\label{fig:admissible-profile}
\end{figure}
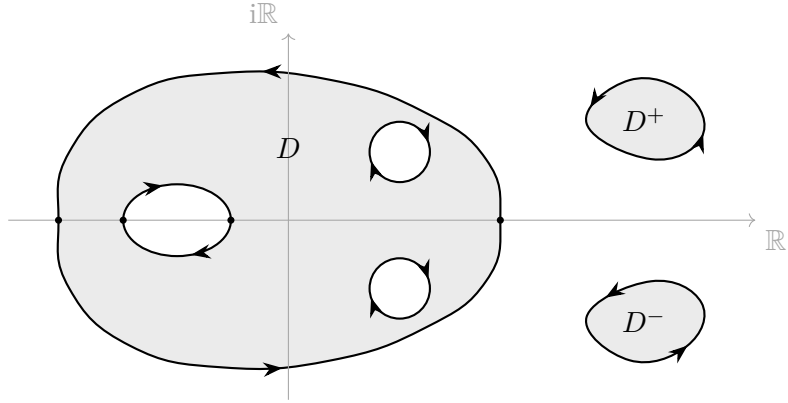

\begin{definition}[Admissible circular open set]
\label{def:weighted-admissible-set}
The circularization
\begin{equation}\label{eq:weighted-admissible-circularization}
  U=\Omega_D=\{x+Ir:x+\ci r\in D,\ I\in\Sphere\}
\end{equation}
of an admissible planar profile \(D\) is called an \emph{admissible circular
open set}. 
\end{definition}
Notice that, in the previous definition
neither \(D\) nor \(U\) need be connected, so that the class
comprises circular sets meeting the real axis, product domains disjoint from
it, and sets with holes.
\begin{definition}[Normalized weighted flux]
\label{def:weighted-flux}
Let \(U\) be an admissible circular open set and let
\(f\colon\partial U\to\Hh\) be continuous. Its \emph{normalized weighted
flux} is
\begin{equation}\label{eq:weighted-flux-definition}
  \boxed{\;
  \Flux{U}{f}
  :=\frac{1}{4\pi^2}
    \int_{\partial U\setminus\R}
    \frac{n(q)f(q)}{|\imq q|^2}\,\dsigma(q).
  \;}
\end{equation}
The integral is taken componentwise, and \(n(q)\) is the outward quaternionic
unit normal.
\end{definition}

\begin{remark}[Why the normal acts on the left]
\label{rem:weighted-left-normal}
The order \(n(q)f(q)\) in
\eqref{eq:weighted-flux-definition} is the four-dimensional counterpart
of the convention in Remark~\ref{rem:complex-left-differential}.
Through the Euclidean metric, the outward normal represents the
outward conormal. In the Green formula for the left Cauchy--Fueter
operator
\[
  D_\ell=\sum_{\mu=0}^3 e_\mu\partial_{x_\mu},
\]
the coefficients \(e_\mu\) act on quaternion-valued functions from the
left, and the corresponding boundary term is
\[
  \int_{\partial W}n(q)h(q)\,\dsigma(q).
\]

The same choice is also forced by the slice reduction. On a slice
\(\C_I\), the identity
\[
  \dd z=I n_I\,\dd s
\]
turns \(n_I f_I\,\dd s\), after normalization, into
\((2\pi I)^{-1}\dd z\,f_I\). Thus the boundary flux reduces to the
left vector-valued contour integral considered in
Section~\ref{sec:complex-model}. Replacing \(n(q)f(q)\) by
\(f(q)n(q)\) would lead instead to a right boundary pairing and to the
right Cauchy--Fueter operator; it would define a legitimate but
different theory.

This convention agrees with the four-dimensional Cauchy formulas for left
slice functions. In the volume formula of
\cite[Theorem~2.3]{GhiloniPerottiVolume}, the boundary integrand has the
order
\[
  C_S(x,q)\,n(q)f(q)\,\dsigma(q).
\]
Likewise, the local formula of
\cite[Theorem~4]{PerottiLocalCauchy} is written using the
quaternion-valued three-form \(Dq\), whose restriction to an outward
oriented hypersurface is \(n(q)\dsigma(q)\), and \(Dq\) acts on the
function from the left. Our weighted flux retains precisely this
normal--function order after the Cauchy kernel is removed.
\end{remark}
The symbol \(\mathcal F\) is the one already used in
\eqref{eq:complex-flux-definition} for the planar normalized flux, however Theorem~\ref{thm:weighted-slice-reduction} will show
that, on a slice function, the four-dimensional quantity
\eqref{eq:weighted-flux-definition} coincides with the planar flux of its
restriction to any full slice. Until that theorem, \(\Flux{U}{f}\) with
\(U\subset\Hh\) is understood as \eqref{eq:weighted-flux-definition} and
\(\Flux{D}{h}\) with \(D\subset\C\) as \eqref{eq:complex-flux-definition}.

The deletion of \(\partial U\cap\R\) in
\eqref{eq:weighted-flux-definition} is only an almost-everywhere convention.
At this stage the factor \(|\imq q|^{-2}\) looks singular, but
Theorem~\ref{thm:weighted-disintegration} will show that the integral is
absolutely convergent even when the boundary meets the real axis. It is an
ordinary Lebesgue integral (and not, for instance, a Cauchy principal value).
The next proposition justifies defining admissibility in terms of \(D\). It
also supplies the boundary regularity needed for the Euclidean normal and
surface measure in Definition~\ref{def:weighted-flux}.

\begin{proposition}[Boundary of an admissible circular set]
\label{prop:weighted-boundary-geometry}
Let \(D\) be an admissible planar profile and \(U=\Omega_D\). Then
\begin{equation}\label{eq:weighted-boundary-circularization}
  \boxed{\;\partial U=\Omega_{\partial D} ,\;}
\end{equation}
and \(\partial U\) is a compact embedded \(C^2\) hypersurface of \(\R^4\). The set \(\partial U\cap\R\) is finite.
\end{proposition}

\begin{proof}
\emph{Step 1: one inclusion.} By
Proposition~\ref{prop:slice-circularization-properties}, applied to the
conjugation-invariant sets \(D\) and \(\C\setminus\overline D\), membership
in \(U\) is decided by the continuous map
\(\pi(q)=\req q+\ci|\imq q|\):
\begin{equation}\label{eq:weighted-U-as-preimage}
  q\in U\iff\pi(q)\in D ,
  \qquad
  q\notin\overline U\iff\pi(q)\notin\overline D .
\end{equation}
A boundary point of \(U\) can therefore satisfy neither alternative, whence
\begin{equation}\label{eq:weighted-first-boundary-inclusion}
  \partial U\subseteq\pi^{-1}(\partial D)=\Omega_{\partial D} .
\end{equation}
The reverse inclusion is proved in Steps~2 and~3, first away from \(\R\) and
then at the real points, where it comes with the regularity statement.

\emph{Step 2: non-real boundary points.} Let \(q_0=x_0+I_0r_0\) with
\(r_0>0\) and \(x_0+\ci r_0\in\partial D\), and let \(\rho(x,r)\) be a local
\(C^2\) defining function for the planar boundary arc through that point,
normalized so that \(D=\{\rho<0\}\) and \(\nabla\rho\neq0\) there. Since
the map \((x,v)\mapsto(x,|v|)\) is smooth of max rank on
\(\R\times(\R^3\setminus\{0\})\),
\begin{equation}\label{eq:weighted-nonreal-defining-function}
  \varrho(x,v):=\rho(x,|v|)
\end{equation}
is a \(C^2\) defining function for \(\partial U\) near \(q_0\), with
non-vanishing gradient; so \(\partial U\) is a \(C^2\) hypersurface away
from \(\R\), and local one-sidedness makes every point of \(\{\varrho=0\}\)
a genuine boundary point, which is the reverse inclusion off the real axis.

\emph{Step 3: real boundary points.} Let \(a\in\partial D\cap\R\). The reflection across \(\R\) fixes \(a\) and
preserves \(D\), so it preserves the tangent to \(\partial D\) at \(a\); of
the lines through \(a\) only the horizontal and the vertical one are
invariant, and transversality rules out the first. The tangent is therefore
vertical. The implicit function theorem then presents the profile
near \(a\) as a graph over the vertical direction,
\begin{equation}\label{eq:weighted-real-profile-graph}
  x=a+\psi(r),\qquad |r|<\varepsilon ,
\end{equation}
with \(\psi\) of class \(C^2\), \(\psi(0)=\psi'(0)=0\) and, by uniqueness of
the graph together with reflection invariance,
\begin{equation}\label{eq:weighted-profile-evenness}
  \psi(-r)=\psi(r) .
\end{equation}
Circularizing turns \eqref{eq:weighted-real-profile-graph} into the radial
graph
\begin{equation}\label{eq:weighted-real-radial-graph}
  x=a+\psi(|v|),\qquad v\in\R^3,\ |v|<\varepsilon ,
\end{equation}
and since \(D\) lies locally on one side of
\eqref{eq:weighted-real-profile-graph}, moving \(x\) across
\eqref{eq:weighted-real-radial-graph} produces points of \(U\) and of its
complement arbitrarily close to \(a\); hence \(a\in\partial U\), and
\eqref{eq:weighted-boundary-circularization} is proved.

\emph{Step 4: regularity at the real points.} It remains to see that the
radial graph \eqref{eq:weighted-real-radial-graph} is of class \(C^2\) at \(v=0\),
where \(|v|\) is not differentiable. For \(r=|v|>0\),
\[
  \nabla_v\bigl(\psi(|v|)\bigr)=\frac{\psi'(r)}{r}\,v ,
  \qquad
  D_v^2\bigl(\psi(|v|)\bigr)
  =\Bigl(\psi''(r)-\frac{\psi'(r)}{r}\Bigr)\frac{vv^T}{r^2}
   +\frac{\psi'(r)}{r}\operatorname{Id}_3 ,
\]
and the evenness \eqref{eq:weighted-profile-evenness} of the \(C^2\)
function \(\psi\) gives \(\psi'(0)=0\), hence
\[
  \frac{\psi'(r)}{r}\longrightarrow\psi''(0) ,
  \qquad
  \psi''(r)-\frac{\psi'(r)}{r}\longrightarrow0
  \qquad(r\to0) .
\]
The gradient therefore extends by \(0\) and the Hessian by
\(\psi''(0)\operatorname{Id}_3\), so the radial graph is \(C^2\) at the
origin; its defining equation having derivative \(1\) in \(x\), it is
embedded there.

Finally, transversality makes the real intersections isolated, and an
infinite family of them on the compact set \(\partial D\) would accumulate
somewhere, against the same local graph description. So \(\partial D\cap\R\)
is finite, and with it \(\partial U\cap\R\).
\end{proof}

The planar outward unit normal of \(\partial D\) will be written
\begin{equation}\label{eq:weighted-planar-normal-components}
  \nu(z)=\nu_1(z)+\ci\nu_2(z) ,
\end{equation}
the letter \(n\) being reserved for the quaternionic normal of
\(\partial U\). Conjugation is an isometry that preserves \(D\), hence
maps \(\partial D\) to itself and the outward side to the outward side, so
it carries the normal at \(z\) to the normal at \(\bar z\); being the
reflection across \(\R\), it does so by changing the sign of the vertical
component alone (see Figure~\ref{fig:normal-parity}). Thus
\begin{equation}\label{eq:weighted-normal-parity}
  \nu_1(\bar z)=\nu_1(z),
  \qquad
  \nu_2(\bar z)=-\nu_2(z) ,
\end{equation}
which is precisely the stem symmetry of
Definition~\ref{def:slice-stem-function}: \(\nu\) is a stem function, and
the outward normal of \(\partial U\) will turn out to be the slice function
it induces.

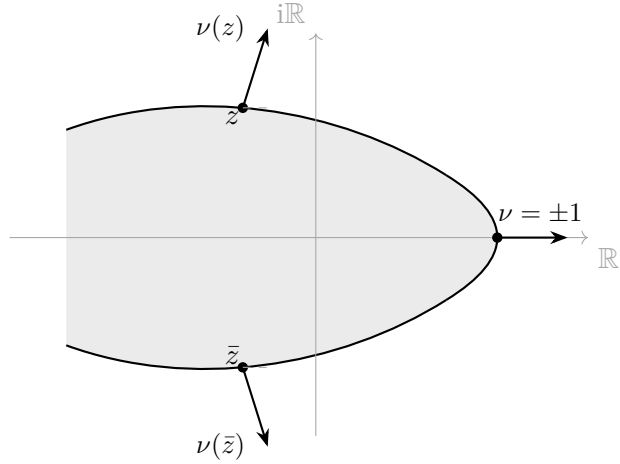
\begin{figure}[htbp]
\centering
\begin{tikzpicture}[scale=1.5]
  \fill[black!8]
    (-2.2,0.95) .. controls (-1.1,1.35) and (0.2,1.15) ..
    (1.15,0.55) .. controls (1.75,0.18) and (1.75,-0.18) ..
    (1.15,-0.55) .. controls (0.2,-1.15) and (-1.1,-1.35) ..
    (-2.2,-0.95) -- (-2.2,0.95);

  \draw[thick]
    (-2.2,0.95) .. controls (-1.1,1.35) and (0.2,1.15) ..
    coordinate[pos=0.45] (zup) (1.15,0.55);
  \draw[thick]
    (-2.2,-0.95) .. controls (-1.1,-1.35) and (0.2,-1.15) ..
    coordinate[pos=0.45] (zdn) (1.15,-0.55);
  \draw[thick]
    (1.15,0.55) .. controls (1.75,0.18) and (1.75,-0.18) ..
    coordinate[pos=0.5] (zre) (1.15,-0.55);

  \draw[->,gray!70] (-2.7,0) -- (2.4,0) node[below right] {$\R$};
  \draw[->,gray!70] (0,-1.75) -- (0,1.8) node[above left] {$\ci\R$};

  \fill (zup) circle (1.3pt);
  \fill (zdn) circle (1.3pt);
  \draw[-{Stealth},thick] (zup) -- ++(0.22,0.70);
  \draw[-{Stealth},thick] (zdn) -- ++(0.22,-0.70);
  \node[below left,inner sep=1pt] at (zup) {$z$};
  \node[above left,inner sep=1pt] at (zdn) {$\bar z$};
  \node[right,font=\small] at ($(zup)+(-0.5,0.70)$) {$\nu(z)$};
  \node[right,font=\small] at ($(zdn)+(-0.5,-0.70)$) {$\nu(\bar z)$};

  \draw[dashed,gray!70] (zup) -- ++(0.22,0);
  \draw[dashed,gray!70] (zdn) -- ++(0.22,0);

  \fill (zre) circle (1.3pt);
  \draw[-{Stealth},thick] (zre) -- ++(0.62,0);
  \node[above,font=\small] at ($(zre)+(0.38,0.04)$) {$\nu=\pm1$};
\end{tikzpicture}
\caption{Conjugation maps $\partial D$ to itself preserving the outward
side, so the outward normals at $z$ and at $\bar z$ are reflections of one
another: the horizontal components agree and the vertical ones are
opposite, which is \eqref{eq:weighted-normal-parity}. At a real boundary
point the normal is its own reflection, hence horizontal, in accordance
with Lemma~\ref{lem:weighted-circular-normal}.}
\label{fig:normal-parity}
\end{figure}

\begin{lemma}[Outward normal after circularization]
\label{lem:weighted-circular-normal}
At a non-real boundary point \(q=x+Ir\) of \(U\), written with a signed
coordinate \(r\neq0\) and with \(z=x+\ci r\in\partial D\), the outward unit
normal is
\begin{equation}\label{eq:weighted-circular-normal}
  \boxed{\;n(q)=\mathcal I(\nu)(q)=\nu_1(z)+I\nu_2(z) .\;}
\end{equation}
At a real boundary point it extends continuously to \(n(a)=\nu_1(a)=\pm1\).
\end{lemma}

\begin{proof}
That the right-hand side of \eqref{eq:weighted-circular-normal} does not
depend on which of the two signed representations of \(q\) is used is the
well-posedness of \eqref{eq:slice-induction}, granted by
\eqref{eq:weighted-normal-parity}.

Let \(s\mapsto z(s)=\alpha(s)+\ci\beta(s)\) parametrize \(\partial D\) by
arclength and let \(E\in T_I\Sphere\). The tangent space of \(\partial U\)
at \(q\) is spanned by
\begin{equation}\label{eq:weighted-boundary-tangent-vectors}
  \alpha'(s)+I\beta'(s)
  \qquad\text{and}\qquad
  \beta(s)E ,
\end{equation}
the second direction being two-dimensional. The quaternion
\(N=\nu_1+I\nu_2\) is a unit vector orthogonal to both: to the first because
\(\nu_1\alpha'+\nu_2\beta'=0\), \(\nu\) being normal to the profile; to the
second because \(E\perp I\) and \(E\perp1\). It points outward, being the
circularization of the outward planar normal, as the local defining function
\eqref{eq:weighted-nonreal-defining-function} also shows.

At a real point \(a\) the planar normal is horizontal by Step~3 of
Proposition~\ref{prop:weighted-boundary-geometry}, so \(\nu_2(a)=0\) and the
stated limit follows.
\end{proof}

The boundary of an admissible circular set carries one direction inherited
from the planar profile and two tangent to the conjugacy spheres. Computing
the corresponding surface measure explains at once the singular weight and
the normalizing constant of Definition~\ref{def:weighted-flux}.

Fix an open hemisphere \(H\subset\Sphere\), of area \(2\pi\), whose
complement is \(-H\) up to their common equator, and parametrize the
boundary by the \emph{full signed profile}:
for each arclength-parametrized component \(z(s)=\alpha(s)+\ci\beta(s)\) of
\(\partial D\), set
\begin{equation}\label{eq:weighted-boundary-parametrization}
  \Xi(s,I)=\alpha(s)+I\beta(s),\qquad I\in H .
\end{equation}
A non-real quaternion has exactly two signed representations, \((I,\beta)\)
and \((-I,-\beta)\), and unless \(I\) lies on the equator exactly one of
\(\pm I\) belongs to \(H\); so \eqref{eq:weighted-boundary-parametrization},
taken over all components of \(\partial D\), covers
\(\partial U\setminus\R\) once, up to a surface-null set.

\begin{theorem}[The weighted boundary measure]
\label{thm:weighted-disintegration}
Let \(U=\Omega_D\) be an admissible circular open set, let
\(H\subset\Sphere\) be an open hemisphere and let
\(G\colon\partial U\to\Hh\) be continuous. Then the integral
\begin{equation}\label{eq:weighted-disintegration}
  \boxed{\;
  \int_{\partial U\setminus\R}\frac{G(q)}{|\imq q|^2}\dsigma(q)
  =\int_H\int_{\partial D}G(\alpha+I\beta)\dd s\,\darea(I) ,
  \qquad z=\alpha+\ci\beta ,
  \;}
\end{equation}
converges absolutely, with
\begin{equation}\label{eq:weighted-integrability-bound}
  \int_{\partial U\setminus\R}\frac{|G(q)|}{|\imq q|^2}\dsigma(q)
  \leq2\pi\operatorname{Length}(\partial D)\,
        \|G\|_{L^\infty(\partial U)} ,
\end{equation}
and near the real axis
\begin{equation}\label{eq:weighted-real-axis-estimate}
  \int_{\partial U\cap\{0<|\imq q|\leq\varepsilon\}}
    \frac{|G(q)|}{|\imq q|^2}\dsigma(q)
  =O(\varepsilon)\,\|G\|_{L^\infty(\partial U)}
  \qquad(\varepsilon\downarrow0) .
\end{equation}
\end{theorem}

\begin{proof}
\emph{The surface Jacobian.} Fix \(s\) with \(\beta(s)\neq0\) and an
orthonormal basis \(E_1,E_2\) of \(T_I\Sphere\). Differentiating
\eqref{eq:weighted-boundary-parametrization},
\begin{equation}\label{eq:weighted-parametrization-derivatives}
  \partial_s\Xi=\alpha'+I\beta' ,
  \qquad
  \dd\Xi(E_1)=\beta E_1 ,
  \qquad
  \dd\Xi(E_2)=\beta E_2 ,
\end{equation}
three mutually orthogonal vectors, since \(E_a\perp I\) and \(E_a\perp1\),
the first of norm \(1\) because \(s\) is arclength and the other two of norm
\(|\beta|\). The Gram determinant is therefore \(\beta^4\) and the surface
Jacobian \(\beta^2\), so that
\begin{equation}\label{eq:weighted-surface-element}
  \boxed{\;
  \dsigma=\beta^2\dd s\,\darea(I) ,
  \qquad
  |\imq\Xi(s,I)|^2=\beta^2 .
  \;}
\end{equation}
This is the geometric cancellation the whole construction rests on: the
weight \(|\imq q|^{-2}\) is the exact inverse of the area scaling of the two
spherical directions.

\emph{Away from the axis.} Restricted to \(|\beta|>\varepsilon\) the
parametrization is non-singular, so the area formula and
\eqref{eq:weighted-surface-element} give
\begin{equation}\label{eq:weighted-cutoff-disintegration}
  \int_{\partial U\cap\{|\imq q|>\varepsilon\}}
    \frac{G(q)}{|\imq q|^2}\dsigma(q)
  =\int_H\int_{\partial D\cap\{|\beta|>\varepsilon\}}
    G(\alpha+I\beta)\dd s\,\darea(I) .
\end{equation}
Applied to \(|G|\), the right-hand side is at most
\(\Area(H)\operatorname{Length}(\partial D)\|G\|_\infty\), so monotone
convergence as \(\varepsilon\downarrow0\) gives both the absolute
convergence and \eqref{eq:weighted-integrability-bound}; dominated
convergence, applied to the four real components of \(G\), then gives
\eqref{eq:weighted-disintegration}.

\emph{Near the axis.} At a point of \(\partial D\cap\R\) the tangent is
vertical, so an arclength parametrization has \(\alpha'(s_0)=0\) and
\(|\beta'(s_0)|=1\); by the inverse function theorem \(\beta\) may be used
as a local parameter, with \(\dd s/\dd\beta\) bounded nearby. Summing over
the finitely many real intersections,
\[
  \operatorname{Length}
  \bigl(\partial D\cap\{|\beta|\leq\varepsilon\}\bigr)\leq C\varepsilon
\]
for small \(\varepsilon\), and \eqref{eq:weighted-cutoff-disintegration}
applied to that thin strip gives \eqref{eq:weighted-real-axis-estimate}.
The analytic meaning is then clear: shrinking a neighbourhood of the real
points removes an amount tending to zero, with no symmetric cutoff and no
principal value.
\end{proof}

\begin{remark}[Why the weight has exponent two]
\label{rem:weighted-exponent-two}
The conjugacy class \(\alpha+\Sphere\beta\) is a two-sphere of radius
\(|\beta|\), so areas in its two tangent directions scale like
\(\beta^2\); as \(\beta\to0\) the sphere collapses to a real point and the
surface element of \(\partial U\) vanishes quadratically with it. The factor
\(|\imq q|^{-2}\) compensates that collapse exactly — neither more nor less,
which is why \eqref{eq:weighted-disintegration} has a bounded integrand on
the right and \eqref{eq:weighted-real-axis-estimate} is \(O(\varepsilon)\)
rather than \(O(1)\) or \(O(\varepsilon^2)\).
\end{remark}

The formula just proved holds for arbitrary continuous quaternionic data on
the boundary. To turn it into a planar integral, free of the angular
variable, we now use the affine dependence on \(I\) peculiar to slice
functions.

By a continuous slice function on \(\partial U\) we mean a boundary map
induced by continuous stem data \(A+\ci B\) on \(\partial D\) — equivalently,
the restriction to \(\partial U\) of a continuous slice function defined on
a circular neighbourhood of it. In signed profile coordinates, if
\(f=\mathcal I(A+\ci B)\) then
\begin{equation}\label{eq:weighted-slice-boundary-values}
  f(\alpha+I\beta)=A(z)+IB(z),
  \qquad z=\alpha+\ci\beta ,
\end{equation}
with the usual stem parities
\begin{equation}\label{eq:weighted-stem-parity}
  A(\bar z)=A(z),\qquad B(\bar z)=-B(z) .
\end{equation}
Here neither differentiability nor slice regularity is assumed.

\begin{theorem}[Reduction of the weighted flux to a full slice]
\label{thm:weighted-slice-reduction}
Let \(U=\Omega_D\) be an admissible circular open set, let
\(f\colon\partial U\to\Hh\) be a continuous slice function, and for
\(I\in\Sphere\) let \(D_I:=U\cap\C_I\), with planar outward unit normal
\(n_I\). Then
\begin{equation}\label{eq:weighted-slice-reduction}
  \boxed{\;
  \Flux{U}{f}
  =\frac{1}{2\pi}\int_{\partial D_I}n_I(z)f_I(z)\dd s
  =(2\pi I)^{-1}\int_{\partial D_I}\dd z\,f_I(z) ,
  \;}
\end{equation}
so that in particular the planar expression does not depend on the slice.
\end{theorem}

In the previous statement,
When \(U\) is a product domain, \(\partial D_I\) comprises the boundaries of
both conjugate profile components.
\begin{proof}
Fix the hemisphere \(H\) of Theorem~\ref{thm:weighted-disintegration} and
write \(J\in H\) for the angular variable. By
Lemma~\ref{lem:weighted-circular-normal} the quaternionic normal is
\(n=\nu_1+J\nu_2\), and multiplying it on the left by
\eqref{eq:weighted-slice-boundary-values} gives
\begin{equation}\label{eq:weighted-normal-times-slice-function}
  nf=(\nu_1+J\nu_2)(A+JB)=\Pi+J\Theta ,
\end{equation}
where
\begin{equation}\label{eq:weighted-even-odd-parts}
  \Pi:=\nu_1A-\nu_2B ,
  \qquad
  \Theta:=\nu_1B+\nu_2A .
\end{equation}
The components \(\nu_1,\nu_2\) are real, hence central.
Comparing \eqref{eq:weighted-normal-parity} with
\eqref{eq:weighted-stem-parity},
\begin{equation}\label{eq:weighted-PQ-parity}
  \Pi(\bar z)=\Pi(z),
  \qquad
  \Theta(\bar z)=-\Theta(z) ,
\end{equation}
and since conjugation preserves arclength and permutes the components of
\(\partial D\), the odd part integrates to zero over the full signed
profile:
\begin{equation}\label{eq:weighted-odd-profile-cancellation}
  \int_{\partial D}\Theta(z)\dd s=0 .
\end{equation}

Applying \eqref{eq:weighted-disintegration} to \(G=nf\) and using
\eqref{eq:weighted-normal-times-slice-function} and
\eqref{eq:weighted-odd-profile-cancellation},
\begin{equation}\label{eq:weighted-angular-reduction}
  \int_{\partial U}\frac{n(q)f(q)}{|\imq q|^2}\dsigma(q)
  =\int_H\int_{\partial D}\bigl(\Pi+J\Theta\bigr)\dd s\,\darea(J)
  =\Area(H)\int_{\partial D}\Pi\dd s
  =2\pi\int_{\partial D}\Pi\dd s .
\end{equation}
On the slice \(\C_I\) of the statement, the same two formulas give
\(n_If_I=\Pi+I\Theta\), whose odd part again integrates to zero, so that
\begin{equation}\label{eq:weighted-planar-P-integral}
  \int_{\partial D_I}n_If_I\dd s=\int_{\partial D}\Pi\dd s .
\end{equation}
Dividing \eqref{eq:weighted-angular-reduction} by \(4\pi^2\) and comparing
with \eqref{eq:weighted-planar-P-integral} proves the first equality in
\eqref{eq:weighted-slice-reduction}.

For the second, the orientation convention of
Lemma~\ref{lem:complex-tangent-normal} makes \(In_I\) the positively
oriented unit tangent of \(\partial D_I\), so that
\begin{equation}\label{eq:weighted-normal-to-dz}
  \dd z=In_I\dd s ,
  \qquad\text{that is}\qquad
  n_I\dd s=-I\dd z ,
\end{equation}
and \((2\pi I)^{-1}=-I/(2\pi)\) turns
\(\frac1{2\pi}\int_{\partial D_I}n_If_I\dd s\) into
\((2\pi I)^{-1}\int_{\partial D_I}\dd z\,f_I\), all scalar factors having
stayed on the left.
\end{proof}

One point of the proof deserves emphasis: the odd part \(\Theta\) was
killed by the integration over the full signed profile, at each fixed
\(J\), and not by the angular integration: the identity
\(\int_HJ\darea(J)=0\) that one might invoke instead is false, the integral
of \(J\) over a hemisphere being in general non-zero, and it is nowhere
used. This is also why \eqref{eq:weighted-boundary-parametrization} was set
up over the whole of \(\partial D\) rather than over its upper half.

\begin{corollary}[Boundedness and right linearity]
\label{cor:weighted-boundedness}
The weighted flux is a bounded right \(\Hh\)-linear functional of the
boundary datum: for continuous \(f,g\colon\partial U\to\Hh\) and
\(a,b\in\Hh\),
\begin{equation}\label{eq:weighted-right-linearity}
  \Flux{U}{fa+gb}=\Flux{U}{f}a+\Flux{U}{g}b ,
\end{equation}
and
\begin{equation}\label{eq:weighted-flux-bound}
  |\Flux{U}{f}|
  \leq\frac{\operatorname{Length}(\partial D)}{2\pi}
      \|f\|_{L^\infty(\partial U)} .
\end{equation}
\end{corollary}

\begin{proof}
Right linearity is due to associativity, \(n(fa)=(nf)a\). The estimate follows from
\(|n|=1\) and \eqref{eq:weighted-integrability-bound} applied to \(G=nf\),
divided by \(4\pi^2\).
\end{proof}

\begin{remark}[Where the factor \(4\pi^2\) comes from]
\label{rem:weighted-normalization}
Equation \eqref{eq:weighted-angular-reduction} exhibits the constant as
\(4\pi^2=(2\pi)\Area(H)\): the first factor is the normalization of the
planar flux \eqref{eq:complex-flux-definition}, the second the area of the
hemisphere, used in tandem with the full signed profile. The geometry behind
this is close to that of the volume Cauchy formulas of
\cite{GhiloniPerottiVolume}; the weighted density is not being proposed here
as a new Cauchy kernel.
\end{remark}
\subsection{Direct checks on balls and tubes}
\label{subsec:weighted-direct-checks}

The reduction theorem already fixes the normalization, but two direct
boundary computations are worth making, for different reasons. The first, on
a ball centred at a real point, shows the constant \(4\pi^2\) emerging from
the geometry rather than from a convention. The second, on a tube around a
non-real sphere, settles a question the first cannot even pose: which of the
two paired coefficients the flux detects. The answer, \(u_{-1}\) and not
\(v_{-1}\), is what makes Definition~\ref{def:sres-residues} the right one.

\begin{proposition}
\label{prop:weighted-ball-extraction}
Let \(a\in\R\), \(R>0\), and let \(f\) admit a normally convergent Laurent
expansion
\begin{equation}\label{eq:weighted-ball-Laurent}
  f(q)=\sum_{m\in\mathbb Z}(q-a)^mc_m ,
  \qquad c_m\in\Hh ,
\end{equation}
on an annulus containing \(\partial B(a,R)\). Then
\begin{equation}\label{eq:weighted-ball-selector}
  \boxed{\;\Flux{B(a,R)}{f}=c_{-1} .\;}
\end{equation}
\end{proposition}

\begin{proof}
Away from the two real points \(a\pm R\), parametrize the boundary
three-sphere by
\begin{equation}\label{eq:weighted-ball-parametrization}
  q=a+R(\cos\theta+I\sin\theta),
  \qquad 0<\theta<\pi,\quad I\in\Sphere ,
\end{equation}
which covers it once, this time with a positive radial coordinate
\(R\sin\theta\) and the whole of \(\Sphere\) (see Figure \ref{fig:weighted-ball-parametrization}). The outward unit normal is
\(n(q)=\cos\theta+I\sin\theta\), the surface element is
\(\dsigma=R^3\sin^2\theta\dd\theta\,\darea(I)\) and
\(|\imq q|^2=R^2\sin^2\theta\), so that the weighted element is simply
\begin{equation}\label{eq:weighted-ball-measure-cancellation}
  \frac{\dsigma}{|\imq q|^2}=R\dd\theta\,\darea(I) .
\end{equation}

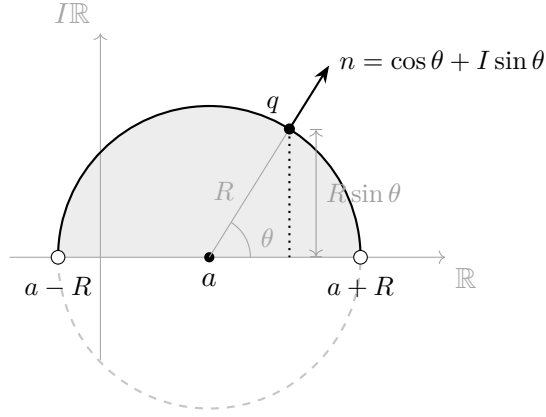
\begin{figure}[htbp]
\centering
\begin{tikzpicture}[scale=1.6]
  \def\aa{0.9}      
  \def\RR{1.25}     
  \pgfmathsetmacro{\th}{58}
  \pgfmathsetmacro{\px}{\aa+\RR*cos(\th)}
  \pgfmathsetmacro{\py}{\RR*sin(\th)}

  \fill[black!7] (\aa,0) ++(0:\RR) arc(0:180:\RR) -- cycle;

  \draw[->,gray!70] (-0.75,0) -- (2.85,0) node[below right] {$\R$};
  \draw[->,gray!70] (0,-0.85) -- (0,1.85) node[above left] {$I\R$};

  \draw[thick] (\aa,0) ++(0:\RR) arc(0:180:\RR);
  \draw[thick,gray!45,dashed] (\aa,0) ++(180:\RR) arc(180:360:\RR);

  \fill (\aa,0) circle (1.2pt);
  \node[below] at (\aa,-0.05) {$a$};
  \draw[gray!70] (\aa,0) -- (\px,\py);
  \node[gray!70,font=\small] at ({\aa+0.42*\RR*cos(\th)-0.16},
                                 {0.42*\RR*sin(\th)+0.10}) {$R$};

  \draw[gray!70] (\aa,0) ++(0:0.34) arc(0:\th:0.34);
  \node[gray!70,font=\small] at ({\aa+0.48},{0.16}) {$\theta$};

  \fill (\px,\py) circle (1.3pt);
  \draw[-{Stealth},thick] (\px,\py) -- ++({0.62*cos(\th)},{0.62*sin(\th)});
  \node[right,font=\small] at ({\px+0.62*cos(\th)},{\py+0.62*sin(\th)+0.04})
    {$n=\cos\theta+I\sin\theta$};
  \node[left,font=\small] at (\px,\py+0.2) {$q$};

  \draw[dotted,thick] (\px,\py) -- (\px,0);
  \draw[|<->|,gray!75] ({\px+0.22},0) -- ({\px+0.22},\py)
    node[midway,right,font=\small] {$R\sin\theta$};

  \fill[white] ({\aa-\RR},0) circle (1.6pt);
  \draw ({\aa-\RR},0) circle (1.6pt);
  \fill[white] ({\aa+\RR},0) circle (1.6pt);
  \draw ({\aa+\RR},0) circle (1.6pt);
  \node[below,font=\small] at ({\aa-\RR},-0.08) {$a-R$};
  \node[below,font=\small] at ({\aa+\RR},-0.08) {$a+R$};
\end{tikzpicture}
\caption{The section by a slice $\C_I$ of the boundary three-sphere of
$B(a,R)$, parametrized as in \eqref{eq:weighted-ball-parametrization}. The
angle $\theta$ runs over the upper half-circle alone, the whole of
$\Sphere$ being swept by $I$; the two real points $a\pm R$, drawn hollow,
correspond to $\theta=0$ and $\theta=\pi$ and are excluded. The outward
normal at $q$ is radial, and $|\imq q|=R\sin\theta$ is the distance from
the real axis, which is what the weight of
Definition~\ref{def:weighted-flux} divides by.}
\label{fig:weighted-ball-parametrization}
\end{figure}

All powers being computed in the commutative plane \(\C_I\),
\[
  n(q)(q-a)^m=R^m(\cos\theta+I\sin\theta)^{m+1}
             =R^m\bigl(\cos((m+1)\theta)+I\sin((m+1)\theta)\bigr) ,
\]
so the term of index \(m\) contributes
\begin{equation}\label{eq:weighted-ball-mode-integral}
  R^{m+1}\int_0^\pi\!\!\int_\Sphere
    \bigl(\cos((m+1)\theta)+I\sin((m+1)\theta)\bigr)c_m
    \,\darea(I)\dd\theta
\end{equation}
to the unnormalized integral. The summand carrying \(I\) vanishes under the
angular integration, by the antipodal symmetry that gives
\begin{equation}\label{eq:weighted-sphere-first-moment}
  \int_\Sphere I\darea(I)=0 ,
  \qquad
  \Area(\Sphere)=4\pi ,
\end{equation}
and the other leaves \(4\pi R^{m+1}c_m\int_0^\pi\cos((m+1)\theta)\dd\theta\),
which vanishes unless \(m=-1\) and equals \(4\pi^2c_{-1}\) when \(m=-1\).
Normal convergence against a finite weighted measure justifies integrating
\eqref{eq:weighted-ball-Laurent} term by term, and division by \(4\pi^2\)
gives \eqref{eq:weighted-ball-selector}.
\end{proof}

Taking a single term, this says that
\begin{equation}\label{eq:weighted-ball-basic-normalization}
  \int_{\partial B(a,R)}
    \frac{n(q)(q-a)^{-1}c}{|\imq q|^2}\dsigma(q)=4\pi^2c
  \qquad(c\in\Hh) ,
\end{equation}
a direct check of the constant and of the position of the coefficient on the
right. Note also that only convergence near the integration sphere was used,
so the proposition applies just as well to an exterior expansion.

The second computation uses a tubular neighbourhood of a non-real conjugacy
sphere.

\begin{definition}[A circular tube]
\label{def:weighted-circular-tube}
For \(\alpha\in\R\), \(\beta>0\) and \(0<\rho<\beta\), set
\begin{equation}\label{eq:weighted-circular-tube}
  T_{\alpha,\beta,\rho}
  :=\bigl\{q\in\Hh:
     (\req q-\alpha)^2+(|\imq q|-\beta)^2<\rho^2\bigr\} .
\end{equation}
\end{definition}

Its planar profile is the pair of conjugate discs of radius \(\rho\) centred
at \(\alpha\pm\ci\beta\), disjoint from \(\R\) precisely because
\(\rho<\beta\); so \(T_{\alpha,\beta,\rho}\) is an admissible circular open
set, it misses the real axis, and it contains the sphere
\(\Sigma=\alpha+\Sphere\beta\). Being circular and disjoint from \(\R\), it
is a product domain, and its trace on any slice is
the pair of discs itself.

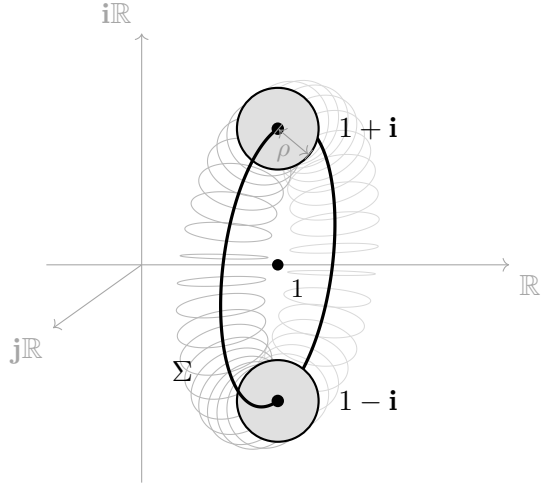
\begin{figure}[htbp]
\centering
\begin{tikzpicture}[scale=1.8]
  \def\rh{0.30}                 
  \pgfmathsetmacro{\jx}{-0.42}  
  \pgfmathsetmacro{\jy}{-0.30}

  \draw[->,gray!70] (-0.70,0) -- (2.70,0) node[below right] {$\R$};
  \draw[->,gray!70] (0,-1.60) -- (0,1.70) node[above left] {$\qi\R$};
  \draw[->,gray!70] (0,0) -- ({1.55*\jx},{1.55*\jy}) node[below left] {$\qj\R$};

  \foreach \ph in {180,190,...,350}{
    \draw[gray!30,line width=0.35pt]
      plot[domain=0:360,samples=70,variable=\t]
        ({1+\rh*cos(\t)+(1+\rh*sin(\t))*sin(\ph)*\jx},
         {(1+\rh*sin(\t))*cos(\ph)+(1+\rh*sin(\t))*sin(\ph)*\jy});}

  \draw[very thick]
    plot[domain=180:360,samples=90,variable=\ph]
      ({1+sin(\ph)*\jx},{cos(\ph)+sin(\ph)*\jy});

  \foreach \ph in {0,10,...,170}{
    \draw[gray!55,line width=0.4pt]
      plot[domain=0:360,samples=70,variable=\t]
        ({1+\rh*cos(\t)+(1+\rh*sin(\t))*sin(\ph)*\jx},
         {(1+\rh*sin(\t))*cos(\ph)+(1+\rh*sin(\t))*sin(\ph)*\jy});}

  \fill[black!12] (1,1) circle (\rh);
  \fill[black!12] (1,-1) circle (\rh);

  \draw[very thick]
    plot[domain=0:180,samples=90,variable=\ph]
      ({1+sin(\ph)*\jx},{cos(\ph)+sin(\ph)*\jy});
  \node at (0.30,-0.78) {$\Sigma$};

  \draw[thick] (1,1) circle (\rh);
  \draw[thick] (1,-1) circle (\rh);
  \fill (1,1) circle (1.3pt);
  \fill (1,-1) circle (1.3pt);
  \node[right] at ({1+\rh+0.07},1) {$1+\qi$};
  \node[right] at ({1+\rh+0.07},-1) {$1-\qi$};

  \fill (1,0) circle (1.2pt);
  \node[below right,font=\small] at (1.02,-0.04) {$1$};
  \draw[|<->|,gray!80] (1,1) -- ++(-40:\rh)
    node[midway,below left,font=\small,inner sep=1pt] {$\rho$};
\end{tikzpicture}
\caption{The circular tube $T_{1,1,\rho}$ in $\R+\qi\R+\qj\R$. The polar
sphere $\Sigma=1+\Sphere$ appears as the heavy circle of radius $1$ in the
plane through $1$ orthogonal to $\R$; the tube is the solid of revolution
obtained by turning a disc of radius $\rho$ about the real axis, so that its
core is $\Sigma$ and its axis is $\R$. The plane of the page is the slice
$\C_\qi$, meeting the tube in the two conjugate discs about $1\pm\qi$. The
condition $\rho<\beta$ keeps the torus clear of the real axis.}
\label{fig:circular-tube}
\end{figure}

\begin{lemma}
\label{lem:weighted-tube-parametrization}
Let \(G\colon\partial T_{\alpha,\beta,\rho}\to\Hh\) be continuous and set
\begin{equation}\label{eq:weighted-tube-w}
  w(t,I):=\rho(\cos t+I\sin t) .
\end{equation}
Then
\begin{equation}\label{eq:weighted-tube-flux}
  \boxed{\;
  \Flux{T_{\alpha,\beta,\rho}}{G}
  =\frac1{4\pi^2}\int_\Sphere\int_0^{2\pi}
     w(t,I)\,G\bigl(\alpha+I\beta+w(t,I)\bigr)\dd t\,\darea(I) ,
  \;}
\end{equation}
with no slice assumption on \(G\).
\end{lemma}

\begin{proof}
Parametrize the boundary by
\begin{equation}\label{eq:weighted-tube-boundary-parametrization}
  q(t,I)=\alpha+\rho\cos t+I(\beta+\rho\sin t)=\alpha+I\beta+w(t,I) ,
\end{equation}
which uses a positive radial coordinate, \(\beta+\rho\sin t>0\), and the
whole of \(\Sphere\), covering the boundary once up to the set of measure
zero where \(t=0\) and \(t=2\pi\) are identified. Its outward unit normal is
\begin{equation}\label{eq:weighted-tube-normal}
  n(q(t,I))=\cos t+I\sin t=\frac{w(t,I)}{\rho} ,
\end{equation}
and the three tangent directions are orthogonal, the one along \(t\) of
length \(\rho\) and the two angular ones scaled by \(\beta+\rho\sin t\), so
that
\begin{equation}\label{eq:weighted-tube-surface-element}
  \dsigma=\rho(\beta+\rho\sin t)^2\dd t\,\darea(I) .
\end{equation}
Since \(|\imq q|^2=(\beta+\rho\sin t)^2\), the two factors
\(\rho(\beta+\rho\sin t)^2\) and \(|\imq q|^{-2}\) collapse against
\eqref{eq:weighted-tube-normal} into
\[
  n(q)\frac{\dsigma}{|\imq q|^2}=w(t,I)\dd t\,\darea(I) ,
\]
the factor \(w\) remaining on the left of \(G\); substituting into
Definition~\ref{def:weighted-flux} gives \eqref{eq:weighted-tube-flux}.
\end{proof}

\begin{proposition}[Direct extraction of a spherical residue]
\label{prop:weighted-tube-extraction}
Let \(u,v\in\Hh\), let \(\Sigma=\alpha+\Sphere\beta\), and set
\begin{equation}\label{eq:weighted-elementary-spherical-pole}
  f_{u,v}(q):=\DeltaSphere{\Sigma}(q)^{-1}(qu+v) .
\end{equation}
Then, for every \(0<\rho<\beta\),
\begin{equation}\label{eq:weighted-tube-selector}
  \boxed{\;
  \Flux{T_{\alpha,\beta,\rho}}{f_{u,v}}=u=\SRes_\Sigma f_{u,v} .
  \;}
\end{equation}
\end{proposition}

\begin{proof}
By Lemma~\ref{lem:weighted-tube-parametrization} the flux is an iterated
integral: one integration around each meridian circle, one over the sphere
of imaginary units. We evaluate the inner integral
first, identifying it with a point residue, and then average that residue
over the sphere.

Fix \(I\in\Sphere\) and set \(y=\alpha+I\beta\), \(\bar y=\alpha-I\beta\).
The curve \(z(t)=y+w(t,I)\) traverses the positively oriented circle of
radius \(\rho\) about \(y\) in \(\C_I\), a circle that leaves the second
pole outside, since \(|y-\bar y|=2\beta>\rho\). Differentiating gives
\(\dd z/\dd t=Iw(t,I)\), hence \(w(t,I)\dd t=-I\dd z\), and the one-variable
residue theorem together with \(-I(2\pi I)=2\pi\) yields
\begin{equation}\label{eq:weighted-tube-inner-residue}
  \int_0^{2\pi}w(t,I)f_{u,v}(z(t))\dd t
  =-I\int_{|z-y|=\rho}\dd z\,f_{u,v}(z)
  =2\pi\Res_y^I(f_{u,v})_I .
\end{equation}

The pole at \(y\) is simple, so that
\begin{equation}\label{eq:weighted-tube-point-residue}
  \Res_y^I(f_{u,v})_I
  =(y-\bar y)^{-1}(yu+v)
  =(2I\beta)^{-1}\bigl((\alpha+I\beta)u+v\bigr)
  =\frac12u-\frac{I}{2\beta}(\alpha u+v) .
\end{equation}
Its two terms behave differently under the angular integration: the first is
independent of \(I\), while the second has \(I\) standing to the left of the
fixed quaternion \((\alpha u+v)/2\beta\) and is therefore annihilated by
\(\int_\Sphere I\darea(I)=0\). Hence
\begin{equation}\label{eq:weighted-average-point-residue}
  \int_\Sphere\Res_y^I(f_{u,v})_I\darea(I)
  =\Area(\Sphere)\cdot\tfrac12u=2\pi u .
\end{equation}

Combining the two steps, the unnormalized integral of
\eqref{eq:weighted-tube-flux} equals \(2\pi\cdot2\pi u=4\pi^2u\), and
dividing by \(4\pi^2\) gives the first equality of
\eqref{eq:weighted-tube-selector}. The second is immediate from the paired
expansion of \(f_{u,v}\), which consists of the single term
\(\DeltaSphere{\Sigma}^{-1}(qu+v)\) and so has \(u_{-1}=u\).
\end{proof}
The computation isolates \(u\) and leaves no trace of \(v\), whatever the
radius of the tube. This is a concrete reason why
Definition~\ref{def:sres-residues} takes \(u_{-1}\), and not \(v_{-1}\) or
some combination of the two, as the spherical residue.

\subsection{The divergence identity and the role of circularity}
\label{subsec:weighted-divergence}

Everything said so far has been geometric: the weighted flux was reduced to
a planar integral without any four-dimensional differential operator, and
that reduction is all the residue theorem will need. It is worth pausing,
nevertheless, on the differential identity behind the word ``flux'', for two
reasons. It connects the density \(nf|\imq q|^{-2}\) to the Cauchy--Fueter
operator, hence to the classical function theory of \cite{Sudbery}; and it shows precisely where circularity enters, as the last assertion of
Proposition~\ref{prop:weighted-divergence} makes explicit.

On \(\Hh\simeq\R^4\) write \(q=x_0+x_1\qi+x_2\qj+x_3\qk\) and
\(r=|\imq q|\), and let
\begin{equation}\label{eq:weighted-left-CF-operator}
  D_\ell:=\partial_{x_0}+\qi\partial_{x_1}+\qj\partial_{x_2}
          +\qk\partial_{x_3}
\end{equation}
be the left Cauchy--Fueter operator. Its interaction with slice regularity
is known: a slice regular function is not Fueter regular, and the failure is
measured exactly by the spherical derivative,
\begin{equation}\label{eq:weighted-perotti-identity}
  D_\ell f=-2\partial_sf ,
\end{equation}
an identity due to Perotti \cite[Proposition~6.1 and
Corollary~6.2]{PerottiHarmonicity}, where it is derived from a decomposition
of \(D_\ell\) into a radial part and the spherical Dirac operator. The same work shows that \(\partial_sf\) is harmonic on \(\Omega\) when
\(f\) is slice regular \cite[Theorem~6.3]{PerottiHarmonicity}, and that
\(f\) itself is therefore biharmonic; see also \cite{PerottiAlmansi} for
the resulting Almansi decomposition. Related manifestations of this
harmonic and biharmonic structure appear in the quaternionic Jensen
formulas of \cite{AltavillaBisiJensen,PerottiJensen}. The next proposition
is \eqref{eq:weighted-perotti-identity} in the weighted form we need,
together with its integral consequence.

\begin{proposition}[Weighted Cauchy--Fueter divergence identity]
\label{prop:weighted-divergence}
Let \(f=\mathcal I(A+\ci B)\) be slice regular on a circular open set
\(\Omega\), written in positive radial coordinates as
\begin{equation}\label{eq:weighted-slice-AB-positive-radius}
  f(x+rI)=A(x,r)+IB(x,r),
  \qquad r>0,\ I\in\Sphere .
\end{equation}
Then, away from \(\R\),
\begin{equation}\label{eq:weighted-pointwise-divergence}
  \boxed{\;
  D_\ell\left(\frac{f}{r^2}\right)=-\frac{2IA}{r^3}
  =-\frac{2}{r^3}\,I(q)\,f_s^\circ(q) .
  \;}
\end{equation}
Consequently, for every bounded open \(W\) with \(C^2\) boundary and
\(\overline W\subset\Omega\setminus\R\).
\begin{equation}\label{eq:weighted-boundary-divergence}
  \boxed{\;
  \int_{\partial W}\frac{n(q)f(q)}{r^2}\dsigma(q)
  =-2\int_W\frac{I(q)\,f_s^\circ(q)}{r^3}\dd V_4(q) ,
  \qquad I(q)=\frac{\imq q}{r} ,
  \;}
\end{equation}
and the right-hand side vanishes when \(W\) is circular. 
\end{proposition}

\begin{proof}
Write \(\underline D:=\qi\partial_{x_1}+\qj\partial_{x_2}+\qk\partial_{x_3}\)
and \(I=\underline x/r\) with
\(\underline x=x_1\qi+x_2\qj+x_3\qk\). For a function \(A=A(x_0,r)\)
independent of the angular variable the chain rule gives
\(\underline DA=IA_r\), while
\begin{equation}\label{eq:weighted-D-I}
  \underline DI
  =\frac1r\sum_{m=1}^3e_m^2-\frac1{r^3}\sum_{m=1}^3e_m\underline x\,x_m
  =-\frac3r-\frac{\underline x^2}{r^3}=-\frac2r ,
\end{equation}
using \(\sum_me_m^2=-3\) and \(\underline x^2=-r^2\), with
\((e_1,e_2,e_3)=(\qi,\qj,\qk)\). For \(B=B(x_0,r)\) the product rule and
\(\sum_me_mI(x_m/r)=I^2=-1\), the coefficients \(x_m/r\) being real, give
\begin{equation}\label{eq:weighted-D-IB}
  \underline D(IB)=(\underline DI)B+I^2B_r=-\frac{2B}{r}-B_r ,
\end{equation}
so that, adding the \(x_0\)-derivative,
\begin{equation}\label{eq:weighted-D-f-before-CR}
  D_\ell f=(A_{x_0}-B_r)+I(B_{x_0}+A_r)-\frac{2B}{r} ,
\end{equation}
which the Cauchy--Riemann equations \eqref{eq:slice-stem-CR} reduce to
\begin{equation}\label{eq:weighted-D-slice-regular}
  D_\ell f=-\frac{2B}{r}=-2\partial_sf ,
\end{equation}
by \eqref{eq:slice-spherical-stem-coordinates}. Since \(r^{-2}\) is real
valued and \(D_\ell r^{-2}=-2I/r^3\), the product rule now gives
\[
  D_\ell(r^{-2}f)=-\frac{2I}{r^3}(A+IB)-\frac{2B}{r^3}
  =-\frac{2IA}{r^3}+\frac{2B}{r^3}-\frac{2B}{r^3}
  =-\frac{2IA}{r^3} ,
\]
which is \eqref{eq:weighted-pointwise-divergence}, the second form following
from \(A=f_s^\circ\).

Formula \eqref{eq:weighted-boundary-divergence} is the real divergence
theorem applied componentwise to \(h=r^{-2}f\): with
\(D_\ell h=\sum_\mu e_\mu\partial_{x_\mu}h\) and
\(n=\sum_\mu e_\mu n_\mu\), \(e_0=1\), one has
\(\int_WD_\ell h\dd V_4=\int_{\partial W}nh\dsigma\), and
\eqref{eq:weighted-pointwise-divergence} evaluates the left-hand side. If \(W\) is circular, the four-dimensional volume splits in positive radial
coordinates as
\begin{equation}\label{eq:weighted-volume-polar-coordinates}
  \dd V_4=r^2\dd x_0\dd r\,\darea(I) ,
\end{equation}
and \(f_s^\circ=A\) does not depend on \(I\); integrating in the angular
variable first, with that factor kept on the right, annihilates the volume
integral by \(\int_\Sphere I\darea(I)=0\).
\end{proof}

Its placement here, after
Theorem~\ref{thm:weighted-slice-reduction}, is deliberate: the identity
explains the geometry of the flux, but the reduction of the boundary
integral, and with it the residue theorem, does not depend on it.  On a
noncircular set the right-hand side need not vanish; in
Theorem~\ref{thm:4d-noncircular-residue} it will become the explicit volume
term that corrects the boundary flux.

\section{The residue theorem in four dimensions}
\label{sec:residue-theorem}

The two constructions of the preceding sections now meet. On one side
Section~\ref{sec:weighted-flux} produced a boundary integral in four
dimensions and showed that it collapses onto a single full slice, where it
becomes an ordinary contour integral; on the other,
Section~\ref{sec:spherical-residues} showed that the two point residues a
conjugate pair carries on that slice add up to a quantity attached to the
sphere itself. Between them stands the classical residue theorem, which
turns the contour integral into the sum of the point residues. Composing
the three steps takes the weighted flux to a sum indexed by the real poles
and the polar spheres of the function, with nothing left of the slice on
which the computation was performed.

The result is more than a bookkeeping of coefficients. Each term on the
right is evaluated by finitely many derivatives of a function made regular
by an intrinsic polynomial factor, so that the theorem is as explicit at a
polar sphere as the classical one is at a pole.

\subsection{Statement and proof}
\label{subsec:4d-statement}

Recall that a semiregular function is induced by a meromorphic stem in the
sense of Definition~\ref{def:slice-semiregular}. Its polar set consists of
isolated real points and nonreal conjugacy spheres. If the function is
semiregular on a circular neighbourhood of the compact set \(\overline U\),
only finitely many of these polar components can meet \(\overline U\). Thus
the finiteness appearing in the next theorem is a consequence of the
hypotheses.

\begin{theorem}[Four-dimensional residue theorem with differential evaluation]
\label{thm:4d-residue}
Let \(U=\Omega_D\) be an admissible circular open set, and let
\(f\) be semiregular on a circular open neighbourhood
\(\Omega\) of \(\overline U\). Assume that no polar component of \(f\)
meets \(\partial U\). Denote the real poles contained in \(U\) by
\[
  a_1,\ldots,a_\ell
\]
and the nonreal polar spheres contained in \(U\) by
\begin{equation}\label{eq:4d-polar-spheres}
  \Sigma_j=\alpha_j+\Sphere\beta_j,
  \qquad \beta_j>0,
  \qquad 1\leq j\leq S.
\end{equation}
The two lists are finite, and either one is allowed to be empty.

For every real pole, choose any integer \(L_i\geq1\) such that
\begin{equation}\label{eq:4d-real-regularizations}
  h_i(q):=(q-a_i)^{L_i}f(q)
\end{equation}
extends slice regularly across \(a_i\). For every nonreal polar sphere,
choose any integer \(M_j\geq1\) such that
\begin{equation}\label{eq:4d-spherical-regularizations}
  g_j(q):=\DeltaSphere{\Sigma_j}(q)^{M_j}f(q)
\end{equation}
extends slice regularly across the whole sphere \(\Sigma_j\) (such exponents
exist by Proposition~\ref{prop:slice-regularization} and need not be minimal).
Finally, choose an arbitrary representative
\[
  y_j=\alpha_j+I_j\beta_j\in\Sigma_j,
  \qquad I_j\in\Sphere,
\]
and put \(\bar y_j=\alpha_j-I_j\beta_j\).

Then the weighted boundary integral is absolutely convergent and
\begin{equation}\label{eq:4d-residue-identity}
  \boxed{\;
  \Flux{U}{f}
  :=\frac{1}{4\pi^2}
    \int_{\partial U\setminus\R}
      \frac{n(q)f(q)}{|\imq q|^2}\,\dsigma(q)
  =\sum_{i=1}^{\ell}\Res_{a_i}f
    +\sum_{j=1}^{S}\SRes_{\Sigma_j}f.
  \;}
\end{equation}
Here \(n\) is the outward Euclidean unit normal, regarded as a quaternion
and multiplied on the left of \(f\). Empty sums in
\eqref{eq:4d-residue-identity} are understood to be zero.
The local residues in this formula admit the evaluations
\begin{align}
  \Res_{a_i}f
  &=\frac{(\partial_c^{L_i-1}h_i)(a_i)}{(L_i-1)!},
  \label{eq:4d-real-differential}\\
  \SRes_{\Sigma_j}f
  &=(g_j)_{I_j}
    [y_j^{(M_j)},\bar y_j^{(M_j)}]
  \notag\\
  &=\frac{
      [\partial_s(\partial_c\partial_s)^{M_j-1}g_j](y_j)
    }{(M_j-1)!}.
  \label{eq:4d-spherical-differential}
\end{align}
Consequently, the residue theorem can be written entirely in differential
form as
\begin{equation}\label{eq:4d-residue-differential-form}
  \boxed{
  \Flux{U}{f}
  =\sum_{i=1}^{\ell}
      \frac{(\partial_c^{L_i-1}h_i)(a_i)}{(L_i-1)!}
    +\sum_{j=1}^{S}
      \frac{
        [\partial_s(\partial_c\partial_s)^{M_j-1}g_j](y_j)
      }{(M_j-1)!}.
}
\end{equation}
The right-hand side is independent of every regularizing exponent
\(L_i,M_j\), every representative \(y_j\), and every slice used in its
calculation. All derivatives are evaluated on the indicated regular
extensions. When \(M_j=1\), the last numerator in
\eqref{eq:4d-spherical-differential} is simply \(\partial_sg_j(y_j)\).
\end{theorem}

\begin{proof}
We give the reduction in detail, including the orientation of all boundary
components.

\smallskip
\noindent\emph{Step 1: finiteness and regularity near the boundary.}
Write \(\Omega=\Omega_E\), and let \(F\) be the meromorphic stem inducing
\(f\). We call \emph{stem pole} a point of \(E\) at which at least one of
the four scalar coordinates of \(F\) has a pole; by
Definition~\ref{def:slice-semiregular} the set \(P_F\) of stem poles is
locally finite in \(E\), and by
Subsection~\ref{subsec:slice-semiregular} its circularization is the polar
set of \(f\), the real stem poles giving the points \(a_i\) and the
non-real conjugate pairs the spheres \(\Sigma_j\).

At profile level, \(\overline D\) is a compact subset of \(E\). Indeed, if
\(z=x+\ci t\in\overline D\), then \(x+It\in\overline U\) for any fixed
\(I\in\Sphere\); the inclusion \(\overline U\subset\Omega_E\) and the
full-section identity \eqref{eq:slice-full-section} imply \(z\in E\). A
locally finite set meets a compact set in finitely many points, since
otherwise compactness would produce an accumulation point; so
\(\overline D\) contains only finitely many stem poles, which is the
finiteness asserted in the statement.

Because no polar component meets \(\partial U\), no stem pole lies on
\(\partial D\). Around each point of this compact boundary one can choose a
small disk avoiding \(P_F\); a finite subcover produces a neighbourhood of
\(\partial D\) on which the stem is holomorphic, so that \(f\) is slice
regular on a circular neighbourhood of \(\partial U\) and its boundary
trace is a continuous slice function. The boundary measure theorem,
Theorem~\ref{thm:weighted-disintegration}, applied to \(G=nf\), now gives
\[
  \int_{\partial U\setminus\R}
  \frac{|n(q)f(q)|}{|\imq q|^2}\,\dsigma(q)<\infty ,
\]
since \(|n|=1\). This is the claimed absolute convergence, and it also
shows that the real boundary points require no principal-value
prescription.

\smallskip
\noindent\emph{Step 2: all possible singularities on one full slice.}
Fix \(I\in\Sphere\), and identify the planar profile \(D\) with the full
slice section
\[
  D_I=U\cap\C_I
  =\{x+It:x+\ci t\in D\}.
\]
By Proposition~\ref{prop:slice-meromorphic-restriction}, the restriction
\(f_I\) is a meromorphic map with values in the complex vector space
\((\Hh,L_I)\). Its possible singular points in \(D_I\) form the finite set
\begin{equation}\label{eq:4d-slice-candidate-poles}
   \mathcal Q_I
  :=\{a_1,\ldots,a_\ell\}
    \cup
    \{p_j^+,p_j^-:1\leq j\leq S\},
  \qquad
  p_j^\pm:=\alpha_j\pm I\beta_j.
\end{equation}
The adjective ``possible'' is important. At an exceptional point of a polar
sphere, cancellation under slice evaluation can lower the pole order or make
the singularity removable. We keep that point in \(\mathcal Q_I\); its
residue will then simply be zero.

If \(D\) is a product domain, its upper and lower components are disjoint.
Nevertheless, \(D_I\) is the \emph{full} slice section, so it contains both
points \(p_j^+\) and \(p_j^-\).

\smallskip
\noindent\emph{Step 3: excision of small disks.}
For every \(c\in\mathcal Q_I\), choose a radius \(\varepsilon_c>0\) such
that the closed planar disks
\[
  \overline{B_I(c,\varepsilon_c)}
  :=\{z\in\C_I:|z-c|\leq\varepsilon_c\}
\]
are pairwise disjoint, contained in \(D_I\), and contain no point of
\(\mathcal Q_I\) other than their centres. This is possible because
\(\mathcal Q_I\) is finite and lies in the open set \(D_I\). Let
\begin{equation}\label{eq:4d-excised-slice}
  D_{I,\varepsilon}
  :=D_I\setminus
    \bigcup_{c\in\mathcal Q_I}
      \overline{B_I(c,\varepsilon_c)},
\end{equation}
and let
\(\gamma_c=\partial B_I(c,\varepsilon_c)\) be oriented
counterclockwise.

Every component of \(\partial D_I\) is already oriented positively relative
to \(D_I\): an exterior boundary is counterclockwise, while the boundary of
a pre-existing hole is clockwise. The new circles \(\gamma_c\), however,
bound holes of the excised set. They must therefore be traversed clockwise
when regarded as components of \(\partial D_{I,\varepsilon}\). In the
notation of oriented one-chains,
\begin{equation}\label{eq:4d-oriented-excised-boundary}
  \partial D_{I,\varepsilon}
  =\partial D_I-
   \sum_{c\in\mathcal Q_I}\gamma_c.
\end{equation}
The previous formula remains valid when \(D_I\) is disconnected and when its
components have holes.

After filling in any removable candidate singularity, \(f_I\) is holomorphic
on a neighbourhood of \(\overline{D_{I,\varepsilon}}\). Choosing
\(J\in\Sphere\) orthogonal to \(I\), the splitting
\eqref{eq:slice-meromorphic-splitting} writes \(f_I=h_1+h_2J\) with
\(h_1,h_2\) holomorphic and \(\C_I\) valued there; applying Cauchy's theorem
to \(h_1\) and to \(h_2\) on every connected component and recombining the
two coordinates gives
\begin{equation}  \label{eq:4d-Cauchy-excision}
  0
  =\int_{\partial D_{I,\varepsilon}}\dd z\,f_I(z)
  =\int_{\partial D_I}\dd z\,f_I(z)
    -\sum_{c\in\mathcal Q_I}
      \int_{\gamma_c}\dd z\,f_I(z).
\end{equation}
This is precisely the vector-valued excision argument underlying
Theorem~\ref{thm:complex-residue-flux}.

By the point-residue normalization
\eqref{eq:slice-point-residue-integral},
\[
  (2\pi I)^{-1}
  \int_{\gamma_c}\dd z\,f_I(z)
  =\Res_c^I f_I.
\]
The equality also holds at a removable candidate point, since both sides are
zero. Multiplying \eqref{eq:4d-Cauchy-excision} on the left by
\((2\pi I)^{-1}\), and grouping the candidate points according to
\eqref{eq:4d-slice-candidate-poles}, yields
\begin{equation}
  \label{eq:4d-slice-residue-sum}
  (2\pi I)^{-1}
  \int_{\partial D_I}\dd z\,f_I(z)=\sum_{i=1}^{\ell}\Res_{a_i}^I f_I
+\sum_{j=1}^{S}
    \left(
      \Res_{p_j^+}^I f_I+
      \Res_{p_j^-}^I f_I
    \right).
\end{equation}
All complex scalars in this calculation act from the left.

\smallskip
\noindent\emph{Step 4: passage from point residues to intrinsic local data.}
At a real pole, the point residue is independent of the slice, and
Proposition~\ref{prop:sres-real-differential} gives
\begin{equation}\label{eq:4d-real-identification}
  \Res_{a_i}^I f_I
  =\Res_{a_i}f
  =\frac{(\partial_c^{L_i-1}h_i)(a_i)}{(L_i-1)!}.
\end{equation}
At a nonreal polar sphere,
Proposition~\ref{prop:sres-paired-point-residues} gives
\begin{equation}\label{eq:4d-spherical-identification}
  \Res_{p_j^+}^I f_I+
  \Res_{p_j^-}^I f_I
  =\SRes_{\Sigma_j}f.
\end{equation}
Thus \eqref{eq:4d-slice-residue-sum} becomes
\begin{equation}\label{eq:4d-grouped-residue-sum}
  (2\pi I)^{-1}
  \int_{\partial D_I}\dd z\,f_I(z)
  =\sum_{i=1}^{\ell}\Res_{a_i}f
   +\sum_{j=1}^{S}\SRes_{\Sigma_j}f.
\end{equation}
Notice that each nonreal contribution is the \emph{sum} of the two point
residues on the full slice, not either one of them separately.

\smallskip
\noindent\emph{Step 5: return to the four-dimensional boundary.}
The reduction theorem for the weighted flux,
Theorem~\ref{thm:weighted-slice-reduction}, states that
\begin{equation}\label{eq:4d-use-flux-reduction}
  \Flux{U}{f}
  =(2\pi I)^{-1}
   \int_{\partial D_I}\dd z\,f_I(z).
\end{equation}
Combining \eqref{eq:4d-grouped-residue-sum} and
\eqref{eq:4d-use-flux-reduction} proves
\eqref{eq:4d-residue-identity}. This also explains the normalization:
the unnormalized four-dimensional boundary integral equals
\(4\pi^2\) times the residue sum.

If a polar sphere lies in a hole of \(U\), its two slice points lie outside
\(D_I\). No disks are excised around them, and they make no contribution.

\smallskip
\noindent\emph{Step 6: differential evaluation and independence of choices.}
Formula \eqref{eq:4d-real-differential} is
Proposition~\ref{prop:sres-real-differential}, while for the non-real
spheres Theorem~\ref{thm:sres-differential} gives
\[
  \SRes_{\Sigma_j}f
  =(g_j)_{I_j}[y_j^{(M_j)},\bar y_j^{(M_j)}]
  =\frac{[\partial_s(\partial_c\partial_s)^{M_j-1}g_j](y_j)}{(M_j-1)!} ,
\]
which is \eqref{eq:4d-spherical-differential}; substituting into
\eqref{eq:4d-residue-identity} proves
\eqref{eq:4d-residue-differential-form}.

The independence of the right-hand side from every auxiliary choice is
Proposition~\ref{prop:sres-three-independences} for the spherical terms,
and for the real ones the observation that every admissible \(L_i\) and
every slice extract the same coefficient of \((q-a_i)^{-1}\); nonminimal
regularizations, which may vanish identically on \(\Sigma_j\), are covered.
\end{proof}

\begin{remark}[The degenerate cases are not exceptions]
\label{rem:4d-degenerate-cases}
Nothing in the theorem needs to be weakened on the domains and for the
functions that the slice algebra makes available. Product domains and
domains disjoint from \(\R\) are covered, as are functions with
\(f\not\equiv0\) and \(f^s\equiv0\), polar spheres carrying an exceptional
point at which the restriction is removable, and regularizations with
\(g_j(y_j)=0\).

The reason is visible in the proof. It used restriction to a slice,
componentwise contour integration, addition of point residues, and
multiplication by the intrinsic polynomials \((q-a_i)^{L_i}\) and
\(\DeltaSphere{\Sigma_j}^{M_j}\); at no stage did it divide by a value of
\(f\), of \(h_i\), of \(g_j\) or of a symmetrization, and at no stage did it
need the two point residues of a conjugate pair to be separately non-zero.
The degenerate configurations are therefore not exceptions to be treated
apart, but cases in which every step reads exactly as before.
\end{remark}

\begin{remark}[What is being integrated]
\label{rem:4d-residue-scope}
The left-hand side of \eqref{eq:4d-residue-identity} is an integral over the
three-dimensional boundary of a set in \(\R^4\), neither a volume integral
nor an unweighted Cauchy--Fueter one. It is not a homological invariant of
an arbitrary three-cycle either: the flux of a function as simple as \(1\)
is generally non-zero on a non-circular boundary, and
Theorem~\ref{thm:4d-noncircular-residue} below measures the discrepancy
exactly. The \(C^2\) hypotheses of
Definition~\ref{def:weighted-admissible-profile} are convenient sufficient
conditions for the boundary geometry and the measure formula; no claim of
optimal regularity is made. Finally, a zero residue does not imply
removability, as Remark~\ref{rem:sres-zero-does-not-remove} records:
different residues may cancel in the global sum, and a single spherical
residue can vanish at a genuine polar sphere.
\end{remark}

\subsection{Local fluxes and domain independence}
\label{subsec:4d-local-fluxes}

The main theorem has a useful local interpretation. A small admissible
circular boundary surrounding exactly one polar component extracts precisely
the residue attached to that component. 

\begin{corollary}[Local extraction at a sphere or a real point]
\label{cor:4d-local-extraction}
Let \(f\) be semiregular on a circular open set \(\Omega\).

\begin{enumerate}[label=\textup{(\roman*)}]
  \item Let \(\Sigma=\alpha+\Sphere\beta\), with \(\beta>0\), be a polar
        sphere. Suppose that \(T_\Sigma\) is an admissible circular open set
        such that
        \[
          \Sigma\subset T_\Sigma,
          \qquad
          \overline{T_\Sigma}\subset\Omega,
        \]
        \(T_\Sigma\) contains no other polar component of \(f\), and no
        polar component meets \(\partial T_\Sigma\). If \(M\geq1\)
        regularizes \(\Sigma\), set
        \(g_M=\DeltaSphere{\Sigma}^Mf\). Then, for every
        \(y\in\Sigma\),
        \begin{equation}\label{eq:4d-local-spherical-extraction}
          \boxed{\;
          \Flux{T_\Sigma}{f}
          =\SRes_\Sigma f
          =\frac{
             [\partial_s(\partial_c\partial_s)^{M-1}g_M](y)
           }{(M-1)!}.
          \;}
        \end{equation}
        In particular, for every sufficiently small
        \(0<\rho<\beta\), one may take the circular tube
        \(T_{\alpha,\beta,\rho}\) from
        Definition~\ref{def:weighted-circular-tube}.

  \item Let \(a\in\R\) be a real pole, and let \(V_a\) be an admissible
        circular open set whose closure lies in \(\Omega\), which contains
        \(a\) and no other polar component, and such that no polar component
        meets \(\partial V_a\). Then
        \begin{equation}\label{eq:4d-local-real-extraction}
          \boxed{\;
          \Flux{V_a}{f}
          =\Res_a f
          =\frac{(\partial_c^{L-1}h_L)(a)}{(L-1)!},
          \qquad
          h_L=(q-a)^Lf,
          \;}
        \end{equation}
        for every admissible regularizing exponent \(L\geq1\).
        In particular, one may take
        \(V_a=B(a,\varepsilon)\) for every sufficiently small
        \(\varepsilon>0\).
\end{enumerate}
\end{corollary}

\begin{proof}
We first justify the existence of the standard local sets. At stem level, the
nonreal sphere \(\Sigma\) corresponds to two isolated conjugate poles
\(p=\alpha+\ci\beta\) and \(\bar p\). Choose \(\rho>0\) so small that the
closed disks of radius \(\rho\) around these two points lie in the profile of
\(\Omega\), are disjoint, and contain no other stem pole. Requiring also
\(\rho<\beta\), their circularization is precisely
\(T_{\alpha,\beta,\rho}\). It is disjoint from the real axis and has
smooth boundary. Local finiteness of the stem poles guarantees that all these
requirements can be satisfied simultaneously. The same argument at a real
stem pole gives a sufficiently small real-centred ball.

For an arbitrary \(T_\Sigma\) satisfying part~\textup{(i)}, the only polar
component counted by Theorem~\ref{thm:4d-residue} is \(\Sigma\). The residue
sum therefore consists of the single term \(\SRes_\Sigma f\), and the
differential equality follows from
Theorem~\ref{thm:sres-differential}. The real case is identical, using
Proposition~\ref{prop:sres-real-differential}. The independence of the chosen
small radius is already contained in these equalities.
\end{proof}

The next statement makes the local-to-global nature of the theorem explicit.
The equality is additivity of residue contributions; it is not obtained by
geometrically decomposing \(\partial U\) into the smaller boundaries.

\begin{corollary}[Local-to-global decomposition]
\label{cor:4d-local-to-global}
Under the hypotheses of Theorem~\ref{thm:4d-residue}, choose pairwise
disjoint admissible circular neighbourhoods
\[
  V_i\Subset U
  \quad(1\leq i\leq\ell),
  \qquad
  T_j\Subset U
  \quad(1\leq j\leq S),
\]
so that \(V_i\) contains only \(a_i\) and \(T_j\) contains only
\(\Sigma_j\). Then
\begin{equation}\label{eq:4d-local-to-global}
  \boxed{\;
  \Flux{U}{f}
  =\sum_{i=1}^{\ell}\Flux{V_i}{f}
   +\sum_{j=1}^{S}\Flux{T_j}{f}.
  \;}
\end{equation}
The sets \(V_i\) may be chosen as real-centred balls and the sets \(T_j\)
as circular tubes.
\end{corollary}

\begin{proof}
The polar components form a finite family of pairwise disjoint compact sets
contained in the open set \(U\). Hence the balls and tubes can be chosen
small enough that their closures lie in \(U\), avoid every other polar
component, and are pairwise disjoint. Corollary~\ref{cor:4d-local-extraction}
gives
\[
  \Flux{V_i}{f}=\Res_{a_i}f,
  \qquad
  \Flux{T_j}{f}=\SRes_{\Sigma_j}f.
\]
Summing these identities and applying
Theorem~\ref{thm:4d-residue} to \(U\) proves
\eqref{eq:4d-local-to-global}.
\end{proof}

\begin{corollary}[Regular functions and domain independence]
\label{cor:4d-domain-independence}
If \(f\) is slice regular on a circular open neighbourhood of
\(\overline U\), then
\begin{equation}\label{eq:4d-regular-zero-flux}
  \Flux{U}{f}=0.
\end{equation}

More generally, let \(f\) be semiregular on a circular open set
\(\Omega\), and let \(U_1,U_2\) be admissible circular open sets with
\[
  \overline{U_1}\cup\overline{U_2}\subset\Omega.
\]
Assume that no polar component meets \(\partial U_1\) or
\(\partial U_2\). If \(U_1\) and \(U_2\) contain exactly the same real
poles and nonreal polar spheres of \(f\), then
\begin{equation}\label{eq:4d-domain-independence}
  \boxed{\;\Flux{U_1}{f}=\Flux{U_2}{f}.\;}
\end{equation}
No inclusion or deformation relation between \(U_1\) and \(U_2\) is
required.
\end{corollary}

\begin{proof}
If \(f\) is slice regular near \(\overline U\), the two residue lists in
Theorem~\ref{thm:4d-residue} are empty. The empty sum is zero, proving
\eqref{eq:4d-regular-zero-flux}.

In the second assertion, apply the theorem separately to \(U_1\) and
\(U_2\). By hypothesis, the two right-hand sides contain the same local
residues, each with coefficient one. They are therefore equal. A polar
component lying in a hole of one of the domains is not contained in that
domain; the assumption requires the containment lists, rather than merely a
visual notion of being surrounded, to agree.
\end{proof}

\subsection{A corrected formula on noncircular domains}
\label{subsec:4d-noncircular}

The term noncircular in this subsection refers to the domain of integration,
not to the domain on which the function is defined.  Throughout, \(f\) is
the restriction of a semiregular slice function on a circular open set
\(\Omega\); only the integration set \(W\Subset\Omega\) is allowed to be
noncircular.  This is the same ambient-domain convention used in Perotti's
local Cauchy formula, where the function is slice regular on an axially
symmetric open set while the integration domain need not be axially
symmetric; see \cite[Theorem~4]{PerottiLocalCauchy}.

Circularity is precisely what makes the volume term in
Proposition~\ref{prop:weighted-divergence} vanish.  Away from the real axis,
the same identity gives a residue formula on an arbitrary smooth integration
domain, provided that the resulting volume correction is retained.  The
polar spheres can then be excised by the local circular tubes used above.

\begin{theorem}[Residue formula with a Cauchy--Fueter volume correction]
\label{thm:4d-noncircular-residue}
Let \(f\) be semiregular on a circular open set \(\Omega\), and let \(W\) be
a bounded open set with \(C^2\) boundary such that
\[
  \overline W\subset\Omega\setminus\R
\]
and no polar sphere of \(f\) meets \(\partial W\).  Let
\(\Sigma_1,\ldots,\Sigma_S\) be the polar spheres that meet \(W\).  Under the
preceding hypothesis, each of them is entirely contained in \(W\).
Choose pairwise disjoint circular tubes \(T_j\) such that
\[
  \Sigma_j\subset T_j,
  \qquad
  \overline{T_j}\subset W,
\]
and \(T_j\) contains no other polar sphere, and set
\begin{equation}\label{eq:4d-noncircular-excised-domain}
  W_{\mathrm{exc}}:=W\setminus\bigcup_{j=1}^{S}\overline{T_j}.
\end{equation}
Then
\begin{equation}\label{eq:4d-noncircular-residue}
  \boxed{
  \frac{1}{4\pi^2}
  \int_{\partial W}
    \frac{n(q)f(q)}{|\imq q|^2}\,\dsigma(q)
  +\frac{1}{2\pi^2}
  \int_{W_{\mathrm{exc}}}
    \frac{I(q)f_s^\circ(q)}{|\imq q|^3}\,\dd V_4(q)
  =\sum_{j=1}^{S}\SRes_{\Sigma_j}f.\;}
\end{equation}
Here \(I(q)=\imq(q)/|\imq q|\).  Both integrals are absolutely convergent,
and their displayed combination is independent of the chosen tubes.

If \(W\) is circular, the tubes may be chosen circular as above and the
volume term vanishes.  Formula \eqref{eq:4d-noncircular-residue} then reduces
to the four-dimensional residue theorem.  For a general \(W\), the volume
integral is the precise correction to the uncorrected boundary flux.
\end{theorem}

\begin{proof}
First note that a polar sphere \(\Sigma\) is connected.  If
\(\Sigma\cap\partial W=\varnothing\), then
\[
  \Sigma=(\Sigma\cap W)
  \mathbin{\dot\cup}
  \bigl(\Sigma\cap(\R^4\setminus\overline W)\bigr)
\]
is a decomposition into relatively open sets.  Hence one of the two sets is
empty: either \(\Sigma\subset W\), or
\(\Sigma\cap\overline W=\varnothing\).  This proves, in particular, the
containment assertion in the statement.

The polar spheres contained in \(W\) form a finite family.  Since none meets
\(\partial W\), a sufficiently thin circular tube around each of them has
closure in \(W\), and the tubes can be chosen pairwise disjoint.  The
function \(f\) is slice regular on a neighbourhood of
\(\overline{W_{\mathrm{exc}}}\).  Applying
Proposition~\ref{prop:weighted-divergence} to \(W_{\mathrm{exc}}\) and
keeping track of the inward orientation of each new boundary component gives
\begin{equation}\label{eq:4d-noncircular-divergence-step}
  \int_{\partial W}
      \frac{n(q)f(q)}{|\imq q|^2}\,\dsigma(q)
   -\sum_{j=1}^{S}
     \int_{\partial T_j}
      \frac{n(q)f(q)}{|\imq q|^2}\,\dsigma(q)
   =-2\int_{W_{\mathrm{exc}}}
      \frac{I(q)f_s^\circ(q)}{|\imq q|^3}\,\dd V_4(q).
\end{equation}
Each tube integral equals \(4\pi^2\SRes_{\Sigma_j}f\) by
Corollary~\ref{cor:4d-local-extraction}.  Dividing
\eqref{eq:4d-noncircular-divergence-step} by \(4\pi^2\) and rearranging
proves \eqref{eq:4d-noncircular-residue}.  Absolute convergence follows
because the excised closure is compact and has positive distance from the
real axis and from the polar set.  Independence of the tubes follows at
once because the resulting expression equals the fixed sum of spherical
residues.

If \(W\) is circular, then \(W_{\mathrm{exc}}\) is circular.  The angular
cancellation in the last part of Proposition~\ref{prop:weighted-divergence}
makes the volume term zero.
\end{proof}

\begin{remark}[Why a polar sphere may not cross the boundary]
\label{rem:4d-noncircular-cut-sphere}
The hypothesis \(\Sigma\cap\partial W=\varnothing\) is not merely a
convenient way to list the poles.  If a polar sphere crosses \(\partial W\),
then \(f\) is singular on the integration surface and the ordinary boundary
integral in \eqref{eq:4d-noncircular-residue} is not defined.  Moreover,
excising only the portion of the sphere lying in \(W\) creates caps whose
limiting contribution depends on the local geometry of the cut.  One could
seek a principal-value or indentation formula after prescribing a
regularization, but there is then no unqualified coefficient-one
contribution of the form \(\SRes_\Sigma f\).  The present theorem therefore
treats exactly the noncircular domains whose boundaries avoid the polar set.
\end{remark}

\begin{remark}[No intrinsic theory on \(W\) is being assumed]
\label{rem:4d-noncircular-regularity-notion}
The theorem makes no assertion about a function given only on the
noncircular set \(W\).  It uses the standard theory on the circular ambient
set \(\Omega\), and then integrates its restriction over \(W\).  There are
more local notions of slice regularity on non-axially symmetric domains,
including the slice-topological approach of
\cite{DouJinRenSabadini} and the theory beyond slice domains developed in
\cite{GhiloniStoppatoBeyond}.  Extending the present residue formula to those
settings would require, in particular, an intrinsic replacement for the
complete spherical Laurent data used in the excision argument; it is not
needed for, and is not claimed by, the result above.
\end{remark}

\begin{remark}[Why the real axis is excluded]
\label{rem:4d-noncircular-real-axis}
The hypothesis \(\overline W\cap\R=\varnothing\) is substantive.  The defect density contains
\(|\imq q|^{-3}\), and angular cancellation is no longer available on a
general noncircular domain approaching the real axis.  Formula
\eqref{eq:4d-noncircular-residue} therefore does not assert a principal-value
extension across \(\R\).  Circular domains meeting \(\R\) remain covered by
Theorem~\ref{thm:4d-residue}, where the boundary measure theorem proves
absolute convergence of the boundary integral directly.
\end{remark}

\subsection{Flux moments and reconstruction of principal parts}
\label{subsec:4d-principal-parts}

The residue is only the first member of a finite system of flux moments.
Multiplication by powers of the local characteristic polynomial shifts each
principal coefficient into residue position.  A second intrinsic weight
separates the two coefficient families carried by a nonreal sphere.  Thus
the same boundary functional recovers the complete principal part.

\begin{theorem}[Recovery of the complete principal part]
\label{thm:4d-principal-part-recovery}
Let \(f\) be semiregular on a circular open set.

\begin{enumerate}[label=\textup{(\roman*)}]
  \item Let \(a\in\R\) be a real pole, let \(V_a\) be a local integration
        set as in Corollary~\ref{cor:4d-local-extraction}, and choose
        \(L\geq1\) such that
        \[
          h_L(q):=(q-a)^Lf(q)
        \]
        extends slice regularly through \(a\).  Write
        \begin{equation}\label{eq:4d-real-local-laurent}
          f(q)=\sum_{n=-L}^{\infty}(q-a)^nc_n
        \end{equation}
        near \(a\), allowing initial zero coefficients when \(L\) is not
        minimal.  Then, for \(0\leq k<L\),
        \begin{equation}\label{eq:4d-real-principal-recovery}
          \boxed{\;
          c_{-k-1}
          =\Flux{V_a}{(q-a)^kf}
          =\frac{(\partial_c^{\,L-k-1}h_L)(a)}
                 {(L-k-1)!}.\;}
        \end{equation}

  \item Let \(\Sigma=\alpha+\Sphere\beta\), \(\beta>0\), be a polar sphere,
        let \(T_\Sigma\) be a local integration set, and choose \(M\geq1\)
        such that
        \[
          g_M(q):=\DeltaSphere{\Sigma}(q)^Mf(q)
        \]
        extends slice regularly through \(\Sigma\).  Write the paired
        Laurent expansion
        \begin{equation}\label{eq:4d-spherical-local-laurent}
          f(q)=\sum_{n=-M}^{\infty}
          \DeltaSphere{\Sigma}(q)^n(qu_n+v_n).
        \end{equation}
        Then, for every \(y\in\Sigma\) and \(0\leq k<M\),
        \begin{align}
          u_{-k-1}
          &=\Flux{T_\Sigma}{\DeltaSphere{\Sigma}^{\,k}f}
          =\frac{
             [\partial_s(\partial_c\partial_s)^{M-k-1}g_M](y)}
             {(M-k-1)!},
          \label{eq:4d-u-principal-recovery}\\
          v_{-k-1}
          &=\Flux{T_\Sigma}{
              (q-2\alpha)\DeltaSphere{\Sigma}^{\,k}f}
          =\frac{
             [(\partial_c\partial_s)^{M-k-1}g_M](y)}
             {(M-k-1)!}
            -y\,u_{-k-1}.
          \label{eq:4d-v-principal-recovery}
        \end{align}
\end{enumerate}
Hence \(L\) flux moments recover the real principal part, while \(2M\)
flux moments recover the paired principal part at a nonreal sphere.  The
formulas are valid for nonminimal regularizing exponents and at exceptional
points of a polar sphere.
\end{theorem}

\begin{proof}
Multiplication of \eqref{eq:4d-real-local-laurent} by \((q-a)^k\) moves
\(c_{-k-1}\) to the coefficient of \((q-a)^{-1}\).  Local extraction gives
the first equality in \eqref{eq:4d-real-principal-recovery}; the second is
the ordinary Taylor formula applied to the regular extension \(h_L\).

For the spherical expansion, multiplication by
\(\DeltaSphere{\Sigma}^{\,k}\) moves \(u_{-k-1}\) to the coefficient of
\(q\DeltaSphere{\Sigma}^{-1}\).  Its spherical residue is therefore
\(u_{-k-1}\), which proves the first equality in
\eqref{eq:4d-u-principal-recovery}.  To select the companion coefficient,
use
\begin{equation}\label{eq:4d-modified-paired-block}
  (q-2\alpha)(qu+v)
  =\DeltaSphere{\Sigma}(q)u+qv
   -(\alpha^2+\beta^2)u-2\alpha v.
\end{equation}
If \(\widetilde u_n,\widetilde v_n\) are the paired coefficients of
\((q-2\alpha)f\), comparison of the coefficients of
\(q\DeltaSphere{\Sigma}^n\) gives
\begin{equation}\label{eq:4d-coefficient-shift}
  \widetilde u_n=v_n,
  \qquad
  \widetilde v_n
  =u_{n-1}-(\alpha^2+\beta^2)u_n-2\alpha v_n.
\end{equation}
Applying this observation after the shift by
\(\DeltaSphere{\Sigma}^{\,k}\) shows that the spherical residue of
\((q-2\alpha)\DeltaSphere{\Sigma}^{\,k}f\) is \(v_{-k-1}\).  Local
extraction proves the first equality in
\eqref{eq:4d-v-principal-recovery}.

Finally, apply Corollary~\ref{cor:sres-all-coefficients} to \(g_M\) with
index \(-k-1\), so that \(M+(-k-1)=M-k-1\).  It gives the differential
expressions in \eqref{eq:4d-u-principal-recovery} and
\eqref{eq:4d-v-principal-recovery}.  None of the arguments requires the
chosen exponent to be minimal or the regularization to be non-zero at the
evaluation point.
\end{proof}

The case \(k=0\) says that the residue theorem detects \(u_{-1}\), while
the degree-one intrinsic weight \(q-2\alpha\) detects its companion
\(v_{-1}\).  The weight \(q-2\alpha\) is the one introduced for the first paired moment
\eqref{eq:complex-first-moment-definition} in
Section~\ref{sec:complex-model}, now read on the sphere rather than on a
conjugate pair.  Comparison with ordinary point residues
also explains why, at higher-order point poles, \(v_{-1}\) is not determined
by the two point residues alone.

\begin{corollary}[Comparison with ordinary point residues]
\label{cor:4d-second-paired-coefficient}
Let \(\Sigma=\alpha+\Sphere\beta\), \(\beta>0\), be a polar sphere of \(f\),
fix \(I\in\Sphere\), put \(p_\pm=\alpha\pm I\beta\), and write the Laurent
expansions on \(\C_I\) as
\begin{equation}\label{eq:4d-point-Laurent-for-v}
  f_I(z)=\sum_{n\geq-m_\pm}(z-p_\pm)^nc_n^\pm
  \qquad\text{near }p_\pm ,
\end{equation}
with \(m_\pm\geq0\), taking \(m_\pm=0\) at a removable point and
\(c_n^\pm=0\) for any absent term. Set
\(r_\pm:=c_{-1}^\pm=\Res_{p_\pm}^If_I\). Then
\begin{equation}\label{eq:4d-v-point-residues}
  \boxed{\;
  v_{-1}=-p_-r_+-p_+r_-+c_{-2}^++c_{-2}^- .
  \;}
\end{equation}
If both point singularities are at most simple then \(c_{-2}^\pm=0\), and
\(v_{-1}\) is determined by the two point residues; at higher-order poles
the coefficients of index \(-2\) cannot in general be omitted.
\end{corollary}

\begin{proof}
By \eqref{eq:4d-v-principal-recovery} with \(k=0\), the coefficient
\(v_{-1}\) is the spherical residue of \(\widetilde f:=(q-2\alpha)f\), a
function with the same polar components as \(f\), the weight being slice
preserving. Near \(p_\pm\), write
\[
  z-2\alpha=(p_\pm-2\alpha)+(z-p_\pm) ,
\]
and multiply by \eqref{eq:4d-point-Laurent-for-v}: the coefficient of
\((z-p_\pm)^{-1}\) receives the constant term \(p_\pm-2\alpha\) against
\(c_{-1}^\pm=r_\pm\), and the term \(z-p_\pm\) against \(c_{-2}^\pm\), so
that
\begin{equation}\label{eq:4d-modified-point-residue}
  \Res_{p_\pm}^I\widetilde f_I=(p_\pm-2\alpha)r_\pm+c_{-2}^\pm .
\end{equation}
Since \(p_++p_-=2\alpha\), one has \(p_+-2\alpha=-p_-\) and
\(p_--2\alpha=-p_+\); adding the two instances of
\eqref{eq:4d-modified-point-residue} and applying
Proposition~\ref{prop:sres-paired-point-residues} to \(\widetilde f\) gives
\eqref{eq:4d-v-point-residues}. The calculation uses only left
multiplication by complex slice scalars, and never commutes \(p_\pm\) past
a quaternionic Laurent coefficient.

Formula \eqref{eq:4d-v-point-residues} is \eqref{eq:cx-first-moment} read on
the slice \(\C_I\), and the identity \(\widetilde u_n=v_n\) of
\eqref{eq:4d-coefficient-shift} shows why the weight of
\eqref{eq:complex-first-moment-definition} was taken to be \(z-2\alpha\):
any other degree-one intrinsic weight would shift the paired coefficients
only up to a multiple of \(u_n\).
\end{proof}

\begin{remark}[The simple spherical-pole case]
\label{rem:4d-simple-spherical-pole}
If \(M=1\) regularizes the sphere, put \(g_1=\DeltaSphere{\Sigma} f\). The two
principal paired coefficients take the particularly transparent form
\[
  u_{-1}=\SRes_\Sigma f=\partial_sg_1(y),
  \qquad
  v_{-1}=g_1(y)-y\partial_sg_1(y).
\]
Thus the restriction of \(g_1\) to the sphere is the affine slice function
\(q\mapsto qu_{-1}+v_{-1}\). For \(M>1\), or for a nonminimal
regularizing exponent, the value of \(g_M\) on \(\Sigma\) may vanish and
the higher alternating derivatives in
\eqref{eq:4d-spherical-differential} and
\eqref{eq:4d-v-principal-recovery} are essential.
\end{remark}

\subsection{Global residue duality}
\label{subsec:4d-residue-duality}

The preceding subsection reconstructs a principal part by weights adapted
to one polar component at a time.  We now solve the corresponding global
problem: a single finite-dimensional space of real polynomial weights
separates all principal coefficients at all the polar components in \(U\).
The result identifies the space of formal principal parts with the full
real dual, with values in \(\Hh\), of a quotient polynomial algebra.  Thus it
contains more information than the unweighted residue sum: the latter is
only the value of the resulting functional at the constant polynomial.

Assume throughout this subsection the hypotheses of
Theorem~\ref{thm:4d-residue}.  Retain the notation \(a_i\) for the real
poles, \(\Sigma_j=\alpha_j+\Sphere\beta_j\) for the nonreal polar spheres,
and \(L_i,M_j\) for admissible order bounds.  The formulas below involve
many characteristic polynomials at once, so in this subsection alone we
abbreviate
\[
  \Delta_j(X):=\DeltaSphere{\Sigma_j}(X)
  =X^2-2\alpha_jX+\alpha_j^2+\beta_j^2 .
\]

\begin{definition}[The weight algebra and the principal-part space]
\label{def:4d-weight-algebra}
Define the real pole polynomial and its quotient algebra by
\begin{equation}\label{eq:4d-global-pole-polynomial}
  Q(X):=
  \prod_{i=1}^{\ell}(X-a_i)^{L_i}
  \prod_{j=1}^{S}\Delta_j(X)^{M_j},
  \qquad
  \mathcal A_Q:=\R[X]/(Q).
\end{equation}
Therefore, two polynomials define the same class precisely when their
difference is divisible by \(Q\).  Every class has a unique representative
of degree strictly smaller than \(\deg Q\); in particular,
\begin{equation}\label{eq:4d-weight-algebra-dimension}
  \dim_{\R}\mathcal A_Q=\deg Q
  =\sum_{i=1}^{\ell}L_i+2\sum_{j=1}^{S}M_j.
\end{equation}

Let \(\mathcal P_Q\) be the right quaternionic vector space of formal
principal parts
\begin{equation}\label{eq:4d-formal-principal-parts}
  \mathfrak p=
  \left(
    (c_{i,-m})_{\substack{1\leq i\leq\ell\\1\leq m\leq L_i}},
    (u_{j,-m},v_{j,-m})_{\substack{1\leq j\leq S\\1\leq m\leq M_j}}
  \right).
\end{equation}
Associate with \(\mathfrak p\) the rational slice function
\begin{equation}\label{eq:4d-principal-part-model}
  R_{\mathfrak p}(q):=
  \sum_{i=1}^{\ell}\sum_{m=1}^{L_i}
    (q-a_i)^{-m}c_{i,-m}
  +\sum_{j=1}^{S}\sum_{m=1}^{M_j}
    \Delta_j(q)^{-m}(qu_{j,-m}+v_{j,-m}).
\end{equation}
It has exactly the prescribed formal principal part at each listed
component and no other singularity.
\end{definition}

\begin{remark}[What \(\mathcal A_Q\) records]
\label{rem:4d-weight-algebra-meaning}
The quotient retains exactly the finite jets of a polynomial that can
interact with the prescribed principal parts.  For example, if
\(Q=(X-a)^2\), then \([P]\) is determined by the two real numbers
\(P(a)\) and \(P'(a)\).  If \(Q=\DeltaSphere{\Sigma}\) for one nonreal
sphere, then each class has a unique affine representative and
\[
  \mathcal A_Q
  =\operatorname{span}_{\R}\{[1],[X-2\alpha]\}.
\]
These two directions match the two quaternionic coefficients \(u_{-1}\)
and \(v_{-1}\) in a simple spherical principal part.  Higher powers of the
local factors retain the corresponding higher jets.
\end{remark}

The factors in \(Q\) are pairwise coprime.  The Chinese remainder theorem
therefore gives
\begin{equation}\label{eq:4d-crt-decomposition}
  \mathcal A_Q\simeq
  \bigoplus_{i=1}^{\ell}
    \R[X]\big/\bigl((X-a_i)^{L_i}\bigr)
  \oplus
  \bigoplus_{j=1}^{S}
    \R[X]\big/\bigl(\Delta_j^{M_j}\bigr);
\end{equation}
see, for example, \cite[Proposition~1.10]{AtiyahMacdonald}.  We record the
particular Hermite--Chinese-remainder selectors that will be used below.

\begin{proposition}[Global polynomial selectors]
\label{prop:4d-hermite-crt-selectors}
For \(1\leq m\leq L_i\), there is a unique class
\([S_{i,m}]\in\mathcal A_Q\) satisfying
\begin{equation}\label{eq:4d-real-selector-congruences}
  S_{i,m}(X)\equiv(X-a_i)^{m-1}
       \pmod{(X-a_i)^{L_i}},
\end{equation}
and congruent to zero modulo every other primary factor of \(Q\).
For \(1\leq m\leq M_j\), there are unique classes
\([\Lambda_{j,m}],[\mathrm M_{j,m}]\in\mathcal A_Q\), zero modulo all the
other primary factors, such that
\begin{equation}\label{eq:4d-spherical-selector-congruences}
  \Lambda_{j,m}(X)\equiv\Delta_j(X)^{m-1},
  \qquad
  \mathrm M_{j,m}(X)\equiv
    (X-2\alpha_j)\Delta_j(X)^{m-1}
  \pmod{\Delta_j(X)^{M_j}}.
\end{equation}
The classes of all the \(S_{i,m},\Lambda_{j,m},\mathrm M_{j,m}\) form a real
basis of \(\mathcal A_Q\).
\end{proposition}

\begin{proof}
If \(Q=1\), then both \(\mathcal A_Q\) and the asserted family of
selectors are zero, so there is nothing to prove. Assume henceforth that
\(\deg Q>0\), and denote the primary factors of \(Q\) by
\[
  Q_i^{\mathrm r}(X):=(X-a_i)^{L_i},
  \qquad
  Q_j^{\mathrm s}(X):=\Delta_j(X)^{M_j}.
\]
For any one of these factors, denoted generically by \(Q_\nu\), put
\[
  B_\nu(X):=\frac{Q(X)}{Q_\nu(X)}.
\]
Since \(B_\nu\) and \(Q_\nu\) are coprime, B\'ezout's identity gives
polynomials \(C_\nu,D_\nu\in\R[X]\) such that
\[
  B_\nu C_\nu+Q_\nu D_\nu=1.
\]
Consequently,
\[
  e_\nu:=[B_\nu C_\nu]\in\mathcal A_Q
\]
satisfies
\[
  e_\nu\equiv1\pmod{Q_\nu},
  \qquad
  e_\nu\equiv0\pmod{Q_\mu}
  \quad\text{for every }\mu\neq\nu.
\]
Thus the \(e_\nu\) are the standard mutually orthogonal
Chinese-remainder idempotents. They may be computed explicitly by the
extended Euclidean algorithm.

Let \(e_i^{\mathrm r}\) and \(e_j^{\mathrm s}\) denote the idempotents
corresponding respectively to \(Q_i^{\mathrm r}\) and \(Q_j^{\mathrm s}\).
The selector classes in the statement are then given explicitly by
\begin{align}
  [S_{i,m}]
  &=e_i^{\mathrm r}[(X-a_i)^{m-1}],
  \label{eq:4d-explicit-real-selectors}\\
  [\Lambda_{j,m}]
  &=e_j^{\mathrm s}[\Delta_j^{m-1}],
  \qquad
  [\mathrm M_{j,m}]
  =e_j^{\mathrm s}
    [(X-2\alpha_j)\Delta_j^{m-1}].
  \label{eq:4d-explicit-spherical-selectors}
\end{align}
The required congruences follow immediately from those satisfied by the
idempotents.

Finally, in the real local factor
\(\R[X]/((X-a_i)^{L_i})\), the classes
\[
  1,(X-a_i),\ldots,(X-a_i)^{L_i-1}
\]
form a basis. Repeated Euclidean division by the monic quadratic
\(\Delta_j\) shows that in
\(\R[X]/(\Delta_j^{M_j})\) the classes
\[
  \bigl\{
    [\Delta_j^{m-1}],
    [(X-2\alpha_j)\Delta_j^{m-1}]
    :1\leq m\leq M_j
  \bigr\}
\]
form a basis. Multiplication by the corresponding idempotents transports
these local bases through the Chinese-remainder isomorphism
\eqref{eq:4d-crt-decomposition}, proving both the asserted formulas and
the global basis statement.
\end{proof}
For a polynomial \(P\in\R[X]\) and a formal principal part
\(\mathfrak p\in\mathcal P_Q\), define provisionally
\begin{equation}\label{eq:4d-residue-pairing}
  \bigl\langle[P],\mathfrak p\bigr\rangle_{\mathrm{res}}
  :=\Flux{U}{P(q)R_{\mathfrak p}(q)}.
\end{equation}
The order of multiplication is important: \(P\) has real coefficients and
is an intrinsic left multiplier, while the coefficients of
\(R_{\mathfrak p}\) remain on the right.

\begin{proposition}[Well-definedness of the residue pairing]
\label{prop:4d-residue-pairing-well-defined}
Formula \eqref{eq:4d-residue-pairing} depends only on the class
\([P]\in\mathcal A_Q\).  It is real linear in \([P]\), right quaternionic
linear in \(\mathfrak p\), and independent of the admissible circular
integration domain \(U\) as long as the enclosed polar components are kept
fixed.
\end{proposition}

\begin{proof}
If \(P-P'=QS\), then
\[
  (P(q)-P'(q))R_{\mathfrak p}(q)
  =Q(q)S(q)R_{\mathfrak p}(q)
\]
extends slice regularly across every \(a_i\) and every \(\Sigma_j\), because
the corresponding factor of \(Q\) cancels each term of
\eqref{eq:4d-principal-part-model} to its prescribed order.  Its flux is
therefore zero by Corollary~\ref{cor:4d-domain-independence}.  The linearity
statements follow from the displayed order of multiplication, and domain
independence follows from the same corollary.
\end{proof}

\begin{theorem}[Residue duality for principal parts]
\label{thm:4d-residue-duality}
The map
\begin{equation}\label{eq:4d-residue-duality-isomorphism}
  \boxed{\;
  \mathcal P_Q
  \longrightarrow
  \operatorname{Hom}_{\R}(\mathcal A_Q,\Hh),
  \qquad
  \mathfrak p\longmapsto
  \bigl([P]\mapsto
        \langle[P],\mathfrak p\rangle_{\mathrm{res}}\bigr)
  \;}
\end{equation}
is an isomorphism of right quaternionic vector spaces.  In particular, the
pairing is nondegenerate in both variables.

If \(\mathfrak p_f\) is the collection of principal coefficients of the
given function \(f\), then
\begin{equation}\label{eq:4d-weighted-residue-functional}
  \boxed{\;
  \Flux{U}{P(q)f(q)}
  =\bigl\langle[P],\mathfrak p_f\bigr\rangle_{\mathrm{res}},
  \qquad [P]\in\mathcal A_Q.\;}
\end{equation}
Thus the finitely many values of the weighted flux functional on any basis
of \(\mathcal A_Q\) determine, and are determined by, the complete principal
parts of \(f\) inside \(U\).
\end{theorem}

\begin{proof}
By Proposition~\ref{prop:4d-residue-pairing-well-defined}, the displayed map
is defined and right quaternionic linear.  The global selectors of
Proposition~\ref{prop:4d-hermite-crt-selectors}, the local recovery formulas
\eqref{eq:4d-real-principal-recovery}--\eqref{eq:4d-v-principal-recovery},
and the local-to-global residue decomposition give
\begin{equation}\label{eq:4d-real-selector}
  \langle[S_{i,m}],\mathfrak p\rangle_{\mathrm{res}}
  =c_{i,-m}
  \qquad
  (1\leq i\leq\ell,\ 1\leq m\leq L_i)
\end{equation}
and
\begin{equation}\label{eq:4d-spherical-selectors}
  \langle[\Lambda_{j,m}],\mathfrak p\rangle_{\mathrm{res}}
  =u_{j,-m},
  \qquad
  \langle[\mathrm M_{j,m}],\mathfrak p\rangle_{\mathrm{res}}
  =v_{j,-m}
  \quad
  (1\leq j\leq S,\ 1\leq m\leq M_j).
\end{equation}
Thus the map in \eqref{eq:4d-residue-duality-isomorphism} is injective.

Moreover,
\[
  \dim_{\Hh}\mathcal P_Q
  =\sum_iL_i+2\sum_jM_j
  =\deg Q
  =\dim_{\R}\mathcal A_Q.
\]
With the right action \((\Lambda c)([P])=\Lambda([P])c\),
\(\operatorname{Hom}_{\R}(\mathcal A_Q,\Hh)\) is a right quaternionic vector
space of the same quaternionic dimension.  The injective map is therefore an
isomorphism.  Since the selector classes form a basis and
\eqref{eq:4d-real-selector}--\eqref{eq:4d-spherical-selectors} recover every
coordinate, the pairing is nondegenerate in both variables.

Finally, \(f-R_{\mathfrak p_f}\) has zero principal part at every polar
component contained in \(U\), so it extends slice regularly across them and
is slice regular on a circular neighbourhood of \(\overline U\) after that
extension.  The same is true after multiplication by \(P\).  Its flux
vanishes by Corollary~\ref{cor:4d-domain-independence}, proving
\eqref{eq:4d-weighted-residue-functional}.
\end{proof}

\begin{remark}[Algebraic form of the finite residue transform]
\label{rem:4d-algebraic-residue-pairing}
Assume that \(d:=\deg Q>0\). Every formal principal part
\(\mathfrak p\in\mathcal P_Q\) determines a quaternionic polynomial
\(N_{\mathfrak p}\in\Hh[X]\), with coefficients displayed on the right,
by
\begin{equation}
  \label{eq:4d-principal-part-numerator}
  N_{\mathfrak p}(X):=
  \sum_{i=1}^{\ell}\sum_{m=1}^{L_i}
    \frac{Q(X)}{(X-a_i)^m}\,c_{i,-m}+
  \sum_{j=1}^{S}\sum_{m=1}^{M_j}
    \frac{Q(X)}{\Delta_j(X)^m}
       \bigl(Xu_{j,-m}+v_{j,-m}\bigr).
\end{equation}
All the quotients appearing here are real polynomials, and
\(\deg N_{\mathfrak p}<d\). By construction,
\begin{equation}\label{eq:4d-principal-part-quotient}
  R_{\mathfrak p}(q)=Q(q)^{-1}N_{\mathfrak p}(q).
\end{equation}
Conversely, the real partial-fraction decomposition associated with
\eqref{eq:4d-crt-decomposition} shows that every quaternionic polynomial
of degree smaller than \(d\) arises from a unique formal principal part.
Thus
\[
  \mathfrak p\longmapsto[N_{\mathfrak p}]
\]
identifies \(\mathcal P_Q\) with
\(\mathcal A_Q\otimes_{\R}\Hh\).

Define
\[
  \lambda_Q\colon\mathcal A_Q\longrightarrow\R,
  \qquad
  \lambda_Q([H])
  :=
  \bigl[X^{d-1}\bigr]\operatorname{rem}_Q(H),
\]
where \(\operatorname{rem}_Q(H)\) is the unique remainder of \(H\) upon
division by the monic polynomial \(Q\), and
\(\bigl[X^{d-1}\bigr]\) denotes extraction of the indicated coefficient.
Extend \(\lambda_Q\) right \(\Hh\)-linearly to
\(\mathcal A_Q\otimes_{\R}\Hh\). Then the residue pairing has the purely
algebraic form
\begin{equation}\label{eq:4d-algebraic-residue-pairing}
  \boxed{\;
  \bigl\langle[P],\mathfrak p\bigr\rangle_{\mathrm{res}}
  =
  \lambda_Q\bigl([P N_{\mathfrak p}]\bigr).
  \;}
\end{equation}
Indeed, on any full complex slice the flux on the left is the sum of the
finite residues of the rational function
\[
  \frac{P(z)N_{\mathfrak p}(z)}{Q(z)}.
\]
This sum is the coefficient of \(z^{-1}\) in its Laurent expansion at
infinity. Polynomial division by \(Q\) does not change that coefficient,
and, because \(Q\) is monic of degree \(d\), the coefficient is precisely
the coefficient of \(X^{d-1}\) in the remainder of
\(PN_{\mathfrak p}\) modulo \(Q\). Notice that the residue at infinity
itself is the negative of this coefficient.

Formula \eqref{eq:4d-algebraic-residue-pairing} also makes
nondegeneracy transparent. If
\[
  N(X)=\sum_{r=0}^{j}X^r n_r,
  \qquad j<d,\qquad n_j\neq0,
\]
then, for \(P=X^{d-1-j}\),
\[
  \lambda_Q([PN])=n_j\neq0.
\]
The same argument, with the roles of the two factors interchanged, detects
every nonzero class \([P]\in\mathcal A_Q\). Hence
\[
  ([P],[N])\longmapsto\lambda_Q([PN])
\]
is the perfect Frobenius pairing of the Artinian Gorenstein algebra
\(\mathcal A_Q=\R[X]/(Q)\), extended to quaternionic coefficients. In this
one-variable and finite-dimensional sense,
Theorem~\ref{thm:4d-residue-duality} gives a concrete instance of the
algebraic residue formalism underlying Grothendieck duality; see
\cite{Hartshorne}.

Accordingly, the functional
\[
  [P]\longmapsto\Flux{U}{P(q)f(q)}
\]
may be viewed as a finite residue transform of \(f\). Once the pole
locations and order bounds are fixed, this transform loses no
principal-part information: ordinary residue summation is its value at
\([1]\), while the remaining classes recover all the higher local
coefficients.
\end{remark}

\begin{example}[One simple spherical pole]
\label{ex:4d-duality-one-sphere}
Let \(\Sigma=\alpha+\Sphere\beta\), put
\(\Delta(X)=\DeltaSphere{\Sigma}(X)\), and suppose that the only principal
part is
\[
  R_{\mathfrak p}(q)=\Delta(q)^{-1}(qu+v).
\]
Here \(Q=\Delta\), and
\(\{[1],[X-2\alpha]\}\) is a real basis of \(\mathcal A_Q\).  The two values
of the residue pairing are
\[
  \langle[1],\mathfrak p\rangle_{\mathrm{res}}=u,
  \qquad
  \langle[X-2\alpha],\mathfrak p\rangle_{\mathrm{res}}=v.
\]
Thus the ordinary spherical residue detects only \(u\), whereas the single
additional weight \(X-2\alpha\) recovers the companion coefficient \(v\).
For a semiregular function with this principal part, both coefficients are
therefore obtained from boundary integrals over any admissible \(U\).
\end{example}

\begin{example}[Separating a real pole from a spherical pole]
\label{ex:4d-duality-mixed-simple-poles}
Suppose that \(U\) contains one simple real pole \(a\) and one simple polar
sphere \(\Sigma=\alpha+\Sphere\beta\).  Write
\[
  \Delta(X)=\DeltaSphere{\Sigma}(X),
  \qquad
  d:=\Delta(a)=(a-\alpha)^2+\beta^2>0,
  \qquad
  Q(X)=(X-a)\Delta(X).
\]
The three selector classes can be represented explicitly by
\begin{equation}\label{eq:4d-mixed-simple-selectors}
  \begin{aligned}
    S_a(X)&:=\frac{\Delta(X)}{d},\\
    \Lambda_\Sigma(X)&:=1-S_a(X)
      =-\frac{(X-a)(X+a-2\alpha)}{d},\\
    \mathrm M_\Sigma(X)&:=(X-2\alpha)\Lambda_\Sigma(X).
  \end{aligned}
\end{equation}
Indeed, \(S_a\equiv1\pmod{X-a}\) and
\(S_a\equiv0\pmod{\Delta}\), while
\[
  \begin{aligned}
    \Lambda_\Sigma&\equiv0\pmod{X-a},
    &\Lambda_\Sigma&\equiv1\pmod{\Delta},\\
    \mathrm M_\Sigma&\equiv0\pmod{X-a},
    &\mathrm M_\Sigma&\equiv X-2\alpha\pmod{\Delta}.
  \end{aligned}
\]
If the formal principal part has coefficients
\(\mathfrak p=(c;u,v)\), then
\[
  \langle[S_a],\mathfrak p\rangle_{\mathrm{res}}=c,
  \qquad
  \langle[\Lambda_\Sigma],\mathfrak p\rangle_{\mathrm{res}}=u,
  \qquad
  \langle[\mathrm M_\Sigma],\mathfrak p\rangle_{\mathrm{res}}=v.
\]
This example shows concretely how global weights cancel all unwanted polar
components before the flux extracts the selected coefficient.
\end{example}
\section{Examples}\label{sec:examples}

The three examples below test, in turn, additivity in the presence of
non-commuting coefficients, higher spherical order with an exceptional
point, and a non-zero function whose symmetrization vanishes identically.
In each of them the boundary integral is computed directly, by
Proposition~\ref{prop:weighted-ball-extraction} or
Lemma~\ref{lem:weighted-tube-parametrization} and independently of the
residue theorem, and only then compared with what the theorem predicts.
Shorter checks follow, on a real pole of higher order, on a polar sphere
whose residue vanishes, and on a function whose deeper paired coefficients
are recovered only by the weighted moments of
Theorem~\ref{thm:4d-principal-part-recovery}.

\begin{example}[Two polar spheres with non-commuting coefficients]
\label{ex:two-spheres}
Let \(\Sigma_1=\Sphere\) and \(\Sigma_2=2+2\Sphere\), with characteristic
polynomials
\[
  \DeltaSphere{\Sigma_1}(q)=q^2+1 ,
  \qquad
  \DeltaSphere{\Sigma_2}(q)=q^2-4q+8 ,
\]
fix the coefficients
\[
  u_1=\qi+\qj,\quad v_1=\qk,
  \qquad
  u_2=1+\qk,\quad v_2=\qj ,
\]
and consider
\begin{equation}\label{eq:two-sphere-example}
  f_2(q)=\DeltaSphere{\Sigma_1}(q)^{-1}(qu_1+v_1)
        +\DeltaSphere{\Sigma_2}(q)^{-1}(qu_2+v_2) .
\end{equation}
Both spheres lie in \(B(0,4)\). The example is designed to be genuinely
non-commutative and genuinely non-real: \(u_1u_2-u_2u_1=2(\qi-\qj)\neq0\),
and \(f_2(0)=\qk+\qj/8\notin\R\), so \(f_2\) is not slice preserving.

Each summand is regular near the other sphere, so the paired principal
coefficients can be read off at once: \((u_1,v_1)\) at \(\Sigma_1\) and
\((u_2,v_2)\) at \(\Sigma_2\). Recovering them differentially from the full
function requires the regularizations
\[
  \gamma_1=\DeltaSphere{\Sigma_1}f_2
     =qu_1+v_1
      +\DeltaSphere{\Sigma_1}\DeltaSphere{\Sigma_2}^{-1}(qu_2+v_2) ,
  \qquad
  \gamma_2=\DeltaSphere{\Sigma_2}f_2 ,
\]
one for each sphere, the subscript numbering the spheres and not, as
elsewhere, the regularizing exponent, which is \(1\) in both cases. Their
remainder terms vanish at both nodes of the respective sphere, leaving
\(\partial_s\gamma_1=u_1\) on \(\Sigma_1\) and
\(\partial_s\gamma_2=u_2\) on \(\Sigma_2\). One can also see this on
explicit values: at \(y_1=\qi\) and \(y_2=2+2\qi\),
\begin{align*}
  \gamma_1(\qi)&=-1+2\qk, & \gamma_1(-\qi)&=1,\\
  \gamma_2(2+2\qi)&=2+2\qi-\qj+2\qk, &
  \gamma_2(2-2\qi)&=2-2\qi+3\qj+2\qk ,
\end{align*}
and the spherical derivative \eqref{eq:slice-spherical-derivative}, which
divides the difference of these values on the left by \(2\qi\) and by
\(4\qi\) respectively, returns \(u_1\) and \(u_2\); the even formula of
Remark~\ref{rem:4d-simple-spherical-pole} then gives
\(\gamma_1(y_1)-y_1u_1=v_1\) and \(\gamma_2(y_2)-y_2u_2=v_2\).

The comparison with ordinary point data is instructive. On any fixed slice,
a single summand \(\DeltaSphere{\Sigma}^{-1}(qu+v)\) has the two point
residues
\begin{equation}\label{eq:ex-simple-pair}
  r_\pm^I=(p_\pm-p_\mp)^{-1}(p_\pm u+v)
        =\frac u2\mp\frac{I}{2\beta}(\alpha u+v) ,
  \qquad p_\pm=\alpha\pm I\beta ,
\end{equation}
which depend on the slice, whereas their sum is \(u\) for every \(I\) — and
this without ever commuting \(I\) past a coefficient. Applying the
computation to both summands reproduces the differential values above.

Finally, the boundary integral. All poles have modulus at most
\(\sqrt8<4\), so on \(|q|=4\) the expansions at infinity converge normally,
with leading terms \(q^{-1}u_1\) and \(q^{-1}u_2\) respectively, and
Proposition~\ref{prop:weighted-ball-extraction} gives
\begin{equation}\label{eq:two-sphere-integral}
  \Flux{B(0,4)}{f_2}=u_1+u_2=1+\qi+\qj+\qk ,
\end{equation}
an unnormalized integral of \(4\pi^2(1+\qi+\qj+\qk)\). This agrees with
Theorem~\ref{thm:4d-residue} and checks two things at once: the constant
\(4\pi^2\), and the additivity of contributions from distinct spheres.
\end{example}

The next example raises the spherical order and places an exceptional point
on the polar sphere, so that the alternating derivatives are genuinely
needed.

\begin{example}[Spherical order four and an exceptional point]
\label{ex:higher}
Let \(\Sigma=\Sphere\), with \(\DeltaSphere{\Sigma}(q)=q^2+1\), and set
\begin{equation}\label{eq:ex-higher}
  f_{\mathrm h}(q)=\DeltaSphere{\Sigma}(q)^{-2}(q^3\qj+q^2\qk) .
\end{equation}
Substituting \(q^2=\DeltaSphere{\Sigma}-1\) in the numerator produces the
paired expansion in closed form,
\[
  f_{\mathrm h}
  =\DeltaSphere{\Sigma}^{-1}(q\qj+\qk)
   -\DeltaSphere{\Sigma}^{-2}(q\qj+\qk) ,
\]
so that \((u_{-1},v_{-1})=(\qj,\qk)\) and
\((u_{-2},v_{-2})=(-\qj,-\qk)\). The least admissible exponent is \(M=2\)
and the spherical order is four; the extension
\(g_2=\DeltaSphere{\Sigma}^2f_{\mathrm h}=q^3\qj+q^2\qk\) vanishes at
\(-\qi\), while \(g_2(\qi)=-2\qk\neq0\), so the sphere carries an
exceptional point.

In signed slice coordinates \(q=x+Ir\), differentiating directly gives
\begin{align*}
  \partial_sg_2(q)&=(3x^2-r^2)\qj+2x\qk ,\\
  (\partial_c\partial_s)g_2(q)&=(3x+Ir)\qj+\qk ,\\
  \partial_s(\partial_c\partial_s)g_2(q)&=\qj ,
\end{align*}
the last identically in \(q\). Formula \eqref{eq:sres-main-differential}
therefore returns \(\SRes_\Sigma f_{\mathrm h}=\qj\) at every
representative, the exceptional point \(-\qi\) included, where the value of
\(g_2\) would have told us nothing; and at \(y=I\in\Sphere\) the even
formula gives \(v_{-1}=(I\qj+\qk)-I\qj=\qk\).

Nothing changes with a nonminimal exponent. For \(M=3\),
\[
  g_3=\DeltaSphere{\Sigma}g_2=q^5\qj+q^4\qk+q^3\qj+q^2\qk ,
  \qquad
  (\partial_c\partial_s)^2g_3(q)=(10x+2Ir)\qj+2\qk ,
\]
whence
\[
  \frac{\partial_s(\partial_c\partial_s)^2g_3(y)}{2!}=\qj ,
  \qquad
  \frac{[(\partial_c\partial_s)^2g_3](y)}{2!}-y\qj=\qk
  \qquad(y\in\Sphere) ,
\]
which checks both the factorial and the independence of the regularizing
exponent, for each of the two coefficients.

On the exceptional slice the picture becomes planar. Since
\(\qk=\qi\qj\),
\[
  (f_{\mathrm h})_{\qi}(z)=\frac{z^2}{(z-\qi)^2(z+\qi)}\,\qj ,
\]
which is the function of Example~\ref{ex:complex-unequal-orders} with
\(\ci\) replaced by \(\qi\) and a factor \(\qj\) on the right: a double pole
at \(\qi\), a simple one at \(-\qi\). The exceptional point has reduced
order without being removable here. Its point residues and coefficients of
index \(-2\) are those computed there,
\[
  r_+=\tfrac34\qj,\quad r_-=\tfrac14\qj,
  \qquad
  c_{-2}^+=\tfrac12\qk,\quad c_{-2}^-=0 ,
\]
and the first-moment formula \eqref{eq:4d-v-point-residues} gives
\[
  v_{-1}=-(-\qi)r_+-\qi r_-+c_{-2}^++c_{-2}^-
        =\tfrac12\qk+\tfrac12\qk=\qk ,
\]
where the correction of index \(-2\) accounts for exactly half the answer:
the two point residues alone would return \(\qk/2\).

The boundary integral confirms both coefficients. The expansions at infinity
begin with
\[
  f_{\mathrm h}(q)=q^{-1}\qj+q^{-2}\qk+O(|q|^{-3}) ,
  \qquad
  qf_{\mathrm h}(q)=\qj+q^{-1}\qk+O(|q|^{-2}) ,
\]
so that Proposition~\ref{prop:weighted-ball-extraction} gives
\begin{equation}\label{eq:ex-higher-flux}
  \Flux{B(0,4)}{f_{\mathrm h}}=\qj ,
  \qquad
  \Flux{B(0,4)}{qf_{\mathrm h}}=\qk ,
\end{equation}
independently of everything above.
\end{example}

The last example is the one the whole apparatus was designed for: a function
that is not zero, whose symmetrization is, on a domain with no real points.

\begin{example}[A zero divisor on a product domain]\label{ex:null}
Fix \(K\in\Sphere\), let \(\Sigma=\Sphere\), and on \(\Hh\setminus\R\) set
\[
  \mathcal J(q)=\frac{\imq q}{|\imq q|} ,
  \qquad
  \eta(q)=1+\mathcal J(q)K ,
  \qquad
  f_{\mathrm n}(q)=\DeltaSphere{\Sigma}(q)^{-1}q\,\eta(q) .
\]
The integration set will be the tube \(T_\rho:=T_{0,1,\rho}\) of
Definition~\ref{def:weighted-circular-tube}, for \(0<\rho<1\): a connected
circular set containing \(\Sphere\), meeting no real point, and with a pair
of disjoint conjugate discs as profile.

The factor \(\eta\) is slice regular, being induced by the locally constant
stem
\[
  H(z)=
  \begin{cases}
    1+\ci K,& \imq z>0,\\
    1-\ci K,& \imq z<0 ,
  \end{cases}
\]
and the identity \((\ci K)^2=1\) governs its whole behaviour: it gives
\(HH^c=0\) on both components, so that \(\eta^s\equiv0\) although
\(\eta\not\equiv0\), and \(H^2=2+2\ci K=2H\), so that \(\eta/2\) is a
non-trivial idempotent. Multiplying by the intrinsic meromorphic factor
\(q/(q^2+1)\) preserves both features, leaving
\(f_{\mathrm n}\not\equiv0\) with \(f_{\mathrm n}^s\equiv0\): a genuine zero
divisor of the slice algebra rather than the zero function, whose only polar
sphere is \(\Sigma\).

\emph{The local spherical expansion.} Write
\(\Delta=\DeltaSphere{\Sigma}\) and, for \(|\Delta(q)|<1\), let
\(\varsigma(q)=\sqrt{1-\Delta(q)}\) be the branch given by the power series
at \(\Delta=0\), of value \(1\) on \(\Sigma\). That branch is
\(-\mathcal J(q)q\): writing \(q=x+Ir\) with \(r>0\), so that
\(\mathcal J(q)=I\), one has \(-Iq=r-Ix\), of positive real part, and
\((-Iq)^2=-q^2=1-\Delta(q)\). Since \(q\) commutes with \(\mathcal J(q)\),
the regularization is therefore
\[
  g(q)=\Delta(q)f_{\mathrm n}(q)=q-\varsigma(q)K ,
\]
and expanding the square root gives, normally on compact subsets of
\(0<|\Delta(q)|<1\),
\begin{equation}\label{eq:ex-null-series}
  f_{\mathrm n}(q)
  =\Delta^{-1}(q-K)+\tfrac12K+\tfrac18\Delta K+\tfrac1{16}\Delta^2K+\cdots ,
\end{equation}
whence \(u_{-1}=1\) and \(v_{-1}=-K\). Note that the auxiliary square root
is intrinsic, and that no inverse of \(\eta\) or of \(f_{\mathrm n}\) has
been taken.

\emph{Differential and slice checks.} For \(I\in\Sphere\) the two values
\(g(I)=I-K\) and \(g(-I)=-I-K\) differ by \(2I\), so that
\[
  \partial_sg(I)=(2I)^{-1}\bigl(g(I)-g(-I)\bigr)=(2I)^{-1}(2I)=1 ,
  \qquad
  g(I)-I\,\partial_sg(I)=-K ,
\]
in agreement with \eqref{eq:ex-null-series}; and the evaluation goes through
at \(I=K\) as well, where \(g(K)=0\). On a full slice \(\C_I\) the
restriction of \(f_{\mathrm n}\) is \(z(z^2+1)^{-1}(1+IK)\) on the upper
component and \(z(z^2+1)^{-1}(1-IK)\) on the lower one, with point residues
\begin{equation}\label{eq:ex-null-pair}
  r_+^I=\frac{1+IK}{2},
  \qquad
  r_-^I=\frac{1-IK}{2},
  \qquad
  r_+^I+r_-^I=1 .
\end{equation}
The slice \(\C_K\) shows why both components are needed: there the upper
restriction vanishes identically, while the lower one has residue \(1\) at
\(-K\), so that the upper half-slice alone would return \(0\) in place of
the spherical residue.

\emph{Direct integration on the tube.} Parametrize the boundary by
\(q=I+w\), \(w=\rho(\cos t+I\sin t)\), as in
Lemma~\ref{lem:weighted-tube-parametrization}. Since \(1+\rho\sin t>0\) one
has \(\mathcal J(q)=I\) throughout, and
\[
  \Delta(q)=w(2I+w) ,
  \qquad
  wf_{\mathrm n}(I+w)=(2I+w)^{-1}(I+w)(1+IK) ,
\]
where every factor preceding \(1+IK\) lies in \(\C_I\). Expanding
\((2I+w)^{-1}\) geometrically, uniformly in \(t\),
\[
  (2I+w)^{-1}(I+w)=\frac12+\frac12w(2I+w)^{-1} ,
\]
and the second term, carrying only positive powers of \(w\), integrates to
zero over a period; hence
\(\int_0^{2\pi}wf_{\mathrm n}(I+w)\dd t=\pi(1+IK)\). Integrating over
\(\Sphere\), where \(\int_\Sphere I\darea(I)=0\) and
\(\Area(\Sphere)=4\pi\), gives
\begin{equation}\label{eq:ex-null-flux}
  \Flux{T_\rho}{f_{\mathrm n}}=1 ,
  \qquad\text{that is}\qquad
  \int_{\partial T_\rho}
    \frac{n(q)f_{\mathrm n}(q)}{|\imq q|^2}\dsigma(q)=4\pi^2 .
\end{equation}
The first moment goes the same way: the constant term of
\((2I+w)^{-1}(I+w)^2\) is \(I/2\), so that
\(\int_0^{2\pi}w(qf_{\mathrm n})(I+w)\dd t=\pi(I-K)\) and
\(\Flux{T_\rho}{qf_{\mathrm n}}=-K=v_{-1}\).

Both computations hold for every \(0<\rho<1\), and neither requires the
local series \eqref{eq:ex-null-series} to converge on the whole boundary of
the tube. A vanishing symmetrization, then, does not force the residue
functional to vanish, and computing it needs neither real points nor any
invertibility.
\end{example}

We close with a few shorter checks.

\begin{remark}[Further checks]\label{rem:ex-further}
Four shorter verifications complete the picture.

At a real point \(a\), the function
\(f_{\mathrm r}(q)=(q-a)^{-3}(\qi+\qj)+(q-a)^{-1}(1+\qk)\) has a pole of
order three, and only the coefficient of index \(-1\) survives:
\(\Flux{B(a,R)}{f_{\mathrm r}}=1+\qk\) for every \(R>0\).

A residue may vanish at a genuine polar sphere. With
\(\Delta=\DeltaSphere{\Sphere}\), the function
\(f_0(q)=\Delta(q)^{-1}\qk\) has \((u_{-1},v_{-1})=(0,\qk)\), hence zero
flux on \(B(0,4)\), although its polar sphere is not removable; the first
moment detects it, \(\Flux{B(0,4)}{qf_0}=\qk\).

Both paired coefficients of index \(-1\) may vanish at once, and then the
deeper moments are needed. For \(\widetilde f_0(q)=\Delta(q)^{-2}\qk\) one
has \((u_{-1},v_{-1})=(0,0)\) while \(v_{-2}=\qk\neq0\), so that
\(\Flux{B(0,4)}{\widetilde f_0}\) and
\(\Flux{B(0,4)}{q\widetilde f_0}\) both vanish in the presence of a pole:
neither the unweighted residue nor its first paired moment sees the
coefficient that is actually there. Weighting by the characteristic
polynomial does see it,
\[
  \Flux{B(0,4)}{\Delta\widetilde f_0}=0,
  \qquad
  \Flux{B(0,4)}{q\Delta\widetilde f_0}=\qk=v_{-2},
\]
which is \eqref{eq:4d-u-principal-recovery}--%
\eqref{eq:4d-v-principal-recovery} with \(k=1\).

Finally, the three kinds of contribution add up as expected. Taking \(a=0\)
in \(f_{\mathrm r}\) and setting
\[
  f_{\mathrm{mix}}(q)=f_{\mathrm r}(q)+f_{\mathrm h}(q)
    +\DeltaSphere{\Sigma_2}(q)^{-1}\bigl(q(\qi-\qk)+\qj\bigr) ,
\]
the local residues at \(0\), at \(\Sphere\) and at \(2+2\Sphere\) are
\(1+\qk\), \(\qj\) and \(\qi-\qk\), whose sum gives
\(\Flux{B(0,4)}{f_{\mathrm{mix}}}=1+\qi+\qj\), the terms in \(\qk\)
cancelling between the real pole and the second sphere.
\end{remark}
\section{Integrals over the imaginary hyperplane}\label{sec:hyperplane}

The residue theorem applies to absolutely convergent integrals over
\(\imq\Hh\simeq\R^3\), once the integrand is known to extend to a slice
function. Not every rational datum does: the admissible ones are exactly
those whose dependence on the imaginary unit is affine, and
Subsection~\ref{subsec:hyp-data} characterizes them and constructs their
canonical extension. That extension is then used twice, first to evaluate
the integral itself and then, after multiplying on the left by the phase
\(e^{-tq}\) with \(t\geq0\), to evaluate a family of oscillatory integrals
whose value turns out to be a finite exponential--polynomial expression in
\(t\). Throughout, \(\dd V_3\) is ordinary Euclidean volume measure on the
imaginary hyperplane.

\subsection{Rational data, extension and decay}\label{subsec:hyp-data}

The datum on the hyperplane and the slice function extending it are objects
of different kinds, and it is convenient to name both at once.

\begin{definition}[Rational data and rational slice functions]
\label{def:rat-data}
An \emph{ordinary rational datum} is a map \(\varphi:\imq\Hh\to\Hh\) whose
four real components are quotients of real polynomials, with denominators
nowhere vanishing on \(\R^3\). It is \emph{spherically compatible} if its
dependence on the imaginary unit is affine, that is, if for every \(r>0\)
there are \(a(r),b(r)\in\Hh\), independent of \(I\), with
\begin{equation}\label{eq:rat-compatibility}
 \varphi(rI)=a(r)+Ib(r)\qquad(I\in\Sphere).
\end{equation}
Separately, a \emph{rational slice function} is one of the form
\(d(q)^{-1}\sum_nq^nc_n\), with \(d\in\R[q]\setminus\{0\}\) and finitely
many \(c_n\in\Hh\), understood semiregularly at its finite singularities.
\end{definition}

In one complex variable, extending the integrand off the real line is not a
step one dwells on: a rational function of \(x\) is already a rational
function of \(z\), and closing the contour is a matter of choosing a
half-plane. In four dimensions the corresponding step has content, because
the datum lives on a three-dimensional hyperplane and a slice function is
determined by two real parameters on each sphere. The next proposition
identifies exactly which rational data extend, and
Remark~\ref{rem:rat-obstruction} exhibits one that does not.

\begin{proposition}[Rational trace and extension]\label{prop:rat-extension}
For an ordinary rational datum \(\varphi\), the following are equivalent:
\begin{enumerate}[label=\textup{(\roman*)}]
 \item \(\varphi\) is the restriction of a stem-induced slice regular
 function on a circular open neighbourhood of \(\imq\Hh\);
 \item \(\varphi\) is spherically compatible;
 \item there are unique quaternion-valued rational functions
 \(\mathsf a(s),\mathsf b(s)\), with real rational components and no poles
 on \([0,\infty)\), such that
 \begin{equation}\label{eq:rat-data-AB}
  \varphi(v)=\mathsf a(|v|^2)+v\,\mathsf b(|v|^2).
 \end{equation}
\end{enumerate}
When these hold, the \emph{canonical extension} of \(\varphi\) is the
rational slice-semiregular function
\begin{equation}\label{eq:rat-extension}
 f_\varphi(q)=\mathsf a(-q^2)+q\,\mathsf b(-q^2) ,
\end{equation}
which has no pole on \(\imq\Hh\) and, elsewhere, only finitely many real
poles and non-real polar spheres. It is unique both as a slice regular germ
along \(\imq\Hh\) and among global rational slice functions.
\end{proposition}

\begin{proof}
\emph{(i)\(\Rightarrow\)(ii)} is the stem representation
\eqref{eq:slice-induction}, whose dependence on \(I\) is affine by
construction.

\emph{(ii)\(\Rightarrow\)(iii).} Fix \(I_0\in\Sphere\) and set
\(u(r)=\varphi(rI_0)\). Componentwise, the two functions
\[
 \frac{u(r)+u(-r)}2 ,
 \qquad
 (rI_0)^{-1}\frac{u(r)-u(-r)}2
\]
are even rational functions of \(r\) with removable singularities at
\(r=0\); and an even real rational function is a rational function of
\(r^2\), as one sees by multiplying numerator and denominator by the
reflection of the denominator. Calling the resulting functions
\(\mathsf a(r^2)\) and \(\mathsf b(r^2)\), neither has a pole on
\([0,\infty)\), and spherical compatibility identifies
\(a(r)=\mathsf a(r^2)\) and \(b(r)=r\,\mathsf b(r^2)\) independently of the
chosen \(I_0\). The same computation gives the uniqueness of the pair.

\emph{(iii)\(\Rightarrow\)(i).} Choose a real denominator \(d(s)\) common to
\(\mathsf a\) and \(\mathsf b\) and non-vanishing on \([0,\infty)\), and
write \(\mathsf a=d^{-1}p\), \(\mathsf b=d^{-1}\tilde p\) with right
quaternionic polynomial coefficients, so that
\[
 f_\varphi(q)=d(-q^2)^{-1}\bigl(p(-q^2)+q\,\tilde p(-q^2)\bigr)
\]
is a rational slice function. Its possible poles are the finitely many real
points and spheres determined by the zeros of \(d(-q^2)\), and none of them
meets \(\imq\Hh\), where \(-v^2=|v|^2\) and \(d\) was assumed non-zero;
its trace on the hyperplane is \(\varphi\).

As for uniqueness, the values at \(rI_0\) and at \(-rI_0\) together
determine the stem on the imaginary profile, so that the complex identity
theorem gives uniqueness of the germ, and rationality upgrades this to
global uniqueness. Note that no uniqueness statement on a single half-slice
has been used, in keeping with
Remark~\ref{rem:slice-full-slice-needed}.
\end{proof}

\begin{remark}[A genuine restriction on the integrand]
\label{rem:rat-obstruction}
The real-valued datum
\[
 \varphi(x_1\qi+x_2\qj+x_3\qk)=(1+x_1^2+2x_2^2+3x_3^2)^{-2}
\]
is rational, real analytic and absolutely integrable, yet not spherically
compatible: being even, it would force \(b(r)=0\) in
\eqref{eq:rat-compatibility} and hence be constant on each sphere, whereas
\(\varphi(\qi)=1/4\) and \(\varphi(\qj)=1/9\). This specific construction, in which the imaginary part of quaternions has a priviledged role, does not
apply to arbitrary rational functions of three real variables.
\end{remark}


Two conditions have to be reconciled before the residue theorem can be
applied: the datum must be integrable on the hyperplane, and its extension
must decay fast enough for the tangent balls to close. The next lemma shows
that the two are the same condition, read on \(\mathsf a\) and
\(\mathsf b\), and records in passing the radial form of the integral.

\begin{lemma}[Decay and the two angular averages]\label{lem:rat-L1}
For a compatible ordinary rational datum written as in
\eqref{eq:rat-data-AB}, absolute integrability is equivalent to
\begin{equation}\label{eq:rat-decay-AB}
 \mathsf a(s)=O(s^{-2}),
 \qquad
 \mathsf b(s)=O(s^{-3})
 \qquad(s\to+\infty) ,
\end{equation}
and these conditions imply the uniform estimate
\begin{equation}\label{eq:rat-decay-f}
 |f_\varphi(q)|\leq C|q|^{-4}\qquad(|q|\geq R_0)
\end{equation}
for any \(R_0\) whose ball contains every pole. Moreover
\begin{equation}\label{eq:rat-radial-reduction}
 \int_{\imq\Hh}\varphi(v)\dd V_3(v)
 =4\pi\int_0^\infty r^2\mathsf a(r^2)\dd r .
\end{equation}
\end{lemma}

\begin{proof}
Integrating \eqref{eq:rat-data-AB} over a sphere of radius \(r\), against
\(1\) and against \(-I\) respectively, identifies the two coefficients as
angular averages of \(\varphi\):
\[
 \mathsf a(r^2)=\frac1{4\pi}\int_\Sphere\varphi(rI)\darea(I) ,
 \qquad
 r\,\mathsf b(r^2)=\frac1{4\pi}\int_\Sphere(-I)\varphi(rI)\darea(I) ,
\]
which, writing \(m(r)=\int_\Sphere|\varphi(rI)|\darea(I)\), give the
two-sided bound
\[
 4\pi\max\{|\mathsf a(r^2)|,r|\mathsf b(r^2)|\}
 \leq m(r)
 \leq 4\pi\bigl(|\mathsf a(r^2)|+r|\mathsf b(r^2)|\bigr) .
\]
Absolute integrability is therefore equivalent to the two conditions
\[
  \int_0^\infty r^2|\mathsf a(r^2)|\dd r<\infty,
  \qquad
  \int_0^\infty r^3|\mathsf b(r^2)|\dd r<\infty.
\]
Neither integral is obstructed at the origin. Since every non-zero rational
component has an integer order at infinity, the two tests are equivalent
precisely to \eqref{eq:rat-decay-AB}.

For \eqref{eq:rat-decay-f}, the leading-term estimates for real polynomials
are uniform in the quaternionic directions, so evaluating the reduced
components at \(-q^2\) turns \eqref{eq:rat-decay-AB} into
\(\mathsf a(-q^2)=O(|q|^{-4})\) and \(q\,\mathsf b(-q^2)=O(|q|^{-5})\).
Finally, Fubini and the first of the two angular averages give
\eqref{eq:rat-radial-reduction}, the odd part integrating to zero over each
sphere.
\end{proof}

In terms of degrees, \eqref{eq:rat-decay-AB} says that in every non-zero
reduced component the denominator exceeds the numerator by at least two for
\(\mathsf a\) and by at least three for \(\mathsf b\), the gaps being
measured in the variable \(s=|v|^2\).

\subsection{Closing the hyperplane}\label{subsec:hyp-tangent}

The domains that close the hyperplane are the four-dimensional counterparts
of the tangent disks of Subsection~\ref{subsec:complex-tangent-disks}. With
\(H_+=\{q:\req q>0\}\), set
\begin{equation}\label{eq:rat-tangent-balls}
 U_R=B(R,R)=\{q:|q|^2<2R\req q\} ,
\end{equation}
whose profile is exactly the disk \(D_R\) of
\eqref{eq:complex-tangent-disk}: these are smooth circular balls increasing
to \(H_+\), tangent to \(\imq\Hh\) at the origin and meeting it nowhere
else, and admissible in the sense of
Definition~\ref{def:weighted-admissible-set} even at that real tangency
point. The next lemma is the analogue of
Lemma~\ref{lem:complex-tangent-disks}, and its proof has the same three
parts: a limit on a bounded range, an estimate on the far part of the
boundary, and an estimate on the tails.

One difference is worth noting at the outset. What the flux integrates here
is not \(f\) but \(G=q^2f\): the factor \(q^2\) is what turns the weight
\(|\imq q|^{-2}\) into the Euclidean volume element of the hyperplane, as
\eqref{eq:rat-graph-geometry} will show, and it is also what makes the
integrand vanish at the tangency point.

\begin{lemma}[Tangent-ball limit]\label{lem:rat-tangent-limit}
Let \(f\) be slice semiregular on a circular open neighbourhood of
\(\overline{H_+}\), regular near \(\imq\Hh\), with finitely many polar
components in \(H_+\), and with \(|f(q)|\leq C|q|^{-4}\) for
\(\req q\geq0\), \(|q|\geq R_0\). Set \(G(q)=q^2f(q)\). Then no pole lies on
\(\partial U_R\) for \(R\) large, and
\begin{equation}\label{eq:rat-tangent-limit}
 \lim_{R\to\infty}\int_{\partial U_R}
       \frac{n_R(q)G(q)}{|\imq q|^2}\dsigma(q)
 =\int_{\imq\Hh}f(v)\dd V_3(v) ,
\end{equation}
the right-hand integral converging absolutely. 
\end{lemma}

In the previous theorem
neither a principal value nor
an extra contribution at the origin is involved.

\begin{proof}
A real pole \(a>0\) lies in \(U_R\) once \(R>a/2\), and a sphere
\(\alpha+\Sphere\beta\) with \(\alpha>0\) once
\(R>(\alpha^2+\beta^2)/(2\alpha)\), so that for large \(R\) every polar
component of \(H_+\) is strictly inside. For \(0\leq r<R\) write
\[
 s_R=\sqrt{R^2-r^2},
 \qquad x_\pm=R\pm s_R,
 \qquad q_\pm=x_\pm+rI ,
\]
so that the near and far parts of the boundary, in the terminology of
Subsection~\ref{subsec:complex-tangent-disks}, carry
\begin{equation}\label{eq:rat-graph-geometry}
 n_\pm=\pm\frac{s_R}{R}+I\frac rR ,
 \qquad
 \frac{\dsigma}{r^2}=\frac{R}{s_R}\dd r\,\darea(I) .
\end{equation}

\emph{The limit on a bounded range.} Uniformly in \(I\) and on every bounded
interval of \(r\),
\[
 n_-G(q_-)\frac{R}{s_R}\longrightarrow-(rI)^2f(rI)=r^2f(rI) ,
\]
the origin included, where \(G(0)=0\). 

\emph{The far part.} There \(|q_+|^2=2Rx_+\geq2R^2\), so the whole
contribution has norm at most
\begin{equation}\label{eq:rat-upper-tail}
 \frac{4\pi C}{2R^2}\int_0^R\frac{R\dd r}{\sqrt{R^2-r^2}}
 =\frac{\pi^2C}{R} ,
\end{equation}
which vanishes as \(R\to\infty\).

\emph{The tails.} For \(M\geq R_0\) and \(R>M\), the tail of the near part
has norm at most
\begin{equation}\label{eq:rat-lower-tail}
 4\pi C\int_M^R\frac{R\dd r}{r^2\sqrt{R^2-r^2}}
 =\frac{4\pi C}{M}\sqrt{1-\frac{M^2}{R^2}}
 \leq\frac{4\pi C}{M} ,
\end{equation}
since \(|q_-|\geq r\) and \(|G(q_-)|\leq C|q_-|^{-2}\); and the tail on the
hyperplane is bounded by \(4\pi C/M\) as well, local boundedness giving the
integrability near zero. Letting \(R\to\infty\) first, on \([0,M]\), and
then \(M\to\infty\), the three estimates prove the assertion. The apparent
singularity of the two graphs at \(r=R\) is integrable, the boundary itself
being smooth there.
\end{proof}

\subsection{Rational and oscillatory integrals}\label{subsec:hyp-rational}

Everything is now in place: the datum extends, the extension decays, and the
tangent balls close the hyperplane. What the residue theorem contributes is
the evaluation of the boundary integral, and the result is a finite sum of
local data at the poles of \(G=q^2f\) lying in the right half-space.

\begin{theorem}[Rational hyperplane integrals]\label{thm:rat-integral}
Let \(\varphi\) be a spherically compatible ordinary rational datum in
\(L^1(\imq\Hh,\Hh)\), let \(f=f_\varphi\) be its canonical extension and
\(G=q^2f\). List the real poles of \(G\) in \(H_+\) as \(a_1,\ldots,a_N>0\)
and its non-real polar spheres there as
\(\Sigma_j=\alpha_j+\Sphere\beta_j\), \(1\leq j\leq S\), with
\(\alpha_j,\beta_j>0\), and choose \(L_i,M_j\geq1\) for which
\[
 h_i(q)=(q-a_i)^{L_i}G(q) ,
 \qquad
 g_j(q)=\DeltaSphere{\Sigma_j}(q)^{M_j}G(q)
\]
extend slice regularly near \(a_i\) and across all of \(\Sigma_j\)
respectively. Then
\begin{equation}\label{eq:rat-integral-residues}
 \int_{\imq\Hh}\varphi(v)\dd V_3(v)
 =4\pi^2\left(\sum_{i=1}^N\Res_{a_i}G
              +\sum_{j=1}^S\SRes_{\Sigma_j}G\right) ,
\end{equation}
and, for any choice of \(y_j\in\Sigma_j\),
\begin{equation}\label{eq:rat-integral-differential}
 \int_{\imq\Hh}\varphi(v)\dd V_3(v)
 =4\pi^2\Bigg[\sum_{i=1}^N
       \frac{(\partial_c^{L_i-1}h_i)(a_i)}{(L_i-1)!}
 +\sum_{j=1}^S
\frac{[\partial_s(\partial_c\partial_s)^{M_j-1}g_j](y_j)}
            {(M_j-1)!}\Bigg] .
\end{equation}
The sums are finite and depend neither on the admissible exponents nor on
the representatives, exceptional evaluation points included.
\end{theorem}

\begin{proof}
Proposition~\ref{prop:rat-extension} and Lemma~\ref{lem:rat-L1} supply the
continuation, the regularity near the hyperplane and the uniform decay
\eqref{eq:rat-decay-f}. For every large \(R\),
Theorem~\ref{thm:4d-residue} applied to \(G\) on \(U_R\) returns a residue
sum that does not depend on \(R\), so that
Lemma~\ref{lem:rat-tangent-limit} gives
\eqref{eq:rat-integral-residues} in the limit. The differential evaluation
of the main theorem then gives
\eqref{eq:rat-integral-differential}, the regularizations being those of
\(G=q^2f\) and not of \(f\) alone.
\end{proof}

\begin{remark}\label{rem:rat-scope}
The application requires an extension regular near the whole imaginary
hyperplane, origin included, and therefore does not accommodate a function
given only on a product domain. The boundary theorem itself still covers
product domains and zero divisors, and no inverse of \(f\) or of a
regularization has been taken here; but in this class the zero-divisor
phenomenon cannot occur at all. Indeed a non-zero global rational slice
function regular near \(0\) has non-vanishing symmetrization, since
\(f^s(x)=|f(x)|^2\) on a real interval and vanishing there would force the
rational function to vanish identically. This is a property of the present
class of data, not an extra hypothesis on Theorem~\ref{thm:4d-residue}.

By \eqref{eq:rat-radial-reduction} these integrals also reduce
componentwise to one-dimensional rational integrals, whose values are of
course classical; what the theorem adds is their explicit expression through
differential spherical residues. Absolute convergence, rather than a
radially truncated principal value, is part of the statement.
\end{remark}

Before turning to the oscillatory version, it is worth separating the
mechanism from the class of data on which it has been exercised. What
Theorem~\ref{thm:rat-integral} really uses is that \(\varphi\) be the trace
on \(\imq\Hh\) of a function semiregular on a circular neighbourhood of
\(\overline{H_+}\), regular near the hyperplane, and decaying like
\(|q|^{-4}\) there; rationality entered only to make those three conditions
checkable, through \eqref{eq:rat-decay-AB}. Any other family of traces for
which they can be verified is covered by the same proof, and which
functions on an arbitrary affine hyperplane of \(\Hh\) arise in this way is
the open problem raised in the introduction. What is tied to \(\imq\Hh\)
here is only the exhausting family: the balls
\eqref{eq:rat-tangent-balls} are centred on the real axis, hence tangent to
the imaginary hyperplane and to no other, so a different hyperplane would
need an exhaustion of its own.

Inserting a phase costs nothing, because the intrinsic exponential is
bounded on the closed half-space where the balls \(U_R\) live: no
Jordan-type argument is required, exactly as in
Proposition~\ref{prop:complex-exponential-factor}. For the same integrable
data and \(t\geq0\), set
\begin{equation}\label{eq:osc-definition}
 \mathcal O_\varphi(t)=\int_{\imq\Hh}e^{-tv}\varphi(v)\dd V_3(v) ,
\end{equation}
the phase multiplying on the left. Its entire intrinsic extension satisfies
\begin{equation}\label{eq:osc-kernel}
 e^{-t(x+rI)}=e^{-tx}\bigl(\cos(tr)-I\sin(tr)\bigr) ,
 \qquad
 |e^{-tq}|=e^{-t\req q} ,
\end{equation}
and preserves both semiregularity and all finite pole orders. No commutation
with arbitrary right quaternionic coefficients is asserted.

\begin{proposition}[Absolute convergence and angular reduction]
\label{prop:osc-angular}
For every \(t\geq0\) the integral \eqref{eq:osc-definition} converges
absolutely, with \(|\mathcal O_\varphi(t)|\leq\|\varphi\|_{L^1}\), and
\begin{equation}\label{eq:osc-angular}
 \mathcal O_\varphi(t)
 =4\pi\int_0^\infty
   \bigl[r^2\cos(tr)\mathsf a(r^2)+r^3\sin(tr)\mathsf b(r^2)\bigr]\dd r .
\end{equation}
It is continuous on \([0,\infty)\), with
\(\mathcal O_\varphi(0)=\int\varphi\dd V_3\); and if
\(\int|v|^m|\varphi(v)|\dd V_3(v)<\infty\) for an integer \(m\geq1\), then
differentiation under the integral sign is valid through order \(m\), giving
\[
 \mathcal O_\varphi^{(m)}(t)
 =\int_{\imq\Hh}(-v)^me^{-tv}\varphi(v)\dd V_3(v) ,
\]
with right derivatives at zero.
\end{proposition}

\begin{proof}
On the hyperplane \(|e^{-tv}|=1\), which gives the bound. At \(v=rI\),
multiplying in the prescribed order,
\[
 e^{-trI}\varphi(rI)
 =\cos(tr)\mathsf a+r\sin(tr)\mathsf b
  +I\bigl(r\cos(tr)\mathsf b-\sin(tr)\mathsf a\bigr) ,
\]
whose \(I\)-term integrates to zero over each full sphere; Fubini then gives
\eqref{eq:osc-angular}. Continuity and, under the moment assumption, the
derivative formula follow by dominated convergence, the lower moments being
integrable because \(|v|^k\leq1+|v|^m\) for \(k\leq m\).
\end{proof}

\begin{theorem}[Oscillatory integrals by differential spherical residues]
\label{thm:osc-integral}
Let \(\mathsf a(s),\mathsf b(s)\) have real rational components, no poles on
\([0,\infty)\), and satisfy \eqref{eq:rat-decay-AB}. Set
\(\varphi(v)=\mathsf a(|v|^2)+v\,\mathsf b(|v|^2)\),
\(f(q)=\mathsf a(-q^2)+q\,\mathsf b(-q^2)\) and \(G_t(q)=q^2e^{-tq}f(q)\),
and keep the polar components \(a_i,\Sigma_j\) and the regularizations
\(h_i,g_j\) of Theorem~\ref{thm:rat-integral}. Then, for every \(t\geq0\),
\begin{equation}\label{eq:osc-residue}
 \mathcal O_\varphi(t)
 =4\pi^2\left(\sum_{i=1}^N\Res_{a_i}G_t
              +\sum_{j=1}^S\SRes_{\Sigma_j}G_t\right) ,
\end{equation}
and, for any \(y_j\in\Sigma_j\),
\begin{equation}\label{eq:osc-differential}
 \mathcal O_\varphi(t)=4\pi^2\Bigg[
  \sum_{i=1}^N
  \frac{[\partial_c^{L_i-1}(e^{-tq}h_i)](a_i)}{(L_i-1)!}+\sum_{j=1}^S
\frac{[\partial_s(\partial_c\partial_s)^{M_j-1}
                  (e^{-tq}g_j)](y_j)}{(M_j-1)!}\Bigg] .
\end{equation}
The integral converges absolutely and the finite sums depend neither on the
admissible exponents nor on the representatives; exceptional evaluation
points are allowed, and no inverse of \(f\), of \(g_j\) or of a
symmetrization is taken.
\end{theorem}

\begin{proof}
The function \(f_t=e^{-tq}f\) satisfies the hypotheses of
Lemma~\ref{lem:rat-tangent-limit}, since \(|f_t(q)|\leq|f(q)|\) for
\(\req q\geq0\) and \(t\geq0\) by \eqref{eq:osc-kernel}, and its finite
polar components are those of \(f\). Applying
Theorem~\ref{thm:4d-residue} to \(G_t=q^2f_t\) on the same balls \(U_R\)
and letting \(R\to\infty\) proves \eqref{eq:osc-residue}. The
regularizations of \(G_t\) are \(e^{-tq}h_i\) and \(e^{-tq}g_j\), only
intrinsic factors having been moved past one another, and their differential
evaluations give \eqref{eq:osc-differential}.
\end{proof}

The residues that appear here are computable in closed form, and this makes
the whole family of integrals elementary.

\begin{corollary}[Finite exponential--polynomial form]\label{cor:osc-finite}
There are quaternion-valued polynomials \(\pi_i,\theta_j,\vartheta_j\), of
degrees at most \(L_i-1\) and \(M_j-1\) respectively, such that
\begin{equation}\label{eq:osc-exponential-polynomial}
 \mathcal O_\varphi(t)
 =\sum_{i=1}^Ne^{-a_it}\pi_i(t)
  +\sum_{j=1}^Se^{-\alpha_jt}
      \bigl(\theta_j(t)\cos(\beta_jt)
            +\vartheta_j(t)\sin(\beta_jt)\bigr) .
\end{equation}
In particular the integral decays exponentially, up to a polynomial factor,
as \(t\to+\infty\); an empty sum is identically zero.
\end{corollary}

\begin{proof}
Let \(\varrho\) be the complex representative of an intrinsic rational
component of \(q^2f\), with principal part
\(\sum_{k=1}^{\mu}c_{-k}(z-p)^{-k}\) at \(p\). By
\eqref{eq:complex-exponential-residue-principal-part},
\[
 \Res_p(e^{-tz}\varrho(z))
 =e^{-tp}\sum_{k=1}^{\mu}\frac{(-t)^{k-1}}{(k-1)!}c_{-k} ,
\]
and conjugate poles carry conjugate Laurent coefficients, so their
contributions combine into the stated cosine--sine form. Recombining the
right quaternionic coefficients and absorbing \(4\pi^2\) gives
\eqref{eq:osc-exponential-polynomial}; the regularizations bound the
ordinary pole orders by \(L_i\) or \(M_j\), and every centre occurring in
the finite sum has positive real part.
\end{proof}

\begin{remark}[What is, and what is not, being evaluated]\label{rem:osc-scope}
The imaginary unit \(I=v/|v|\) varies with the integration point, so
\(\mathcal O_\varphi\) is not the three-dimensional Fourier transform with a
fixed frequency vector: by \eqref{eq:osc-angular} it couples a cosine
integral of \(\mathsf a\) with a sine integral of \(\mathsf b\), and the odd
part \(v\,\mathsf b(|v|^2)\) of the datum can therefore contribute for
\(t>0\), as Example~\ref{ex:quartic} will show. A fixed-unit phase, or the
reversed multiplication order, is not covered by the formula. Nor does the
finite expression of Corollary~\ref{cor:osc-finite} justify by itself
differentiating the original integral at zero, which needs the moment
assumption of Proposition~\ref{prop:osc-angular}; and the restrictions on
the class of data recorded in Remark~\ref{rem:rat-scope} remain in force.
\end{remark}

\subsection{Examples}\label{subsec:hyp-examples}

The first example is a two-parameter family covering all the phenomena at
once: a sphere of spherical order four, an odd datum contributing only for
\(t>0\), an exceptional evaluation point at which the regularization
vanishes, and a degeneration to a real double pole.

\begin{example}[The quartic family]\label{ex:quartic}
Let \(a,b>0\), put \(\kappa=a^2+b^2\), and set
\begin{equation}\label{eq:rat-quartic}
 \begin{aligned}
 \Sigma_\pm&=\pm a+\Sphere b, &
 \DeltaSphere{\Sigma_\pm}(q)&=(q\mp a)^2+b^2,\\
 \Pi(q)&=\DeltaSphere{\Sigma_+}(q)\DeltaSphere{\Sigma_-}(q)
       =q^4+2(b^2-a^2)q^2+\kappa^2, &
 K_{a,b}(s)&=(s-\kappa)^2+4a^2s .
 \end{aligned}
\end{equation}
For arbitrary right coefficients \(C,D,E\in\Hh\), consider
\begin{equation}\label{eq:osc-higher-datum}
 \begin{aligned}
 \varphi(v)&=K_{a,b}(|v|^2)^{-1}C+K_{a,b}(|v|^2)^{-2}D
                         +vK_{a,b}(|v|^2)^{-2}E,\\
 f(q)&=\Pi(q)^{-1}C+\Pi(q)^{-2}D+q\,\Pi(q)^{-2}E ,
 \end{aligned}
\end{equation}
whose denominator is positive for \(s\geq0\), so that all the integrability
conditions hold. Of the two possible polar spheres \(\Sigma_\pm\), only
\(\Sigma_+\) lies in \(H_+\) and contributes.

For \(G=q^2f\), the exponent \(M=2\) always regularizes \(\Sigma_+\):
\begin{equation}\label{eq:rat-higher-reg}
 g_2=\DeltaSphere{\Sigma_+}^2G=\Psi_0C+\Psi_1D+\Psi_2E,
 \quad
 \Psi_0=\frac{\DeltaSphere{\Sigma_+}q^2}{\DeltaSphere{\Sigma_-}},\quad
 \Psi_1=\frac{q^2}{\DeltaSphere{\Sigma_-}^2},\quad
 \Psi_2=\frac{q^3}{\DeltaSphere{\Sigma_-}^2} ,
\end{equation}
each \(\Psi_\ell\) being intrinsic and regular near \(\Sigma_+\). The
alternating derivative of such a factor has a closed form: for \(r>0\) and
an intrinsic regular \(\Psi(x+Ir)=u(x,r)+Iw(x,r)\), the definitions and the
Cauchy--Riemann equations give
\begin{equation}\label{eq:osc-M2}
 [\partial_s\partial_c\partial_s\Psi](x+Ir)
 =\frac{w-r\,\req\Psi'(x+Ir)}{2r^3} ,
 \qquad \Psi'=\partial_c\Psi ,
\end{equation}
since \(\partial_s\Psi=w/r\) and \(w_r=\req\Psi'\). At \(y=a+Ib\), where
\(\DeltaSphere{\Sigma_-}(y)=4ay\), the three regular factors have jets
\begin{equation}\label{eq:osc-H-jets}
 \begin{aligned}
 \Psi_0(y)&=0,& \Psi_0'(y)&=-\frac{b^2}{2a}+I\frac b2,\\
 \Psi_1(y)&=\frac1{16a^2},&
 \Psi_1'(y)&=\frac{-b^2-Iab}{16a^3\kappa},\\
 \Psi_2(y)&=\frac{y}{16a^2},&
 \Psi_2'(y)&=\frac{a-Ib}{16a^3} ,
 \end{aligned}
\end{equation}
and, for \(\Psi_{\ell,t}=e^{-tq}\Psi_\ell\), inserting
\(\Psi_{\ell,t}(y)=e^{-ty}\Psi_\ell(y)\) and
\(\Psi_{\ell,t}'(y)=e^{-ty}(\Psi_\ell'(y)-t\Psi_\ell(y))\) into
\eqref{eq:osc-M2}, Theorem~\ref{thm:osc-integral} yields
\begin{equation}\label{eq:osc-higher-result}
 \mathcal O_\varphi(t)
 =\mathcal F_0(t)C+\mathcal F_1(t)D+\mathcal F_2(t)E
\end{equation}
with the real scalar coefficients
\begin{center}
\renewcommand{\arraystretch}{1.8}
\begin{tabular}{clc}
\toprule
\(\ell\) & \(\mathcal F_\ell(t)\) & \(\mathcal F_\ell(0)\)\\
\midrule
\(0\) & \(\displaystyle\frac{\pi^2e^{-at}}a
 \left(\cos(bt)-\frac ab\sin(bt)\right)\)
 & \(\displaystyle\frac{\pi^2}a\)\\
\(1\) & \(\displaystyle\frac{\pi^2e^{-at}}{8a^3\kappa}
 \left[\left(1+\frac{a\kappa}{b^2}t\right)\cos(bt)
              -\frac{a^3}{b^3}\sin(bt)\right]\)
 & \(\displaystyle\frac{\pi^2}{8a^3\kappa}\)\\
\(2\) & \(\displaystyle\frac{\pi^2e^{-at}}{8a^3b^3}
 \left[a^2bt\cos(bt)+(b^2-a^2+ab^2t)\sin(bt)\right]\)
 & \(0\)\\
\bottomrule
\end{tabular}
\end{center}
equivalently, \([\partial_s\partial_c\partial_s(e^{-tq}\Psi_\ell)](y)
=\mathcal F_\ell(t)/(4\pi^2)\) for \(\ell=0,1,2\). Two remarks on the
exponents: the \(C\)-summand would need only \(M=1\) when non-zero, so
computing it with \(M=2\) is itself a check of the nonminimal
regularization, while the \(D\)-summand genuinely has spherical order four.
At \(t=0\) the table gives
\begin{equation}\label{eq:rat-higher-integral}
 \int_{\imq\Hh}\varphi(v)\dd V_3(v)
 =\frac{\pi^2}{a}C+\frac{\pi^2}{8a^3\kappa}D .
\end{equation}

It is worth comparing \(\mathcal F_0\) with the scalar complex integral
without the radial weight,
\begin{equation}\label{eq:cx-osc-pair}
 \int_\R\frac{e^{\ci tx}}{K_{a,b}(x^2)}\dd x
 =\frac{\pi e^{-at}}{2a\kappa}
       \left(\cos(bt)+\frac ab\sin(bt)\right) ,
\end{equation}
which carries the opposite sign in front of the sine. The discrepancy is not
a matter of normalization: the hyperplane integral carries the weight
\(r^2\), and the two integrands, differing by that factor, pick up the two
conjugate poles with different relative phases.

\emph{The odd contribution.} Taking \(a=b=1\), \(C=D=0\) and \(E=\qk\), the
table gives
\begin{equation}\label{eq:osc-odd-example}
 \int_{\imq\Hh}e^{-tv}\frac{v\qk}{(|v|^4+4)^2}\dd V_3(v)
 =\frac{\pi^2}{8}te^{-t}(\cos t+\sin t)\qk ,
\end{equation}
an integral that vanishes at \(t=0\), as any odd datum must, but not for
\(t>0\): this is the phenomenon announced in
Remark~\ref{rem:osc-scope}.

\emph{An exceptional evaluation point.} Take now \(C=0\), \(D=\qj-\qk\) and
\(E=\qk\), still with \(a=b=1\). At \(y_0=1+\qi\) the regularization
vanishes,
\[
 (\DeltaSphere{\Sigma_+}^2G_t)(y_0)
 =\frac{e^{-ty_0}}{16}\bigl(D+y_0E\bigr)=0 ,
\]
and yet \(y_0\) is a perfectly legitimate evaluation point, the formula
returning
\begin{equation}\label{eq:osc-exceptional-result}
 \begin{aligned}
 \mathcal O_\varphi(t)=\frac{\pi^2e^{-t}}{16}\Bigl\{&
    [(1+2t)\cos t-\sin t]\qj\\
   &+[(1+2t)\sin t-\cos t]\qk\Bigr\} ,
 \end{aligned}
\end{equation}
of value \(\pi^2(\qj-\qk)/16\) at zero. At no stage is there a division by
the vanishing regularized value.

\emph{Coalescence to a real double pole.} Letting \(b\to0^+\) with
\(D=E=0\), the two conjugate spheres merge into the real point \(a\) and
\(K_{a,b}(s)\to(s+a^2)^2\). Since \(K_{a,b}(s)\geq(s+a^2)^2/4\) for
\(s\geq0\) and \(0\leq b\leq a/2\), dominated convergence gives
\begin{equation}\label{eq:osc-real}
 \int_{\imq\Hh}e^{-tv}\frac{C}{(|v|^2+a^2)^2}\dd V_3(v)
 =\frac{\pi^2}{a}e^{-at}(1-at)C ,
\end{equation}
and the limit is confirmed by the real formula: the extension is
\((a^2-q^2)^{-2}C\), whose regularization at the positive real pole \(a\) is
\(h(q)=q^2(q+a)^{-2}C\), with \(h(a)=C/4\) and \(h'(a)=C/(4a)\). For
\(C\neq0\) the right derivative at zero is \(-2\pi^2C\), although the
formally differentiated integrand, odd but not absolutely integrable, has
infinite first moment: differentiating under the integral sign would not
have been legitimate here, and the value comes from the closed form instead.
\end{example}

The second example puts two spheres inside the half-space, with coefficients
that do not commute.

\begin{example}[Two spheres and non-commuting coefficients]
\label{ex:two-inside}
Take
\[
 \varphi(v)=\frac{\qi}{|v|^4+4}+\frac{\qj}{|v|^4+6|v|^2+25} ,
 \qquad
 f(q)=(q^4+4)^{-1}\qi+(q^4-6q^2+25)^{-1}\qj ,
\]
whose polar spheres in \(H_+\) are \(1+\Sphere\) and \(2+\Sphere\). Each
summand being regular near the other sphere, the two local residues are
those of the quartic family with \((a,b)=(1,1)\) and \((2,1)\), and
Theorem~\ref{thm:osc-integral} gives
\begin{equation}\label{eq:osc-two-spheres}
 \begin{aligned}
 \mathcal O_\varphi(t)
 ={}&\pi^2e^{-t}(\cos t-\sin t)\qi\\
 &+\frac{\pi^2}{2}e^{-2t}(\cos t-2\sin t)\qj ,
 \end{aligned}
\end{equation}
so that in particular
\begin{equation}\label{eq:rat-two-inside}
 \int_{\imq\Hh}\varphi(v)\dd V_3(v)=\pi^2\qi+\frac{\pi^2}{2}\qj .
\end{equation}
The datum is not a slice-preserving function times a fixed quaternion: the
coefficients satisfy \([\qi,\qj]=2\qk\neq0\) and the two radial profiles are
distinct. Finally, every ball \(B(R,R)\) with \(R>5/4\) contains both
positive spheres and neither negative one, so the limit
\eqref{eq:rat-tangent-limit} is reached with a residue sum that is constant
in \(R\).
\end{example}

The three coefficients \(\mathcal F_0,\mathcal F_1,\mathcal F_2\) can also
be obtained without the residue theorem, by one-variable methods. We record
that computation, both as a check and because the two routes are
instructive to compare.

\begin{remark}[Independent verification of the three coefficients]
\label{rem:hyp-independent-check}
Regard \(a,\kappa>0\) as independent parameters, write
\(K(s)=(s-\kappa)^2+4a^2s\) and \(\lambda=4a^2/\kappa-2>-2\). Scaling
\(r=\sqrt\kappa\,x\), then using \(x\mapsto1/x\) and \(u=x-1/x\),
\[
 \begin{aligned}
 \mathcal N(a,\kappa)&=4\pi\int_0^\infty\frac{r^2\dd r}{K(r^2)}
       =\frac{4\pi}{\sqrt\kappa}
         \int_0^\infty\frac{x^2\dd x}{x^4+\lambda x^2+1},\\
 2\int_0^\infty\frac{x^2\dd x}{x^4+\lambda x^2+1}
 &=\int_\R\frac{\dd u}{u^2+\lambda+2}
   =\frac{\pi}{\sqrt{\lambda+2}},
 \qquad \mathcal N(a,\kappa)=\frac{\pi^2}{a} .
 \end{aligned}
\]
With \(J_j=\int_{\R^3}|v|^jK(|v|^2)^{-2}\dd V_3(v)\), differentiating in the
parameters under locally uniform integrable majorants gives
\[
 \partial_a\mathcal N=-8aJ_2=-\frac{\pi^2}{a^2},
 \qquad
 \partial_\kappa\mathcal N=2(J_2-\kappa J_0)=0,
 \qquad
 J_0=\frac{\pi^2}{8a^3\kappa} ,
\]
which are the two values of \(\mathcal F_0(0)\) and \(\mathcal F_1(0)\).

For \(\kappa=a^2+b^2\), factor
\(K(r^2)=((r-b)^2+a^2)((r+b)^2+a^2)\) and close
\(z^2e^{\ci tz}/K(z^2)\) in the upper half-plane: the poles there are
\(\zeta_\pm=\pm b+\ci a\), with residues
\(\pm\zeta_\pm e^{\ci t\zeta_\pm}/(8\ci ab)\), and the arc contributes
\(O(R^{-1})\) for \(t\geq0\), whence
\[
 4\pi\int_0^\infty\frac{r^2\cos(tr)}{K(r^2)}\dd r
 =4\pi^2\ci\left(
   \frac{\zeta_+e^{\ci t\zeta_+}}{8\ci ab}
  -\frac{\zeta_-e^{\ci t\zeta_-}}{8\ci ab}\right)
 =\mathcal F_0(t) .
\]
Finally put \(F(a,\kappa;t)=\mathcal F_0(t)\) with
\(b=\sqrt{\kappa-a^2}\), and
\(T_j=4\pi\int_0^\infty r^{2+j}\cos(tr)K(r^2)^{-2}\dd r\) for \(j=0,2\).
Holding \(\kappa\) fixed in \(\partial_a\) and \(a\) fixed in
\(\partial_\kappa\),
\[
 \partial_aF=-8aT_2,
 \qquad
 \partial_\kappa F=2(T_2-\kappa T_0),
 \qquad
 T_0=-\frac{\partial_aF}{8a\kappa}-\frac{\partial_\kappa F}{2\kappa}
    =\mathcal F_1(t) ,
\]
and, \(r^3/K(r^2)^2\) being integrable,
\[
 4\pi\int_0^\infty\frac{r^3\sin(tr)}{K(r^2)^2}\dd r
 =-\mathcal F_1'(t)=\mathcal F_2(t) .
\]
\end{remark}

\section*{Acknowledgements}

The author is a member of the Gruppo Nazionale per le Strutture
Algebriche, Geometriche e le loro Applicazioni (GNSAGA) of the
Istituto Nazionale di Alta Matematica ``Francesco Severi'' (INdAM)
and gratefully acknowledges its support.

\subsection*{Declaration on the use of generative artificial intelligence}

During the preparation of this manuscript, the author used OpenAI's
GPT-5.6 Sol for language editing and
grammatical revision, and as a computational aid in working out the
examples. All AI-assisted computations were subsequently checked
independently by the author. The author reviewed and edited all AI-assisted
material and takes full responsibility for the content of the manuscript.

\subsection*{Conflict of interest}

The author declares that there is no conflict of interest.

\end{document}